\documentclass[11pt]{amsbook}

\usepackage{amssymb,amscd,amsthm,amsxtra}
\usepackage{latexsym}

\bibliography{refs}

\usepackage{hyperref}

\usepackage[mathscr]{euscript}
\renewcommand{\mathcal}[1]{{\mathscr#1}}

\usepackage{graphicx,epstopdf,color}

\newtheorem{theorem}{Theorem}[section]
\newtheorem{corollary}[theorem]{Corollary}
\newtheorem{lemma}[theorem]{Lemma}
\newtheorem{prop}[theorem]{Proposition}

\theoremstyle{definition}
\newtheorem{defn}[theorem]{Definition}
\theoremstyle{remark}
\newtheorem{rem}[theorem]{Remark}
\numberwithin{section}{chapter}
\numberwithin{equation}{section}

\newcommand{\R}{{\mathbb R}}
\newcommand{\N}{{\mathbb N}}

\renewcommand{\leq}{\leqslant}
\renewcommand{\le}{\leqslant}
\renewcommand{\geq}{\geqslant}
\renewcommand{\ge}{\geqslant}

\newcommand{\eps}{\varepsilon }
\renewcommand{\epsilon}{\varepsilon }

\def\Xint#1{\mathchoice
{\XXint\displaystyle\textstyle{#1}}%
{\XXint\textstyle\scriptstyle{#1}}%
{\XXint\scriptstyle\scriptscriptstyle{#1}}%
{\XXint\scriptscriptstyle\scriptscriptstyle{#1}}%
\!\int}
\def\XXint#1#2#3{{\setbox0=\hbox{$#1{#2#3}{\int}$ }
\vcenter{\hbox{$#2#3$ }}\kern-.6\wd0}}
\def\dashint{\Xint-}

\newlength{\defbaselineskip}
\newcommand{\setlinespacing}[1]
           {\setlength{\baselineskip}{#1 \defbaselineskip}}

\author[S. Dipierro]{Serena Dipierro}
\address[Serena Dipierro]{School of Mathematics and Statistics,
University of Melbourne,
813 Swanston Street, Parkville VIC 3010, Australia}
\email{serena.dipierro@ed.ac.uk}

\author[M. Medina]{Mar\'{i}a Medina}
\address[Mar\'{i}a Medina]{Departamento de Matem\'aticas, Universidad Aut\'onoma de Madrid,
        28049, Madrid, Spain. }
\email{maria.medina@uam.es}

\author[E. Valdinoci]{Enrico Valdinoci}
\address[Enrico Valdinoci]{School of Mathematics and Statistics,
University of Melbourne,
813 Swanston Street, Parkville VIC 3010, Australia, and
Universit\`a degli studi di Milano,
Via Saldini 50, 20133 Milan, Italy, and
Weierstra{\ss} Institut f\"ur 
Angewandte Analysis und Stochastik, Hausvogteiplatz 11A, 10117 Berlin, Germany.}
\email{enrico.valdinoci@wias-berlin.de}

\subjclass[2010]{35A15, 35B40, 35D30, 35J20, 35R11, 49N60.}

\keywords{Fractional equation, critical problem, concentration-compactness principle, 
Mountain Pass Theorem.}

\thanks{{\it Acknowledgements}.
The first author has been supported by EPSRC grant  EP/K024566/1
\emph{Monotonicity formula methods for nonlinear PDEs}
and by the Alexander von Humboldt Foundation.
The second author has been supported by projects  MTM2010-18128 and MTM2013-40846-P, MINECO.
The third author has been supported by ERC grant 277749 \emph{EPSILON Elliptic
Pde's and Symmetry of Interfaces and Layers for Odd Nonlinearities}
and PRIN grant
201274FYK7 \emph{Critical Point Theory
and Perturbative Methods for Nonlinear Differential Equations}.
}

\title[A fractional problem]{Fractional elliptic problems\\
with critical growth\\
in the whole of $\R^n$}

\begin{document}

\maketitle

{\small
\tableofcontents 
}

\chapter*{Abstract}
This is a research monograph devoted
to the analysis of a nonlocal equation in the whole of
the Euclidean space. 
In studying this equation, we will introduce all the necessary
material in the most self-contained way as possible,
giving precise reference to the literature when necessary.

In further detail,
we study here
the following nonlinear and nonlocal elliptic equation in~$\R^n$
$$
(-\Delta)^s u = \epsilon\,h\,u^q + u^p \ {\mbox{ in }}\R^n, 
$$
where~$s\in(0,1)$, $n>2s$, $\epsilon>0$ is a small parameter, $p=\frac{n+2s}{n-2s}$, $q\in(0,1)$, and~$h\in L^1(\R^n)\cap L^\infty(\R^n)$. 
The problem has a variational structure, and this allows us to find 
a positive solution by looking at critical points of a suitable energy 
functional. In particular, in this monograph, we find a local minimum 
and a different
solution of this functional (this second solution
is found by a contradiction argument which uses a
mountain pass technique, so the solution is not necessarily proven to
be of mountain pass type).
 
One of the crucial ingredient in the proof is the use of a suitable
Concentration-Compactness 
principle.  

Some difficulties arise from 
the nonlocal structure of the problem and from the fact that we deal with an 
equation in the whole of~$\R^n$ (and this causes lack of compactness of some 
embeddings). We overcome these difficulties 
by looking at an equivalent extended problem. 
\medskip

This monograph is organized as follows.

Chapter \ref{90f56rfxFFFj}
gives an elementary introduction to
the techniques involved, providing also some motivations
for nonlocal equations and auxiliary remarks on critical point theory.

Chapter \ref{intro-mono} gives a detailed description of the class of problems
under consideration (including the main equation studied in this
monograph) and provides further motivations.

Chapter \ref{FAS} introduces the analytic setting necessary
for the study of nonlocal and nonlinear equations
(this part is of rather general interest, since 
the functional analytic setting is common in different
problems in this area).

The research oriented part of the monograph is mainly concentrated
in Chapters \ref{ECXMII}, \ref{7xucjhgfgh345678}
and \ref{EMP:CHAP}
(as a matter of fact, Chapter~\ref{7xucjhgfgh345678} may also be of general
interest, since it deals with a regularity theory
for a general class of equations).

\chapter{Introduction}\label{90f56rfxFFFj}

This research monograph deals with a nonlocal problem with critical
nonlinearities. The techniques used are variational and they
rely on classical variational methods, such as the Mountain Pass Theorem
and the Concentration-Compactness Principle (suitably adapted, in order
to fit with the nonlocal structure of the problem under consideration).
The subsequent sections will give a brief introduction
to the fractional Laplacian and to the variational methods exploited.

Of course, a comprehensive introduction goes far beyond the scopes
of a research monograph, but we will try to let the interested
reader get acquainted with the problem under consideration and
with the methods used in a rather self-contained form,
by keeping the discussion at the simplest possible level
(but trying to avoid oversimplifications).
The expert reader may well skip this initial overview and
go directly to Chapter~\ref{intro-mono}.

\section{The fractional Laplacian}\label{pruzzo}

The operator dealt with in this paper is the so-called
fractional Laplacian.

For a ``nice'' function~$u$ (for instance, if~$u$
lies in the Schwartz Class of smooth and rapidly
decreasing functions), the $s$ power of the Laplacian,
for~$s\in(0,1)$, can be easily defined in the Fourier frequency space.
Namely, by taking the Fourier transform
$$ \hat u(\xi) = {\mathcal{F}} u(\xi)=
\int_{\R^n} u(x)\,e^{-2\pi i x\cdot\xi} \,dx,$$
and by looking at the Fourier Inversion Formula
$$ u(x) = {\mathcal{F}}^{-1} \hat u(x)=
\int_{\R^n} \hat u(\xi)\,e^{2\pi i x\cdot\xi} \,d\xi,$$
one notices that the derivative (say, in the $k$th coordinate
direction) in the original variables
corresponds to the multiplication by~$2\pi i \xi_k$ in the frequency
variables, that is
$$ \partial_k u(x) =
\int_{\R^n} 2\pi i \xi_k \,\hat u(\xi)\,e^{2\pi i x\cdot\xi} \,d\xi
= {\mathcal{F}}^{-1} \big( 2\pi i \xi_k \,\hat u\big).
$$
Accordingly, the operator~$(-\Delta)=-\sum_{k=1}^n \partial^2_{k}$
corresponds to the multiplication by~$(2\pi |\xi|)^2$ in the frequency
variables, that is
$$ -\Delta u(x) =                              
\int_{\R^n} (2\pi |\xi|)^2 \,\hat u(\xi)\,e^{2\pi i x\cdot\xi} \,d\xi
= {\mathcal{F}}^{-1} \big( (2\pi | \xi|)^2\,\hat u\big).
$$
With this respect, it is not too surprising to define
the power $s$ of the operator~$(-\Delta)$
as the multiplication by~$(2\pi |\xi|)^{2s}$ in the frequency
variables, that is
\begin{equation}\label{DEF:1}
(-\Delta)^s u(x)                    
:= {\mathcal{F}}^{-1} \big( (2\pi | \xi|)^{2s}\,\hat u\big).
\end{equation}
Another possible approach to the fractional Laplacian
comes from the theory of semigroups and fractional 
calculus. Namely, for any~$\lambda>0$,
using the substitution~$\tau=\lambda t$ and an integration by parts, one sees that
$$ \int_0^{+\infty} t^{-s-1}(e^{-\lambda t}-1)\,dt
= \Gamma(-s)\,\lambda^{s},$$
where~$\Gamma$ is the Euler's Gamma-function.
Once again, not too surprising, one can define the fractional
power of the Laplacian by formally replacing
the positive real number~$\lambda$
with the positive operator~$-\Delta$ in the above formula, that is
$$ (-\Delta)^{s} :=
\frac{1}{\Gamma(-s)} \int_0^{+\infty} t^{-s-1}(e^{\Delta t}-1)\,dt,$$
which reads as
\begin{equation}\label{DEF:2}
(-\Delta)^{s} u(x) =
\frac{1}{\Gamma(-s)} \int_0^{+\infty} t^{-s-1}(e^{\Delta t}u(x)-u(x))\,dt.
\end{equation}
Here above, the function~$U(x,t)=e^{\Delta t}u(x)$
is the solution of the heat equation~$\partial_t U=\Delta U$
with initial datum~$U|_{t=0}=u$.

The equivalence between the two definitions
in~\eqref{DEF:1} and~\eqref{DEF:2}
can be proved by suitable elementary calculations,
see e.g.~\cite{Bucur}.

The two definitions
in~\eqref{DEF:1} and~\eqref{DEF:2} are both useful for many
properties and they give different useful pieces of information.
Nevertheless, in this monograph, we will
take another definition, which is equivalent to the
ones in~\eqref{DEF:1} and~\eqref{DEF:2}
(at least for nice functions), but which is
more flexible for our purposes.
Namely, we set
\begin{equation}\label{DEF:3}\begin{split}
(-\Delta)^s u(x)\,&:= c_{n,s}\,PV \int_{\R^n}\frac{u(x)-u(y)}{|x-y|^{n+2s}}\,dy 
\\&:=c_{n,s}\,\lim_{r\to0} \int_{\R^n\setminus B_r(x)}
\frac{u(x)-u(y)}{|x-y|^{n+2s}}\,dy,\end{split}\end{equation}
where
$$ c_{n,s}:=\frac{2^{2s}\,s\,\Gamma\left( \frac{n}{2}+s\right)}{\pi^{n/2}\,
\Gamma(1-s)}.$$
See for instance~\cite{Bucur}
for the equivalence of~\eqref{DEF:3}
with~\eqref{DEF:1} and~\eqref{DEF:2}.

Roughly speaking, our preference (at least for what concerns
this monograph) for the
definition in~\eqref{DEF:3} lies in the following
features. First of all, the definition in~\eqref{DEF:3}
is more unpleasant, but geometrically more intuitive
(and often somehow more treatable) than the ones
in~\eqref{DEF:1} and~\eqref{DEF:2}, since
it describes an incremental quotient (of differential order~$2s$)
weighted in the whole of~$\R^n$. As a consequence,
one may obtain a ``rough'' idea on how~$(-\Delta)^s$ looks like
by considering the oscillations of the original function~$u$,
suitably weighted.

Conversely, the definitions in~\eqref{DEF:1} and~\eqref{DEF:2}
are perhaps shorter and more evocative, but 
they require some ``hidden calculations''
since they involve either the Fourier transform 
or the heat flow of the function~$u$, rather than the
function~$u$ itself.

Moreover, the definition in~\eqref{DEF:3}
has straightforward probabilistic interpretations (see e.g.~\cite{Bucur}
and references therein) and can be directly generalized
to other singular integrodifferential kernels
(of course, in many cases, even when dealing in principle
with the definition in~\eqref{DEF:3},
the other equivalent definitions do provide additional results).

In addition, by taking the definition in~\eqref{DEF:3},
we do not need~$u$ to be necessarily in the
Schwartz Class, but we can look at weak, distributional
solutions, in a similar way to the theory of
classical Sobolev spaces. We refer for instance
to~\cite{DPV} for a basic discussion on the fractional Sobolev spaces
and to~\cite{SV-2} for the main functional analytic
setting needed in the study of variational problems. \medskip

To complete this short introduction to the fractional Laplacian,
we briefly describe a simple probabilistic motivation
arising from game theory on a traced space
(here, we keep the discussion at a simple, and even heuristic level,
see for instance~\cite{Bertoin}, \cite{MOL}
and the references therein
for further details). The following discussion describes
the fractional Laplacian occurring as a consequence
of a classical random process in one additional dimension
(see Figure~\ref{BROWN}).

\begin{figure}
    \centering
    \includegraphics[width=12.4cm]{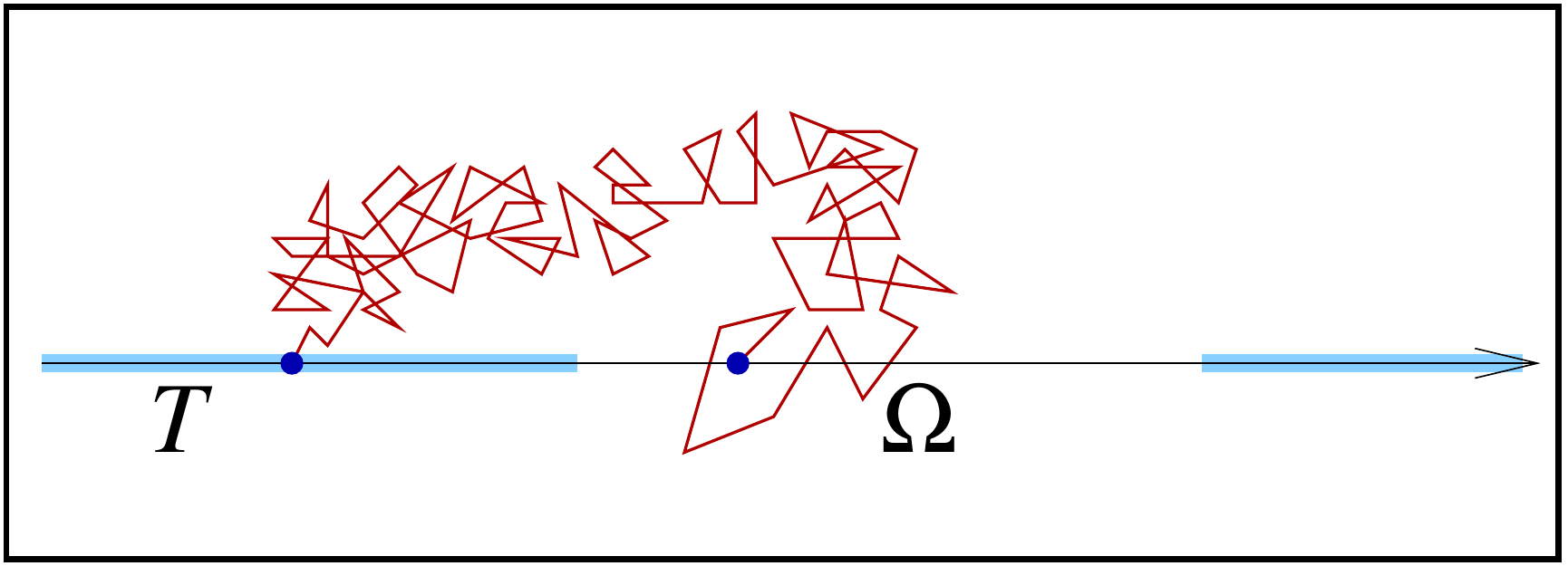}
    \caption{A Brownian motion on~$\R^{n+1}$ and the payoff
on~$\R^n\times\{0\}$.}
    \label{BROWN}
\end{figure}

We consider a bounded and smooth domain~$\Omega\subset\R^n$
and a nice (and, for simplicity, rapidly
decaying) payoff function~$f:\R^n\setminus\Omega\to [0,1]$.
We immerse this problem into~$\R^{n+1}$, by defining~$\Omega_*:=\Omega
\times\{0\}$.

The game goes as follows: we start at some point of~$\Omega_*$
and we move randomly in~$\R^{n+1}$ by following a Brownian motion,
till we hit~$(\R^n\setminus\Omega)\times\{0\}$ at some point~$p$:
in this case we receive a payoff of~$f(p)$ livres.

For any~$x\in\R^n$,
we denote by~$u(x)$ the expected value of the payoff
when we start at the point~$(x,0)\in\R^{n+1}$
(that is, roughly speaking, how much we expect to win
if we start the game from the point~$(x,0)\in\Omega\times\{0\}$).
We will show that~$u$ is solution of the fractional equation
\begin{equation}\label{EQ-LW-iu-1}
\left\{ \begin{matrix}
(-\Delta)^{1/2} u =0 & {\mbox{ in }}\Omega,\\
u =f & {\mbox{ in }}\R^n\setminus\Omega.
\end{matrix}
\right.
\end{equation}
Notice that the condition~$u=f$ in~$\R^n\setminus\Omega$
is obvious from the construction (if we start directly at a place
where a payoff is given, we get that). So the real issue
is to understand the equation satisfied by~$u$.

For this scope, for any~$(x,y)\in\R^n\times\R=\R^{n+1}$,
we denote by~$U(x,y)$ the expected value of the payoff
when we start at the point~$(x,y)$. We observe that~$U(x,0)=u(x)$.
Also, we define~$T:=(\R^n\setminus\Omega)\times\{0\}$
($T$ is our ``target'' domain) and we
claim that~$U$ is harmonic in~$\R^{n+1}\setminus T$,
i.e.
\begin{equation}\label{EQ-LW-iu-2}
{\mbox{$\Delta U=0$ in $\R^{n+1}\setminus T$.}}
\end{equation}
To prove this, we argue as follows. Fix~$P\in \R^{n+1}\setminus T$
and a small ball of radius~$r>0$ around it, such that~$B_r(P)\subset
\R^{n+1}\setminus T$. Then, the expected value
that we receive starting the game from $P$ should be the average
of the expected value that we receive starting the game from another
point~$Q\in \partial B_r(P)$ times the probability of drifting from~$Q$
to~$P$. Since the Brownian motion is rotationally invariant,
all the points on the sphere have the same probability of drifting
towards~$P$, and this gives that
$$ U(P) = \dashint_{\partial B_r(P)} U(Q)\,d{\mathcal{H}}^n(Q).$$
That is, $U$ satisfies the mean value property of
harmonic functions, and this establishes~\eqref{EQ-LW-iu-2}.

Furthermore, since the problem is symmetric with respect to the $(n+1)$th
variable we also have that~$U(x,y)=U(x,-y)$ for any~$x\in\Omega$
and so
\begin{equation}\label{EQ-LW-iu-3}
{\mbox{$\partial_y U(x,0)=0$ for any~$x\in\Omega$.}}
\end{equation}
Now we take Fourier transforms in the variable~$x\in\R^n$,
for a fixed~$y>0$. {F}rom~\eqref{EQ-LW-iu-2}, we know that~$\Delta U(x,y)=0$
for any~$x\in\R^n$ and~$y>0$, therefore
$$ -(2\pi|\xi|)^2 \hat U(\xi,y)+\partial_{yy} \hat U(\xi,y)=0,$$
for any~$\xi\in\R^n$ and~$y>0$. This is an ordinary differential equation
in~$y>0$, which
can be explicitly solved: we find that
$$ \hat U(\xi,y) = \alpha(\xi)\, e^{2\pi|\xi| y} + \beta(\xi)\, e^{-2\pi |\xi| y},$$
for suitable functions~$\alpha$ and~$\beta$.
As a matter of fact, since
$$ \lim_{y\to+\infty} e^{2\pi|\xi| y}=+\infty,$$
to keep~$\hat U$ bounded we have that~$\alpha(\xi)=0$ for any~$\xi\in\R^n$.
This gives that
$$ \hat U(\xi,y) = \beta(\xi)\, e^{-2\pi|\xi| y}.$$
We now observe that
$$ \hat u(\xi) =\hat U(\xi, 0)=\beta(\xi),$$
therefore
$$ \hat U(\xi,y) = \hat u(\xi)\, e^{-2\pi|\xi| y}$$
and so
$$ {\mathcal{F}}(\partial_y U)(\xi,y)=
\partial_y \hat U(\xi,y) = -2\pi|\xi|\,\hat u(\xi)\, e^{-2\pi|\xi| y}.$$
In particular, ${\mathcal{F}}(\partial_y U)(\xi,0)=-2\pi|\xi|\,\hat u(\xi)$.
Hence we exploit~\eqref{EQ-LW-iu-3}
(and we also recall~\eqref{DEF:1}): in this way,
we obtain that, for any~$x\in\Omega$,
$$ 0= \partial_y U(x,0)= - {\mathcal{F}}^{-1} \Big(
2\pi\,|\xi|\,\hat u(\xi)\Big) (x)= -(-\Delta)^{1/2} u(x),$$
which proves \eqref{EQ-LW-iu-1}.

\section{The Mountain Pass Theorem}

Many of the problems in mathematical analysis deal with the
construction of suitable solutions. The word ``construction''
is often intended in a ``weak'' sense, not only because the
solutions found are taken in a ``distributional'' sense,
but also because the proof of the existence of the solution
is often somehow not constructive
(though some qualitative or quantitative properties
of the solutions may be often additionally found).

In some cases, the problem taken into account presents a variational
structure, namely the desired solutions may be found as critical
points of a functional (this functional is often called ``energy''
in the literature, though it is in many cases related more to
a ``Lagrangian action'' from the physical point of view).

When the problem has a variational structure,
it can be attacked by all the methods which aim to prove
that a functional indeed possesses a critical point.
Some of these methods arise as the ``natural'' generalizations
from basic Calculus to advanced Functional Analysis:
for instance, by a variation of the classical Weierstra{\ss} Theorem,
one can find solutions corresponding to local (or sometimes global)
minima of the functional.
\medskip

In many circumstances, these minimal solutions do not exhaust
the complexity of the problem itself. For instance,
the minimal solutions happen in many cases to be ``trivial''
(for example, corresponding to the zero solution). Or, in any case,
solutions different from the minimal ones may exist, and they may
indeed have interesting properties. For example, the fact that
they come from a ``higher energy level'' may allow them to show
additional oscillations, or having ``directions'' along which
the energy is not minimized may produce some intriguing forms
of instabilities.

Detecting non-minimal solutions is of course, in principle,
harder than finding minimal ones, since the direct methods
leading to the Weierstra{\ss} Theorem (basically reducing
to compactness and some sort of continuity) are in general not enough.

As a matter of fact, these methods need to be implemented with the aid of
additional ``topological'' methods, mostly inspired by
Morse Theory (see~\cite{Milnor}). Roughly speaking, these methods
rely on the idea that critical points prevent the energy graph
to be continuously deformed by following lines of steepest
descent (i.e. gradient flows).
\medskip

One of the most important devices to detect critical points of non-minimal
type is the so called Mountain Pass Theorem.
This result can be pictorially depicted
by thinking that the energy functional is simply
the elevation ${\mathcal{E}}$ of a given point on the Earth.
The basic assumption of the Mountain Pass Theorem
is that there are
(at least) two low spots in the landscape,
for instance, the origin, which (up to translations)
is supposed to lie at the sea level (say, ${\mathcal{E}}(0)=0$)
and a far-away place~$p$ which also lies at the sea level,
or even below (say, ${\mathcal{E}}(p)\leq0$). 

The origin is also supposed to be surrounded by points
of higher elevation (namely,
there exist~$r$, $a>0$ such that~${\mathcal{E}}(u)\geq a$
if~$|u| =r$). Under this assumption,
any path joining
the origin with~$p$ is supposed to ``climb up'' some mountains
(i.e., it has to go up, at least at level~$a>0$, and then
reach again the sea level in order to reach~$p$).

Thus, each of the path joining~$0$ to~$p$ will have a highest
point. If one needs to travel in ``real life'' from~$0$ to~$p$,
then (s)he would like to minimize the value of this highest point,
to make the effort as small as possible. This corresponds,
in mathematical jargon, to the search of the value
\begin{equation}\label{78ghKKK}
c:=\inf_{\Gamma}\sup_{t\in[0,1]} {\mathcal{E}}(g(t)),
\end{equation}
where~$\Gamma$ is the collection of all possible path~$g$
such that~$g(0)=0$ and~$g(1)=p$.

Roughly speaking, one should expect~$c$ to be a critical value
of saddle type, since the ``minimal path'' has a maximum
in the direction ``transversal to the range of mountains'',
but has a minimum with respect to the tangential directions,
since the competing paths reach a higher altitude.

\begin{figure}
    \centering
    \includegraphics[height=7.5cm]{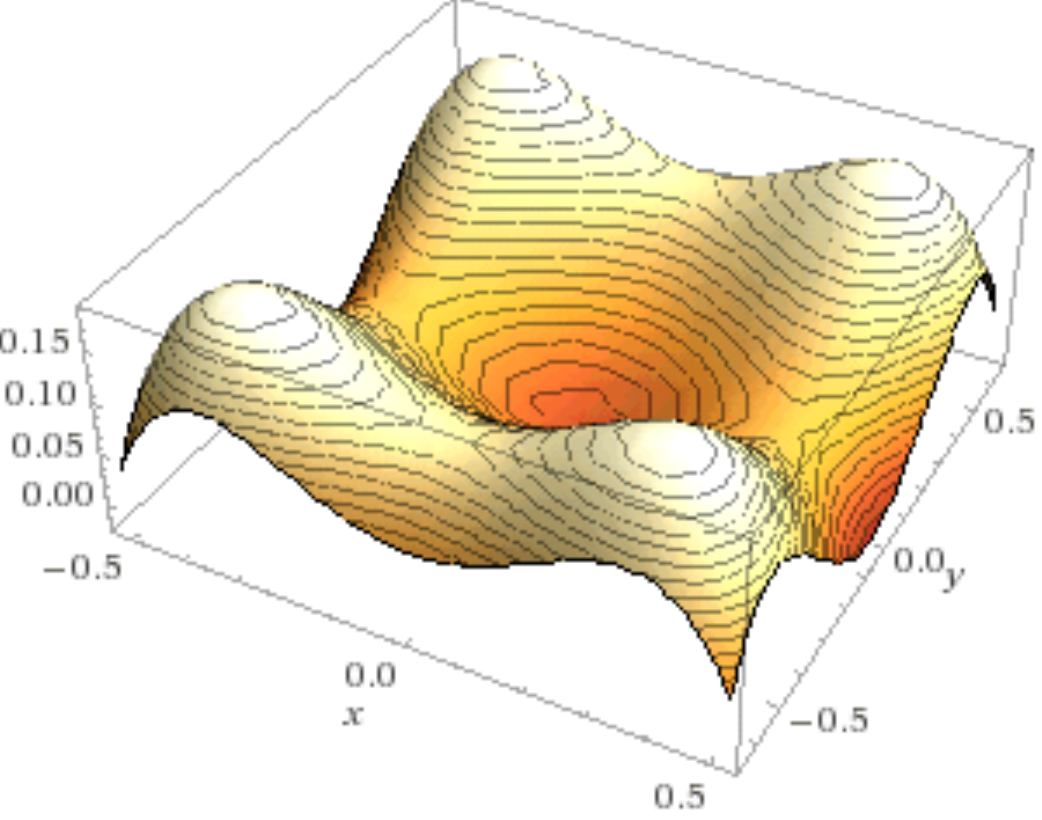} \\
    \includegraphics[height=7.5cm]{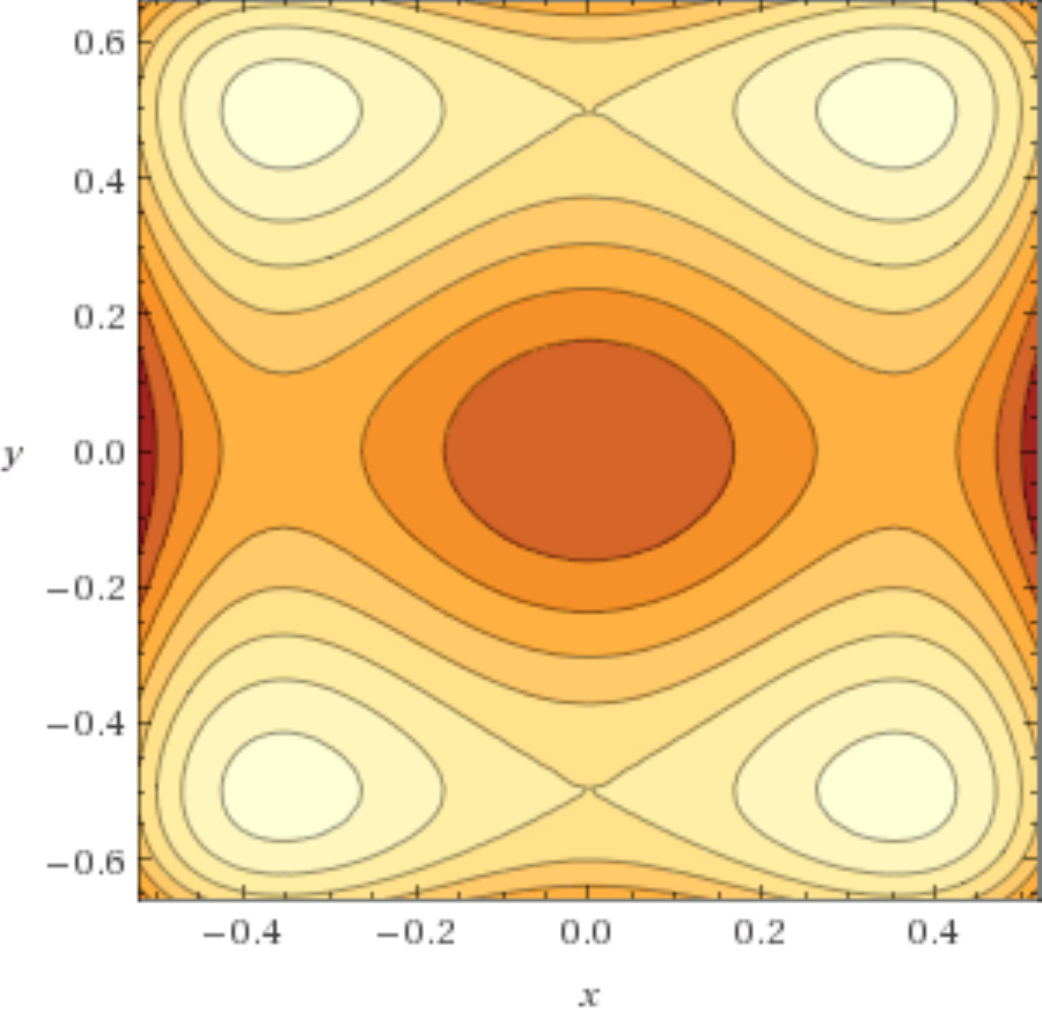}
    \caption{The function $z=x^2+y^2-4x^4-2y^4$ ($3$D plot and
level sets).}
    \label{FIG-1}
\end{figure}

A possible picture of the structure of this mountain pass
is depicted in Figure~\ref{FIG-1}.
On the other hand, to make the argument really work,
one needs a compactness condition, in order to avoid that
the critical point ``drifts to infinity''.

We stress that this loss of compactness for critical points
is not necessarily due to the fact that one works in complicate
functional spaces, and indeed simple examples
can be given even in Calculus curses, see
for instance the following example taken from
Exercise~5.42 in~\cite{Chierchia}.
One can consider the function of two real variables
$$ f(x,y)= (e^x+e^{-x^2})y^2(2-y^2)-e^{-x^2}+1.$$
By construction~$f(0,0)=0$,
\begin{eqnarray*}
\partial_x f &=& (e^x-2x e^{-x^2})y^2(2-y^2)+2xe^{-x^2}, \\
\partial_y f &=& 2(e^x+e^{-x^2})y(2-y^2)-2(e^x+e^{-x^2})y^3\\
{\mbox{and }}\qquad
D^2f(0,0) &=&\left(
\begin{matrix}
2 & 0 \\
0 & 8
\end{matrix}
\right).
\end{eqnarray*}
As a consequence, the origin is a nondegenerate local minimum for~$f$.
In addition, $f(0,\sqrt{2})=0$, so the geometry of
the mountain pass is satisfied. Nevertheless, the function~$f$
does not have any other critical points except the origin.
Indeed, a critical point should satisfy
\begin{eqnarray}
\label{R5:LK1}
&& (e^x-2x e^{-x^2})y^2(2-y^2)+2xe^{-x^2}=0 \\
\label{R5:LK2}
{\mbox{and }}&&
2(e^x+e^{-x^2})y(2-y^2)-2(e^x+e^{-x^2})y^3=0.
\end{eqnarray}
If~$y=0$, then we deduce from~\eqref{R5:LK1} that also~$x=0$,
which gives the origin. So we can suppose that~$y\ne0$
and write~\eqref{R5:LK2} as
$$ 2(e^x+e^{-x^2})(2-y^2)-2(e^x+e^{-x^2})y^2=0,$$
which, after a further simplification gives~$(2-y^2)-y^2=0$,
and therefore~$y=\pm 1$.

By inserting this into~\eqref{R5:LK1}, we obtain that
$$ 0 =
(e^x-2x e^{-x^2})+2xe^{-x^2}= e^x,$$
which produces no solutions. This shows that this example
provides no additional critical points than the origin,
in spite of its mountain pass structure.

The reason for this is that the critical point 
has somehow drifted to infinity: indeed
$$ \lim_{n\to+\infty} \nabla f( -n,1)=0.$$
\medskip

To avoid this type of pathologies of critical points\footnote{We
observe that this pathology does not occur for functions in~$\R$,
since, in one variable, the conditions~${\mathcal{E}}(0)=0$,
${\mathcal{E}}(r)=a>0$
and~${\mathcal{E}}(p)\leq0$ imply the existence of
another critical point, by 
Rolle's Theorem.}
drifting to infinity, one requires an assumption
that provides the compactness (up to subsequences)
of ``almost critical'' points.

This additional compactness assumption is called in the literature
``Palais-Smale condition'' and 
requires that if a sequence~$u_k$ is such that~${\mathcal{E}}(u_k)$
is bounded and~${\mathcal{E}}'(u_k)$ is infinitesimal, then~$u_k$
has a convergent subsequence.

We remark that in~$\R^n$
a condition of this sort
is satisfied automatically for proper maps
(i.e., for functions which do not take unbounded sets into bounded sets),
but in functional spaces the situation is definitely more delicate.
We refer to~\cite{Mawhin-Willem}
and the references therein for a throughout discussion about the
Palais-Smale condition.

\begin{figure}
    \centering
    \includegraphics[height=7.5cm]{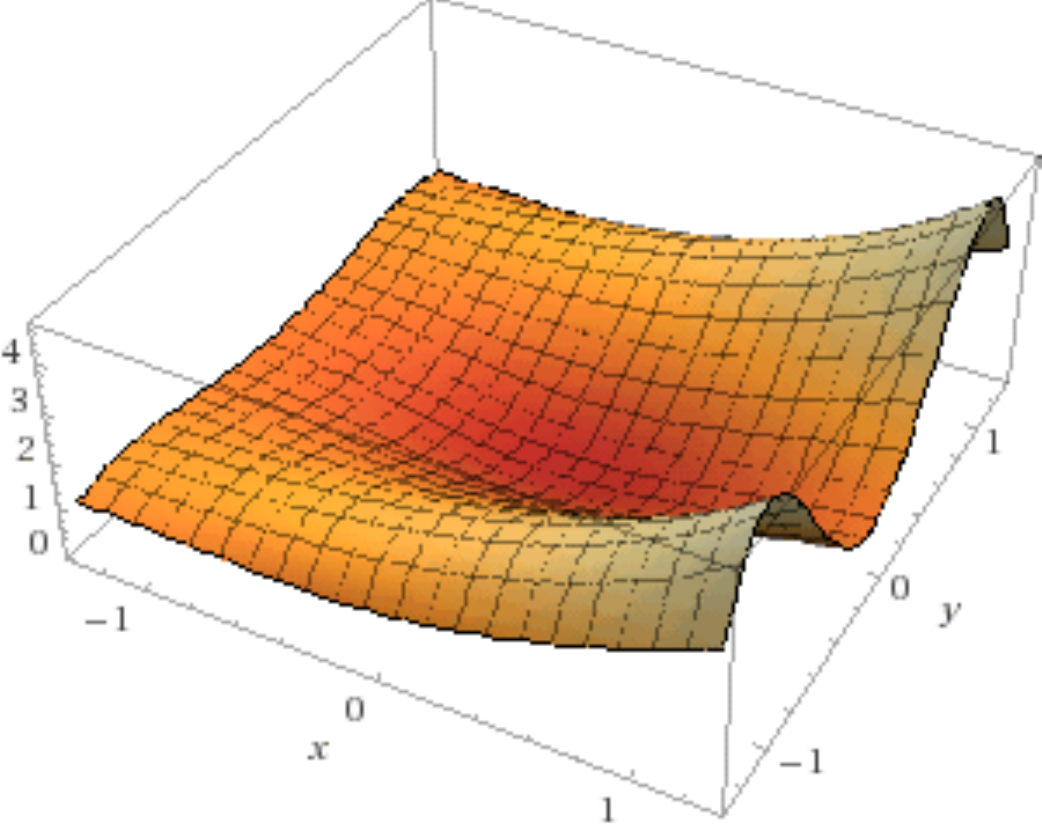} \\
    \includegraphics[height=7.5cm]{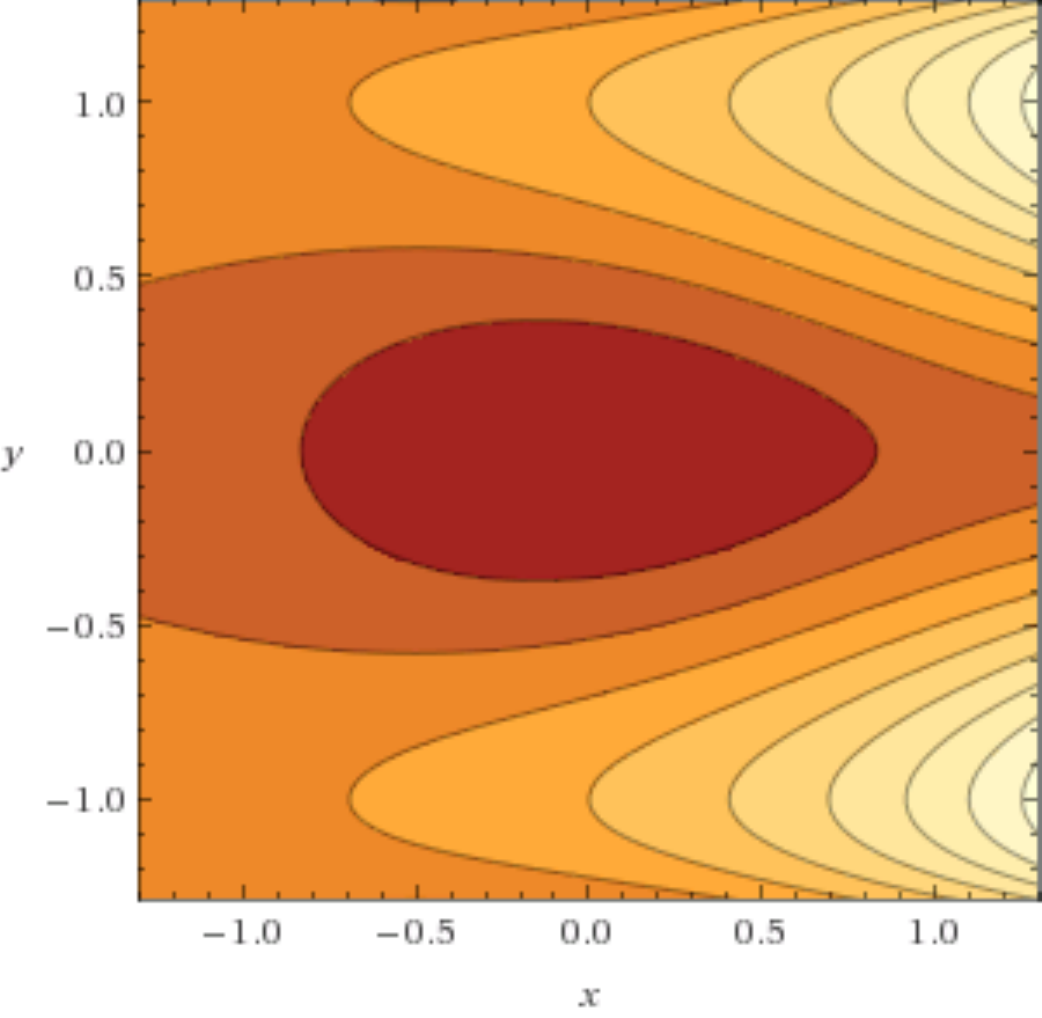}
    \caption{The function $z=(e^x+e^{-x^2})y^2(2-y^2)-e^{-x^2}+1$ ($3$D plot and
level sets).}
    \label{FIG-2}
\end{figure}

The standard version of the Mountain Pass Theorem is due to~\cite{AmRab}, and goes as follows:

\begin{theorem}\label{mpt}
Let~${\mathcal{H}}$ be a Hilbert space and let~${\mathcal{E}}$ be
in~$C^1({\mathcal{H}},\R)$. 
Suppose that there exist~$u_0$, $u_1\in {\mathcal{H}}$ and~$r>0$ such that 
\begin{eqnarray}
\label{1.5} && \inf_{\|u-u_0\|=r}{\mathcal{E}}(u)> {\mathcal{E}}(u_0),\\
\label{1.6}&& \|u_1-u_0\|> r \;{\mbox{ and }}\; {\mathcal{E}}(u_1)\le {\mathcal{E}}(u_0).
\end{eqnarray}
Suppose also that the Palais-Smale condition holds at level~$c$,
with
$$ c:=\inf_{\Gamma}\sup_{t\in[0,1]} {\mathcal{E}}(g(t)),$$
where~$\Gamma$ is the collection of all possible path~$g$
such that~$g(0)=u_0$ and~$g(1)=u_1$.

Then~${\mathcal{E}}$ has a critical point at level~$c$. 
\end{theorem}

Notice that hypothesis~\eqref{1.5} has a strict sign,
therefore it requires the existence of a real 
mountain (i.e., with a strictly positive elevation
with respect to~$u_0$) surrounding~$u_0$.

Mathematically, this condition easily follows if, for 
instance, we previously know that $u_0$ is a strict local minimum. 
Unfortunately, in the applications it is not so common to have as much 
information, being more likely to only know that $u_0$ is just a local, possibly
degenerate,
minimum. Thus, a natural question arises: what happens if we can cross 
the mountain through a flat path? 
That is, what if the separating mountain range has zero altitude?
Does the Mountain Pass Theorem hold in 
this limiting case? \medskip

The answer is yes. In \cite{GG} N. Ghoussub and N. Preiss refined the result of A. Ambrosetti and P. Rabinowitz to overcome this difficulty. Indeed, they proved that the conclusion in Theorem \ref{mpt} holds if we replace the hypotheses by
\begin{eqnarray}
\label{1.5bis}&& \inf_{\|u-u_0\|=r}{\mathcal{E}}(u)\geq {\mathcal{E}}(u_0),\\
\label{1.6bis}&& \|u_1-u_0\|> r \;{\mbox{ and }}\; {\mathcal{E}}(u_1)\le {\mathcal{E}}(u_0).
\end{eqnarray}
Now, the first condition is satisfied if we prove that $u_0$ is just a (not necessarily strict) local minimum. In fact, this will be the version of the Mountain Pass Theorem that we will apply in this monograph since, as we will see in Chapter \ref{EMP:CHAP}, our paths across the mountain will start from a point that, as far as we know, is only a local minimum.\medskip

As a matter of fact, we point out that the results in~\cite{GG}
are more general than assumptions~\eqref{1.5bis}
and~\eqref{1.6bis}, and they are based on the notion
of ``separating set''. Namely, one says that a closed set~$S$
separates~$u_0$ and~$u_1$ if~$u_0$ and~$u_1$
belong to disjoint connected
components of the complement of~$S$.

With this notion, it is proved that
the Mountain Pass Theorem holds if there exists a closed
set~$F$ such that~$F\cap \{{\mathcal{E}}\ge c\}$
separates~$u_0$ and~$u_1$ (see in particular Theorem~(1. bis)
in~\cite{GG}).\medskip

Let us briefly observe that conditions~\eqref{1.5bis}
and~\eqref{1.6bis} indeed imply the existence of
a separating set. For this, we first observe that, for any path~$g$
which joins~$u_0$ and~$u_1$, we have that
$$ \sup_{t\in[0,1]} {\mathcal{E}}(g(t)) \ge {\mathcal{E}}(g(0))
={\mathcal{E}}(u_0),$$
and so, taking the infimum, we obtain that~$c\ge {\mathcal{E}}(u_0)$.

So we can distinguish two cases:
\begin{eqnarray}
\label{primo caso} && {\mbox{either $c>{\mathcal{E}}(u_0)$}}\\
\label{secondo caso} && {\mbox{or $c={\mathcal{E}}(u_0)$}}.
\end{eqnarray}
If \eqref{primo caso} holds true, than we choose~$F$
to be the whole of the space, and we check
that the set~$F\cap \{{\mathcal{E}}\ge c\}=
\{{\mathcal{E}}\ge c\}$ separates $u_0$ and~$u_1$.

To this goal, we first notice that~$c>{\mathcal{E}}(u_0)\ge
{\mathcal{E}}(u_1)$, due to~\eqref{1.6bis}
and~\eqref{primo caso}, thus~$u_0$ and~$u_1$ belong to the complement
of~$\{{\mathcal{E}}\ge c\}$, which is~$\{{\mathcal{E}}<c\}$. They 
cannot lie in the same
connected component of such set, otherwise there would be a joining path~$g$
all contained in~$\{{\mathcal{E}}<c\}$, i.e.
$$ \sup_{t\in [0,1]}{\mathcal{E}} (g(t))=
\max_{t\in [0,1]}{\mathcal{E}} (g(t)) < c,$$
which is in contradiction with~\eqref{78ghKKK}.

This proves that
conditions~\eqref{1.5bis}
and~\eqref{1.6bis} imply the existence of
a separating set in case~\eqref{primo caso}.
Let us now deal with case~\eqref{secondo caso}.
In this case, we choose~$F=\{ u {\mbox{ s.t. }} \|u-u_0\|=r\}$,
where~$r$ is given by~\eqref{1.5bis}
and~\eqref{1.6bis}. Notice that, by~\eqref{1.5bis}
and~\eqref{secondo caso}, we know that
${\mathcal{E}}(u)\ge {\mathcal{E}}(u_0)=c$ for any~$u\in F$,
and therefore~$F\subseteq \{{\mathcal{E}}\ge c\}$.
As a consequence~$F\cap \{{\mathcal{E}}\ge c\} = F$,
which is a sphere of radius~$r$ which contains~$u_0$ in its
center and has~$u_1$ in its exterior, and so
it separates the two points, as desired.

\section{The Concentration-Compactness Principle}

The methods based on concentration and compactness, or compensated
compactness are based on a careful analysis, which
aims to recover compactness (whenever possible) from rather
general conditions. These methods are indeed very powerful
and find applications in many different contexts.
Of course, a throughout exposition of these techniques
would require a monograph in itself, so we limit ourselves
to a simple description on the application of these methods
in our concrete case. \medskip

As we have already noticed in the discussion about the Palais-Smale condition,
one of the most important difficulties to take care when
dealing with critical point theory is the possible loss of
compactness.

Of course, boundedness is the first necessary requirement
towards compactness, but in infinitely dimensional spaces
this is of course not enough. The easiest example
of loss of
compactness in spaces of functions defined in the whole of~$\R^n$
is provided by the translations: namely, given a (smooth,
compactly supported, not identically zero) function~$f$,
the sequence of functions~$U_k(x):=U(x-k e_1)$ is not precompact.

One possibility that avoids this possible ``drifting towards
infinity'' of the mass of the sequence is the request that
the sequence is ``tight'', i.e. the amount of mass at infinity
goes to zero uniformly. 
Of course, the appropriate choice of the norm used to measure
such tightness depends on the problem considered.
In our case, we are interested in controlling the weighted tail of
the gradient at infinity
(a formal statement about this will 
be given in Definition~\ref{defTight}).
\medskip

Other tools to use in order to prove compactness often rely
on boundedness of sequences in possibly stronger norms.
For instance, in variational problems, one often obtains
uniform energy bounds which control the sequence in a suitable
Sobolev norm: for instance, one may control uniformly
the $L^2$-norm of the first derivatives.
This, together with the compactness of 
the measures (see e.g.~\cite{evans}) implies that the squared 
norm of the gradient converges in the sense of measures
(formal details about this will be given in Definition~\ref{convMeasures}).
\medskip

The version of the Concentration-Compactness Principle
that will be used in
this monograph is due to P. L. Lions in \cite{L1} and~\cite{L2},
and will be explicitly recalled in Proposition~\ref{CCP},
after having introduced the necessary functional setting.

\chapter{The problem studied in this monograph}\label{intro-mono}

\section{Fractional critical problems}

A classical topic in nonlinear analysis is the study of
the existence and multiplicity of solutions
for nonlinear equations.
Typically, the equations under consideration possess some
kind of ellipticity, which translates into additional
regularity and compactness properties at a functional level.\medskip

In this framework, an important distinction
arises between ``subcritical'' problems
and ``critical'' ones. Namely, in subcritical problems
the exponent of the nonlinearity is smaller than the Sobolev
exponent, and this gives that any reasonable bound on
the Sobolev seminorm implies convergence in some $L^p$-spaces:
for instance, minimizing sequences, or Palais-Smale sequences,
usually possess naturally a uniform bound in the Sobolev
seminorm, and this endows the subcritical problems
with additional compactness properties that lead to
existence results via purely functional analytic methods.\medskip

The situation of critical problems is different,
since in this case the exponent of the nonlinearity
coincides with the Sobolev exponent and therefore
no additional $L^p$-convergence may be obtained
only from bounds in Sobolev spaces.
As a matter of fact, many critical problems do not
possess any solution. Nevertheless, as discovered in~\cite{BN},
critical problems do possess solutions once suitable lower order
perturbations are taken into account.
Roughly speaking, these perturbations are capable to
modify the geometry of the energy functional associated to the problem,
avoiding the critical points to ``drift towards infinity'',
at least at some appropriate energy level.
Of course, to make such argument work, a careful
analysis of the variational structure of the problem
is in order, joint with an appropriate use of topological methods
that detect the existence of the critical points of the functional
via its geometric features.\medskip

Recently, a great attention has also been devoted to
problems driven by nonlocal operators. In this case,
the ``classical'' ellipticity (usually modeled by the Laplace
operator) is replaced by a ``long range, ferromagnetic interaction'',
which penalizes the oscillation of the function
(roughly speaking, the function is seen as a state
parameter, whose value at a given point of the space
influences the values at all the other points, in order to
avoid sharp fluctuations). The ellipticity condition
in this cases reduces to the validity of some sort of
maximum principle, and the prototype nonlocal operators
studied in the literature are the fractional
powers of the Laplacian.\medskip

In this research monograph we deal with the problem 
\begin{equation}\label{problem}
(-\Delta)^s u = \epsilon\,h\,u^q + u^p \ {\mbox{ in }}\R^n, 
\end{equation}
where $s\in(0,1)$ and $(-\Delta)^s$ is the so-called fractional Laplacian
(as introduced in Section \ref{pruzzo}), that is 
\begin{equation}\label{laplacian}
 (-\Delta)^s u(x)= c_{n,s}\,PV \int_{\R^n}\frac{u(x)-u(y)}{|x-y|^{n+2s}}\,dy \ 
{\mbox{ for }}x\in\R^n,
\end{equation}
where $c_{n,s}$ is a suitable positive constant 
(see~\cite{DPV, Silv} for the definition and the basic properties). 
Moreover, $n>2s$, $\epsilon>0$ is a small parameter, $0<q<1$, and $p=\frac{n+2s}{n-2s}$ is 
the fractional critical Sobolev exponent.

Problems of this type have widely appeared in the literature.
In the classical case, when~$s=1$, the equation
considered here arises in differential geometry, in the context
of the so-called Yamabe problem, i.e. the
search of Riemannian metrics with constant scalar curvature.

The fractional analogue of the Yamabe problem has been
introduced in~\cite{Chang}
and was also studied in detail,
see e.g.~\cite{JIE13}, \cite{WANG15} and the references therein.

Here we suppose that $h$ satisfies 
\begin{eqnarray}
&& h\in L^1(\mathbb{R}^n)\cap L^\infty(\mathbb{R}^n),\label{h1}\\
{\mbox{and }} && {\mbox{there exists a ball~$B\subset\R^n$ such that }}
\inf_B h>0. \label{h2}
\end{eqnarray}
Notice that condition~\eqref{h1} implies that 
\begin{equation}\label{h0}
h\in L^r(\R^n) \quad {\mbox{ for any }}r\in(1,+\infty).
\end{equation}

\medskip

In the classical case, that is when~$s=1$ and the fractional Laplacian
boils down to the classical Laplacian, there is an intense literature 
regarding this type of problems, see \cite{ABC, AGP1, AGP, ALM, AM, AM1, BertiM, BN, 
CW, Cing, Da, Mal, Mal2}, and references therein. 
See also~\cite{GMP}, where the concave term appears for the first time. 

In a nonlocal setting, in~\cite{bcss} the authors deal with problem~\eqref{problem} in a bounded domain with Dirichlet boundary condition. 
Problems related to ours have been also studied in~\cite{SV-2, SV-1, SV}. 

We would also like to mention that the very recent literature is
rich of new contributions on fractional problems related
to Riemannian geometry: see for instance the articles~\cite{NEW2a}
and~\cite{NEW2b}, where a fractional Nirenberg problem is taken into account.
Differently from the case treated here, the main operator
under consideration in these papers is the square root of the identity
plus the Laplacian (here, any fractional power is considered
and no additional invertibility comes from the identity part);
also, the nonlinearity treated here is different from the
one considered in~\cite{NEW2a, NEW2b}, since the lower order terms produce new phenomena and
additional difficulties (moreover, the techniques in
these papers have a different flavor than the ones in this monograph
and are related also to the Leray-Schauder degree).\medskip

Furthermore, in~\cite{vecchio}, we find solutions to~\eqref{problem} 
by considering the equation as a perturbation of the problem 
with the fractional critical Sobolev exponent, 
that is 
$$ (-\Delta)^s u=u^{\frac{n+2s}{n-2s}} \quad {\mbox{ in }}\R^n.$$
Indeed, it is known that the minimizers of the Sobolev embedding in~$\R^n$ 
are unique, up to translations and positive dilations, and non-degenerate 
(see~\cite{vecchio} and references therein,
and also~\cite{NEW1} for related results; see also~\cite{WEI}
and the references therein for classical counterparts
in Riemannian geometry). In particular, in~\cite{vecchio}
we used perturbation methods 
and a Lyapunov-Schmidt reduction to find solutions to~\eqref{problem} that 
bifurcate\footnote{As a technical observation,
we stress that papers like~\cite{NEW1} and~\cite{vecchio}
deal with a bifurcation method
from a ground state solution that is already known to 
be non-degenerate, while in the
present monograph we find:  (i) solutions that
bifurcate from zero, (ii) different solutions obtained
by a contradiction argument (e.g., if no other solutions
exist, then one gains enough compactness to find
a mountain pass solution).

Notice that this second class of solutions is not necessarily
of mountain pass type, due to the initial
contradictory assumption.

In particular, the
methods of~\cite{NEW1} and~\cite{vecchio}
cannot be applied in the framework of this monograph.
An additional difficulty in our setting with respect to~\cite{NEW1}
is that 
we deal also with a subcritical power in the nonlinearity which makes the functional not twice differentiable and requires a different
functional setting.}
from these minimizers. The explicit form of the
fractional Sobolev minimizers was found in~\cite{COZ}
and it is given by
\begin{equation}\label{tale}
z(x):= \frac{c_\star}{\big( 1+|x|^2\big)^{\frac{n-2s}{2}}},
\end{equation}
for a suitable~$c_\star>0$, depending on~$n$ and~$s$.
\medskip

In order to state our main results, we introduce some notation. We set
$$ [u]_{\dot H^s(\R^n)}^2:= \frac{c_{n,s}}{2}\iint_{\R^{2n}}\frac{|u(x)-u(y)|^2}{|x-y|^{n+2s}}\,dx\,dy, $$
and we define the space $\dot{H}^s(\R^n)$ as the completion of
the space of smooth and rapidly decreasing functions (the so-called
Schwartz space)
with respect to the norm $[u]_{\dot H^s(\R^n)}+\|u\|_{L^{2^*_s}(\R^n)}$,
where
$$ 2^*_s=\frac{2n}{n-2s}$$
is the fractional critical exponent. 
Notice that we can also define $\dot{H}^s(\R^n)$ as the space 
of measurable functions $u:\R^n\to\R$ such that the norm 
$[u]_{\dot H^s(\R^n)}+\|u\|_{L^{2^*_s}(\R^n)}$ is finite, thanks to a 
density result, see e.g.~\cite{DipVAl}. 

Given $f\in L^\beta(\R^n)$, where $\beta:=\frac{2n}{n+2s}$,
we say that $u\in \dot{H}^s(\mathbb{R}^{n})$ is a (weak) solution to
$(-\Delta)^s u=f$ in $\R^n$ if
$$ \frac{c_{n,s}}{2}\iint_{\R^{2n}}
\frac{\big( u(x)-u(y)\big)\,\big( \varphi(x)-\varphi(y)\big)}{|x-y|^{n+2s}}
\,dx\,dy=\int_{\R^n} f\,\varphi\,dx,$$
for any $\varphi\in \dot{H}^s(\mathbb{R}^{n})$.

Thus, we can state the following
\begin{theorem}\label{TH1}
Let $0<q<1$. Suppose that $h$ satisfies \eqref{h1} and \eqref{h2}. 
Then there exists $\varepsilon_0>0$ such that for all $\varepsilon\in (0,\varepsilon_0)$ 
problem \eqref{problem} has at least two nonnegative solutions. 
Furthermore, if $h\geq 0$ then the solutions are strictly positive.
\end{theorem}

This result can be seen as the nonlocal counterpart of Theorem 1.3 in \cite{AGP}. 
To prove it we will take advantage of the variational structure of the problem. 
The idea is first to ``localize'' the problem, via the extension introduced in~\cite{CS} and consider a functional in the extended variables. 
More precisely, this extended functional will be introduced in
the forthcoming formula~\eqref{f ext}. 
It turns out that the existence of critical points 
of the ``extended'' functional implies the existence 
of critical points for the functional on the trace, 
that is related to problem~\eqref{problem}.
The functional in the original variables will be introduced in~\eqref{f giu}, 
see Section~\ref{sec:ext} for the precise framework. 

The proof of Theorem~\ref{TH1} is divided in two parts. 
More precisely, in the first part we obtain the existence 
of the first solution, that turns out to be a minimum 
for the extended functional introduced in the forthcoming Section~\ref{sec:ext}. 
Then in the second part we will find a 
second solution, 
by applying the Mountain Pass Theorem introduced in~\cite{AmRab}. 
We stress, however, that this additional solution
is not necessarily of mountain pass type, since, in order
to obtain the necessary compactness, one adopts the contradiction
assumption that no other solution exists.

Notice that in~\cite{vecchio} we have proved that if~$h$ changes sign 
then there exist two distinct solutions of~\eqref{problem} that bifurcate 
from a non trivial critical manifold. 
Here we also show that there exists 
a third solution that bifurcates from~$u=0$. This means that 
when $h$ changes sign, problem~\eqref{problem} admits at least three 
different solutions. 

Let us point out that, when~$h$ changes sign, the solution~$u_{1,\epsilon}$ 
found in~\cite{vecchio} can possibly coincide with the second
solution 
that we construct in this monograph. 
It would be an interesting open problem to investigate 
on the Morse index of the second solution found. 

So the main point is to show that the extended functional satisfies a compactness property. 
In particular, for the existence of the minimum, we will prove that 
a Palais-Smale condition holds true below a certain energy level, see Proposition~\ref{PScond}. 
Then the existence of the minimum will be ensured 
by the fact that the critical level lies below this threshold. 

In order to show the Palais-Smale condition we will use 
a version of the Concentration-Compactness Principle, 
see Section~\ref{sec:CC},  
and for this we will borrow some ideas from~\cite{L1, L2}. 
Differently from \cite{bcss}, here we are dealing with a problem in the whole of~$\R^n$,
therefore, in order to apply the Concentration-Compactness Principle, we also
need to show a tightness property (see Definition~\ref{defTight}). 
Of course, fractional problems may, in principle,
complicate the tightness issues, since the
nonlocal interaction could produce (or send) additional mass
from (or to) infinity.\medskip

As customary in many fractional problems, see~\cite{CS},
we will work in an extended space, which reduces the fractional
operator to a local (but possibly singular and degenerate) one,
confining the nonlocal feature to a boundary reaction problem.
This functional simplification (in terms of nonlocality)
creates
additional difficulties coming from the fact that 
the extended functional is not homogeneous. Hence,
we will have to deal with weighted Sobolev spaces, 
and so we have to prove some weighted embedding to obtain
some convergences needed throughout the monograph, see Section \ref{sec:weighted}. 
\medskip

A further source of difficulty is that the exponent~$q$
in~\eqref{problem} is below~$1$, hence the associated energy
is not convex and not smooth.\medskip

In the forthcoming Section \ref{sec:ext} we present the 
variational setting of the problem, both in the original and in the extended variables, 
and we state the main results of this monograph. In particular, we first 
introduce the material that we are going to use 
in order to construct the first solution, that is the minimum solution. 
Then, starting from this minimum, we introduce a translated functional, 
that we will exploit to obtain the existence of a mountain pass solution
(under the contradictory assumption that no other solutions exist). 

\section{An extended problem and statement of the main results}\label{sec:ext}

In this section we introduce the variational setting, 
we present a related extended problem, and we state the main results of this 
monograph.

Since we are looking for positive solutions, we will consider the following problem:
\begin{equation}\label{problem-1}
(-\Delta)^s u = \epsilon h u_+^q + u_+^p \quad {\mbox{ in }}\R^n.
\end{equation}
Hence, we say that $u\in\dot{H}^s(\R^n)$ is a (weak) solution to \eqref{problem-1} 
if for every $v\in\dot{H}^s(\R^n)$ we have 
\begin{eqnarray*}&& \frac{c_{n,s}}{2}\iint_{\R^{2n}}\frac{(u(x)-u(y))(v(x)-v(y))}{|x-y|^{n+2s}}\,dx\,dy \\&&= 
\int_{\R^n}h(x) u_+^{q}(x)v(x)\,dx + \int_{\R^n}u_+^p(x)v(x).\end{eqnarray*}
It turns out that if $u$ is a solution to \eqref{problem-1}, then it is nonnegative in $\R^n$ 
(see the forthcoming Proposition~\ref{prop:pos}, and also  
Section~\ref{sec:positivity}
for the discussion about the positivity of the solutions). 
Therefore, $u$ is also a solution of \eqref{problem}.

Notice that problem \eqref{problem-1} has a variational structure. 
Namely, solutions to \eqref{problem-1} 
can be found as critical points of the 
functional $f_\epsilon:\dot{H}^s(\R^n)\rightarrow\R$ defined by 
\begin{equation}\begin{split}\label{f giu}
f_\epsilon(u) :=\,&\frac{c_{n,s}}{4}\iint_{\R^{2n}}\frac{|u(x)-u(y)|^2}{|x-y|^{n+2s}}\,dx\,dy \\
& \quad -\frac{\epsilon}{q+1}\,\int_{\R^n}h(x)\,u_+^{q+1}(x)\,dx 
-\frac{1}{p+1}\,\int_{\R^n}u_+^{p+1}(x)\,dx.
\end{split}\end{equation}
However, instead of working with this 
framework derived from Definition \ref{laplacian} of the Laplacian, 
we will consider the extended operator given by \cite{CS}, 
that allows us to transform a nonlocal problem into a local one by adding one variable. 

For this, we will denote by $\mathbb{R}^{n+1}_+:=\mathbb{R}^n\times(0,+\infty)$. 
Also, for a point $X\in\mathbb{R}^{n+1}_+$, 
we will use the notation $X=(x,y)$, with $x\in\mathbb{R}^n$ and $y>0$. 

Moreover, for $x\in\R^{n}$ and $r>0$, we will denote 
by $B_r(x)$ the ball in $\R^n$ centered at $x$ with radius $r$, i.e. 
$$ B_r(x):=\{x'\in\R^n : |x-x'|<r\},$$ 
and, for $X\in\R^{n+1}_+$ and $r>0$, $B^+_r(X)$ will be the ball in $\R^{n+1}_+$ 
centered at $X$ with radius $r$, that is
$$ B_r^+(X):=\{X'\in\R^{n+1}_+ : |X-X'|<r\}. $$

Now, given a function~$u:\R^n\to\R$, we associate a function~$U$ 
defined in~$\R^{n+1}_+$ as 
\begin{equation}\label{poisson}
U(\cdot,z)=u*P_s(\cdot,z), \quad {\mbox{ where }}P_s(x,z):=c_{n,s}\frac{z^{2s}}{(|x|^2+z^2)^{(n+2s)/2}}.
\end{equation}
Here $c_{n,s}$ is a normalizing constant depending on~$n$ and~$s$. 

Set also $a:=1-2s$, and 
\begin{equation}\label{extNorm}
[U]_a^*:=\left(\kappa_s\int_{\mathbb{R}^{n+1}}{y^a|\nabla U|^2\,dX}\right)^{1/2},
\end{equation}
where $\kappa_s$ is a normalization constant. We define the spaces
\begin{equation*}
\dot{H}^s_a(\mathbb{R}^{n+1}):=\overline{C_0^\infty(\mathbb{R}^{n+1})}^{[\cdot]_a^*},
\end{equation*}
and
\begin{equation}\begin{split}\label{D123}
\dot{H}^s_a(\mathbb{R}^{n+1}_+):=\{&U:=\tilde{U}|_{\mathbb{R}^{n+1}_+}\mbox{ s.t. }\tilde{U}\in \dot{H}^s_a(\mathbb{R}^{n+1}),\\
&\tilde{U}(x,y)=\tilde{U}(x,-y)\hbox{ a.e. in }\mathbb{R}^n\times\mathbb{R}\},
\end{split}\end{equation}
endowed with the norm 
\begin{equation*}
[U]_a:=\left(\kappa_s\int_{\mathbb{R}^{n+1}_+}{y^a|\nabla U|^2\,dX}\right)^{1/2}.
\end{equation*}
From now on, for simplicity, 
we will neglect the dimensional constants $c_{n,s}$ and $\kappa_s$. 
It is known that finding a solution $u\in \dot{H}^s(\mathbb{R}^n)$ to a problem
$$(-\Delta)^su=f(u)\quad\hbox{ in }\mathbb{R}^n $$
is equivalent to find $U\in\dot{H}^s_a(\mathbb{R}^{n+1}_+)$ solving the local problem
\begin{equation*}
\begin{cases}
\hbox{div}(y^a\nabla U)=0\quad\hbox{ in }\mathbb{R}^{n+1}_+,\\
-\displaystyle \lim_{y\rightarrow 0^+}y^a\,\frac{\partial U}{\partial\nu}=f(u),
\end{cases}
\end{equation*}
and that this extension is 
an isometry between $\dot{H}^s(\mathbb{R}^n)$ and $\dot{H}^s_a(\mathbb{R}^{n+1}_+)$ 
(again up to constants), that is,
\begin{equation}\label{equivNorms}
[U]_a=[u]_{\dot{H}^s(\mathbb{R}^n)},
\end{equation}
where we make the identification $u(x)=U(x,0)$, with $U(x,0)$ understood in the sense of traces
(see e.g. \cite{CS} and \cite{Bucur}). 

Also, we recall that the Sobolev embedding in $\dot{H}^s(\mathbb{R}^n)$ gives that 
$$ S\,\|u\|^2_{L^{2^*_s}(\mathbb{R}^n)}\leq [u]^2_{\dot{H}^s(\mathbb{R}^n)}, $$
where $S$ is the usual constant of the Sobolev embedding of $\dot{H}^s(\mathbb{R}^n)$, 
see for instance Theorem 6.5 in \cite{DPV}. 
As a consequence of this and \eqref{equivNorms} we have the following result.
\begin{prop}[Trace inequality]\label{traceIneq}
Let $U\in \dot{H}^s_a(\mathbb{R}^{n+1}_+)$. Then,
\begin{equation}\label{TraceIneq}
S\,\|U(\cdot,0)\|^2_{L^{2^*_s}(\mathbb{R}^n)}\leq [U]^2_a.
\end{equation}
\end{prop}

Therefore, we can reformulate problem \eqref{problem-1} as
\begin{equation}\label{ExtendedProblem}
\begin{cases}
\hbox{div}(y^a\nabla U)=0\quad\hbox{ in }\mathbb{R}^{n+1}_+,\\
-\displaystyle \lim_{y\rightarrow 0^+}y^a\,\frac{\partial U}{\partial\nu}=\varepsilon h u_+^q+u_+^p.
\end{cases}
\end{equation}
In particular, we will say that $U\in\dot{H}^s_a(\mathbb{R}^{n+1}_+)$ 
is a (weak) solution of problem \eqref{ExtendedProblem} if
\begin{equation}\label{weak sol}
\int_{\mathbb{R}^{n+1}_+}{y^a\langle\nabla U,\nabla \varphi\rangle\,dX}
=\int_{\mathbb{R}^n}{(\varepsilon h(x)U_+^q(x,0)+U_+^p(x,0))\varphi(x,0)\,dx},
\end{equation}
for every $\varphi\in\dot{H}^s_a(\mathbb{R}^{n+1}_+)$. 
Likewise, the energy functional associated to the problem \eqref{ExtendedProblem} is
\begin{equation}\begin{split}\label{f ext}
\mathcal{F}_\varepsilon(U):=&
\frac{1}{2}\int_{\mathbb{R}^{n+1}_+}{y^a|\nabla U|^2\,dX}
-\frac{\varepsilon}{q+1}\int_{\mathbb{R}^n}{h(x)U_+^{q+1}(x,0)\,dx}\\
&\quad -\frac{1}{p+1}\int_{\mathbb{R}^n}{U_+^{p+1}(x,0)\,dx}.
\end{split}\end{equation}
Notice that for any $U,V\in\dot{H}^s_a(\R^{n+1}_+)$ we have 
\begin{equation}\begin{split}\label{pqoeopwoegi}
&  \langle \mathcal{F}_\varepsilon'(U),V\rangle = 
\int_{\R^{n+1}_+}y^a\langle\nabla U, \nabla V\rangle\,dX \\
&\quad - 
\epsilon \int_{\R^n}h(x)U_+^{q}(x,0)\,V(x,0)\,dx - 
\int_{\R^n}U_+^p(x,0)\,V(x,0)\,dx.
\end{split}\end{equation}
Hence, if $U$ is a critical point of $\mathcal{F}_\epsilon$, 
then it is a weak solution of \eqref{ExtendedProblem}, according to \eqref{weak sol}. 
Therefore $u:=U(\cdot,0)$ is a solution to \eqref{problem-1}.

Moreover, if $U$ is a minimum of $\mathcal{F}_\varepsilon$, 
then $u(x):=U(x,0)$ is a minimum of $f_\varepsilon$, thanks to \eqref{equivNorms},  
and so $u$ is a solution to problem \eqref{problem-1}.

In this setting, we can prove the existence of a first solution 
of problem \eqref{ExtendedProblem}, and consequently of problem \eqref{problem-1}.

\begin{theorem}\label{MINIMUM}
Let $0<q<1$ and suppose that $h$ satisfies \eqref{h1} and \eqref{h2}. 
Then, there exists $\varepsilon_0>0$ such that $\mathcal{F}_\varepsilon$ 
has a local minimum $U_\varepsilon\neq 0$, for any $\epsilon<\epsilon_0$. 
Moreover, $U_\varepsilon\rightarrow 0$ in $\dot{H}^s_a(\mathbb{R}^{n+1}_+)$ 
when $\varepsilon\rightarrow 0$.
\end{theorem}

We now set $u_\epsilon:=U_\epsilon(\cdot,0)$, where $U_\epsilon$ is the 
local minimum of $\mathcal{F}_\epsilon$ found in Theorem~\ref{MINIMUM}.
Then, according to \eqref{equivNorms}, $u_\epsilon$ is a local minimum 
of $f_\epsilon$, and so a solution to \eqref{problem-1}. 

Notice that, again by \eqref{equivNorms},  
$$ [u_\epsilon]_{\dot{H}^s(\R^n)}= [U_\epsilon]_a\to 0 \quad {\mbox{ as }}\epsilon\to 0.$$ 
In this sense, the solution $u_\epsilon$ obtained by minimizing the functional 
bifurcates from the solution $u=0$. 

Furthermore, $u_\epsilon$ is nonnegative, and thus $u_\epsilon$ is a true solution of \eqref{problem}.
Indeed, we can prove the following:

\begin{prop}\label{prop:pos}
Let~$u\in\dot{H}^s(\R^n)$ be a nontrivial
solution of \eqref{problem-1} and let $U$
be its extension, according to \eqref{poisson}. 
Then, $u\ge0$ and $U>0$.  
\end{prop}

\begin{proof}
We set~$u_-(x):=-\min\{u(x),0\}$, namely~$u_-$ is the negative part of~$u$, 
and we claim that 
\begin{equation}\label{u meno}
u_-=0.
\end{equation}
For this, we multiply~\eqref{problem-1} by~$u_-$ and we integrate over~$\R^n$: we obtain
$$ \int_{\R^n}(-\Delta)^su\,u_-\,dx=\int_{\R^n}\left(\epsilon h(x)u_+^q+u_+^p\right)u_-\,dx=0.$$
Hence, by an integration by parts we get 
\begin{equation}\label{negggg-1}
\iint_{\R^{2n}}\frac{(u(x)-u(y))(u_-(x)-u_-(y))}{|x-y|^{n+2s}}\,dx\,dy=0.
\end{equation}
Now, we observe that 
\begin{equation}\label{neggg}
(u(x)-u(y))(u_-(x)-u_-(y))\ge |u_-(x)-u_-(y)|^2. 
\end{equation}
Indeed, if both~$u(x)\ge0$ and~$u(y)\ge0$ and if both~$u(x)\le0$ and~$u(y)\le0$ 
then the claim trivially follows. Therefore, we suppose that~$u(x)\ge0$ 
and~$u(y)\le0$ (the symmetric situation is analogous). In this case 
\begin{eqnarray*}
&& (u(x)-u(y))(u_-(x)-u_-(y))= -(u(x)-u(y))u(y)\\
&&\qquad\quad =-u(x)u(y)+|u(y)|^2\ge |u(y)|^2
=|u_-(x)-u_-(y)|^2, 
\end{eqnarray*}
which implies~\eqref{neggg}. 

From~\eqref{negggg-1} and~\eqref{neggg}, we obtain that 
$$ \iint_{\R^{2n}}\frac{|u_-(x)-u_-(y)|^2}{|x-y|^{n+2s}}\,dx\,dy\le 0,$$
and this implies~\eqref{u meno}, since $u\in \dot{H}^s(\R^n)$. 
Hence~$u\ge0$. 
This implies that $U>0$, being a convolution 
of~$u$ with a positive kernel.  
\end{proof}

We can also prove the existence of a second solution 
of problem \eqref{ExtendedProblem}, and consequently of problem \eqref{problem}.

\begin{theorem}\label{TH:MP}
Let $0<q<1$ and suppose that $h$ satisfies \eqref{h1} and \eqref{h2}. 
Then, there exists $\varepsilon_0>0$ such that $\mathcal{F}_\varepsilon$ 
has a second solution $\overline{U}_{\varepsilon}\neq 0$, for any $\epsilon<\epsilon_0$. 
\end{theorem}

To prove the existence of a second solution of problem \eqref{ExtendedProblem} 
we consider a translated functional. 
Namely, we let~$U_\varepsilon$ be the local minimum of the functional~\eqref{f ext} 
(already found in Theorem \ref{MINIMUM}), 
and we consider the functional 
$\mathcal{I}_\varepsilon:\dot{H}^s_a(\mathbb{R}^{n+1}_+)\rightarrow \mathbb{R}$ defined as
\begin{equation}\label{def I}
\mathcal{I}_\varepsilon(U)=
\frac{1}{2}\int_{\mathbb{R}^{n+1}_+}{y^a|\nabla U|^2\,dX}
-\int_{\mathbb{R}^n}{G(x,U(x,0))\,dx},
\end{equation}
where 
\begin{equation*}
G(x,U):=\int_0^{U}{g(x,t)\,dt},
\end{equation*}
and
\begin{equation}\label{def g small}
g(x,t):=
\begin{cases}
\varepsilon h(x)((U_\varepsilon +t)^q-U_\varepsilon^q )+(U_\varepsilon +t)^p-U_\varepsilon^p,\;\hbox{ if }t\geq 0,\\
0\;\hbox{ if }t< 0.
\end{cases}
\end{equation}
Explicitly, 
\begin{equation}\begin{split}\label{def G}
G(x,U)=&\,\frac{\epsilon\,h(x)}{q+1}
\left((U_\epsilon +U_+)^{q+1}-U_\epsilon^{q+1}\right) 
-\epsilon\,h(x)\,U_\epsilon^q U_+ \\
&\qquad +\frac{1}{p+1}\left((U_\epsilon +U_+)^{p+1}-U_\epsilon^{p+1}\right)
-U_\epsilon^pU_+.
\end{split}\end{equation}
Moreover, for any~$U,V\in\dot{H}^s_a(\R^{n+1}_+)$, we have that
\begin{equation}\label{DER}
\langle \mathcal{I}_\varepsilon'(U),V\rangle = 
\int_{\R^{n+1}_+}y^a\langle\nabla U, \nabla V\rangle\,dX - 
\int_{\R^n}g(x,U(x,0))\,V(x,0)\,dx.
\end{equation}
Notice that a critical point of~\eqref{def I} is a solution to the following problem
\begin{equation}\label{alksjaksfhhf}
\begin{cases}
\hbox{div}(y^a\nabla U)=0\quad\hbox{ in }\mathbb{R}^{n+1}_+,\\
-\displaystyle \lim_{y\rightarrow 0^+}y^a\,\frac{\partial U}{\partial\nu}=g(x,U(x,0)).
\end{cases}
\end{equation}
One can prove that a solution~$U$ to this problem is positive, 
as stated in the forthcoming Lemma~\ref{lemma:max}. 
Therefore,~$\overline{U}:=U_\epsilon+U>0$, thanks to Proposition \ref{prop:pos}. 
Also,~$\overline{U}$ will be the second solution of~\eqref{weak sol}, 
and so~$\overline{u}:=\overline{U}(\cdot,0)$ will be the second solution 
to~\eqref{problem}.

\begin{lemma}\label{lemma:max}
Let~$U\in\dot{H}_a^s(\mathbb{R}^{n+1}_+)$, $U\neq0$, be a solution to~\eqref{alksjaksfhhf}. 
Then~$U$ is positive. 
\end{lemma}

\begin{proof} 
We first observe that, if~$U$ is a solution to~\eqref{alksjaksfhhf}, 
then~$u:=U(\cdot,0)$ is a solution of 
\begin{equation}\label{alksjaksfhhf-1}
(-\Delta)^su=g(x,u) \quad {\mbox{ in }}\R^n.
\end{equation}
Now, we set~$u_-(x):=-\min\{u(x),0\}$, namely~$u_-$ is the negative part of~$u$, 
and we claim that 
\begin{equation}\label{u meno-1}
u_-=0.
\end{equation}
For this, we multiply~\eqref{alksjaksfhhf-1} by~$u_-$ and we integrate over~$\R^n$: we obtain
$$ \int_{\R^n}(-\Delta)^su\,u_-\,dx=\int_{\R^n}g(x,u)\,u_-\,dx.$$
Recalling the definition of~$g$ in~\eqref{def g small}, we have that 
$$ \int_{\R^n}g(x,u)\,u_-\,dx=0.$$
Hence, by an integration by parts we get 
\begin{equation}\label{negggg-12}
\iint_{\R^{2n}}\frac{(u(x)-u(y))(u_-(x)-u_-(y))}{|x-y|^{n+2s}}\,dx\,dy=0.
\end{equation}
Now, we observe that 
\begin{equation}\label{neggg00}
(u(x)-u(y))(u_-(x)-u_-(y))\ge |u_-(x)-u_-(y)|^2. 
\end{equation}
Indeed, if both~$u(x)\ge0$ and~$u(y)\ge0$ and if both~$u(x)\le0$ and~$u(y)\le0$ 
then the claim trivially follows. Therefore, we suppose that~$u(x)\ge0$ 
and~$u(y)\le0$ (the symmetric situation is analogous). In this case 
\begin{eqnarray*}
&& (u(x)-u(y))(u_-(x)-u_-(y))= -(u(x)-u(y))u(y)\\
&&\qquad\quad =-u(x)u(y)+|u(y)|^2\ge |u(y)|^2
=|u_-(x)-u_-(y)|^2, 
\end{eqnarray*}
which implies~\eqref{neggg00}. 

From~\eqref{negggg-12} and~\eqref{neggg00}, we obtain that 
$$ \iint_{\R^{2n}}\frac{|u_-(x)-u_-(y)|^2}{|x-y|^{n+2s}}\,dx\,dy\le 0,$$
and this implies~\eqref{u meno-1}, since $u\in H^s(\R^n)$. 
Hence~$u\ge0$. 
This implies that $U>0$, being a convolution 
of~$u$ with a positive kernel. 
\end{proof}

The next sections will be devoted to the proof of 
Theorems \ref{MINIMUM} and \ref{TH:MP}. 

More precisely, for this goal
some preliminary material from functional analysis
is needed. The main analytic tools are contained in Chapter~\ref{FAS}.
Namely,
since we will work with an extended functional (that also 
contains terms with weighted Sobolev norms), 
we devote Section~\ref{sec:weighted} to show some weighted Sobolev embeddings 
and Section~\ref{sec:CC} to prove a suitable
Concentration-Compactness Principle. 

The existence of a minimal solution is discussed in
Chapter~\ref{ECXMII}. In particular,
in Section~\ref{sec:conv} we deal with some convergence results, 
that we need in the subsequent Section~\ref{sec:PS}, where 
we show that under a given level the Palais-Smale condition holds true 
for the extended functional. Then, in Section~\ref{concaveFirstSol} 
we complete the proof of Theorem~\ref{MINIMUM}. 

In Chapter~\ref{7xucjhgfgh345678},
we discuss some regularity and positivity issues
about the solution that we constructed.
More precisely,
in Section \ref{sec:reg} we show some regularity results, 
and in Section \ref{sec:positivity} we prove the positivity of the solutions to \eqref{problem}, 
making use of a strong maximum principle for weak solutions. 

Then, in Chapter~\ref{EMP:CHAP}
we deal with the existence of the mountain pass solution
(under the contradictory assumption that the solution is unique). 
We first show, in Section \ref{sec:exi}, that the translated functional introduced in Section 
\ref{sec:ext} has $U=0$ as a local minimum (notice that this is a consequence of 
the fact that we are translating the original functional with respect to its 
local minimum). 

Sections \ref{sec:prelim} and \ref{sec:prelim2} are devoted to some preliminary 
results. We will exploit these basic lemmata in the subsequent Section \ref{sec:PS MP}, 
where we prove that the above-mentioned translated functional satisfies a Palais-Smale 
condition. 

In Section~\ref{sec:BMMV}
we estimate the minimax value along a suitable path
(roughly speaking, the linear path constructed
along a suitably cut-off minimizer of the fractional Sobolev inequality).
This estimate is needed to exploit the Mountain Pass Theorem
via the convergence of the Palais-Smale sequences at appropriate
energy levels.
With this, in Section \ref{sec:proof} we finish the proof of Theorem \ref{TH:MP}. 

\chapter{Functional analytical setting}\label{FAS}

\section{Weighted Sobolev embeddings}\label{sec:weighted}

For any $r\in(1,+\infty)$, we denote by $L^r(\R^{n+1}_+,y^a)$ the weighted\footnote{Some of the results presented here
are valid for more general families of weights,
in the setting of Muckenhoupt classes.
Nevertheless, we focused on the monomial weights
both for the sake of concreteness and simplicity,
and because some
more general results follow in a straightforward way
from the ones presented here. With this respect,
for further comments that compare monomial
and Muckenhoupt weights,
see the end of Section~1
in~\cite{CR}.}
Lebesgue space, 
endowed with the norm 
$$ \|U\|_{L^r(\R^{n+1}_+,y^a)}:=\left(\int_{\R^{n+1}_+} y^a |U|^r\,dX\right)^{1/r}.$$ 
The following result shows that $\dot{H}^s_a(\mathbb{R}^{n+1}_+)$ 
is continuously embedded in $L^{2\gamma}(\R^{n+1}_+,y^a)$. 

\begin{prop}[Sobolev embedding]\label{WeightedSob}
There exists a constant $\hat{S}>0$ such that for all $U\in \dot{H}_a^s(\mathbb{R}^{n+1}_+)$ it holds
\begin{equation}\label{SobIneq}
\left(\int_{\mathbb{R}^{n+1}_+}{y^a|U|^{2\gamma}\,dX}\right)^{1/2\gamma}\leq 
\hat{S}\left(\int_{\mathbb{R}^{n+1}_+}{y^a|\nabla U|^2\,dX}\right)^{1/2},
\end{equation}
where $\gamma=1+\dfrac{2}{n-2s}$.
\end{prop}

\begin{proof}
Let us first prove the result for $U\in C_0^\infty(\mathbb{R}^{n+1})$. If $s\in(0,1/2)$, inequality \eqref{SobIneq} is easily deduced from Theorem 1.3 of \cite{CR}. By a density argument, we obtain that inequality \eqref{SobIneq} holds 
for any function $U\in\dot{H}^s_a(\mathbb{R}^{n+1}_+)$. 
Indeed, if $U\in\dot{H}^s_a(\mathbb{R}^{n+1}_+)$, then there exists 
a sequence of functions $\{U_k\}_{k\in\mathbb{N}}\in C^\infty_0(\R^{n+1})$ such that 
$U_k$ converges to some $\tilde{U}$ in $\dot{H}^s_a(\mathbb{R}^{n+1})$ as $k\to\infty$, where $U=\tilde{U}$ in $\mathbb{R}^{n+1}_+$ and $\tilde{U}$ is even with respect to the $(n+1)$-th variable. Hence, for any $k$, we have 
\begin{equation}\begin{split}\label{sobUk}
\left(\int_{\mathbb{R}^{n+1}_+}{y^a|U_k|^{2\gamma}\,dX}\right)^{1/2\gamma}&\leq 
\hat{S}\left(\int_{\mathbb{R}^{n+1}_+}{y^a|\nabla U_k|^2\,dX}\right)^{1/2}\\
&\leq 
\hat{S}\left(\int_{\mathbb{R}^{n+1}}{y^a|\nabla U_k|^2\,dX}\right)^{1/2}.
\end{split}\end{equation}
Moreover, given two functions of the approximating sequence, there holds
\begin{equation*}\begin{split}
\left(\int_{\mathbb{R}^{n+1}_+}{y^a|U_k-U_m|^{2\gamma}\,dX}\right)^{1/2\gamma}
&\leq 
\hat{S}\left(\int_{\mathbb{R}^{n+1}}{y^a|\nabla (U_k-U_m)|^2\,dX}\right)^{1/2}\rightarrow 0,
\end{split}\end{equation*}
and thus, up to a subsequence,
\begin{equation*}\begin{split}
U_k\rightarrow \tilde{U}&\hbox{ in }L^{2\gamma}(\mathbb{R}^{n+1}_+,y^a),\\
U_k\rightarrow \tilde{U}&\hbox{ a.e. in }\mathbb{R}^{n+1}_+.
\end{split}\end{equation*}
Hence, by Fatou's Lemma and \eqref{sobUk} we get 
\begin{equation}\begin{split}\label{density}
\left(\int_{\mathbb{R}^{n+1}_+}{y^a|U|^{2\gamma}\,dX}\right)^{1/2\gamma}&=\left(\int_{\mathbb{R}^{n+1}_+}{y^a|\tilde{U}|^{2\gamma}\,dX}\right)^{1/2\gamma}\\
&\leq 
\lim_{k\to+\infty}\left(\int_{\mathbb{R}^{n+1}_+}{y^a|U_k|^{2\gamma}\,dX}\right)^{1/2\gamma}\\ 
&\leq \lim_{k\to+\infty}
\hat{S}\left(\int_{\mathbb{R}^{n+1}_+}{y^a|\nabla U_k|^2\,dX}\right)^{1/2}\\
&\leq \lim_{k\to+\infty}
\hat{S}\left(\int_{\mathbb{R}^{n+1}}{|y|^a|\nabla U_k|^2\,dX}\right)^{1/2}\\
&=\hat{S}\left(\int_{\mathbb{R}^{n+1}}{|y|^a|\nabla \tilde{U}|^2\,dX}\right)^{1/2}\\
&=\hat{S}\left(2\int_{\mathbb{R}^{n+1}_+}{y^a|\nabla U|^2\,dX}\right)^{1/2},
\end{split}\end{equation}
which shows that Proposition~\ref{WeightedSob} holds true
for any function $U\in\dot{H}^s_a(\mathbb{R}^{n+1}_+)$, up to renaming $\hat{S}$.

On the other hand, the case $s=\frac{1}{2}$ corresponds to the classical Sobolev inequality, so we can now concentrate on the range $s\in (1/2,1)$, that can be derived from Theorem 1.2 of \cite{FKS} by arguing as follows. 

Let us denote
$$w(X):=|y|^a.$$
Thus, it can be checked that 
\begin{equation}\label{Muck}
w\in A_q \hbox{ for every }q\in (2-2s,2],
\end{equation}
where $A_q$ denotes the class of Muckenhoupt weights of order $q$.  Since in particular $w\in A_2$, by Theorem 1.2 of \cite{FKS}, we know that there exist positive constants $C$ and $\delta$ such that for all balls $B_R\subset \mathbb{R}^{n+1}$, all $u\in C_0^\infty(B_R)$ and all $\gamma$ satisfying $1\leq \gamma\leq \frac{n+1}{n}+\delta$, one has
\begin{equation}\label{ineqFKS}
\left(\frac{1}{w(B_R)}\int_{B_R}{|U|^{2\gamma}w\,dX}\right)^{1/2\gamma}\leq CR\left(\frac{1}{w(B_R)}\int_{B_R}{|\nabla U|^2w\,dX}\right)^{1/2}.
\end{equation}
In particular, it yields
\begin{eqnarray*}
w(B_R)&=&\int_{B_R}{|y|^a\,dX}=\int_{|x|^2+y^2\leq R^2}{|y|^a\,dx\,dy}\\
&=&\int_{|\eta|^2+\xi^2\leq 1}{|\xi|^a R^{n+1+a}\,d\xi\,d\eta}
= CR^{a+n+1}=CR^{2-2s+n},
\end{eqnarray*}
with $C$ independent of $R$. Thus, 
$$Rw(B_R)^{\frac{1}{2\gamma}-\frac{1}{2}}= CR^{1+\frac{(1-\gamma)(2-2s+n)}{2\gamma}},$$
and plugging this into \eqref{ineqFKS} we get
\begin{equation*}
\left(\int_{B_R}{|U|^{2\gamma}w\,dX}\right)^{1/2\gamma}\leq CR^{1+\frac{(1-\gamma)(2-2s+n)}{2\gamma}}\left(\int_{B_R}{|\nabla U|^2w\,dX}\right)^{1/2},
\end{equation*}
where $C$ is a constant independent of $R$. In particular, if we set $\gamma=1+\frac{2}{n-2s}$, then
$$1+\frac{(1-\gamma)(2-2s+n)}{2\gamma}=0,$$
and the inequality holds for every ball with the same constant. It remains to check that this value of $\gamma$ is under the hypotheses of Theorem 1.2 of \cite{FKS}, that is, $1\leq \gamma\leq \frac{n+1}{n}+\delta$. Keeping track of $\delta$ in \cite{FKS}, this condition actually becomes
$$1\leq \gamma\leq \frac{n+1}{n+1-\frac{2}{q}},$$
for every $q<2$ such that $w\in A_q$. Thus, by \eqref{Muck}, we can choose any $q\in (2-2s,2)$.
Since $\gamma$ is clearly greater than $1$, we have to prove the upper bound, that is,
$$1+\frac{2}{n-2s}\leq  \frac{n+1}{n+1-\frac{2}{q}},$$
but this is equivalent to ask 
$$q\leq \frac{n-2s+2}{n+1}.$$ 
Since we can choose $q$ as close as we want to $2-2s$, this inequality will be true whenever
$$2-2s<\frac{n-2s+2}{n+1},$$
which holds if and only if $s>\frac{1}{2}$. Summarizing, we have that
\begin{equation}
\left(\int_{B_R}{|y|^a|U|^{2\gamma}\,dX}\right)^{1/2\gamma}\leq C\left(\int_{B_R}{|y|^a|\nabla U|^2\,dX}\right)^{1/2},
\end{equation}
where $\gamma=1+\frac{2}{n-2s}$ and $C$ is a constant independent of the domain. Choosing $R$ large enough, it yields
\begin{equation}
\left(\int_{\mathbb{R}^{n+1}}{|y|^a|U|^{2\gamma}\,dX}\right)^{1/2\gamma}\leq C\left(\int_{\mathbb{R}^{n+1}}{|y|^a|\nabla U|^2\,dX}\right)^{1/2},
\end{equation}
Consider now $U\in\dot{H}^s_a(\mathbb{R}^{n+1}_+)$. We perform the same density argument as in the case $s\in (0,1/2)$, with the only difference that instead of \eqref{density} we have
\begin{equation*}\begin{split}
\left(\int_{\mathbb{R}^{n+1}_+}{y^a|U|^{2\gamma}\,dX}\right)^{1/2\gamma}&=\left(\int_{\mathbb{R}^{n+1}_+}{y^a|\tilde{U}|^{2\gamma}\,dX}\right)^{1/2\gamma}\\
&\leq 
\lim_{k\to+\infty}\left(\int_{\mathbb{R}^{n+1}_+}{y^a|U_k|^{2\gamma}\,dX}\right)^{1/2\gamma}\\ 
&\leq 
\lim_{k\to+\infty}\left(\int_{\mathbb{R}^{n+1}}{|y|^a|U_k|^{2\gamma}\,dX}\right)^{1/2\gamma}\\ 
&\leq \lim_{k\to+\infty}
\hat{S}\left(\int_{\mathbb{R}^{n+1}}{|y|^a|\nabla U_k|^2\,dX}\right)^{1/2}\\
&=\hat{S}\left(\int_{\mathbb{R}^{n+1}}{|y|^a|\nabla \tilde{U}|^2\,dX}\right)^{1/2}\\
&=\hat{S}\left(2\int_{\mathbb{R}^{n+1}_+}{y^a|\nabla U|^2\,dX}\right)^{1/2}.\qedhere
\end{split}\end{equation*}
\end{proof}

We also show a compactness result that we will need in the sequel. 
More precisely, we prove that $\dot{H}^s_a(\mathbb{R}^{n+1}_+)$ 
is locally compactly embedded in $L^2(\R^{n+1}_+,y^a)$. 
The precise statement goes as follows: 
 
\begin{lemma}\label{lemma:compact}
Let $R>0$ and let $\mathcal{J}$ be a subset of $\dot{H}^s_a(\R^{n+1}_+)$ such that 
$$ \sup_{U\in\mathcal{J}}\int_{\R^{n+1}_+}y^a|\nabla U|^2\,dX<+\infty.$$
Then $\mathcal{J}$ is precompact in $L^2(B^+_R,y^a)$.
\end{lemma}

\begin{proof}
We will prove that $\mathcal{J}$ is totally bounded in $L^2(B^+_R,y^a)$, 
i.e. for any $\epsilon>0$ there exist $M$ and $U_1,\ldots,U_M\in L^2(B^+_R,y^a)$ 
such that for any $U\in\mathcal{J}$ there exists $i\in\{1,\ldots,M\}$ such that
\begin{equation}\label{alewqewgew}
\|U_i-U\|_{L^2(B^+_R,y^a)}\le\epsilon.
\end{equation}
For this, we fix $\epsilon>0$, we set 
\begin{equation}\label{pwqrtt000}
A:=\sup_{U\in\mathcal{J}}\int_{\R^{n+1}_+}y^a|\nabla U|^2\,dX<+\infty 
\end{equation}
and we let 
\begin{equation}\label{eta def}
\eta:=\left[\frac{\epsilon^2}{2\hat{S}^2A}\left(\frac{a+1}{|B_R|}\right)^{\frac{\gamma-1}{\gamma}} 
\right]^{\frac{\gamma}{(\gamma-1)(a+1)}},
\end{equation}
where $\gamma$ and $\hat{S}$ are the constants introduced in the statement 
of Proposition \ref{WeightedSob}, and $|B_R|$ is the Lebesgue measure of the ball $B_R$ in $\R^n$. 

Now, notice that 
\begin{equation}\label{lskfjgrekge}
{\mbox{if $X\in B_R^+\cap\{y\ge\eta\}$ then 
$y^a\ge\min\{\eta^a,R^a\}$}}. \end{equation}
Indeed, if $a\ge0$ (that is $s\in(0,1/2]$) then $y^a\ge\eta^a$, 
while if $a<0$ (that is $s\in(1/2,1)$) then we use that $y\le R$, and so $y^a\ge R^a$, 
thus proving \eqref{lskfjgrekge}. Analogously, one can prove that 
\begin{equation}\label{lskfjgrekge-1}
{\mbox{if $X\in B_R^+\cap\{y\ge\eta\}$ then 
$y^a\le\max\{\eta^a,R^a\}$}}. \end{equation}
Therefore, using \eqref{lskfjgrekge}, we have that, for any $U\in\mathcal{J}$,
$$ A\ge \int_{B_R^+\cap\{y\ge\eta\}}y^a|\nabla U|^2\,dX\ge  
\min\{\eta^a,R^a\}\int_{B_R^+\cap\{y\ge\eta\}}|\nabla U|^2\,dX.$$ 
Hence, 
$$ \int_{B_R^+\cap\{y\ge\eta\}}|\nabla U|^2\,dX<+\infty$$
for any $U\in\mathcal{J}$. So by the Rellich-Kondrachov theorem we have that $\mathcal{J}$ is totally 
bounded in $L^2(B_R^+\cap\{y\ge\eta\})$. Namely, 
there exist $\tilde{U}_1,\ldots,\tilde{U}_M\in L^2(B_R^+\cap\{y\ge\eta\})$ 
such that for any $U\in\mathcal{J}$ there exists $i\in\{1,\ldots,M\}$ such that
\begin{equation}\label{qeqwptoepyrekh}
\|U_i-U\|_{L^2(B_R^+\cap\{y\ge\eta\})}\le\frac{\epsilon^2}{2\max\{\eta^a,R^a\}}.
\end{equation}
Now for any $i\in\{1,\ldots,M\}$ we set 
$$ U_i:= 
\begin{cases}
\tilde{U}_i & {\mbox{ if }} y\ge\eta,\\
0 & {\mbox{ if }} y<\eta. 
\end{cases}$$
Notice that $U_i\in L^2(B^+_R,y^a)$ for any $i\in\{1,\ldots,M\}$. 
Indeed, fixed $i\in\{1,\ldots,M\}$, we have that
\begin{eqnarray*}
\int_{B_R^+}y^a|U_i|^2\,dX &=& \int_{B_R^+\cap \{y<\eta\}}y^a|U_i|^2\,dX +
\int_{B_R^+\cap \{y\ge\eta\}}y^a|U_i|^2\,dX \\
&=& 0 + \int_{B_R^+\cap \{y\ge\eta\}}y^a|\tilde{U}_i|^2\,dX \\
&\le & \max\{\eta^a,R^a\}\int_{B_R^+\cap \{y\ge\eta\}}|U_i|^2\,dX <+\infty,
\end{eqnarray*}
thanks to \eqref{lskfjgrekge-1} and the fact that $\tilde{U}_i\in L^2(B_R^+\cap\{y\ge\eta\})$ 
for any $i\in\{1,\ldots,M\}$. 

It remains to show \eqref{alewqewgew}. For this, we first observe that 
\begin{equation}\label{alewqewgew-3}
\|U_i-U\|_{L^2(B^+_R,y^a)}^2 = \int_{B_R^+\cap\{y<\eta\}}y^a |U|^2\,dX + 
\int_{B_R^+\cap\{y\ge \eta\}}y^a |\tilde{U}_i-U|^2\,dX.\end{equation} 
Using the
H\"older inequality with exponents $\gamma$ and $\frac{\gamma}{\gamma-1}$ 
and Proposition \ref{WeightedSob} and recalling \eqref{pwqrtt000} and \eqref{eta def}, we obtain that
\begin{eqnarray*}
\int_{B_R^+\cap\{y<\eta\}}y^a |U|^2\,dX &=& 
\int_{B_R^+\cap\{y<\eta\}}y^{\frac{a}{\gamma}} |U|^2 y^{\frac{a(\gamma-1)}{\gamma}}\,dX \\
& \le & 
\left(\int_{B_R^+\cap\{y<\eta\}}y^a |U|^{2\gamma}\,dX\right)^{\frac{1}{\gamma}} 
\left(\int_{B_R^+\cap\{y<\eta\}}y^a \,dX\right)^{\frac{\gamma-1}{\gamma}} \\
&\le & \hat{S}^2\int_{\R^{n+1}_+}y^a |\nabla U|^2\,dX \, 
\left(\frac{|B_R|}{a+1}\right)^{\frac{\gamma-1}{\gamma}}\eta^{\frac{(a+1)(\gamma-1)}{\gamma}}\\
&\le& \hat{S}^2 A \left(\frac{|B_R|}{a+1}\right)^{\frac{\gamma-1}{\gamma}}
\eta^{\frac{(a+1)(\gamma-1)}{\gamma}}\\
&=& \frac{\epsilon^2}{2}.
\end{eqnarray*}
Moreover, making use of \eqref{lskfjgrekge-1} and \eqref{qeqwptoepyrekh}, we have that 
$$ \int_{B_R^+\cap\{y\ge \eta\}}y^a |\tilde{U}_i-U|^2\,dX \le \max\{\eta^a,R^a\}
\int_{B_R^+\cap\{y\ge \eta\}}|\tilde{U}_i-U|^2\,dX \le \frac{\epsilon^2}{2}.
$$
Plugging the last two formulas into \eqref{alewqewgew-3}, we get 
$$ \|U_i-U\|_{L^2(B^+_R,y^a)}^2\le \frac{\epsilon^2}{2} +\frac{\epsilon^2}{2}=\epsilon^2,$$
which implies \eqref{alewqewgew} and thus concludes the proof of Lemma \ref{lemma:compact}.  
\end{proof}

\section{A Concentration-Compactness Principle}\label{sec:CC}

In this section we show a Concentration-Compactness Principle, 
in the spirit of the original result proved by P. L. Lions in \cite{L1} and \cite{L2}. 
In particular, we want to adapt Lemma 2.3 of \cite{L2}. 
See also, \cite{AGP, giampiero}, where this principle was proved 
for different problems. 

For this, we recall the following definitions: 
\begin{defn} \label{defTight}
We say that a sequence $\{U_k\}_{k\in\mathbb{N}}$ is tight 
if for every $\eta>0$ there exists $\rho>0$ such that
\begin{equation*}
\int_{\mathbb{R}^{n+1}_+\setminus B_\rho^+}{y^a|\nabla U_k|^2\,dX}\leq \eta \quad  {\mbox{ for any }}k.
\end{equation*}
\end{defn}

\begin{defn}\label{convMeasures}
Let $\{\mu_k\}_{k\in\mathbb{N}}$ be a sequence of measures on a topological space $X$. We say that $\mu_k$ converges to $\mu$ in $X$ if and only if
$$\lim_{k\to+\infty}\int_X{\varphi\,d\mu_k}= \int_X{\varphi\,d\mu},$$
for every $\varphi\in C_0(X)$.
\end{defn}

This definition is standard, see for instance Definition~1.1.2 in~\cite{evans}. 
In particular, we will consider measures on $\mathbb{R}^n$ and $\mathbb{R}^{n+1}_+$.

\begin{prop}[Concentration-Compactness Principle]\label{CCP}
Let $\{U_k\}_{k\in\mathbb{N}}$ be a bounded tight sequence in $\dot{H}^s_a(\mathbb{R}^{n+1}_+)$, such that $U_k$ converges weakly to $U$ in $\dot{H}^s_a(\mathbb{R}^{n+1}_+)$. Let $\mu,\nu$ be two nonnegative measures on $\mathbb{R}^{n+1}_+$ and $\mathbb{R}^n$ respectively and such that
\begin{equation}\label{first conv}
\lim_{k\to+\infty}y^a|\nabla U_k|^2=\mu
\end{equation}
and 
\begin{equation}\label{second conv}
\lim_{k\to+\infty}|U_k(x,0)|^{2^*_s}=\nu
\end{equation}
in the sense of Definition \ref{convMeasures}. 

Then, there exist an at most countable set $J$ and three families 
$\{x_j\}_{j\in J}\in \mathbb{R}^n$, $\{\nu_j\}_{j\in J}$, $\{\mu_j\}_{j\in J}$, $\nu_j,\mu_j\geq0$ such that
\begin{itemize}
\item[(i)] $\displaystyle \nu = |U(x,0)|^{2^*_s}+\sum_{j\in J}{\nu_j\delta_{x_j}}$,
\item[(ii)] $\displaystyle \mu \geq y^a|\nabla U|^2+\sum_{j\in J}{\mu_j\delta_{(x_j,0)}}$,
\item[(iii)] $\mu_j\geq S\nu_j^{2/2^*_s}$ for all $j\in J$.
\end{itemize}
\end{prop}

\begin{proof}
We first suppose that $U\equiv 0$. 
We claim that 
\begin{equation}\label{lemma lions}
\left(\int_{\R^n} |\varphi(x,0)|^{2^*_s}\, d\nu\right)^{2/2^*_s}\le C\int_{\R^{n+1}_+} \varphi^2 \, d\mu, \quad {\mbox{ for any }} \varphi\in C^\infty_0(\R^{n+1}_+),
\end{equation}
for some~$C>0$. For this, let $\varphi\in C_0^\infty (\mathbb{R}^{n+1}_+)$ and $K:=\,$supp$(\varphi)$. 
By Proposition \ref{traceIneq}, we have that
\begin{equation}\label{traceVarphi}
\left(\int_{\mathbb{R}^n}{|(\varphi U_k)(x,0)|^{2^*_s}\,dx}\right)^{2/2^*_s}\leq C
\int_{\mathbb{R}^{n+1}_+}{y^a|\nabla (\varphi U_k)|^2\,dX},
\end{equation}
for a suitable positive constant $C$. By~\eqref{second conv}, we deduce 
\begin{equation}\label{owqurwqhgf}
\lim_{k\to+\infty} \int_{\mathbb{R}^n}{|(\varphi U_k)(x,0)|^{2^*_s}\,dx}=
\int_{\mathbb{R}^n}{|\varphi(x,0)|^{2^*_s}\,d\nu}.
\end{equation}
On the other hand, the right hand side in~\eqref{traceVarphi} can be written as
\begin{eqnarray}\label{convRHS}
\int_{\mathbb{R}^{n+1}_+}{y^a|\nabla (\varphi U_k)|^2\,dX}&=& \int_{\mathbb{R}^{n+1}_+}{y^a\varphi^2|\nabla U_k|^2\,dX}+\int_{\mathbb{R}^{n+1}_+}{y^aU_k^2|\nabla \varphi|^2\,dX}\nonumber\\
&&\qquad +2\int_{\mathbb{R}^{n+1}_+}{y^a\varphi\,U_k\,\langle\nabla \varphi,\nabla U_k\rangle\,dX}.
\end{eqnarray}

Now we observe that 
\begin{equation}\label{pqotugeohgvg}
[U_k]_a\leq C \end{equation}
for some~$C>0$ independent of~$k$, 
and so, 
by Lemma \ref{lemma:compact}, we have that, up to a subsequence,
\begin{equation}\label{wjdehqwjpfognjbvn}
{\mbox{$U_k$ converges to~$U=0$ in~$L^2_{\rm loc}(\R^{n+1}_+,y^a)$ as~$k\to+\infty$.}}\end{equation} 
Therefore, 
\begin{equation}\label{laksjkashgj}
\lim_{k\to+\infty}\int_{\mathbb{R}^{n+1}_+}{y^aU_k^2|\nabla \varphi|^2\,dX}\leq C \lim_{k\to+\infty}\int_{K}{y^a U_k^2\,dX}=0.
\end{equation} 
Also, by the H\"older inequality and~\eqref{pqotugeohgvg}, 
\begin{eqnarray*}
&&\left|\int_{\mathbb{R}^{n+1}_+}{y^a\varphi\,U_k\,\langle\nabla \varphi,\nabla U_k\rangle\,dX}\right|\\ &\le & 
\left(\int_{\mathbb{R}^{n+1}_+}{y^a|\varphi|^2\,|\nabla U_k|^2\,dX}\right)^{1/2}\, 
\left(\int_{\mathbb{R}^{n+1}_+}{y^a|\nabla\varphi|^2\,|U_k|^2\,dX}\right)^{1/2}\\
&\le & C\, \left(\int_{\mathbb{R}^{n+1}_+}{y^a|\nabla U_k|^2\,dX}\right)^{1/2}\, 
\left(\int_{K}{y^a |U_k|^2\,dX}\right)^{1/2}\\
&\le & C\,\left(\int_{K}{y^a |U_k|^2\,dX}\right)^{1/2},
\end{eqnarray*}
where~$C$ may change from line to line. 
Hence, from~\eqref{wjdehqwjpfognjbvn} we have that 
$$ \lim_{k\to+\infty}\int_{\mathbb{R}^{n+1}_+}{y^a\varphi\,U_k\,\langle\nabla \varphi,\nabla U_k\rangle\,dX}=0.$$
Thus, plugging this and~\eqref{laksjkashgj} into \eqref{convRHS}, 
and using~\eqref{first conv}, 
we obtain
\begin{equation*}
\lim_{k\to+\infty}\int_{\mathbb{R}^{n+1}_+}{y^a|\nabla (\varphi U_k)|^2\,dX}= 
\int_{\mathbb{R}^{n+1}_+}{\varphi^2\,d\mu}.
\end{equation*}
Therefore, taking the limit in \eqref{traceVarphi} as~$k\to+\infty$, 
and using~\eqref{owqurwqhgf}, we get
\begin{equation*}
\left(\int_{\mathbb{R}^n}{|\varphi(x,0)|^{2^*_s}\,d\nu}\right)^{2/2^*_s}\leq 
C \int_{\mathbb{R}^{n+1}_+}{\varphi^2\,d\mu},\qquad\hbox{for all }\varphi\in C_0^\infty(\mathbb{R}^{n+1}_+), 
\end{equation*}
which shows~\eqref{lemma lions} in the case~$U\equiv 0$.

Let us consider now the case $U\not\equiv 0$. 
First, we define a function $V_k:=U_k-U$, and we observe that~$V_k\in\dot{H}^s_a(\mathbb{R}^{n+1}_+)$, 
and 
\begin{equation}\label{2.3base}
{\mbox{$V_k$ converges weakly to~0 in~$\dot{H}^s_a(\mathbb{R}^{n+1}_+)$
as~$k\to+\infty$.}}
\end{equation}
Also, we denote by
\begin{equation}\label{2.3bis}
\tilde{\nu}:=\lim_{k\rightarrow\infty}{|V_k(x,0)|^{2^*_s}}\qquad\hbox{ and }\qquad \tilde{\mu}:=\lim_{k\rightarrow\infty}{y^a|\nabla V_k|^2},
\end{equation}
where both limits are understood in the sense of Definition \ref{convMeasures}.
Then, we are in the previous case, and so we can apply~\eqref{lemma lions}, 
that is
\begin{equation}\label{traceVarphiVk}
\left(\int_{\mathbb{R}^n}{|\varphi(x,0)|^{2^*_s}\,d\tilde{\nu}}\right)^{2/2^*_s}\leq 
C \int_{\mathbb{R}^{n+1}_+}{\varphi^2\,d\tilde{\mu}},\qquad\hbox{for all }\varphi\in C_0^\infty(\mathbb{R}^{n+1}_+).
\end{equation}
Furthermore, by \cite{BL}, we know that
\begin{equation*}
\lim_{k\rightarrow\infty}{\int_{\mathbb{R}^n}{|(\varphi V_k)(x,0)|^{2^*_s}\,dx}}
=\lim_{k\rightarrow\infty}{\int_{\mathbb{R}^n}{|(\varphi U_k)(x,0)|^{2^*_s}\,dx}}-
\int_{\mathbb{R}^n}{|(\varphi U)(x,0)|^{2^*_s}\,dx},
\end{equation*}
that is, recalling~\eqref{2.3bis},
\begin{equation*}
\int_{\mathbb{R}^n}{|\varphi(x,0)|^{2^*_s}\,d\tilde{\nu}}=\int_{\mathbb{R}^n}{|\varphi (x,0)|^{2^*_s}\,d\nu}-
\int_{\mathbb{R}^n}{|(\varphi U)(x,0)|^{2^*_s}\,dx}.
\end{equation*}
Therefore 
\begin{equation}\label{2.3ter}
\nu=\tilde{\nu}+|U(x,0)|^{2^*_s}.
\end{equation}

On the other hand,
\begin{eqnarray*}
\int_{\mathbb{R}^{n+1}_+}{y^a\varphi^2|\nabla U_k|^2\,dX}&=& \int_{\mathbb{R}^{n+1}_+}{y^a\varphi^2|\nabla U|^2\,dX}+\int_{\mathbb{R}^{n+1}_+}{y^a\varphi^2|\nabla V_k|^2\,dX}\\
&&\qquad +2\int_{\mathbb{R}^{n+1}_+}{y^a\varphi^2\langle\nabla V_k,\nabla U\rangle\,dX}.
\end{eqnarray*}
Now we take the limit as~$k\to+\infty$, we use~\eqref{second conv}, \eqref{2.3bis}
and~\eqref{2.3base}, and we obtain 
\begin{equation*}
\int_{\mathbb{R}^{n+1}_+}{\varphi^2\,d\mu} =\int_{\mathbb{R}^{n+1}_+}{y^a\varphi^2|\nabla U|^2\,dX}+ \int_{\mathbb{R}^{n+1}_+}{\varphi^2\,d\tilde{\mu}},
\end{equation*}
i.e., 
\begin{equation}\label{2.45}
\mu=\tilde{\mu}+y^a|\nabla U|^2.
\end{equation} 

Now, since inequality \eqref{traceVarphiVk} is satisfied, 
we can apply Lemma~1.2 in \cite{L1} to~$\tilde{\nu}$ and~$\tilde{\mu}$
(see also Lemma 2.3 in~\cite{L2}). 
Therefore, there exist an at most countable set $J$ and families $\{x_j\}_{j\in J}\in \mathbb{R}^n$, $\{\nu_j\}_{j\in J}$, $\{\mu_j\}_{j\in J}$, with~$\nu_j\ge0$ and~$\mu_j>0$, such that
$$\tilde{\nu}=\sum_{j\in J}{\nu_j\delta_{x_j}} \quad{\mbox{ and }}\quad \tilde{\mu}\geq\sum_{j\in J}{\mu_j\delta_{(x_j,0)}}.$$
So the proof is finished, thanks to~\eqref{2.3ter} and~\eqref{2.45}.
\end{proof}

\chapter{Existence of a minimal solution and proof of Theorem~\ref{MINIMUM}}\label{ECXMII} 

\section{Some convergence results in view of Theorem~\ref{MINIMUM}}\label{sec:conv}

In this section we collect some results about the convergence 
of sequences of functions in suitable~$L^r(\R^n)$ spaces. 
We will exploit the following lemmata in the forthcoming Section~\ref{sec:PS}, 
see in particular the proof of Proposition~\ref{PScond}. 

The first result that we prove is the following:

\begin{lemma}\label{PSL-1}
Let~$v_k\in L^{2^*_s}(\R^n,[0,+\infty))$ be a sequence converging to some~$v$
in~$L^{2^*_s}(\R^n)$. Then
\begin{eqnarray}
\label{R45-2}
&&
\lim_{k\to+\infty}\int_{\R^n} |v_k^q(x)-v^q(x)|^{\frac{2^*_s}{q}}\,dx=0
\\ 
\label{R45-1}
{\mbox{and }}&&
\lim_{k\to+\infty}\int_{\R^n} |v_k^p(x)-v^p(x)|^{\frac{2n}{n+2s}}\,dx=0.
\end{eqnarray}
\end{lemma}

\begin{proof} For any~$t\ge-1$, let
$$ f(t):=\frac{|(1+t)^q-1|}{|t|^q}.$$
We have that~$(1+t)^q =1+qt+o(t)$ for~$t$ close to~$0$,
and therefore
$$ f(t)= \frac{|qt+o(t)|}{|t|^q}\to 0$$
as~$t\to0$. In addition, $f(-1)=1$ and
$$\lim_{t\to+\infty} f(t)=1.$$
As a consequence, we can define
$$ L:=\sup_{t\ge-1} f(t)$$
and we have that~$L\in[1,+\infty)$.
Now we show that
\begin{equation}\label{fdcvbpfp11}
|a^q-b^q|\le L|a-b|^q
\end{equation}
for any~$a$, $b\ge0$. To prove this, we can suppose that~$b\ne0$,
otherwise we are done, and we write~$t:=\frac{a}{b}-1$. Then we have that
$$ |a^q-b^q| = b^q |(1+t)^q -1|\le L b^q |t|^q=L|a-b|^q,$$
which proves~\eqref{fdcvbpfp11}.

As a consequence of this and of the convergence of~$v_k$, we have that
\begin{equation*}
\int_{\R^n} |v_k^q(x)-v^q(x)|^{\frac{2^*_s}{q}}\,dx
\le L\int_{\R^n} |v_k(x)-v(x)|^{2^*_s}\,dx \to0,
\end{equation*}
as~$k\to+\infty$,
which establishes~\eqref{R45-2}.
Now we prove~\eqref{R45-1}.
For this, given~$a\ge b\ge0$, we notice that
$$ a^p-b^p=
p\int_b^a t^{p-1}\,dt\le pa^{p-1}(a-b)\le
p(a+b)^{p-1}(a-b).$$
By possibly exchanging the roles of~$a$ and~$b$, we conclude that,
for any~$a$, $b\ge0$,
$$ |a^p-b^p|\le p(a+b)^{p-1}|a-b|.$$
Accordingly, for any~$a$, $b\ge0$,
$$ |a^p-b^p|^{\frac{2n}{n+2s}}
\le p^{\frac{2n}{n+2s}} (a+b)^{\frac{2n(p-1)}{n+2s}}
|a-b|^{\frac{2n}{n+2s}} =
p^{\frac{2n}{n+2s}} (a+b)^{\frac{8sn}{(n-2s)(n+2s)}}
|a-b|^{\frac{2n}{n+2s}}.$$
We use this and the H\"older inequality with exponents~$\frac{n+2s}{4s}$
and~$\frac{n+2s}{n-2s}$ to deduce that
\begin{eqnarray*}
&& \int_{\R^n} \big|v_k^p(x)-v^p(x)\big|^{\frac{2n}{n+2s}}\,dx\\
&\le & p^{\frac{2n}{n+2s}}
\int_{\R^n} \big( 
v_k(x)+v(x) \big)^{\frac{8sn}{(n-2s)(n+2s)}}
\big|
v_k(x)-v(x) \big|^{\frac{2n}{n+2s}}\,dx\\
&\le &
p^{\frac{2n}{n+2s}} \left(
\int_{\R^n} \big( 
v_k(x)+v(x) \big)^{\frac{2n}{n-2s}}
\,dx \right)^{ \frac{4s}{n+2s} }
\left( \int_{\R^n} 
\big|v_k(x)-v(x) \big|^{\frac{2n}{n-2s}}\,dx\right)^{\frac{n-2s}{n+2s}}\\
&= & p^{\frac{2n}{n+2s}}
\| v_k+v\|_{L^{2^*_s}(\R^n)}^{ \frac{8sn}{(n-2s)(n+2s)} }
\| v_k-v\|_{L^{2^*_s}(\R^n)}^{\frac{2n}{n+2s}}.
\end{eqnarray*}
{F}rom the convergence of~$v_k$, we have that~$\| v_k+v\|_{L^{2^*_s}(\R^n)}
\le\| v_k\|_{L^{2^*_s}(\R^n)}
+\| v\|_{L^{2^*_s}(\R^n)}$ is bounded uniformly in~$k$,
while~$\| v_k-v\|_{L^{2^*_s}(\R^n)}$ in infinitesimal as~$k\to+\infty$,
therefore~\eqref{R45-1}
now plainly follows.
\end{proof}

Next result shows that we can deduce strong convergence 
in~$L^{2^*_s}(\R^n)$ from the convergence in the sense of Definition~\ref{convMeasures}. 

\begin{lemma}\label{PSL-2}
Let~$v_k\in L^{2^*_s}(\R^n,[0,+\infty))$
be a sequence converging to some~$v$ a.e. in~$\R^n$.
Assume also that~$v_k^{2^*_s}$ converges to~$v^{2^*_s}$ in the
measure sense given in Definition~\ref{convMeasures}, i.e.
\begin{equation}\label{09ngjhgfnmxxu}
\lim_{k\to+\infty} \int_{\R^n}v_k^{2^*_s} \varphi\,dx=
\int_{\R^n}v^{2^*_s} \varphi\,dx
\end{equation}
for any~$\varphi\in C_0(\R^n)$.

In addition, assume that for any~$\eta>0$ there exists~$\rho>0$
such that
\begin{equation}\label{TI67}
\int_{\R^n\setminus B_\rho }v_k^{2^*_s} (x)\,dx <\eta.
\end{equation}
Then, $v_k\to v$ in~$L^{2^*_s}(\R^n,[0,+\infty))$
as~$k\to+\infty$.
\end{lemma}

\begin{proof} First of all, by Fatou's lemma,
\begin{equation}\label{FA6ichofhav}
\lim_{k\to+\infty} \int_{\R^n}v_k^{2^*_s} \,dx\geq
\int_{\R^n}v^{2^*_s}\,dx.
\end{equation}
Now we fix~$\eta>0$ and we take~$\rho=\rho(\eta)$ such that~\eqref{TI67}
holds true. Let~$
\varphi_\rho\in C^\infty_0(B_{\rho+1},[0,1])$
such that~$\varphi_\rho=1$ in~$B_\rho$. Then, by~\eqref{TI67}
$$ \int_{\R^n}v_k^{2^*_s} \,dx <
\int_{B_\rho }v_k^{2^*_s} \,dx +\eta
\le \int_{\R^n}v_k^{2^*_s}\varphi_\rho \,dx +\eta .$$
Hence, exploiting~\eqref{09ngjhgfnmxxu},
$$ \lim_{k\to+\infty} \int_{\R^n}v_k^{2^*_s} \,dx
\le
\lim_{k\to+\infty} 
\int_{\R^n}v_k^{2^*_s}\varphi_\rho \,dx +\eta =
\int_{\R^n}v^{2^*_s}\varphi_\rho \,dx +\eta.$$
Since~$\varphi_\rho\le 1$, this gives that
$$ \lim_{k\to+\infty} \int_{\R^n}v_k^{2^*_s} \,dx  
\le   
\int_{\R^n}v^{2^*_s}\,dx +\eta.$$
Since~$\eta$ can be taken arbitrarily small,
we obtain that
$$ \lim_{k\to+\infty} \int_{\R^n}v_k^{2^*_s} \,dx
\le
\int_{\R^n}v^{2^*_s}\,dx .$$
This, together with~\eqref{FA6ichofhav}, proves that
$$ \lim_{k\to+\infty}
\|v_k\|_{L^{2^*_s}(\R^n)}^{2^*_s}=
\lim_{k\to+\infty} \int_{\R^n}v_k^{2^*_s} \,dx=
\int_{\R^n}v^{2^*_s}\,dx=\|v\|_{L^{2^*_s}(\R^n)}^{2^*_s}.$$
This and the Brezis-Lieb lemma (see e.g. formula~(1)
in~\cite{BL}) implies the desired result.
\end{proof}

\section{Palais-Smale condition for~${\mathcal{F}}_\epsilon$}\label{sec:PS}

In this section we show that the functional~$\mathcal{F}_\epsilon$ 
introduced in \eqref{f ext} satisfies a Palais-Smale condition. 
The precise statement is contained in the following proposition. 

\begin{prop}[Palais-Smale condition]\label{PScond}
There exists~$\bar C, c_1>0$, depending on~$h$, $q$, $n$
and~$s$, such that the following statement
holds true.

Let $\{U_k\}_{k\in\mathbb{N}}\subset \dot{H}^s_a(\mathbb{R}^{n+1}_+)$ 
be a sequence satisfying
\begin{enumerate}
\item[(i)]$\displaystyle\lim_{k\to+\infty}\mathcal{F}_\epsilon(U_k)= 
c_\epsilon$, with 
\begin{equation}\label{ceps}
c_\epsilon+c_1\varepsilon^{1/\gamma}+\bar C \epsilon^{\frac{p+1}{p-q}}<
\dfrac{s}{n}S^{\frac{n}{2s}},\end{equation} 
where $\gamma=1+\frac{2}{n-2s}$ and $S$ is the Sobolev constant appearing in Proposition~\ref{traceIneq},
\item[(ii)]$\displaystyle\lim_{k\to+\infty}\mathcal{F}'_\epsilon(U_k)= 0.$
\end{enumerate}
Then there exists a subsequence, still denoted by~$\{U_k\}_{k\in\mathbb{N}}$, 
which is strongly convergent in $\dot{H}^s_a(\mathbb{R}^{n+1}_+)$ as~$k\to+\infty$.
\end{prop}

\begin{rem} \label{rem:3.3-1}
The limit in ii) is intended in the following way 
\begin{eqnarray*}
&& \lim_{k\to+\infty}\|\mathcal{F}'_\epsilon(U_k)\|_
{\mathcal{L}(\dot{H}^s_a(\mathbb{R}^{n+1}_+),\dot{H}^s_a(\mathbb{R}^{n+1}_+))} 
\\ &&\qquad = \lim_{k\to+\infty}\sup_{{V\in \dot{H}^s_a(\R^{n+1}_+)}\atop{ [V]_a =1 }} 
\left|\langle \mathcal{F}'_\epsilon(U_k), V\rangle\right|
=0,
\end{eqnarray*}
where $\mathcal{L}(\dot{H}^s_a(\mathbb{R}^{n+1}_+),\dot{H}^s_a(\mathbb{R}^{n+1}_+))$ 
consists of all the linear functional from $\dot{H}^s_a(\mathbb{R}^{n+1}_+)$ 
in $\dot{H}^s_a(\mathbb{R}^{n+1}_+)$.
\end{rem}

First we show that a sequence that satisfies the assumptions in Proposition \ref{PScond} is bounded. 

\begin{lemma}\label{lemma bound}
Let~$\epsilon$, $\kappa>0$.
Let $\{U_k\}_{k\in\mathbb{N}}\subset \dot{H}^s_a(\mathbb{R}^{n+1}_+)$ be a sequence satisfying
\begin{equation}\label{9sd45678trdfghbvcrtyujbv}
|{\mathcal{F}}_\epsilon (U_k)| +
\sup_{{V\in \dot{H}^s_a(\R^{n+1}_+)}\atop{ [V]_a =1 }}
\big|\langle {\mathcal{F}}_\epsilon'(U_k),V\rangle\big|\le\kappa,\end{equation}
for any~$k\in\N$.

Then there exists $M>0$ such that
\begin{equation}\label{bound}
[U_k]_a\leq M.
\end{equation}
\end{lemma}

\begin{proof} If~$[U_k]_a=0$ we are done. So we can suppose that~$[U_k]_a\ne0$
and use~\eqref{9sd45678trdfghbvcrtyujbv} to obtain
$$
\left|\mathcal{F}_\epsilon(U_k)\right|\le \kappa ,\hbox{ and } 
\left|\langle \mathcal{F}'_\epsilon(U_k), U_k/[U_k]_{a}\rangle\right| \le \kappa.$$
Therefore, we have that 
\begin{equation}\label{lakjdkfeowpt}
\mathcal{F}_\epsilon(U_k)-\frac{1}{p+1}\langle \mathcal{F}'_\epsilon(U_k), U_k\rangle 
\le \kappa\,\left(1+ [U_k]_a\right). 
\end{equation}
On the other hand, by the H\"older inequality and Proposition \ref{traceIneq}, we obtain 
\begin{eqnarray*}
&& \mathcal{F}_\epsilon(U_k)-\frac{1}{p+1}\langle \mathcal{F}'_\epsilon(U_k), U_k\rangle\\
& =&\left(\frac{1}{2}-\frac{1}{p+1}\right)\int_{\mathbb{R}^{n+1}}{y^a|\nabla U_k|^2\,dX}
-\varepsilon\left(\frac{1}{q+1}-\frac{1}{p+1}\right)\int_{\mathbb{R}^n}{h(x)(U_k)_+^{q+1}(x,0)\,dx}\\
&\geq & \left(\frac{1}{2}-\frac{1}{p+1}\right)[U_k]_a^2
-\varepsilon C\left(\frac{1}{q+1}-\frac{1}{p+1}\right)
\|h\|_{L^{\frac{p+1}{p-q}}(\mathbb{R}^n)}[U_k]_a^{q+1}.
\end{eqnarray*}
From this and \eqref{lakjdkfeowpt} we conclude that $[U_k]_a$ 
must be bounded (recall also \eqref{h0} and that $q+1<2$). 
So we obtain the desired result. 
\end{proof}

In order to prove that $\mathcal{F}_\epsilon$ satisfies 
the Palais-Smale condition, we need to show that the sequence of functions 
satisfying the hypotheses of Proposition \ref{PScond} is tight, 
according to Definition \ref{defTight}. 

First we make the following preliminary observation:
\begin{lemma}\label{lemma basic}
Let $m:=\frac{p+1}{p-q}$. 
Then there exists a constant $\bar{C}=\bar{C}(n,s,p,q,\|h\|_{L^m(\mathbb{R}^n)})>0$ such that, 
for any $\alpha>0$, 
\begin{equation*}
\frac{s}{n}\alpha^{p+1}-\epsilon\left(\frac{1}{q+1}-
\frac{1}{2}\right)\|h\|_{L^m(\mathbb{R}^n)}\alpha^{q+1}
\geq -\bar{C}\epsilon^{\frac{p+1}{p-q}}.
\end{equation*}
\end{lemma}

\begin{proof}
Let us define the function $f:(0,+\infty)\rightarrow\mathbb{R}$ as
$$f(\alpha):=c_1\alpha^{p+1}-\varepsilon c_2 \alpha^{q+1}, \;\quad c_1:=\frac{s}{n},\;\quad c_2:=\left(\frac{1}{q+1}-
\frac{1}{2}\right)\|h\|_{L^m(\mathbb{R}^n)}.$$
Differentiating, we obtain that
$$f'(\alpha)=\alpha^q((p+1)c_1\alpha^{p-q}-\varepsilon (q+1)c_2),$$
and thus, $f$ has a local minimum at the point 
$$\overline{\alpha}:= c_3\varepsilon^{\frac{1}{p-q}},\quad c_3=c_3(n,s,p,q,\|h\|_{L^m(\mathbb{R}^n)}):=\left(\frac{c_2(q+1)}{c_1(p+1)}\right)^{\frac{1}{p-q}}.$$
Evaluating $f$ at $\overline{\alpha}$, we obtain that the minimum value that $f$ will reach is
$$f(\overline{\alpha})=c_4\varepsilon^{\frac{p+1}{p-q}},$$
with $c_4$ a constant depending on $n$, $s$, $p$, $q$ and $\|h\|_{L^m(\mathbb{R}^n)}$. Therefore, there exists $\bar{C}=\bar{C}(n,s,p,q,\|h\|_{L^m(\mathbb{R}^n)})>0$ such that
$$f(\alpha)\geq f(\overline{\alpha})\geq -\bar{C}\varepsilon^{\frac{p+1}{p-q}},$$
for any $\alpha>0$, and this concludes the proof.
\end{proof}

The tightness of the sequence in Proposition \ref{PScond} is 
contained in the following lemma: 

\begin{lemma}[Tightness]\label{tightness}
Let $\{U_k\}_{k\in\mathbb{N}}\subset \dot{H}^s_a(\mathbb{R}^{n+1}_+)$ be a sequence satisfying
the hypotheses of Proposition \ref{PScond}. 

Then for all $\eta>0$ there exists $\rho>0$ such that for every $k\in\mathbb{N}$ it holds
$$\int_{\mathbb{R}^{n+1}_+\setminus B_\rho^+}{y^a|\nabla U_k|^2\,dX}
+\int_{\mathbb{R}^n\setminus\{B_\rho\cap\{y=0\}\}}{(U_k)_+^{2^*_s}(x,0)\,dx}
<\eta.$$
In particular, the sequence $\{U_k\}_{k\in\mathbb{N}}$ is tight. 
\end{lemma}

\begin{proof}
First we notice that~\eqref{9sd45678trdfghbvcrtyujbv}
holds in this case, due to conditions~(i) and~(ii) in
Proposition~\ref{PScond}. Hence,
Lemma \ref{lemma bound} 
gives that the sequence 
$\{U_k\}_{k\in\mathbb{N}}$ is bounded in $\dot{H}^s_a(\mathbb{R}^{n+1}_+)$, that is~$[U_k]_a\leq M$. Thus,
\begin{equation}\begin{split}\label{weak convergence-1}
& U_k\rightharpoonup U \quad \hbox{ in }\dot{H}^s_a(\mathbb{R}^{n+1}_+) \quad {\mbox{ as }}k\to+\infty \\
{\mbox{and }}& U_k\rightarrow U\quad  \hbox{ a.e. in }\mathbb{R}^{n+1}_+\quad {\mbox{ as }}k\to+\infty.
\end{split}\end{equation}

Now, we proceed by contradiction. Suppose that there exists $\eta_0>0$ 
such that for all $\rho>0$ there exists~$k=k(\rho)\in\N$ such that
\begin{equation}\label{contrad}
\int_{\mathbb{R}^{n+1}_+\setminus B_\rho^+}{y^a|\nabla U_k|^2\,dX}
+\int_{\mathbb{R}^n\setminus\{B_\rho\cap\{y=0\}\}}{(U_k)_+^{2^*_s}(x,0)\,dx}
\geq \eta_0.
\end{equation}
We observe that 
\begin{equation}\label{forse0} 
k\to+\infty \quad {\mbox{ as }}\rho\to+\infty.
\end{equation}
Indeed, let us take a sequence $\{\rho_i\}_{i\in\mathbb{N}}$ such that $\rho_i\rightarrow +\infty$ as $i\rightarrow +\infty$, and suppose that $k_i:=k(\rho_i)$ given by \eqref{contrad} is a bounded sequence. That is, the set $F:=\{k_i:\;i\in\N\}$ is a finite set of integers.

Hence, there exists an integer $k^\star$ so that we can extract a subsequence $\{k_{i_j}\}_{j\in\N}$ satisfying $k_{i_j}=k^\star$ for any $j\in\N$. Therefore,
\begin{equation}\label{contrad2}
\int_{\mathbb{R}^{n+1}_+\setminus B_{\rho_{i_j}}^+}{y^a|\nabla U_{k^\star}|^2\,dX}
+\int_{\mathbb{R}^n\setminus\{B_{\rho_{i_j}}\cap\{y=0\}\}}{(U_{k^\star})_+^{2^*_s}(x,0)\,dx}
\geq \eta_0,
\end{equation}
for any $j\in \N$. 

But on the other hand, since~$U_{k^\star}$
belongs to~$\dot{H}^s_a(\mathbb{R}^{n+1}_+)$ 
(and so $U_{k^\star}(\cdot,0)\in L^{2^*_s}(\R^n)$ 
thanks to Proposition \ref{traceIneq}), for $j$ large enough there holds
\begin{equation*}
\int_{\mathbb{R}^{n+1}_+\setminus B_{\rho_{i_j}}^+}{y^a|\nabla U_{k^\star}|^2\,dX}
+\int_{\mathbb{R}^n\setminus\{B_{\rho_{i_j}}\cap\{y=0\}\}}{(U_{k^\star})_+^{2^*_s}(x,0)\,dx}
\leq \frac{\eta_0}{2},
\end{equation*}
which is a contradiction with \eqref{contrad2}.
This shows~\eqref{forse0}. 

Now, since $U$ given in~\eqref{weak convergence-1} belongs to~$\in\dot{H}^s_a(\mathbb{R}^{n+1}_+)$, by Propositions~\ref{WeightedSob} 
and~\ref{traceIneq}, 
we have that, for a fixed $\varepsilon>0$, there exists $r_\epsilon>0$ such that
$$\int_{\mathbb{R}^{n+1}_+\setminus B_{r_\epsilon}^+}{y^a|\nabla U|^2\,dX}
+\int_{\mathbb{R}^{n+1}_+\setminus B_{r_\epsilon}^+}{y^a|U|^{2\gamma}\,dX}
+\int_{\mathbb{R}^n\setminus\{B_{r_\epsilon}\cap\{y=0\}\}}{|U(x,0)|^{2^*_s}\,dx}<\varepsilon.$$
Notice that 
\begin{equation}\label{eps to zero}
{\mbox{$r_\epsilon\to +\infty$ as $\epsilon\to 0$.}} 
\end{equation}

Moreover, by \eqref{bound} and again by Propositions~\ref{WeightedSob} 
and~\ref{traceIneq}, we obtain that there exists $\tilde{M}>0$ such that
\begin{equation}\label{boundk}
\int_{\mathbb{R}^{n+1}_+}{y^a|\nabla U_k|^2\,dX}+\int_{\mathbb{R}^{n+1}_+}{y^a|U_k|^{2\gamma}\,dX}
+\int_{\mathbb{R}^n}{|U_k(x,0)|^{2^*_s}\,dx}\leq \tilde{M}.
\end{equation}

Now let $j_\epsilon\in\mathbb{N}$ be the integer part of $\frac{\tilde{M}}{\varepsilon}$. Notice that~$j_\epsilon$ tends to~$+\infty$ as~$\epsilon$ 
tends to~0. We also set
$$ I_l:=\{(x,y)\in\mathbb{R}^{n+1}_+:r+l\leq |(x,y)|\leq r+(l+1)\},\;l=0,1,\cdots,j_\epsilon.$$
Thus, from~\eqref{boundk} we get
\begin{eqnarray*}
(j_\epsilon+1)\varepsilon &\geq &
\frac{\tilde{M}}{\epsilon}\epsilon\\&\ge &
 \sum_{l=0}^{j_\epsilon}\left({\int_{I_l}{y^a|\nabla U_k|^2\,dX}
 +\int_{I_l}{y^a|U_k|^{2\gamma}\,dX}
+\int_{I_l\cap\{y=0\}}{|U_k(x,0)|^{2^*_s}\,dx}}\right).
\end{eqnarray*}
This implies that there exists $\bar{l}\in\{0,1,\cdots, j_\epsilon\}$ such that, 
up to a subsequence, 
\begin{equation}\label{epsBound}
\int_{I_{\bar{l}}}{y^a|\nabla U_k|^2\,dX}+\int_{I_{\bar{l}}}{y^a|U_k|^{2\gamma}\,dX}
+\int_{I_{\bar{l}}\cap\{y=0\}}{|U_k(x,0)|^{2^*_s}\,dx}\leq \varepsilon.
\end{equation}

Now we take a cut-off function~$\chi\in C^\infty_0(\R^{n+1}_+,[0,1])$, 
such that
\begin{equation}\label{3.4bis}
\chi(x,y)=\begin{cases}
1,\quad |(x,y)|\leq r+\bar{l}\\
0,\quad |(x,y)|\geq r+(\bar{l}+1),
\end{cases}
\end{equation}
and 
\begin{equation}\label{3.4bisbis}
|\nabla \chi|\leq 2.
\end{equation} 
We also define 
\begin{equation}\label{3.4ter}
V_k:=\chi U_k \quad {\mbox{ and }}\quad W_k:=(1-\chi)U_k.
\end{equation}
We estimate
\begin{equation}\begin{split}\label{math F}
&|\langle \mathcal{F}'_\epsilon(U_k)-\mathcal{F}'_\epsilon(V_k),V_k\rangle |\\
&\quad = \bigg|\int_{\mathbb{R}^{n+1}_+}{y^a\langle\nabla U_k,\nabla V_k\rangle\,dX}
-\epsilon \int_{\mathbb{R}^{n}}{h(x) (U_k)_+^q(x,0)\,V_k(x,0)\,dx}\\
&\qquad\quad -\int_{\mathbb{R}^{n}}{(U_k)_+^p(x,0)\,V_k(x,0)\,dx}
-\int_{\mathbb{R}^{n+1}_+}{y^a\langle\nabla V_k,\nabla V_k\rangle\,dX}\\
&\qquad\quad +\epsilon \int_{\mathbb{R}^{n}}{h(x) (V_k)_+^{q+1}(x,0)\,dx}
+\int_{\mathbb{R}^{n}}{(V_k)_+^{p+1}(x,0)\,dx}\bigg|.
\end{split}\end{equation}
First, we observe that
\begin{equation}\begin{split}\label{AA}
&\bigg|\int_{\mathbb{R}^{n+1}_+}{y^a\langle\nabla U_k,\nabla V_k\rangle\,dX}-\int_{\mathbb{R}^{n+1}_+}{y^a\langle\nabla V_k,\nabla V_k\rangle\,dX}\bigg|\\
&\qquad\leq \int_{I_{\overline{l}}}{y^a|\nabla U_k|^2|\chi||1-\chi|\,dX}+\int_{I_{\overline{l}}}{y^a|\nabla U_k||U_k||\nabla\chi|\,dX}\\
&\qquad\qquad +2\int_{I_{\overline{l}}}{y^a|U_k||\nabla U_k||\nabla\chi||\chi|\,dX}+\int_{I_{\overline{l}}}{y^a|U_k|^2|\nabla \chi|^2\,dX}\\
&\qquad =:A_1+A_2+A_3+A_4.
\end{split}\end{equation}
By \eqref{epsBound}, we have that $A_1\leq C\varepsilon$, for some $C>0$. 
Furthermore, by the H\"older inequality, \eqref{3.4bisbis} and \eqref{epsBound}, we obtain
\begin{eqnarray*}
A_2 &\leq& 2\int_{I_{\overline{l}}}{y^a|\nabla U_k||U_k|\,dX}\leq 2\left(\int_{I_{\overline{l}}}{y^a|\nabla U_k|^2\,dX}\right)^{1/2}\left(\int_{I_{\overline{l}}}{y^a| U_k|^2\,dX}\right)^{1/2}\\
&\leq& 2\varepsilon^{1/2} \left(\int_{I_{\overline{l}}}{y^a|U_k|^{2\gamma}\,dX}\right)^{1/{2\gamma}}
\left(\int_{I_{\overline{l}}}{y^{(a-\frac{a}{\gamma})m}\,dX}\right)^{1/2m},
\end{eqnarray*}
where $m=\dfrac{n+2-2s}{2}$. Since 
$\left(a-\dfrac{a}{\gamma}\right)m=a=(1-2s)>-1$,  we have that the second integral is finite, 
and therefore, for $\varepsilon<1$,
\begin{equation*}
A_2\leq \tilde{C}\varepsilon^{1/2} \left(\int_{ I_{\overline{l} }}{y^a|U_k|^{2\gamma}\,
dX}\right)^{1/{2\gamma}}\leq C\varepsilon^{1/2}\epsilon^{1/2\gamma}\le C\epsilon^{1/\gamma},
\end{equation*}
where \eqref{epsBound} was used once again. 
In the same way, we get that $A_3\leq C\varepsilon^{1/\gamma}$. Finally, by \eqref{epsBound},
\begin{equation*}
A_4\leq C\left(\int_{I_{ \overline{l} }}{y^a|U_k|^{2\gamma}\,dX}\right)^{1/{\gamma}}
\left(\int_{I_{\overline{l}}}{y^{(a-\frac{a}{\gamma})m}\,dX}\right)^{1/m}\leq C\varepsilon^{1/\gamma}.
\end{equation*}
Using these informations in \eqref{AA}, we obtain that 
$$ \bigg|\int_{\mathbb{R}^{n+1}_+}{y^a\langle\nabla U_k,\nabla V_k\rangle\,dX}
-\int_{\mathbb{R}^{n+1}_+}{y^a\langle\nabla V_k,\nabla V_k\rangle\,dX}\bigg|
\le C\epsilon^{1/\gamma}, $$
up to renaming the constant $C$. 

On the other hand, since $p+1=2^*_s$, by \eqref{3.4ter} and \eqref{epsBound}, 
\begin{eqnarray*}
\bigg|\int_{\mathbb{R}^n}{( (U_k)_+^p(x,0)\,V_k(x,0)-(V_k)_+^{p+1}(x,0))\,dx}\bigg|
&\le &\int_{\mathbb{R}^n}{|1-\chi^p||\chi| |U_k(x,0)|^{p+1}\,dx}\\
&\leq& C\int_{I_{\overline{l}}\cap\{y=0\}}{|U_k(x,0)|^{2^*_s}\,dx}\leq C\varepsilon.
\end{eqnarray*}
In the same way, applying the H\"older inequality, one obtains
\begin{eqnarray*}
&&\bigg|\int_{\mathbb{R}^n}{h(x)\,((U_k)_+^q(x,0)\,V_k(x,0)-(V_k)_+^{q+1}(x,0))\,dx}\bigg|
\\&&\qquad \le \int_{\mathbb{R}^n}{|h(x)|\,|1-\chi^q||\chi| |U_k(x,0)|^{q+1}\,dx}\\
&&\qquad \leq C\, \|h\|_{L^\infty(\R^n)}\,
\int_{I_{\overline{l}}\cap\{y=0\}}{|U_k(x,0)|^{2^*_s}\,dx}\leq C\varepsilon.
\end{eqnarray*}

All in all, plugging these observations in \eqref{math F}, we obtain that 
\begin{equation}\label{boundV}
|\langle \mathcal{F}'_\epsilon(U_k)-\mathcal{F}'_\epsilon(V_k),V_k\rangle|
\leq C\varepsilon^{1/\gamma}.
\end{equation}
Likewise, one can see that
\begin{equation}\label{boundW}
|\langle \mathcal{F}'_\epsilon(U_k)-\mathcal{F}'_\epsilon(W_k),W_k\rangle|
\leq C\varepsilon^{1/\gamma}.
\end{equation}

Now we claim that 
\begin{equation}\label{fprimeV}
|\langle \mathcal{F}'_\epsilon(V_k),V_k\rangle|\leq C\varepsilon^{1/\gamma}+o_k(1),
\end{equation}
where $o_k(1)$ denotes (here and in the rest of this monograph)
a quantity that tends to 0 as $k$ tends to $+\infty$. 
For this, we first observe that 
\begin{equation}\label{bbbb}
[V_k]_a\le C,
\end{equation}
for some $C>0$. Indeed, recalling \eqref{3.4ter} and using \eqref{3.4bis} 
and \eqref{3.4bisbis}, we have 
\begin{eqnarray*}
[V_k]_a^2 &=& \int_{\R^{n+1}_+}y^a|\nabla V_k|^2\,dX \\ 
&=& \int_{\R^{n+1}_+}y^a|\nabla\chi|^2|U_k|^2\,dX + 
\int_{\R^{n+1}_+}y^a\,\chi^2|\nabla U_k|^2\,dX + 2\int_{\R^{n+1}_+}y^a\,\chi\,U_k
\ \langle \nabla U_k, \nabla\chi\rangle\,dX\\
&\le & 4 \int_{I_{\overline{l}} }y^a| U_k|^2\,dX + [U_k]_a^2 +
C\left(\int_{ I_{\overline{l}}}y^a|\nabla U_k|^2\,dX\right)^{1/2}\, 
\left(\int_{I_{\overline{l}}}y^a |U_k|^2\,dX\right)^{1/2}\\
&\le & C \left(\int_{I_{\overline{l}} }y^a| U_k|^{2\gamma}\,dX\right)^{1/\gamma} 
+ [U_k]_a^2 +
C\, [U_k]_a\, \left(\int_{I_{\overline{l}}}y^a |U_k|^{2\gamma}\,dX\right)^{1/2\gamma}, 
\end{eqnarray*}
where the H\"older inequality was used in the last two lines.  
Hence, from Proposition \ref{WeightedSob} and \eqref{bound}, we obtain \eqref{bbbb}.

Now, we notice that 
\begin{eqnarray*}
|\langle \mathcal{F}'_\epsilon(V_k),V_k\rangle| \le 
|\langle \mathcal{F}'_\epsilon(V_k)-\mathcal{F}'_\epsilon(U_k),V_k\rangle| + 
|\langle \mathcal{F}'_\epsilon(U_k),V_k\rangle| \le  
C\,\epsilon^{1/\gamma} +|\langle \mathcal{F}'_\epsilon(U_k),V_k\rangle|,
\end{eqnarray*}
thanks to \eqref{boundV}. 
Thus, from \eqref{bbbb} and assumption (ii) in Proposition \ref{PScond} 
we get the desired claim in \eqref{fprimeV}. 

Analogously (but making use of \eqref{boundW}), one can see that  
\begin{equation}\label{fprimeW}
|\langle \mathcal{F}'_\epsilon(W_k),W_k\rangle|\leq C\varepsilon^{1/\gamma}+o_k(1),
\end{equation}

From now on, we divide the proof in three main steps: 
we first show lower bounds for $\mathcal{F}_\epsilon(V_k)$ 
and $\mathcal{F}_\epsilon(W_k)$ (see Step 1 and Step 2, respectively), 
then in Step 3 we obtain a lower bound for $\mathcal{F}_\epsilon(U_k)$, 
which will give a contradiction with the hypotheses on $\mathcal{F}_\epsilon$, 
and so the conclusion of Lemma \ref{tightness}.

\medskip

\noindent {\it Step 1: Lower bound for $\mathcal{F}_\epsilon(V_k)$.} 
By \eqref{fprimeV} we obtain that 
\begin{equation}\label{skajgfoew}
\mathcal{F}_\epsilon(V_k)\geq \mathcal{F}_\epsilon(V_k)
-\frac{1}{2}\langle \mathcal{F}'_\epsilon(V_k), V_k\rangle-\frac{1}{2}C\varepsilon^{1/\gamma}
+o_k(1).
\end{equation}
Using the H\"older inequality, it yields
\begin{equation*}\begin{split}
\mathcal{F}_\epsilon(V_k)&-\frac{1}{2}\langle \mathcal{F}'_\epsilon(V_k), V_k\rangle 
= \left(\frac{1}{2}-\frac{1}{p+1}\right)\|(V_k)_+(\cdot,0)\|_{L^{p+1}(\mathbb{R}^n)}^{p+1}\\
&-\epsilon\left(\frac{1}{q+1}-\frac{1}{2}\right)\int_{\mathbb{R}^n}{h(x)(V_k)_+^{q+1}(x,0)\,dx}\\
&\geq\frac{s}{n}\|(V_k)_+(\cdot,0)\|_{L^{p+1}(\mathbb{R}^n)}^{p+1}
-\epsilon\left(\frac{1}{q+1}-\frac{1}{2}\right)\|h\|_{L^m(\mathbb{R}^n)}\|(V_k)_+(\cdot,0)\|_{L^{p+1}(\mathbb{R}^n)}^{q+1},
\end{split}\end{equation*}
with $m=\dfrac{p+1}{p-q}$ 
(recall \eqref{h0}).  
Therefore, from Lemma \ref{lemma basic} 
(applied here with $\alpha:=\|(V_k)_+(\cdot,0)\|_{L^{p+1}(\mathbb{R}^n)}$) we deduce that 
$$ \mathcal{F}_\epsilon(V_k)-\frac{1}{2}\langle \mathcal{F}'_\epsilon(V_k), V_k\rangle 
\ge -\bar{C}\,\epsilon^{\frac{p+1}{p-q}}.$$
Going back to \eqref{skajgfoew}, this implies that 
\begin{equation}\label{LowerBoundV}
\mathcal{F}_\varepsilon(V_k)\geq -c_0\epsilon^{1/\gamma}
-\bar{C}\,\epsilon^{\frac{p+1}{p-q}}+o_k(1).
\end{equation}
\\

\noindent {\it Step 2: Lower bound for $\mathcal{F}_\epsilon(W_k)$.} 
First of all, by the definition of $W_k$ in \eqref{3.4ter}, Proposition \ref{traceIneq} 
and \eqref{bound}, we have that 
\begin{equation}\begin{split}\label{upBoundWq}
\bigg|\epsilon\int_{\mathbb{R}^n}{ h(x)(W_k)_+^{q+1}(x,0)\,dx }\bigg|
\leq\,&\epsilon\, \|h\|_{L^m(\mathbb{R}^n)}
\|(W_k)_+(\cdot,0)\|_{L^{2^*_s}(\mathbb{R}^n)}^{q+1}\\
\leq\,& \epsilon\, C\,\|h\|_{L^m(\mathbb{R}^n)}\|(U_k)_+(\cdot,0)\|_{L^{2^*_s}(\mathbb{R}^n)}^{q+1}\\
\leq\,& \epsilon\, C\, \|h\|_{L^m(\mathbb{R}^n)}[U_k]_a^{q+1}\leq C\epsilon.
\end{split}\end{equation}
Thus, from \eqref{boundW} we get 
\begin{equation}\begin{split}\label{Wbound}
&\bigg|\int_{\mathbb{R}^{n+1}_+} {y^a|\nabla W_k|^2\,dX}
-\int_{\mathbb{R}^n}  {(W_k)_+^{p+1} (x,0)\,dx} \bigg|
\\ &\qquad\le \left|\langle \mathcal{F}'_\epsilon(W_k),W_k\rangle\right| + 
\left|\epsilon\int_{\mathbb{R}^n}{ h(x) (W_k)_+^{q+1}(x,0)\,dx}\right|\\
&\qquad \leq C\varepsilon^{1/\gamma} + o_k(1),
\end{split}\end{equation}
where~\eqref{fprimeW} was also used in the last passage. 
Moreover, notice that $W_k=U_k$ in $\R^{n+1}_+\setminus B_{r+\overline{l}+1}$ 
(recall \eqref{3.4bis} and \eqref{3.4ter}). 
Hence, using \eqref{contrad} with $\rho:=r+\overline{l}+1$, we get 
\begin{equation}\begin{split}\label{espero}
&\int_{\mathbb{R}^{n+1}_+\setminus B^+_{r+\bar{l}+1}}{y^a|\nabla W_k|^2\,dX}
+\int_{\mathbb{R}^n\setminus\{B_{r+\bar{l}+1}\cap\{y=0\}\}}{(W_k)_+^{2^*_s}(x,0)\,dx}\\
&\qquad =\int_{\mathbb{R}^{n+1}_+\setminus B^+_{r+\bar{l}+1}}{y^a|\nabla U_k|^2\,dX}
+\int_{\mathbb{R}^n\setminus\{B_{r+\bar{l}+1}\cap\{y=0\}\}}
{(U_k)_+^{2^*_s}(x,0)\,dx}
\geq \eta_0,
\end{split}\end{equation}
for $k=k(\rho)$. 
We observe that $k$ tends to $+\infty$ as $\epsilon\to 0$, 
thanks to \eqref{forse0} and \eqref{eps to zero}. 

From \eqref{espero} we obtain that either 
$$ \int_{\mathbb{R}^n\setminus\{ B_{r+\bar{l}+1} \cap\{y=0\}\} }
{(W_k)_+^{2^*_s}(x,0)\,dx}
\ge\frac{\eta_0}{2}$$
or
$$ \int_{\mathbb{R}^{n+1}_+\setminus B^+_{r+\bar{l}+1}}{y^a|\nabla W_k|^2\,dX}
\ge\frac{\eta_0}{2}.
$$
In the first case, we get that 
$$\int_{\mathbb{R}^n} { (W_k)_+^{2^*_s}(x,0)\,dx}(x,0)\ge \int_{\mathbb{R}^n\setminus\{B_{r+\bar{l}+1}\cap\{y=0\}\}}
{(W_k)_+^{2^*_s}(x,0)\,dx}(x,0)\ge 
\frac{\eta_0}{2}.$$
In the second case, taking $\varepsilon$ small (and so $k$ large enough), by \eqref{Wbound} 
we obtain that
\begin{equation*}\begin{split}
\int_{\mathbb{R}^n}{(W_k)_+^{p+1}(x,0)\,dx} & \geq 
\int_{\mathbb{R}^{n+1}_+}{y^a|\nabla W_k|^2\,dX}-C\varepsilon^{1/\gamma}-o_k(1)
\\ &\geq \int_{\mathbb{R}^{n+1}_+\setminus B^+_{r+\bar{l}+1}}
{y^a|\nabla W_k|^2\,dX}-C\varepsilon^{1/\gamma}-o_k(1)
\\ & >\frac{\eta_0}{4}.
\end{split}\end{equation*}
Hence, in both the cases we have that 
\begin{equation}\label{lowBoundWp}
\int_{\mathbb{R}^n}{(W_k)_+^{p+1}(x,0)\,dx} >\frac{\eta_0}{4}.
\end{equation}

Now we define $\psi_k:=\alpha_kW_k$, with
$$ \alpha_k^{p-1}:=\frac{[W_k]_a^2}{\|(W_k)_+(\cdot,0)\|_{L^{p+1}(\mathbb{R}^n)}^{p+1}}.$$
Notice that from \eqref{fprimeW} we have that
\begin{eqnarray*}
[W_k]_a^2 &\le & \|(W_k)_+(\cdot,0)\|_{L^{p+1}(\mathbb{R}^n)}^{p+1}
+\left|\epsilon\int_{\R^n}h(x)(W_k)_+^{q+1}(x,0)\,dx\right|
+C\,\epsilon^{1/\gamma} +o_k(1)\\
&\le & \|(W_k)_+(\cdot,0)\|_{L^{p+1}(\mathbb{R}^n)}^{p+1}
+C\,\epsilon^{1/\gamma} +o_k(1)
\end{eqnarray*}
where \eqref{upBoundWq} was used in the last line. 
Hence, thanks to \eqref{lowBoundWp}, we get that 
\begin{equation}\label{star-1}
\alpha_k^{p-1}\leq 1+C\varepsilon^{1/2\gamma} +o_k(1).\end{equation}
Also, we notice that for this value of $\alpha_k$, we have the following: 
$$[\psi_k]^2_a=\alpha_k^2[W_k]_a^2=\alpha^{p+1}_k
\|(W_k)_+(\cdot,0)\|_{L^{p+1}(\mathbb{R}^n)}^{p+1}
=\|(\psi_k)_+(\cdot,0)\|_{L^{p+1}(\mathbb{R}^n)}^{p+1}.$$
Thus, by~\eqref{equivNorms} and Proposition \ref{traceIneq}, we obtain 
\begin{eqnarray*}
&& S\leq \frac{[\psi_k(\cdot,0)]^2_{\dot{H}^s(\mathbb{R}^n)}}
{\|(\psi_k)_+(\cdot,0)\|_{L^{p+1}(\mathbb{R}^n)}^{2}}
=\frac{[\psi_k]_a^2}{\|(\psi_k)_+(\cdot,0)\|_{L^{p+1}(\mathbb{R}^n)}^{2}}
\\&&\qquad\qquad =\frac{\|(\psi_k)_+(\cdot,0)\|^{p+1}_{L^{p+1}(\mathbb{R}^n)}}
{\|(\psi_k)_+(\cdot,0)\|_{L^{p+1}(\mathbb{R}^n)}^{2}}
=\|(\psi_k)_+(\cdot,0)\|_{L^{p+1}(\mathbb{R}^n)}^{\frac{4s}{n-2s}}.
\end{eqnarray*}
In the last equality we have used the fact that $p+1=2^*_s$. Consequently,
$$\|(W_k)_+(\cdot,0)\|_{L^{p+1}(\mathbb{R}^n)}^{p+1}
=\frac{\|(\psi_k)_+(\cdot,0)\|_{L^{p+1}(\mathbb{R}^n)}^{p+1}}{\alpha_k^{p+1}}
\geq S^{n/2s}\frac{1}{\alpha_k^{p+1}}.$$
This together with \eqref{star-1} give that 
\begin{eqnarray*}
S^{n/2s}&\leq&(1+C\varepsilon^{1/\gamma}+o_k(1))^{\frac{p+1}{p-1}}
\|(W_k)_+(\cdot,0)\|_{L^{p+1}(\mathbb{R}^n)}^{p+1}\\
&\leq& \|(W_k)_+(\cdot,0)\|_{L^{p+1}(\mathbb{R}^n)}^{p+1}+C\varepsilon^{1/\gamma}+o_k(1).
\end{eqnarray*}
Also, we observe that 
$$ \frac12-\frac{1}{p+1}=\frac{s}{n}. $$
Hence, 
\begin{eqnarray*}
\mathcal{F}_\epsilon(W_k)-\frac{1}{2}\langle \mathcal{F}'_\epsilon(W_k),W_k\rangle
&=&\frac{s}{n}\|(W_k)_+(\cdot,0)\|_{L^{p+1}(\mathbb{R}^n)}^{p+1}\\
&&\qquad -\epsilon\left(\frac{1}{q+1}-\frac{1}{2}\right)
\int_{\mathbb{R}^n}{h(x)(W_k)_+^{q+1}(x,0)\,dx}\\
&\geq&\frac{s}{n}S^{n/2s}-C\varepsilon^{1/\gamma}+o_k(1).
\end{eqnarray*}
Finally, using also \eqref{fprimeW}, we get 
\begin{equation}\label{LowBoundFW}
\mathcal{F}_\epsilon(W_k)\geq \frac{s}{n}S^{n/2s}-C\varepsilon^{1/\gamma}+o_k(1).
\end{equation}
\\

\noindent {\it Step 3: Lower bound for $\mathcal{F}_\epsilon(U_k)$.} 
We first observe that, thanks to \eqref{3.4ter}, 
we can write 
\begin{equation}\label{adwetperigyrejh}
U_k=(1-\chi)U_k+\chi U_k=W_k+V_k.\end{equation} 
Therefore
\begin{equation}\begin{split}\label{sumF}
\mathcal{F}_\epsilon(U_k) = &\, \mathcal{F}_\epsilon(V_k)+\mathcal{F}_\epsilon(W_k)
+\int_{\mathbb{R}^{n+1}_+}{y^a\langle\nabla V_k,\nabla W_k\rangle\,dX} \\
&\quad +\frac{1}{p+1}\int_{\mathbb{R}^n}{(V_k)_+^{p+1}(x,0)\,dx}
\\&\quad
+\frac{\epsilon}{q+1}\int_{\mathbb{R}^n}{h(x)(V_k)_+^{q+1}(x,0)\,dx}\\
&\quad +\frac{1}{p+1}\int_{\mathbb{R}^n}{(W_k)_+^{p+1}(x,0)\,dx}
\\&\quad
+\frac{\epsilon}{q+1}\int_{\mathbb{R}^n}{h(x)(W_k)_+^{q+1}(x,0)\,dx}\\
&\quad -\frac{1}{p+1}\int_{\mathbb{R}^n}{(U_k)_+^{p+1}(x,0)\,dx}
\\&\quad
-\frac{\epsilon}{q+1}\int_{\mathbb{R}^n}{h(x)(U_k)_+^{q+1}(x,0)\,dx}.
\end{split}\end{equation}
On the other hand,
\begin{eqnarray*}
&& \int_{\mathbb{R}^{n+1}_+}{y^a\langle\nabla V_k,\nabla W_k\rangle\,dX} \\
&&\qquad = \frac{1}{2}\int_{\mathbb{R}^{n+1}_+}{y^a\langle\nabla U_k-\nabla V_k,\nabla V_k\rangle\,dX} +\frac{1}{2}\int_{\mathbb{R}^{n+1}_+}{y^a\langle\nabla U_k-\nabla W_k,\nabla W_k\rangle\,dX}.
\end{eqnarray*}
Recall also that, according to~\eqref{pqoeopwoegi},
\begin{eqnarray*}
&&\langle \mathcal{F}_\varepsilon'(U_k)-\mathcal{F}_\varepsilon'(V_k),V_k
\rangle \\&=& 
\int_{\R^{n+1}_+}y^a\langle\nabla U_k-\nabla V_k, \nabla V_k\rangle\,dX \\
&&\qquad - \epsilon \int_{\R^n}h(x)(U_k)_+^q(x,0)\,V_k(x,0)\,dx - 
\int_{\R^n}(U_k(x,0))_+^p\,V_k(x,0)\,dx\\
&&\qquad +\epsilon\int_{\R^n}h(x)(V_k)_+^{q+1}(x,0)\,dx + 
\int_{\R^n}(V_k)_+^{p+1}(x,0)\,dx,
\end{eqnarray*}
and 
\begin{eqnarray*}
&&\langle \mathcal{F}_\varepsilon'(U_k)-\mathcal{F}_\varepsilon'(W_k),W_k
\rangle \\&=& 
\int_{\R^{n+1}_+}y^a\langle\nabla U_k-\nabla W_k, \nabla W_k\rangle\,dX \\
&&\qquad - \epsilon \int_{\R^n}h(x)(U_k)_+^q(x,0)\,W_k(x,0)\,dx - 
\int_{\R^n}(U_k)_+^p(x,0)\,W_k(x,0)\,dx\\
&&\qquad +\epsilon\int_{\R^n}h(x)(W_k)_+^{q+1}(x,0)\,dx + 
\int_{\R^n}(W_k)_+^{p+1}(x,0)\,dx.
\end{eqnarray*}
Hence, plugging the three formulas above into \eqref{sumF} we get
\begin{equation*}\begin{split}
\mathcal{F}_\epsilon(U_k) = &\, \mathcal{F}_\epsilon(V_k)+\mathcal{F}_\epsilon(W_k)+\frac12
\langle \mathcal{F}_\varepsilon'(U_k)-\mathcal{F}_\varepsilon'(V_k),V_k
\rangle +\frac12\langle \mathcal{F}_\varepsilon'(U_k)-\mathcal{F}_\varepsilon'(W_k),W_k
\rangle \\
&\quad +\frac{1}{p+1}\int_{\mathbb{R}^n}{(V_k)_+^{p+1}(x,0)\,dx}
+\frac{\epsilon}{q+1}\int_{\mathbb{R}^n}{h(x)(V_k)_+^{q+1}(x,0)\,dx}\\
&\quad +\frac{1}{p+1}\int_{\mathbb{R}^n}{(W_k)_+^{p+1}(x,0)\,dx}
+\frac{\epsilon}{q+1}\int_{\mathbb{R}^n}{h(x)(W_k)_+^{q+1}(x,0)\,dx}\\
&\quad -\frac{1}{p+1}\int_{\mathbb{R}^n}{(U_k)_+^{p+1}(x,0)\,dx}
-\frac{\epsilon}{q+1}\int_{\mathbb{R}^n}{h(x)(U_k)_+^{q+1}(x,0)\,dx}\\
&\quad +\frac{\epsilon}{2}
\int_{\R^n}h(x)(U_k)_+^q(x,0)\,V_k(x,0)\,dx + 
\frac12 \int_{\R^n}(U_k)_+^p(x,0)\,V_k(x,0)\,dx\\
&\quad -\frac{\epsilon}{2}\int_{\R^n}h(x)(V_k)_+^{q+1}(x,0)\,dx - 
\frac12\int_{\R^n}(V_k)_+^{p+1}(x,0)\,dx\\
&\quad + \frac{\epsilon}{2} \int_{\R^n}h(x)(U_k)_+^q(x,0)
\,W_k(x,0)\,dx 
+\frac12 \int_{\R^n}(U_k)_+^{p}(x,0)\,W_k(x,0)\,dx\\
&\quad -\frac{\epsilon}{2}\int_{\R^n}h(x)(W_k)_+^{q+1}(x,0)\,dx - \frac12
\int_{\R^n}(W_k)_+^{p+1}(x,0)\,dx.
\end{split}\end{equation*}
Notice that all the integrals with~$\epsilon$ in front are bounded. 
Therefore, using this and~\eqref{boundV} and~\eqref{boundW} we obtain that 
\begin{equation}\begin{split}\label{qwqweteyryhg}
\mathcal{F}_\epsilon(U_k) \ge &\, \mathcal{F}_\epsilon(V_k)+\mathcal{F}_\epsilon(W_k)\\
&\quad +\frac{1}{p+1}\int_{\mathbb{R}^n}{(V_k)_+^{p+1}(x,0)\,dx}
+\frac{1}{p+1}\int_{\mathbb{R}^n}{(W_k)_+^{p+1}(x,0)\,dx} \\
&\quad -\frac{1}{p+1}\int_{\mathbb{R}^n}{(U_k)_+^{p+1}(x,0)\,dx}
+\frac12 \int_{\R^n}(U_k)_+^{p}(x,0)\,V_k(x,0)\,dx\\
&\quad -
\frac12\int_{\R^n}(V_k)_+^{p+1}(x,0)\,dx
 +\frac12 \int_{\R^n}(U_k)_+^p(x,0)\,W_k(x,0)\,dx\\
&\quad - \frac12
\int_{\R^n}(W_k)_+^{p+1}(x,0)\,dx -C\epsilon^{1/\gamma},
\end{split}\end{equation}
for some positive~$C$. We observe that, thanks to~\eqref{adwetperigyrejh}, 
\begin{eqnarray*}
&&\int_{\R^n}(U_k)_+^p(x,0)\,V_k(x,0)\,dx + \int_{\R^n}(U_k)_+^p(x,0)\,W_k(x,0)\,dx\\
&=& \int_{\R^n}(U_k)_+^p(x,0)\,\big(V_k(x,0)+W_k(x,0)\big)\,dx\\
&=&\int_{\R^n}(U_k)_+^{p+1}(x,0)\,dx.
\end{eqnarray*}
Therefore, \eqref{qwqweteyryhg} becomes
\begin{equation*}\begin{split}
\mathcal{F}_\epsilon(U_k) \ge &\, \mathcal{F}_\epsilon(V_k)+\mathcal{F}_\epsilon(W_k)\\
&\quad +\frac{1}{p+1}\int_{\mathbb{R}^n}{(V_k)_+^{p+1}(x,0)\,dx}
+\frac{1}{p+1}\int_{\mathbb{R}^n}{(W_k)_+^{p+1}(x,0)\,dx} \\
&\quad -\frac{1}{p+1}\int_{\mathbb{R}^n}{(U_k)_+^{p+1}(x,0)\,dx}
- \frac12 \int_{\R^n}(V_k)_+^{p+1}(x,0)\,dx\\
&\quad +\frac12 \int_{\R^n}(U_k)_+^{p+1}(x,0)\,dx- \frac12
\int_{\R^n}(W_k)_+^{p+1}(x,0)^{p+1}\,dx -C\epsilon^{1/\gamma}\\
= &\, \mathcal{F}_\epsilon(V_k)+\mathcal{F}_\epsilon(W_k)\\
&\quad +\left(\frac12-\frac{1}{p+1}\right)\int_{\mathbb{R}^n}\left(
(U_k)_+^{p+1}(x,0)-(V_k)_+^{p+1}(x,0)-(W_k)_+^{p+1}(x,0)\right)\,dx-C\epsilon^{1/\gamma}\\
= &\, \mathcal{F}_\epsilon(V_k)+\mathcal{F}_\epsilon(W_k)\\
&\quad +\left(\frac12-\frac{1}{p+1}\right)\int_{\mathbb{R}^n}(U_k)_+^{p+1}(x,0)\left(1-\chi^{p+1}(x,0)-(1-\chi(x,0))^{p+1}\right)\,dx-C\epsilon^{1/\gamma},
\end{split}\end{equation*}
where~\eqref{3.4ter} was used in the last line. 
Since $p+1>2$ and 
$$ 1-\chi^{p+1}(x,0)-(1-\chi(x,0))^{p+1}\ge0 \quad {\mbox{ for any }}x\in\R^n,$$
this implies that 
\begin{equation*}
\mathcal{F}_\epsilon(U_k) \ge  \mathcal{F}_\epsilon(V_k)+\mathcal{F}_\epsilon(W_k)-C\epsilon^{1/\gamma}.
\end{equation*}
This, \eqref{LowerBoundV} and \eqref{LowBoundFW} imply that 
\begin{equation*}
\mathcal{F}_\epsilon(U_k)\geq 
\frac{s}{n}S^{n/2s}-c_1\varepsilon^{1/\gamma}-\bar{C}\varepsilon^{\frac{p+1}{p-q}}+o_k(1).
\end{equation*}
Hence, taking the limit as $k\to+\infty$ we obtain that 
$$ c_\epsilon=\lim_{k\to+\infty}\mathcal{F}_\epsilon(U_k)\geq 
\frac{s}{n}S^{n/2s}-c_1\varepsilon^{1/\gamma}-\bar{C}\varepsilon^{\frac{p+1}{p-q}},$$
which is a contradiction with assumption (i) of Proposition \ref{PScond}. 
This concludes the proof of Lemma \ref{tightness}.
\end{proof}

\begin{proof}[Proof of Proposition \ref{PScond}]
By Lemma \ref{tightness}, we know that $\{U_k\}_{k\in\mathbb{N}}$ is a tight sequence. 
Moreover, from Lemma \ref{lemma bound} we have that $[U_k]_a\leq M$, for $M>0$. 
Hence, also~$\{(U_k)_+\}_{k\in\N}$ is a bounded tight sequence in~$\dot{H}^s_a(\R^{n+1}_+)$. Therefore, there exists $\overline{U}\in \dot{H}^s_a(\mathbb{R}^{n+1}_+)$ such that
$$(U_k)_+\rightharpoonup \overline U \qquad\hbox{ in }\dot{H}^s_a(\mathbb{R}^{n+1}_+).$$

Also, we observe that Theorem 1.1.4 in \cite{evans} implies that 
there exist two measures on $\R^n$ and $\R^{n+1}_+$, $\nu$ and $\mu$ respectively, such that 
$(U_k)_+^{2^*_s}(x,0)$ converges to $\nu$ and 
$y^a|\nabla (U_k)_+|^2$ converges to $\mu$ 
as $k\to+\infty$, according to Definition \ref{convMeasures}.  

Hence, we can apply Proposition \ref{CCP} and we obtain that
\begin{equation}\label{point a}
{\mbox{$\displaystyle (U_k)_+^{2^*_s}(\cdot,0)$ converges to 
$\nu = \overline U^{2^*_s}(\cdot,0) +\sum_{j\in J}{\nu_j\delta_{x_j}}$ as $k\to+\infty$, with $\nu_j\geq 0$,}}
\end{equation}
\begin{equation}\label{point b}
{\mbox{$\displaystyle y^a|\nabla (U_k)_+|^2$ converges to 
$\mu\geq y^a|\nabla\overline U|^2+\sum_{j\in J}{\mu_j\delta_{(x_j,0)}}$ as $k\to+\infty$, 
with $\mu_j\geq 0$,}}
\end{equation}
and 
\begin{equation}\label{point c}
\mu_j\geq S \nu_j^{2/2^*_s} \ {\mbox{ for all }} j\in J,
\end{equation}
where $J$ is an at most countable set. 

We want to prove that $\mu_j=\nu_j=0$ for any $j\in J$. 
For this, we suppose by contradiction that there exists $j\in J$ such that $\mu_j\neq 0$. 
We denote $X_j:=(x_j,0)$. Moreover, we fix $\delta>0$ and we 
consider a cut-off function $\phi_\delta\in C^\infty(\mathbb{R}^{n+1}_+,[0,1])$, defined as
\begin{equation*}
\phi_\delta(X)=
\begin{cases}
1,\qquad\hbox{ if }X\in B_{\delta/2}^+(X_j),\\
0,\qquad\hbox{ if }X\in (B_{\delta}^+(X_j))^c,
\end{cases}
\end{equation*}
with $|\nabla \phi_\delta|\leq \frac{C}{\delta}$. 

We claim that there exists a constant $C>0$ such that
\begin{equation}\label{fi delta bounded}
[\phi_\delta\,(U_k)_+]_a\le C.
\end{equation}
Indeed, we compute 
\begin{eqnarray*}
[\phi_\delta\,(U_k)_+]_a^2 
&=&\int_{\R^{n+1}_+}y^a|\nabla(\phi_\delta (U_k)_+)|^2\,dX \\
&=& \int_{B^+_\delta(X_j)}y^a|\nabla\phi_\delta|^2(U_k)_+^2\,dX 
+\int_{B^+_\delta(X_j)}y^a\,\phi_\delta^2|\nabla (U_k)_+|^2\,dX\\&&\qquad  
+2\int_{B^+_\delta(X_j)}y^a\,\phi_\delta\,(U_k)_+
\langle\nabla\phi_\delta,\nabla (U_k)_+\rangle\,dX\\
&\le & \frac{C^2}{\delta^2}\int_{B^+_\delta(X_j)}y^a|U_k|^2\,dX 
+\int_{B^+_\delta(X_j)}y^a|\nabla U_k|^2\,dX \\&&\qquad
+\frac{2C}{\delta}\int_{B^+_\delta(X_j)}y^a|U_k|\,|\nabla U_k|\,dX\\
& \le & C\,\left(\int_{B^+_\delta(X_j)}y^a|U_k|^{2\gamma}\,dX\right)^{1/\gamma} 
+\int_{B^+_\delta(X_j)}y^a|\nabla U_k|^2\,dX \\
&& \qquad +C\,\left(\int_{B^+_\delta(X_j)}y^a|U_k|^{2\gamma}\,dX\right)^{1/2\gamma}
\left(\int_{B^+_\delta(X_j)}y^a|\nabla U_k|^2\,dX\right)^{1/2}\\
&\le & C\,M^2,
\end{eqnarray*}
up to renaming $C$, where we have used Proposition \ref{WeightedSob} 
and Lemma \ref{lemma bound} in the last step. This shows \eqref{fi delta bounded}. 

Hence, from~\eqref{pqoeopwoegi}, \eqref{fi delta bounded} and (ii) in Proposition \ref{PScond} we deduce that
\begin{equation}\begin{split}\label{prima esp}
0 =&\,\lim_{k\rightarrow\infty}{\langle \mathcal{F}'_\varepsilon(U_k),\phi_\delta (U_k)_+\rangle}\\
=&\,\lim_{k\rightarrow\infty}\left(\int_{\mathbb{R}^{n+1}_+}{y^a\langle\nabla U_k,\nabla (\phi_\delta (U_k)_+)\rangle\,dX}\right.\\
&\qquad -\left.\varepsilon\int_{\mathbb{R}^n}{h(x)\phi_\delta(x,0)
(U_k)_+^{q+1}(x,0)\,dx}
-\int_{\mathbb{R}^n}{\phi_\delta(x,0)(U_k)_+^{p+1}(x,0)\,dx}\right)\\
=&\,\lim_{k\rightarrow\infty}\left(\int_{\mathbb{R}^{n+1}_+}{y^a|\nabla (U_k)_+|^2\phi_\delta\,dX}
+\int_{\mathbb{R}^{n+1}_+}{y^a\langle\nabla (U_k)_+,\nabla \phi_\delta \rangle (U_k)_+\,dX}\right.\\
&\qquad -\left.\varepsilon\int_{\mathbb{R}^n}{h(x)\phi_\delta(x,0)(
U_k)_+^{q+1}(x,0)\,dx}
-\int_{\mathbb{R}^n}{\phi_\delta(x,0)(U_k)_+^{p+1}(x,0)\,dx}\right).
\end{split}\end{equation}
Now we recall that $p+1=2^*_s$, and so, using \eqref{point a} and \eqref{point b}, we have that 
\begin{eqnarray}
\label{conv12}
\lim_{k\to+\infty}
\int_{\mathbb{R}^n}{\phi_\delta(x,0)(U_k)_+^{p+1}(x,0)\,dx}&=&
\int_{\mathbb{R}^n}{\phi_\delta(x,0)\,d\nu}\\
\label{conv11}
{\mbox{and }}\ \lim_{k\to+\infty}
\int_{\mathbb{R}^{n+1}_+}{y^a|\nabla (U_k)_+|^2\phi_\delta\,dX}&=& 
\int_{\mathbb{R}^{n+1}_+}{\phi_\delta\,d\mu}.
\end{eqnarray}
Also, we observe that supp$(\phi_\delta)\subseteq B_{\delta}^+(X_j)$. 
Moreover $[(U_k)_+(\cdot,0)]_{\dot{H}^s(\R^n)}=[(U_k)_+]_a\leq M$, 
thanks to \eqref{equivNorms} and Lemma \ref{lemma bound}. 
Finally, the H\"older inequality and Proposition \ref{traceIneq} imply that 
$\|(U_k)_+(\cdot,0)\|_{L^2(B_{\delta}^+(x_j))}\le C$, for a suitable positive constant $C$. 
Hence, we can apply Theorem 7.1 in \cite{DPV} and we obtain that 
\begin{equation}\label{q conv}
{\mbox{$(U_k)_+(\cdot,0)$ converges to $\overline U(\cdot,0)$ strongly 
in $L^{r}(B_{\delta}^+(x_j))$ as $k\to+\infty$, for any $r\in[1,2]$.}}
\end{equation}
Therefore, 
\begin{eqnarray*}
&&\left| \int_{\mathbb{R}^n}{h(x)\phi_\delta(x,0)(U_k)_+^{q+1}(x,0)\,dx} - 
\int_{\mathbb{R}^n} {h(x)\phi_\delta(x,0)\overline U^{q+1}(x,0)\,dx} \right| 
\\
&&\qquad = \left|\int_{  B_{\delta}^+(X_j) \cap \{y=0\} }  h(x)\,\phi_\delta(x,0)
((U_k)_+^{q+1}(x,0)-\overline U^{q+1}(x,0))\,dx\right| 
\\
&&\qquad \le  \|h\|_{L^\infty(\R^n)}
\left|\int_{   B_{\delta}^+(X_j)\cap \{y=0\} } ((U_k)_+^{q+1}(x,0)-\overline U^{q+1}(x,0))\,dx\right|,
\end{eqnarray*}
which together with \eqref{q conv} implies that
\begin{equation}\label{conv2}
\lim_{k\to+\infty}
\int_{\mathbb{R}^n}{h(x)\phi_\delta(x,0)(U_k)_+^{q+1}(x,0)\,dx}=
\int_{B_\delta^+(X_j)\cap\{y=0\}}{h(x)\phi_\delta(x,0)\,\overline{U}^{q+1}(x,0)\,dx}.
\end{equation}
Finally, taking the limit as $\delta\to 0$ we get 
\begin{equation}\begin{split}\label{conv222}
&\lim_{\delta\to 0}\lim_{k\to+\infty}
\int_{\mathbb{R}^n}{h(x)\phi_\delta(x,0)(U_k)_+^{q+1}(x,0)\,dx}\\
&\qquad =\,
\lim_{\delta \to 0}\int_{B_\delta^+(X_j)\cap\{y=0\}}{h(x)\phi_\delta(x,0)\,\overline U^{q+1}(x,0)\,dx}=0. 
\end{split}
\end{equation}

Also, by the H\"older inequality and Lemma \ref{lemma bound} we obtain that 
\begin{equation}\begin{split}\label{conv3}
&\bigg|\int_{\mathbb{R}^{n+1}_+}y^a\langle\nabla (U_k)_+,
\nabla \phi_\delta \rangle (U_k)_+\,dX\bigg|
= \bigg|\int_{B_{\delta}^+(X_j) }y^a\langle\nabla (U_k)_+,\nabla \phi_\delta \rangle (U_k)_+\,dX\bigg|\\
&\qquad\leq \left(\int_{B_{\delta}^+(X_j)}{y^a|\nabla U_k|^2\,dX}\right)^{1/2}
\left(\int_{B_{\delta}^+(X_j)}{y^a(U_k)_+^2|\nabla \phi_\delta|^2\,dX}\right)^{1/2}\\
&\qquad\leq M\,\left(\int_{B_{\delta}^+(X_j)}{y^a(U_k)_+^2|\nabla \phi_\delta|^2\,dX}\right)^{1/2}.
\end{split}
\end{equation}
Notice that, since $\{(U_k)_+\}$ is a bounded sequence in~$\dot{H}^s_a(\R^{n+1}_+)$, using Lemma~\ref{lemma:compact}, we have
\begin{equation}\label{wfrewphyrtyrjh}
\lim_{k\to+\infty}\int_{B_{\delta}^+(X_j)}{y^a(U_k)_+^2|\nabla \phi_\delta|^2\,dX}= 
\int_{B_{\delta}^+(X_j)}{y^a\overline U^2|\nabla \phi_\delta|^2\,dX}.\end{equation}
Moreover, by the H\"{o}lder inequality,
\begin{equation}\label{conv4}
\int_{B_{\delta}^+(X_j)}y^a\overline U^2|\nabla \phi_\delta|^2\,dX\leq \left(\int_{B_{\delta}^+(X_j)}{y^a\overline U^{2\gamma}\,dX}\right)^{1/\gamma}\left(\int_{B_{\delta}^+(X_j)}{y^a|\nabla \phi_\delta|^{2\gamma'}\,dX}\right)^{1/\gamma'},
\end{equation}
where  
\begin{equation}\label{gamma primo}
\gamma'=\frac{n-2s+2}{2}.\end{equation} 
Thus, taking into account that $|\nabla\phi_\delta|\leq\frac{C}{\delta}$, we have
\begin{eqnarray*}
\left(\int_{B_{\delta}^+(X_j)}{y^a|\nabla\phi_\delta|^{2\gamma'}\,dX}\right)^{1/\gamma'} 
&\leq & \frac{C^2}{\delta^2}
\left(\int_{B_{\delta}^+(X_j)}{y^a\,dX}\right)^{1/\gamma'} \\
&\le & \frac{C^2}{\delta^2}\,\delta^{\frac{n+1+a}{\gamma'}}.
\end{eqnarray*}
We recall \eqref{gamma primo} and that $a=1-2s$, and we obtain that 
$$ \frac{n+1+a}{\gamma'}-2= 0,$$
and so 
$$\left(\int_{B_{\delta}^+(X_j)}{y^a|\nabla\phi_\delta|^{2\gamma'}\,dX}\right)^{1/\gamma'} 
\le C^2.$$
This and \eqref{conv4} give that 
$$ \int_{B_{\delta}^+(X_j)}y^a\overline U^2|\nabla \phi_\delta|^2\,dX\leq C^2
\left(\int_{B_{\delta}^+(X_j)}{y^a\overline U^{2\gamma}\,dX}\right)^{1/\gamma},$$
Hence 
$$ \lim_{\delta\to 0}\int_{B_{\delta}^+(X_j)}y^a\overline U^2|\nabla \phi_\delta|^2\,dX
\leq C^2\lim_{\delta\to 0} 
\left(\int_{B_{\delta}^+(X_j)}{y^a\overline U^{2\gamma}\,dX}\right)^{1/\gamma}= 0.
$$
From this and \eqref{wfrewphyrtyrjh} we obtain 
\begin{equation}\label{conv5}
\lim_{\delta\to 0}\lim_{k\to+\infty}\int_{B_{\delta}^+(X_j)}{y^a(U_k)_+^2
|\nabla \phi_\delta|^2\,dX}  = 
\lim_{\delta\to 0}\int_{B_{\delta}^+(X_j)}{y^a\overline U^2|\nabla \phi_\delta|^2\,dX}=0.\end{equation}

Using \eqref{conv12}, \eqref{conv11}, \eqref{conv222} and \eqref{conv5} in \eqref{prima esp}, 
we obtain that 
\begin{equation}\begin{split}\label{limk}
 0 =&\,\lim_{\delta\to 0}\lim_{k\rightarrow\infty}{\langle \mathcal{F}'_\varepsilon(U_k),
\phi_\delta (U_k)_+\rangle}
\\=&\, \lim_{\delta\to 0}\left(\int_{\mathbb{R}^{n+1}_+}{\phi_\delta\,d\mu}
- \int_{\mathbb{R}^n}{\phi_\delta(x,0)\,d\nu}\right)
\\=&\, \lim_{\delta\to 0}\left(\int_{B_\delta^+(X_j)}{\phi_\delta\,d\mu}
- \int_{B_\delta^+(X_j)\cap \{y=0\}}{\phi_\delta(x,0)\,d\nu}\right)\\
\ge &\, \mu_j-\nu_j,
\end{split}\end{equation}
thanks to \eqref{point a} and \eqref{point b}. 
Therefore, recalling \eqref{point c}, we obtain that 
$$ \nu_j\ge \mu_j\ge S\, \nu_j^{2/2^*_s}.$$
Hence, either $\nu_j=\mu_j=0$ or $\nu_j^{1-2/2^*}\ge S$. 
Since we are assuming that $\mu_j\neq 0$, the first possibility cannot occur, 
and so, from the second one, we have that 
\begin{equation}
\label{maggiore-1} \nu_j\ge S^{n/2s}. \end{equation}  

Now, from Lemma \ref{lemma bound} we know that $[(U_k)_+(\cdot,0)]_a\leq M$. 
Moreover we observe that $q+1<2<2^*_s$. Hence Proposition \ref{traceIneq} 
and the compact embedding in Theorem 7.1 in \cite{DPV} imply that
\begin{equation*}\begin{split}
&\|(U_k)_+(\cdot,0)-\overline U(\cdot,0)\|_{L^{2^*_s}(\mathbb{R}^n)}\leq 2M\\
\hbox{ and }\quad & (U_k)_+(\cdot,0)\rightarrow \overline U(\cdot,0)\hbox{ in }L^{q+1}_{loc}(\mathbb{R}^n)
{\mbox{ as }}k\to+\infty.
\end{split}\end{equation*}
Therefore, recalling \eqref{h0} and~\eqref{h1},
and using the H\"older inequality, we obtain
\begin{eqnarray*}
&&\left|\int_{\mathbb{R}^{n}}h(x)\left((U_k)_+(x,0)-\overline U(x,0)\right)^{q+1}\,dx\right|\\
&&\qquad \leq \int_{B_R}|h(x)|\,|(U_k)_+(x,0)-\overline U(x,0)|^{q+1}\,dx
+\int_{\mathbb{R}^{n}\setminus B_R}|h(x)|\,
|(U_k)_+(x,0)- \overline U(x,0)|^{q+1}\,dx\\
&&\qquad \leq \|h\|_{L^\infty(\mathbb{R}^n)}\|((U_k)_+-\overline U)(\cdot,0)\|_{L^{q+1}(B_R)}
+\|h\|_{L^\alpha(\mathbb{R}^n\setminus B_R)}\|((U_k)_+-\overline U)(\cdot,0)\|_{L^{2^*_s}(\mathbb{R}^n)}^{q+1}\\
&&\qquad \le C\, \|((U_k)_+-\overline U)(\cdot,0)\|_{L^{q+1}(B_R)} 
+(2M)^{q+1}\|h\|_{L^\alpha(\mathbb{R}^n\setminus B_R)}, 
\end{eqnarray*}
where $\alpha$ satisfies $\displaystyle \frac{1}{\alpha}=1-\frac{q+1}{2^*_s}$. 
Hence, letting first $k\to+\infty$ and then $R\to+\infty$, we conclude that
\begin{equation}\label{doppiabis}
\lim_{k\to+\infty}\int_{\mathbb{R}^{n}}h(x){(U_k)_+^{q+1}(x,0)\,dx}= 
\int_{\mathbb{R}^n}{h(x)\overline U^{q+1}(x,0)\,dx}.
\end{equation}

On the other hand, let $\{\varphi_m\}_{m\in\mathbb{N}}\in C_0^\infty(\mathbb{R}^n)$ 
be a sequence such that $0\leq \varphi_m\leq 1$ and 
$\displaystyle \lim_{m\rightarrow\infty}\varphi_m(x)=1$ for all $x\in\mathbb{R}^n$. 
Thus, by \eqref{point a}, we have that 
\begin{equation*}
\lim_{k\to+\infty}\int_{\mathbb{R}^n}{(U_k)_+^{p+1}(x,0)\,dx}\geq 
\lim_{k\to +\infty}\int_{\mathbb{R}^n}{(U_k)^{p+1}_+(x,0)\varphi_m\,dx}
=\int_{\mathbb{R}^n}{\varphi_m\,d\nu}.
\end{equation*}
Furthermore, by Fatou's lemma and \eqref{maggiore-1},
\begin{equation*}
\lim_{m\rightarrow\infty}{\int_{\mathbb{R}^n}{\varphi_m\,d\nu}}\geq 
\int_{\mathbb{R}^n}{\,d\nu}\geq S^{n/2s}+\int_{\mathbb{R}^n}{\overline U^{p+1}(x,0)\,dx}.
\end{equation*}
So, using the last two formulas we get 
\begin{equation}\begin{split}\label{pwqotpytiyoi}
\lim_{k\rightarrow\infty}{\int_{\mathbb{R}^n}{(U_k)_+^{p+1}(x,0)\,dx}} =&\, 
\lim_{m\to+\infty}\lim_{k\rightarrow\infty}{\int_{\mathbb{R}^n}{(U_k)_+^{p+1}(x,0)\,dx}}	\\
\geq &\, \lim_{m\to+\infty} {\int_{\mathbb{R}^n}{\varphi_m\,d\nu}}\\
\ge &\, S^{n/2s}+\int_{\mathbb{R}^n}{\overline U^{p+1}(x,0)\,dx}.
\end{split}\end{equation}

Now, since $[U_k]_a\le M$ (thanks to Lemma \ref{lemma bound}), 
from (ii) in Proposition \ref{PScond} we have that 
$$ \lim_{k\to+\infty}\langle \mathcal{F}_\varepsilon'(U_k),U_k\rangle=0,$$
and so, by hypothesis (i) we get 
\begin{equation}\label{alklasfhhsgdkjg}
\lim_{k\rightarrow\infty}{\left(\mathcal{F}_\varepsilon(U_k)
-\frac{1}{2}\langle \mathcal{F}_\varepsilon'(U_k),U_k\rangle\right)}=c_\epsilon.
\end{equation}
On the other hand, 
\begin{eqnarray*}
&&\lim_{k\rightarrow\infty}{\left(\mathcal{F}_\varepsilon(U_k)
-\frac{1}{2}\langle \mathcal{F}_\varepsilon'(U_k),U_k\rangle\right)}\\
&=& \lim_{k\to+\infty}\left(\left(\frac12-\frac{1}{p+1}\right)\int_{\R^n}(U_k)_+^{p+1}(x,0)\,dx
-\epsilon\left(\frac{1}{q+1}-\frac12\right)\int_{\R^n}h(x)(U_k)_+^{q+1}(x,0)\,dx\right).
\end{eqnarray*}
We notice that 
$$ \frac12-\frac{1}{p+1}=\frac{s}{n},$$ 
and so from~\eqref{doppiabis} and~\eqref{pwqotpytiyoi} we obtain that 
\begin{eqnarray*}
&&\lim_{k\rightarrow\infty}{\left(\mathcal{F}_\varepsilon(U_k)
-\frac{1}{2}\langle \mathcal{F}_\varepsilon'(U_k),U_k\rangle\right)}\\
&\geq &\frac{s}{n}S^{n/2s}+\frac{s}{n}\int_{\mathbb{R}^n}{\overline U^{p+1}(x,0)\,dx}
-\varepsilon\left(\frac{1}{q+1}-\frac{1}{2}\right)\int_{\mathbb{R}^n}{h(x)\overline U^{q+1}(x,0)\,dx}\\
&\geq &\frac{s}{n}S^{n/2s}+\frac{s}{n}\|\overline U(\cdot,0)\|^{p+1}_{L^{p+1}(\R^n)}
-\varepsilon\left(\frac{1}{q+1}-\frac{1}{2}\right)\|h\|_{L^m(\R^n)}
\|\overline U(\cdot,0)\|^{q+1}_{L^{p+1}(\R^n)}\\
&\geq& \frac{s}{n}S^{n/2s}-\bar{C}\epsilon^{\frac{p+1}{p-q}},
\end{eqnarray*}
where we have applied Lemma \ref{lemma basic} with $\alpha:=\|\overline U(\cdot,0)\|_{L^{p+1}(\R^n)}$ 
in the last line. This and \eqref{alklasfhhsgdkjg} imply that 
$$ c_\epsilon \ge \frac{s}{n}S^{n/2s}-\bar{C}\epsilon^{\frac{p+1}{p-q}},$$
which gives a contradiction with~\eqref{ceps}.

Therefore, necessarily $\mu_j=\nu_j=0$. 
Repeating this argument for every $j\in J$, 
we obtain that  $\mu_j=\nu_j=0$ for any $j\in J$. 
Hence, by~\eqref{point a},
\begin{equation}\label{chiama-1}
\lim_{k\to+\infty}\int_{\R^n}(U_k)_+^{2^*_s}(x,0)\varphi\,dx =
\int_{\R^n}\overline U^{2^*_s}(x,0)\varphi\,dx,
\end{equation}
for any~$\varphi\in C_0(\R^n)$. 

Then the desired result will follow. Indeed, 
we use Lemmata~\ref{PSL-1} and~\ref{PSL-2}, with~$v_k(x):=(U_k)_+(x,0)$
and~$v(x):=\overline U(x,0)$.
More precisely, 
condition~\eqref{09ngjhgfnmxxu} is guaranteed by~\eqref{chiama-1},
while condition~\eqref{TI67} follows from Lemma~\ref{tightness}.
This says that we can use
Lemma~\ref{PSL-2} and obtain that
$(U_k)_+(\cdot,0)\to \overline U(\cdot,0)$ in~$L^{2^*_s}(\R^n,[0,+\infty))$.
With this, the assumptions of Lemma~\ref{PSL-1}
are satisfied, which in turn gives that
\begin{eqnarray*}
&&
\lim_{k\to+\infty}\int_{\R^n} |(U_k)_+^q(x,0)
-\overline U^q(x,0)|^{\frac{2^*_s}{q}}\,dx=0
\\
{\mbox{and }}&&
\lim_{k\to+\infty}\int_{\R^n} |(U_k)_+^p(x,0)
-\overline U^p(x,0)|^{\frac{2n}{n+2s}}\,dx=0.
\end{eqnarray*}
Therefore, we can fix~$\delta\in(0,1)$, to be taken arbitrarily small
in the sequel, and say that
\begin{equation}\label{0vjewrsjn029}
\begin{split}
& \int_{\R^n} |(U_k)_+^q(x,0)
-(U_m)_+^q(x,0)|^{\frac{2^*_s}{q}}\,dx
\\ &\qquad+ \int_{\R^n} |(U_k)_+^p(x,0)
-(U_m)_+^p(x,0)|^{\frac{2n}{n+2s}}\,dx\le\delta
\end{split}\end{equation}
for any~$k$, $m$ large enough, say larger than some~$k_\star(\delta)$.

Let us now take~$\Phi\in \dot H^s_a(\R^{n+1}_+)$ with
\begin{equation}\label{11-0dvf67dd}
[\Phi]_a=1.\end{equation}
By assumption~(ii) in Proposition~\ref{PScond} we know that
for large~$k$ (again, say, up to renaming quantities,
that~$k\ge k_\star(\delta)$),
$$ |\langle {\mathcal{F}}'_\epsilon(U_k),\Phi\rangle|\le \delta.$$
This and~\eqref{pqoeopwoegi}
say that
\begin{eqnarray*}
&&\Big| \int_{\R^{n+1}_+} y^a \langle\nabla U_k(X),\nabla \Phi(X)\rangle\,dX\\
&&\qquad -\epsilon\int_{\R^n} h(x)\,(U_k)_+^{q}(x,0) \phi(x)\,dx
-\int_{\R^n} (U_k)_+^p(x,0)\phi(x)\,dx
\Big|\le\delta,
\end{eqnarray*}
where~$\phi(x):=\Phi(x,0)$. 
In particular, for~$k$, $m\ge k_\star(\delta)$,
\begin{eqnarray*}
&&\Big| \int_{\R^{n+1}_+} y^a \langle\nabla (U_k(X)-U_m(X)),\nabla \Phi(X)\rangle\,dX\\
&&\qquad\qquad
-\epsilon\int_{\R^n} h(x)\,\big( (U_k)_+^q(x,0)-(U_m)_+^q(x,0)\big) \phi(x)\,dx
\\&&\qquad\qquad -\int_{\R^n} \big((U_k)_+^p(x,0)-(U_m)_+^p(x,0)\big)\phi(x)\,dx
\Big|\le2\delta.
\end{eqnarray*}
So, using the H\"older inequality
with exponents~$\frac{2n}{n+2s-q(n-2s)}$
$\frac{2^*_s}{q}=\frac{2n}{q(n-2s)}$ and~$2^*_s=\frac{2n}{n-2s}$,
and with exponents~$\frac{2^*_s}{p}=\frac{2n}{n+2s}$
and~$2^*_s$, we obtain
\begin{eqnarray*}
&& \left| \int_{\R^{n+1}_+} y^a \langle\nabla (U_k(X)-U_m(X)),\nabla \Phi(X)\rangle\,dX\right|
\\&\le&
\left|\int_{\R^n} h(x)\,\big( (U_k)_+^q(x,0)-(U_m)_+^q(x,0)\big) 
\phi(x)\,dx\right| \\&&\qquad +
\left|\int_{\R^n} \big((U_k)_+^p(x,0)-(U_m)_+^p(x,0)\big)\phi(x)\,dx
\right|+2\delta\\
&\le&
\left[ \int_{\R^n} |h(x)|^{\frac{2n}{n+2s-q(n-2s)}} \,dx \right]^{
\frac{n+2s-q(n-2s)}{2n} }\\ &&\qquad\qquad \cdot
\left[ \int_{\R^n} \big| (U_k)_+^q(x,0)-(U_m)_+^q(x,0)\big|^{\frac{2^*_s}{q}}
\,dx \right]^{
\frac{q(n-2s)}{2n} }
\left[ \int_{\R^n} |\phi(x)|^{2^*_s}\,dx \right]^{\frac{1}{2^*_s}}
\\ &&\qquad+
\left[
\int_{\R^n} \big|(U_k)_+^p(x,0)-(U_m)_+^p(x,0)\big|^{
\frac{2n}{n+2s} }\,dx
\right]^{ \frac{n+2s}{2n} }
\left[ \int_{\R^n} |\phi(x)|^{2^*_s}\,dx \right]^{\frac{1}{2^*_s}}
+2\delta.
\end{eqnarray*}
Thus, by~\eqref{h0} and~\eqref{0vjewrsjn029},
$$ \left| \int_{\R^{n+1}_+} y^a \langle\nabla (U_k(X)-U_m(X)),\nabla \Phi(X)\rangle\,dX\right|\le 
C\, \delta^{ \frac{q(n-2s)}{2n} }\,\|\phi\|_{L^{2^*_s}(\R^n)}
+ C\,\delta^{ \frac{n+2s}{2n} } \,\|\phi\|_{L^{2^*_s}(\R^n)}
+2\delta,
$$
for some~$C>0$.
Now, by~\eqref{TraceIneq} and~\eqref{11-0dvf67dd},
we have that~$\|\phi\|_{L^{2^*_s}(\R^n)}\le S^{-1/2}
[\Phi]_a=S^{-1/2}$, therefore, up to renaming constants,
$$ \left| \int_{\R^{n+1}_+} y^a \langle\nabla (U_k(X)-U_m(X)),\nabla \Phi(X)\rangle\,dX\right|\le
C\delta^\gamma,$$
for some~$C$, $\gamma>0$, as long as~$k$, $m\ge
k_\star(\delta)$. Since this inequality is valid for
any~$\Phi$ satisfying~\eqref{11-0dvf67dd}, we have proved that
$$ [ U_k-U_m]_a\le C\delta^\gamma,$$
that says that~$U_k$ is a Cauchy sequence in~$\dot H^s_a(\R^{n+1}_+)$,
and then the desired result follows.
\end{proof}

\section{Proof of Theorem \ref{MINIMUM}}\label{concaveFirstSol}

With all this, we are in the position to prove Theorem \ref{MINIMUM}. 

\begin{proof}[Proof of Theorem \ref{MINIMUM}]
We recall that \eqref{h0} holds true. 
Thus, applying the H\"older inequality and Proposition \ref{traceIneq}, 
for $U\in\dot{H}^s_a(\mathbb{R}^{n+1}_+)$ we have 
\begin{eqnarray*}
\mathcal{F}_\varepsilon(U)&\geq& 
\frac{1}{2}[U]_a^2-c_1\|U_+(\cdot,0)\|_{L^{p+1}(\mathbb{R}^n)}^{p+1}
- c_2\epsilon\,\|h\|_{L^m(\R^n)}\|U_+(\cdot,0)\|_{L^{p+1}(\mathbb{R}^n)}^{q+1}\\
&\geq&\frac{1}{2}[U]_a^2-\tilde{c}_1[U]_a^{p+1}-\varepsilon \tilde{c}_2[U]_a^{q+1}.
\end{eqnarray*}
We consider the function 
$$\phi(t)=\frac{1}{2}t^2-\tilde{c}_1 t^{p+1}-\varepsilon\tilde{c}_2 t^{q+1}, \qquad t\geq 0.$$
Since $q+1<2<p+1$, we have that for every $\varepsilon>0$ 
we can find $\rho=\rho(\varepsilon)>0$ satisfying $\phi(\rho)=0$ 
and $\phi(t)<0$ for any $t\in(0,\rho)$. 
As a matter of fact, $\rho$ is the first zero of the function $\phi$. 

Furthermore, it is not difficult to see that
\begin{equation}\label{rho0}
\rho(\varepsilon)\rightarrow 0\quad \hbox{ as }\varepsilon\rightarrow 0.
\end{equation}
Thus, there exist $c_0>0$ and $\epsilon_0>0$ 
such that for all $\varepsilon<\varepsilon_0$,
\begin{equation}\label{behavF}\begin{cases}
\mathcal{F}_\varepsilon(U)\geq -c_0\qquad \hbox{ if }[U]_a<\rho(\epsilon_0),\\
\mathcal{F}_\varepsilon(U)>0\qquad \hbox{ if }[U]_a=\rho(\epsilon_0).
\end{cases}
\end{equation}
Now we take $\varphi\in C_0^\infty(\mathbb{R}^{n+1}_+)$, 
$\varphi\geq 0$, $[\varphi]_a=1$, 
and such that supp$(\varphi(\cdot,0))\subset B$,
where~$B$ is given in condition~\eqref{h2}.
Hence, for any~$t>0$,
\begin{eqnarray*} \mathcal{F}_\varepsilon(t\varphi)&=&
\frac{1}{2}t^2-\frac{\varepsilon}{q+1}t^{q+1}
\int_{\mathbb{R}^n}{h(x)\varphi^{q+1}(x,0)\,dx}-
\frac{t^{p+1}}{p+1}\int_{\mathbb{R}^n}{\varphi^{p+1}(x,0)\,dx}
\\&\le&
\frac{1}{2}t^2-\frac{\varepsilon}{q+1}t^{q+1}\inf_B h
\int_{B}{\varphi^{q+1}(x,0)\,dx}-
\frac{t^{p+1}}{p+1}\int_{B}{\varphi^{p+1}(x,0)\,dx}
.\end{eqnarray*}
This inequality and condition~\eqref{h2} give that,
for any~$\epsilon<\epsilon_0$ (possibly taking $\epsilon_0$ smaller) 
there exists $t_0<\rho(\epsilon_0)$ such that, for any $t<t_0$, we have 
\begin{equation*}
\mathcal{F}_\varepsilon{(t\varphi)}<0.
\end{equation*}
This implies that 
$$ i_\varepsilon:=\inf_{U\in\dot{H}^s_a(\mathbb{R}^{n+1}_+), [U]_a< \rho(\epsilon_0)}
{\mathcal{F}_\varepsilon(U)}<0.$$
This and \eqref{behavF} give that, for $0<\varepsilon<\varepsilon_0$,
$$-\infty<-c_0\leq i_\varepsilon<0.$$

Now we take a minimizing sequence $\{U_k\}$ and we observe that 
$$ \lim_{\epsilon\to 0}\lim_{k\to+\infty} {\mathcal{F}_\varepsilon(U_k)}
=\lim_{\epsilon\to 0}i_\epsilon\le 0<\dfrac{s}{n}S^{\frac{n}{2s}}.$$
Hence, condition~\eqref{ceps}
is satisfied
with $c_\epsilon:=i_\epsilon$,
and so
we can apply Proposition \ref{PScond} 
and we conclude that $i_\varepsilon$ is attained at some minimum $U_\varepsilon$.

Finally, since $[U_\varepsilon]_a\leq \rho(\varepsilon_0)$, 
\eqref{rho0} implies that $U_\epsilon$ converges to 0 in $\dot{H}^s_a(\R^{n+1}_+)$
as $\epsilon$ tends to 0. This concludes the proof
of Theorem \ref{MINIMUM}.
\end{proof}

\chapter{Regularity and positivity of the solution}\label{7xucjhgfgh345678}

\section{A regularity result}\label{sec:reg}

In this section we show a regularity result\footnote{We remark
that many of the techniques presented in this part have a very general
flavor and do not rely on variational principles, so they can be
possibly applied to integrodifferential operators with general kernels
and to nonlinear operators with suitable growth conditions,
and the method can interplay with tools for viscosity solutions. Nevertheless,
rather than trying to exhaust the many possible applications
of this theory in different frameworks,
which would require a detailed list of cases and different assumptions,
we presented here the basic theory just for the fractional
Laplacian, both for the sake of simplicity and because in the 
problem considered in the main results of this monograph
(such as Theorems~\ref{TH1}, \ref{MINIMUM} and~\ref{TH:MP})
we treat an equation arising from the variational
energy introduced in~\eqref{f giu}
(hence the general fully nonlinear case cannot be comprised
at that level)
and extension methods will be exploited in Chapter~\ref{EMP:CHAP}
(recall Section~\ref{sec:ext}: the case of
general integrodifferentiable kernels is not comprised
the variational principle discussed there).}
that allows us to say that a nonnegative solution to~\eqref{problem} is bounded. 

\begin{prop}\label{prop:bound}
Let $u\in \dot{H}^s(\R^n)$ be a nonnegative solution to the problem 
$$(-\Delta)^su=f(x,u)\qquad\hbox{ in }\R^n,$$
and assume that 
$|f(x,t)|\leq C(1+|t|^p)$, for some $1\leq p\leq 2^*_s-1$ and $C>0$. Then $u\in L^\infty(\R^n)$.
\end{prop}

\begin{proof}
Let $\beta\geq 1$ and $T>0$, and let us define
\begin{equation*}
\varphi(t)=\begin{cases} 0, \hbox{ if }t\leq 0,\\
t^\beta, \hbox{ if }0<t<T,\\
\beta T^{\beta-1}(t-T)+T^\beta,\hbox{ if }t\geq T.
\end{cases}
\end{equation*}
Since $\varphi$ is convex and Lipschitz, 
\begin{equation}\label{unoqqq}
\varphi(u)\in \dot{H}^s(\R^n)\end{equation} 
and
\begin{equation}\label{dueqqq}
(-\Delta)^s\varphi(u)\leq \varphi'(u)(-\Delta)^su
\end{equation}
in the weak sense. 

We recall that~\eqref{equivNorms} and Proposition~\ref{traceIneq} imply that, 
for any~$u\in\dot{H}^s(\R^n)$,
$$ \|u\|_{L^{2^*_s}(\R^n)}\le S^{-2}[u]_{\dot{H}^s(\R^n)}.$$
Moreover, by Proposition~3.6 in~\cite{DPV} we have that 
$$ [u]_{\dot{H}^s(\R^n)}= \|(-\Delta)^{s/2}u\|_{L^2(\R^n)}.$$
Hence, from~\eqref{unoqqq}, an integration by parts 
and~\eqref{dueqqq} we deduce that
\begin{eqnarray*}
&&\|\varphi(u)\|_{L^{2^*_s}(\R^n)}^2\le S^{-1}\int_{\R^n}|(-\Delta)^{s/2}
\varphi(u)|^2\,dx \\&&\qquad\quad =  
S^{-1}\int_{\R^n}\varphi(u)(-\Delta)^{s}
\varphi(u)\,dx  \le S^{-1}\int_{\R^n}\varphi(u)\,\varphi'(u)(-\Delta)^{s}u\,dx. 
\end{eqnarray*}
Therefore, from the assumption on~$u$ and~$f$ we obtain 
\begin{eqnarray*}
\|\varphi(u)\|_{L^{2^*_s}(\R^n)}^2&\le & S^{-1}\int_{\R^n}\varphi(u)\,\varphi'(u)(1+u^{2^*_s-1})\,dx \\&=& S^{-1}\left(\int_{\R^n}\varphi(u)\,\varphi'(u)\,dx
+\int_{\R^n}\varphi(u)\,\varphi'(u)\,u^{2^*_s-1}\,dx\right).
\end{eqnarray*}
Using that $\varphi(u)\varphi'(u)\leq \beta u^{2\beta-1}$ and $u\varphi'(u)\leq \beta \varphi(u)$, we have
\begin{equation}\label{boundBeta}
\left(\int_{\R^n}(\varphi(u))^{2^*_s}\right)^{2/2^*_s}\leq C\beta\left(\int_{\R^n}u^{2\beta-1}\,dx+\int_{\R^n}(\varphi(u))^2u^{2^*_s-2}\,dx\right),
\end{equation}
where $C$ is a positive constant that does not depend on $\beta$. Notice that the last integral is well defined for every $T$ in the definition of $\varphi$. Indeed,
\begin{equation*}\begin{split}
\int_{\R^n}(\varphi(u))^2u^{2^*_s-2}\,dx&\,=\int_{\{u\leq T\}}(\varphi(u))^2u^{2^*_s-2}\,dx+\int_{\{u>T\}}(\varphi(u))^2u^{2^*_s-2}\,dx\\
&\leq T^{2\beta-2}\int_{\R^n}u^{2^*_s}\,dx+C\int_{\R^n}u^{2^*_s}\,dx<+\infty,
\end{split}\end{equation*}
where we have used that $\beta>1$ and that $\varphi(u)$ is linear when $u\geq T$. We choose now $\beta$ in \eqref{boundBeta} such that $2\beta-1=2^*_s$, and we name it $\beta_1$, that is, 
\begin{equation}\label{scelta beta}
\beta_1:=\frac{2^*_s+1}{2}.
\end{equation} 

Let $R>0$ to be fixed later. Attending to the last integral in \eqref{boundBeta} and applying the H\"older's inequality with exponents $r:=2^*_s/2$ and $r':=2^*_s/(2^*_s-2)$,
\begin{equation}\begin{split}\label{boundBeta2}
\int_{\R^n}(\varphi(u))^2&u^{2^*_s-2}\,dx=\int_{\{u\leq R\}}(\varphi(u))^2u^{2^*_s-2}\,dx+\int_{\{u>R\}}(\varphi(u))^2u^{2^*_s-2}\,dx\\
&\leq \int_{\{u\leq R\}}\frac{(\varphi(u))^2}{u}R^{2^*_s-1}\,dx+\left(\int_{\R^n}(\varphi(u))^{2^*_s}\,dx\right)^{2/2^*_s}\left(\int_{\{u>R\}}u^{2^*_s}\,dx\right)^{\frac{2^*_s-2}{2^*_s}}.
\end{split}\end{equation}
By the Monotone Convergence Theorem, we can choose $R$ large enough so that
$$\left(\int_{\{u>R\}}u^{2^*_s}\,dx\right)^{\frac{2^*_s-2}{2^*_s}}\leq \frac{1}{2C\beta_1},$$
where $C$ is the constant appearing in \eqref{boundBeta}. Therefore, we can absorb the last term in \eqref{boundBeta2} by the left hand side of \eqref{boundBeta} to get
\begin{equation*}
\left(\int_{\R^n}(\varphi(u))^{2^*_s}\,dx\right)^{2/2^*_s}\leq 2C\beta_1\left(\int_{\R^n}u^{2^*_s}\,dx+R^{2^*_s-1}\int_{\R^n}\frac{(\varphi(u))^2}{u}\,dx\right),
\end{equation*}
where~\eqref{scelta beta} is also used. 
Now we use that~$\varphi(u)\le u^{\beta_1}$ and~\eqref{scelta beta} once again
in the right hand side and we
take~$T\rightarrow\infty$: we obtain $$\left(\int_{\R^n}u^{2^*_s\beta_1}\,dx\right)^{2/2^*_s}\leq 2C\beta_1\left(\int_{\R^n}u^{2^*_s}\,dx+R^{2^*_s-1}\int_{\R^n}u^{2^*_s}\,dx\right)<+\infty,$$
and therefore 
\begin{equation}\label{alkdhghuhu}
u\in L^{2^*_s\beta_1}(\R^n).
\end{equation}

Let us suppose now $\beta>\beta_1$. Thus, using that $\varphi(u)\leq u^\beta$ in the right hand side of \eqref{boundBeta} and letting $T\rightarrow\infty$ we get
\begin{equation}\label{boundBeta3}
\left(\int_{\R^n}u^{2^*_s\beta}\,dx\right)^{2/2^*_s}\leq C\beta \left(\int_{\R^n}u^{2\beta-1}\,dx+\int_{\R^n}u^{2\beta+2^*_s-2}\,dx\right).
\end{equation}
Furthermore, we can write
$$u^{2\beta-1}=u^au^b,$$
with $a:=\frac{2^*_s(2^*_s-1)}{2(\beta-1)}$ and $b:=2\beta-1-a.$ Notice that, since $\beta>\beta_1$, then $0<a,b<2^*_s$. Hence, applying Young's inequality with exponents 
$$r:=2^*_s/a \quad \hbox{ and }\quad r':=2^*_s/(2^*_s-a),$$
there holds
\begin{equation*}\begin{split}
\int_{\R^n}u^{2\beta-1}\,dx &\leq \frac{a}{2^*_s}\int_{\R^n}u^{2^*_s}\,dx+\frac{2^*_s-a}{2^*_s}\int_{\R^n}u^{\frac{2^*_sb}{2^*_s-a}}\,dx\\
&\leq \int_{\R^n}u^{2^*_s}\,dx+\int_{\R^n}u^{2\beta+2^*_s-2}\,dx\\
&\leq C\left(1+\int_{\R^n}u^{2\beta+2^*_s-2}\,dx\right),
\end{split}\end{equation*}
with~$C>0$ independent of $\beta$. Plugging this into \eqref{boundBeta3},
\begin{equation*}
\left(\int_{\R^n}u^{2^*_s\beta}\,dx\right)^{2/2^*_s}\leq C\beta \left(1+\int_{\R^n}u^{2\beta+2^*_s-2}\,dx\right),
\end{equation*}
with $C$ changing from line to line, but remaining independent of $\beta$. Therefore,
\begin{equation}\label{boundBeta4}
\left(1+\int_{\R^n}u^{2^*_s\beta}\,dx\right)^{\frac{1}{2^*_s(\beta-1)}}\leq (C\beta)^{\frac{1}{2(\beta-1)}} \left(1+\int_{\R^n}u^{2\beta+2^*_s-2}\,dx\right)^{\frac{1}{2(\beta-1)}},
\end{equation}
that is (2.6) in \cite[Proposition 2.2]{bcss}. From now on, we follow exactly their iterative argument. That is, we define $\beta_{m+1}$, $m\geq 1$, so that
$$2\beta_{m+1}+2^*_s-2=2^*_s\beta_m.$$
Thus,
$$\beta_{m+1}-1=\left(\frac{2^*_s}{2}\right)^m(\beta_1-1),$$
and replacing in \eqref{boundBeta4} it yields
\begin{equation*}
\left(1+\int_{\R^n}u^{2^*_s\beta_{m+1}}\,dx\right)^{\frac{1}{2^*_s(\beta_{m+1}-1)}}\leq (C\beta_{m+1})^{\frac{1}{2(\beta_{m+1}-1)}} \left(1+\int_{\R^n}u^{2^*_s\beta_m}\,dx\right)^{\frac{1}{2^*_s(\beta_m-1)}}.
\end{equation*}
Defining $C_{m+1}:=C\beta_{m+1}$ and
$$A_m:=\left(1+\int_{\R^n}u^{2^*_s\beta_m}\,dx\right)^{\frac{1}{2^*_s(\beta_m-1)}},$$
we conclude that there exists a constant $C_0>0$ independent of~$m$, 
such that
$$A_{m+1}\leq \prod_{k=2}^{m+1}C_k^{\frac{1}{2(\beta_k-1)}}A_1\leq C_0A_1.$$
Thus, 
$$\|u\|_{L^\infty(\R^n)}\leq C_0A_1<+\infty,$$
thanks to~\eqref{alkdhghuhu}. This finishes the proof of Proposition~\ref{prop:bound}.
\end{proof}

\begin{corollary}\label{coro:bound}
Let~$u\in\dot{H}^s(\R^n)$ be a solution of \eqref{problem-1} and let $U$ 
be its extension, according to \eqref{poisson}. 

Then~$u\in L^\infty(\R^n)$, and~$U\in L^\infty(\R^{n+1}_+)$.
\end{corollary}

\begin{proof}
First we observe that $u\ge0$, thanks to Proposition \ref{prop:pos}. 
Moreover, since $u$ is a solution to~\eqref{problem}, it solves
$$ (-\Delta)^su=f(x,u) \quad {\mbox{ in }}\R^n,$$
where~$f(x,t):=\epsilon h(x)t^q_++t^p_+$. 
It is easy to check that $f$ satisfies the hypotheses of Proposition~\ref{prop:bound}.
Hence the boundedness of~$u$ simply follows from Proposition~\ref{prop:bound}. 

Let us now show the $L^\infty$ estimate for $U$. 
According to~\eqref{poisson}, for any~$(x,z)\in\R^{n+1}_+$,
$$ U(x,z)=\int_{\R^n}u(x-y)\,P_s(y,z)\,dy.$$
Therefore, 
$$ |U(x,z)|\le \|u\|_{L^\infty(\R^n)}\int_{\R^n}P_s(y,z)\,dy=\|u\|_{L^\infty(\R^n)},$$
for any~$(x,z)\in\R^{n+1}_+$, which implies the~$L^\infty$-bound for~$U$, 
and concludes the proof of the corollary.
\end{proof}

Finally, we can prove that a solution to~\eqref{problem-1} 
is continuous, as stated in the following:

\begin{corollary}\label{coro:conti}
Let~$u\in\dot{H}^s(\R^n)$ be a solution of \eqref{problem-1} and let $U$ 
be its extension, according to \eqref{poisson}. 

Then~$u\in C^\alpha(\R^n)$, 
for any~$\alpha\in(0,\min\{2s,1\})$, and~$U\in C(\overline{\R^{n+1}_+})$.
\end{corollary}

\begin{proof}
The regularity of~$u$ follows from Corollary~\ref{coro:bound} 
and Proposition~5 in~\cite{SV:weak}, being~$u$ a solution to~\eqref{problem-1}. 
The continuity of~$U$ follows from Lemma~4.4 in~\cite{CabS}.
\end{proof}

\section{A strong maximum principle and positivity of the solutions} \label{sec:positivity}

In this section we deal with the problem of the positivity of the solutions to \eqref{problem-1}. 
We have shown in Proposition \ref{prop:pos} that a solution of \eqref{problem-1} is nonnegative. 
Here we prove that if $h\ge0$ then the solution is strictly positive. 

Following is the strong maximum principle for weak solutions
needed for our purposes:

\begin{prop}\label{prop:maxpr}
Let~$u$ be a bounded, continuous function,
with~$u\ge0$ in~$\R^n$
and~$(-\Delta)^s u\ge0$ in the weak sense in~$\Omega$.
If there exists~$x_\star\in\Omega$ such that~$u(x_\star)=0$, then~$u$
vanishes identically in~$\R^n$.
\end{prop}

\begin{proof} 
Let~$R>0$ such that~$B_R(x_\star)\subset\Omega$.
For any~$r\in(0,R)$, we consider the solution of
\begin{equation}\label{0dvmnb7whhh12} \left\{
\begin{matrix}
(-\Delta)^s v_r = 0 & {\mbox{ in }} B_r(x_\star),\\
v_r =u & {\mbox{ in }} \R^n\setminus B_r(x_\star)
\end{matrix}
\right. \end{equation}
Notice that~$v_r$ may be obtained by direct minimization
and~$v_r$ is continuous in the whole of~$\R^n$ (see e.g. Theorem~2
in \cite{SV:weak}).
Moreover, if~$w_r:= v_r-u$, we have that~$(-\Delta)^s w_r\le0$
in the weak sense in~$B_r(x_\star)$, and~$w_r$ vanishes outside~$B_r(x_\star)$.
Accordingly, by the weak maximum principle for weak solutions
(see e.g. Lemma~6 in \cite{SV:weak}), we have that~$w_r\le0$
in the whole of~$\R^n$, which gives that~$v_r\le u$.
In particular,
\begin{equation}\label{89990}
v_r(x_\star)\le u(x_\star)=0.\end{equation}
The weak maximum principle for weak solutions
and the fact that~$v_r=u\ge0$ outside~$B_r(x_\star)$
also imply that~$v_r\ge0$ in~$\R^n$. This and~\eqref{89990}
say that
\begin{equation}\label{89991}
\min_{\R^n} v_r=v_r(x_\star)=0.\end{equation}
In addition, $v_r$ is also a solution of~\eqref{0dvmnb7whhh12}
in the viscosity sense (see e.g. Theorem~1 in \cite{SV:weak}),
hence it is smooth in the interior of~$B_r(x_\star)$,
and we can compute~$(-\Delta)^s v_r(x_\star)$ and obtain from~\eqref{89991}
that
$$ 0=\int_{\R^n} \frac{ v_r(x_\star+y)+v_r(x_\star-y)-2v_r(x_\star)}{
|y|^{n+2s}}\,dy \ge 0.$$
This implies that~$v_r$ is constant in~$\R^n$, that is
$v_r(x)=v_r(x_\star)=0$ for any~$x\in\R^n$.
In particular~$0=v_r(x)=u(x)$ for any~$x\in \R^n\setminus B_r(x_\star)$.
By taking~$r$ arbitrarily small, we obtain that~$u(x)=0$
for any~$x\in \R^n\setminus\{x_\star\}$, and the desired result
plainly follows.
\end{proof}

Thanks to Proposition \ref{prop:maxpr} we now show the positivity of solutions of \eqref{problem-1}. 

\begin{corollary}
Let~$u\in\dot{H}^s(\R^n)$, $u\neq0$, be a solution of \eqref{problem-1}. Suppose also that $h\ge0$. Then, $u>0$.
\end{corollary}

\begin{proof}
First we observe that $u\in C^{\alpha}(\R^n)\cap L^\infty(\R^n)$, for some $\alpha\in(0,1)$, 
thanks to Corollaries \ref{coro:bound} and \ref{coro:conti}. Also, by Proposition \ref{prop:pos} 
we have that $u\ge0$. 
Moreover, since $u$ is a solution to \eqref{problem-1} with $h\ge0$, then $u$ satisfies 
$$ (-\Delta)^s u \ge 0 \quad {\mbox{ in }}\R^n.$$
This means that the hypotheses of Proposition \ref{prop:maxpr} are satisfied, 
and so if $u$ is equal to zero at some point then $u$ must be identically zero in $\R^n$. 
This contradicts the fact that $u\neq0$, and thus implies the desired result. 
\end{proof}

\chapter{Existence of a second
solution and proof of Theorem~\ref{TH:MP}}\label{EMP:CHAP}

In this chapter, we complete the proof of
Theorem~\ref{TH:MP}.
The computations needed for this
are delicate and somehow technical. Many calculations
are based on general
ideas of Taylor expansions and can be adapted to other types of nonlinearities
(though other estimates do take into account the precise growth conditions
of the main term of the nonlinearity and of its perturbation). Rather
than trying to list abstract conditions on the nonlinearity
which would allow similar techniques to work (possibly at the price
of more careful Taylor expansions), we remark that the
case treated here is somehow classical and motivated from geometry.

Namely, the power-like nonlinearity~$u^p$ is inherited by
Riemannian geometry
and the critical exponent~$p=\frac{n+2s}{n-2s}$
comes from conformal invariance in the classical case
(see also the recent fractional contributions in~\cite{NEW2a, NEW2b}
in which the same term is taken into account).
The perturbative term~$u^q$ is somehow more ``arbitrary''
and it could be generalized: we stick to this choice
both to compare with the classical cases in~\cite{AGP} and references
therein and in order to emphasize the role of such perturbation,
which is to produce a small, but not negligible, subcritical growth.
One of the features of this perturbation is indeed to modify the geometry
of the energy functional near the origin, without spoiling the
energy properties at infinity. Indeed, roughly speaking,
since~$p>1$, the unperturbed energy term~$u^{p+1}$ has
a higher order of homogeneity with respect to the diffusive part of
the energy, which is quadratic. Conversely, terms which behave
like~$u^q$ with~$q\in(0,1)$ near the origin 
induce a negative energy
term which may produce (and indeed produces) nonzero
local minima (the advantage of having a pure power
in the perturbation is that the inclusion in classical Lebesgue
spaces becomes explicit, but of course more general terms
can be taken into account, at a price of more involved computations).

Once a critical point~$U_\eps$ of minimal type is produced near the origin
(as given by Theorem~\ref{MINIMUM}, whose proof has been completed in
Chapter~\ref{ECXMII}), an additional critical point
is created by the behavior of the functional
at a large scale. Indeed, while the $(q+1)$-power
becomes dominant near zero and the $(p+1)$-power leads the
profile of the energy at infinity towards negative values (recall that~$q+1<2<p+1$),
in an intermediate regime the quadratic part of the energy that comes
from fractional diffusion endows the functional
with a new critical point along the path joining~$U_\eps$ to infinity.
In Figures~\ref{mafig} and~\ref{mafig2} 
we try to depict this phenomenon with a one dimensional
picture, by plotting the graphs of~$y=x^2 - |x|^{p+1}-\epsilon \,|x|^{q+1}$,
with~$p=2$, $q=1/2$ and~$\epsilon\in
\left\{0, \,\frac{1}{4},\,\frac{1}{3},\,\frac{5}{14}\right\}$.
Of course, the infinite dimensional analysis that follows
is much harder than the elementary twodimensional picture, which does not
even take into account the possible saddle properties of the critical
points ``in other directions'' and only 
serves to favor a basic intuition.

\begin{figure}
    \centering
    \includegraphics[width=12.4cm]{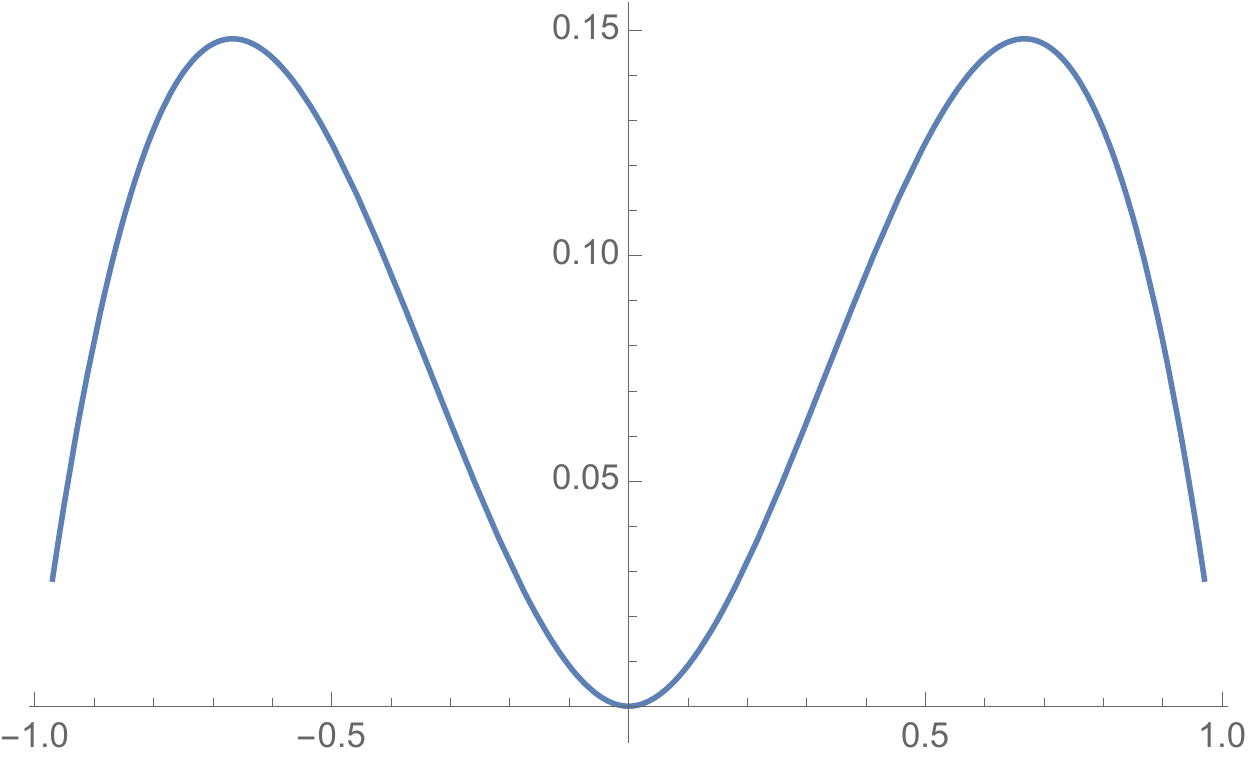}
    \caption{The function~$y=x^2 - |x|^{3}-\epsilon \,|x|^{3/2}$,
with $\epsilon
=0$.}
    \label{mafig2}
\end{figure}

\begin{figure}
    \centering
    \includegraphics[width=12.4cm]{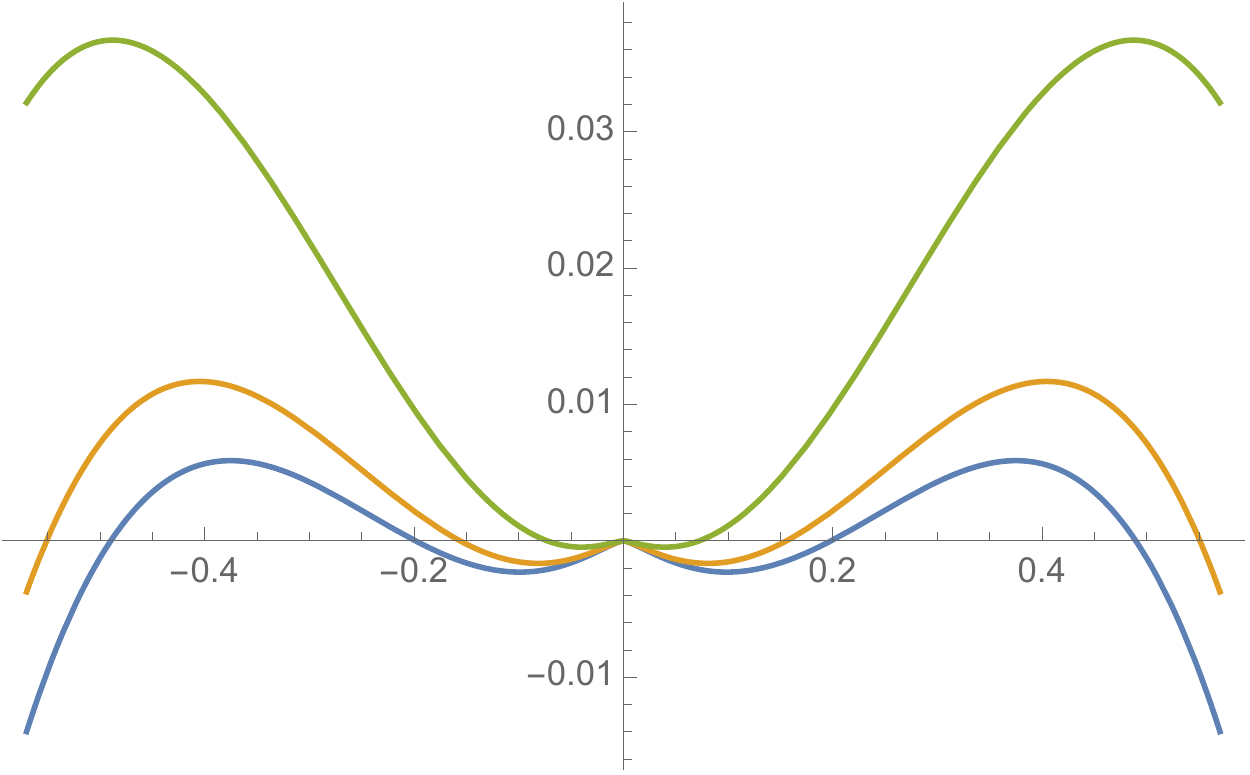} 
    \caption{The function~$y=x^2 - |x|^{3}-\epsilon \,|x|^{3/2}$,
with $\epsilon\in
\left\{\frac{1}{4},\,\frac{1}{3},\,\frac{5}{14}\right\}$. }
    \label{mafig}
\end{figure}

\section{Existence of a local minimum for~$\mathcal{I}_\epsilon$} \label{sec:exi}

In this section we show that~$U=0$ is a local minimum
for~$\mathcal{I}_\varepsilon$.

\begin{prop}\label{prop:zero}
Let $U_\varepsilon$ be a local positive minimum of~$\mathcal{F}_\varepsilon$ in $\dot{H}_a^s(\mathbb{R}^{n+1}_+)$. 
Then $U=0$ is a local minimum of $\mathcal{I}_\varepsilon$
in  $\dot{H}_a^s(\mathbb{R}^{n+1}_+)$.
\end{prop}

\begin{proof}
Let $U_\varepsilon$ be a local minimum of $\mathcal{F}_\varepsilon$ in 
$\dot{H}_a^s(\mathbb{R}^{n+1}_+)$. Then, there exists $\eta>0$ such that
\begin{equation}\label{UepsMin}
\mathcal{F}_\varepsilon(U_\varepsilon+U)\geq \mathcal{F}_\varepsilon(U_\varepsilon),
\ \hbox{ if }\;u\in\dot{H}^s_a(\R^{n+1}_+)\  {\mbox{ s.t. }} [U]_a\leq \eta.
\end{equation}
Moreover, since $U_\epsilon$ is a positive critical point of $\mathcal{F}_\varepsilon$, 
we have that, for every $V\in \dot{H}^s_a(\mathbb{R}^{n+1}_+)$, 
\begin{equation}\begin{split}\label{UepsSol}
\int_{\mathbb{R}^{n+1}_+}y^a\langle\nabla U_\varepsilon&,\nabla V\rangle\,dX-\int_{\mathbb{R}^n}{U_\varepsilon(x,0)^{p}V(x,0)\,dx}\\
&-\varepsilon\int_{\mathbb{R}^n}{h(x)U_\varepsilon(x,0)^{q}V(x,0)\,dx}=0.
\end{split}\end{equation}
Now, we take $U\in\dot{H}_a^s(\mathbb{R}^{n+1}_+)$ such 
that
\begin{equation}\label{wqsdhjykp}
[U]_a\leq \eta.
\end{equation} 
From~\eqref{def I} and~\eqref{def G}, we have that 
\begin{eqnarray*}
\mathcal{I}_\varepsilon(U) &=& \frac{1}{2}\int_{\mathbb{R}^{n+1}_+}{y^a|\nabla U|^2\,dX} \\
&&\qquad -\frac{\epsilon}{q+1}\int_{\R^n}h(x)\,\left((U_\epsilon +U_+)^{q+1}-U_\epsilon^{q+1}\right)\,dx +\epsilon\int_{\R^n}h(x)\,U_\epsilon^q U_+\,dx \\
&&\qquad -\frac{1}{p+1}\int_{\R^n}\left((U_\epsilon +U_+)^{p+1}-U_\epsilon^{p+1}\right)\,dx +\int_{\R^n}U_\epsilon^pU_+\,dx.
\end{eqnarray*}
On the other hand, recalling the definition of~$\mathcal{F}_\epsilon$ in~\eqref{f ext}, we
have that 
\begin{eqnarray*}
&&\mathcal{F}_\varepsilon(U_\epsilon+U_+) -\mathcal{F}_\epsilon(U_\epsilon)\\
&=& \frac{1}{2}\int_{\mathbb{R}^{n+1}_+}{y^a|\nabla U_+|^2\,dX} +
\int_{\mathbb{R}^{n+1}_+}{y^a\langle\nabla U_\epsilon,\nabla U_+\rangle\,dX} \\
&&\qquad -\frac{\epsilon}{q+1}\int_{\R^n}h(x)\,\left((U_\epsilon +U_+)^{q+1}-U_\epsilon^{q+1}\right)\,dx \\
&&\qquad -\frac{1}{p+1}\int_{\R^n}\left((U_\epsilon +U_+)^{p+1}-U_\epsilon^{p+1}\right)\,dx \\
&=& \frac{1}{2}\int_{\mathbb{R}^{n+1}_+}{y^a|\nabla U_+|^2\,dX} +
\int_{\mathbb{R}^{n+1}_+}{y^a\langle\nabla U_\epsilon,\nabla U_+\rangle\,dX} \\
&&\qquad -\frac{\epsilon}{q+1}\int_{\R^n}h(x)\,\left((U_\epsilon +U_+)^{q+1}-U_\epsilon^{q+1}\right)\,dx \\
&&\qquad -\frac{1}{p+1}\int_{\R^n}\left((U_\epsilon +U_+)^{p+1}-U_\epsilon
^{p+1}\right)\,dx,
\end{eqnarray*}
where in the last equality we have used the fact that both~$U_\epsilon$ 
and~$U_\epsilon+U_+$ are positive. Hence, the last two formulas give that 
\begin{eqnarray*}
\mathcal{I}_\varepsilon(U)&=&\frac12\int_{\mathbb{R}^{n+1}_+}{y^a|\nabla U_-|^2\,dX}+ \mathcal{F}_\varepsilon(U_\varepsilon +U_+)-\mathcal{F}_\varepsilon(U_\varepsilon)\\
&&\quad -\int_{\mathbb{R}^{n+1}_+}{y^a\langle\nabla U_\varepsilon,\nabla U_+\rangle\,dX}\\
&&\quad+\varepsilon\int_{\mathbb{R}^n}{h(x)U_\varepsilon^{q}U_+\,dx}
+\int_{\mathbb{R}^n}{U_\varepsilon^{p}U_+\,dx}.
\end{eqnarray*}
Using \eqref{UepsSol} with~$V:=U_+$, we obtain 
$$
\mathcal{I}_\varepsilon(U)=\frac12\int_{\mathbb{R}^{n+1}_+}{y^a|\nabla U_-|^2\,dX}+ \mathcal{F}_\varepsilon(U_\varepsilon +U_+)-\mathcal{F}_\varepsilon(U_\varepsilon).$$
Moreover, we observe that~$[U_+]_a\leq \eta$, 
thanks to~\eqref{wqsdhjykp}. Hence, from~\eqref{UepsMin} we deduce that 
$$\mathcal{I}_\varepsilon(U)\ge \frac12\int_{\mathbb{R}^{n+1}_+}{y^a|\nabla U_-|^2\,dX}\geq 0=\mathcal{I}_\varepsilon(0).$$
This shows the desired result. 
\end{proof}

\section{Some preliminary lemmata towards the proof of Theorem~\ref{TH:MP}}\label{sec:prelim}

In this section we show some preliminary lemmata, 
that we will use in the sequel to prove that
a Palais-Smale sequence is bounded.  

We start with a basic inequality.

\begin{lemma}
For every~$\delta>0$ there exists~$M_\delta>0$ such that
the following inequality holds true for every~$\alpha$, $\beta\ge0$
and~$m>0$:
\begin{equation}\label{9sdf5568gg}
(\alpha+\beta)^{m+1} - \alpha^{m+1}-(m+1)\alpha^m \beta-
\beta\big( (\alpha+\beta)^m-\alpha^m\big)\le
\delta \beta^{m+1} +M_\delta \alpha^{m+1}.\end{equation}
\end{lemma}

\begin{proof} First of all, we observe that
the left hand side of~\eqref{9sdf5568gg}
vanishes when~$\alpha=0$, therefore we can suppose that
\begin{equation}\label{9ref1wfgfds}
\alpha\ne0.
\end{equation}
For any~$\tau\ge0$, let
$$ f(\tau):= 
(1+\tau)^{m+1} - 1-(m+1)\tau-
\tau\big( (1+\tau)^m-1\big).$$
We observe that
$$ \lim_{\tau\to+\infty} \frac{f(\tau)}{\tau^{m+1}}=0$$
therefore there exists~$\tau_\delta>0$ such that~$\frac{f(\tau)}{\tau^{m+1}}
\le\delta$ for any~$\tau\ge\tau_\delta$.
Let also
$$ M_\delta:= \max_{\tau\in[0,\tau_\delta]} f(\tau).$$
Then, by looking separately at the cases~$\tau\in[0,\tau_\delta]$
and~$\tau\in [\tau_\delta,+\infty)$, we see that
$$ f(\tau)\le \delta \tau^{m+1}+M_\delta.$$
As a consequence, recalling~\eqref{9ref1wfgfds}
and taking~$\tau:=\beta/\alpha$,
\begin{equation*}\begin{split}
& (\alpha+\beta)^{m+1} - \alpha^{m+1}-(m+1)\alpha^m \beta-
\beta\big( (\alpha+\beta)^m-\alpha^m\big) \\
&\qquad= \alpha^{m+1} \Big[
(1+\tau)^{m+1} - 1-(m+1)\tau-
\tau\big( (1+\tau)^m-1\big)\Big] \\
&\qquad= \alpha^{m+1} f(\tau)\\
&\qquad\le \alpha^{m+1} \big( \delta \tau^{m+1}+M_\delta\big)
\\&\qquad= \delta \beta^{m+1}+M_\delta\alpha^{m+1}
.\qedhere\end{split}\end{equation*}
\end{proof}

We recall~\eqref{def g small} and~\eqref{def G}, and we have the following 
estimates.

\begin{corollary}
For any~$U\in H^s_a(\R^{n+1}_+)$ and any~$\delta\in(0,1)$, we have that
\begin{equation}\label{LA1-00}
\int_{\R^n} G(x,U(x,0))-\frac{1}{p+1} g(x,U(x,0))\,U(x,0)\,dx\\
\le
C (\epsilon+\delta) \| U_+(\cdot,0)\|_{L^{2^*_s}(\R^n)}^{2^*_s}
+ C_{\delta,\epsilon},\end{equation}
for suitable~$C$, $C_{\delta,\epsilon}>0$.
Moreover,
\begin{equation}\label{LA2-00}
\int_{\R^n} g(x,U(x,0))\,U(x,0)\,dx\ge 
\frac{\|U_+(\cdot,0)\|^{2^*_s}_{L^{2^*_s}(\R^n)}}{8}-C_\epsilon,
\end{equation}
for a suitable~$C_\epsilon>0$.
\end{corollary}

\begin{proof} By~\eqref{def g small} and~\eqref{def G},
we can write~$g=g_1+g_2$ and~$G=G_1+G_2$, where
\begin{eqnarray*}
&& g_1(x,t) :={\epsilon h(x)} 
\big((U_\epsilon(x,0)+t_+)^q-U_\epsilon^q(x,0)\big),\\
&& g_2(x,t) := (U_\epsilon(x,0)+t_+)^p-U_\epsilon^p(x,0),\\
&& G_1(x,t) :=\frac{\epsilon h(x)}{q+1}
\big((U_\epsilon(x,0)+t_+)^{q+1}-U_\epsilon^{q+1}(x,0)\big)-\epsilon h(x)
U_\epsilon^q(x,0) \,t_+\\
{\mbox{and }}&& G_2(x,t):=
\frac{1}{p+1} \big((U_\epsilon(x,0)+t_+)^{p+1}-U_\epsilon^{p+1}(x,0)\big)
-U_\epsilon^p(x,0) \,t_+
.\end{eqnarray*}
We observe that for any~$\tau\ge0$,
\begin{eqnarray*}
&& (1+\tau)^q-1 =q\int_0^\tau (1+\theta)^{q-1} \,d\theta
\le q\int_0^\tau \theta^{q-1} \,d\theta=
\tau^q\\
{\mbox{and }}&&(1+\tau)^{q+1}-1 =(q+1)
\int_0^\tau (1+\theta)^{q} \,d\theta\le (q+1)(1+\tau)^q\tau,
\end{eqnarray*}
since~$q\in(0,1)$.
Therefore, taking~$\tau:=t_+/U_\epsilon$,
$$ (U_\epsilon+t_+)^q-U_\epsilon^q=
U_\epsilon^q \big( (1+\tau)^q-1\big)\le U_\epsilon^q \tau^q=
t_+^q$$
and
\begin{eqnarray*}
&& (U_\epsilon+t_+)^{q+1}-U_\epsilon^{q+1}=U_\epsilon^{q+1}\big(
(1+\tau)^{q+1}-1\big)\\&&\qquad \qquad \le (q+1)U_\epsilon^{q+1}(1+\tau)^q\tau
=(q+1) (U_\epsilon+t_+)^q t_+.
\end{eqnarray*}
As a consequence,
\begin{eqnarray*}
&&|g_1|\le \epsilon |h| t_+^q \\{\mbox{and }}&& 
|G_1|\le \epsilon |h|(U_\epsilon+t_+)^q t_+ 
+\epsilon |h|U_\epsilon^q t_+ \le 
2\epsilon|h|(U_\epsilon+t_+)^q t_+.
\end{eqnarray*}
Thus we obtain
\begin{equation*}
|G_1(x,t)|+|g_1(x,t) t|\le
2\epsilon|h|(U_\epsilon+t_+)^q t_++\epsilon |h| t_+^{q+1}
\le 3\epsilon|h|(U_\epsilon+t_+)^q t_+. 
\end{equation*}
Since~$U_\epsilon$ is bounded (recall Corollary~\ref{coro:bound}), we obtain that
\begin{equation}\label{9dkjcc9gg-1}
|G_1(x,t)|+|g_1(x,t) t|\le C\epsilon|h| \,(1+t_+^{q+1}),
\end{equation}
for some~$C>0$.
By considering the cases~$t_+\in [0,1]$
and~$t_+\in [1,+\infty)$, we see that
$$ t_+^{q+1} \le t_+^{p+1} +1,$$
since~$q<p$.
This and~\eqref{9dkjcc9gg-1}
give that
\begin{equation}\label{9dkjcc9gg-2}
|G_1(x,t)|+|g_1(x,t) t|\le C\epsilon|h| \,(1+t_+^{p+1}),
\end{equation}
up to changing the constants.
Now we fix~$\delta\in(0,1)$.
Using~\eqref{9sdf5568gg} with~$\alpha:=U_\epsilon$,
$\beta:=t_+$ and~$m:=p$, we have that
\begin{eqnarray*}
&& G_2(x,t)-\frac{1}{p+1}g_2(x,t)t\\
&=&
\frac{1}{p+1} \big((U_\epsilon+t_+)^{p+1}-U_\epsilon^{p+1}\big)
-U_\epsilon^p \,t_+
-(U_\epsilon+t_+)^p+U_\epsilon^p
\\ &=& \frac{1}{p+1}\left(
(\alpha+\beta)^{m+1} - \alpha^{m+1}-(m+1)\alpha^m \beta-
\beta\big( (\alpha+\beta)^m-\alpha^m\big)\right)\\
&\le& \frac{1}{p+1}\left(
\delta \beta^{m+1} +M_\delta \alpha^{m+1} \right)\\
&\le& \delta\,t_+^{p+1} +M_\delta U_\epsilon^{p+1}.
\end{eqnarray*}
This, together with~\eqref{9dkjcc9gg-2}, implies that
\begin{eqnarray*}
&& G(x,t)-\frac{1}{p+1} g(x,t)t \\
&=& G_1(x,t)-\frac{1}{p+1}g_1(x,t)t + G_2(x,t)-\frac{1}{p+1}g_2(x,t)t\\
&\le& C\epsilon|h| \,(1+t_+^{p+1}) +
\delta\,t_+^{p+1}  +M_\delta U_\epsilon^{p+1}\\
&\le& C (\epsilon+\delta)t_+^{p+1} + C|h| + M_\delta U_\epsilon^{p+1}.
\end{eqnarray*}
As a consequence, and recalling that~$p+1=2^*_s$, we obtain
$$ \int_{\R^n} G(x,U(x,0))-\frac{1}{p+1} g(x,U(x,0))\,U(x,0)\,dx\\
\le
C (\epsilon+\delta) \| U_+(\cdot,0)\|_{L^{2^*_s}(\R^n)}^{2^*_s}
+ C_{\delta,\epsilon}$$
for some~$C$, $C_{\delta,\epsilon}>0$. This proves~\eqref{LA1-00}
and we now focus on the proof of~\eqref{LA2-00}.
To this goal, for any~$\tau\ge0$, we set~$\ell(\tau):=\frac{\tau^p}{2}
-(1+\tau)^p+1$. We observe that~$\ell(0)=0$ and
$$\lim_{\tau\to+\infty} \ell(\tau)=-\infty,$$
therefore
$$ L:=\sup_{\tau\ge0} \ell(\tau) \in [0,+\infty).$$
As a consequence,
$$ (1+\tau)^p-1 = \frac{\tau^p}{2}-\ell(\tau)\ge \frac{\tau^p}{2}-L.$$
By taking~$\tau:= \frac{U_+}{U_\epsilon}$, this implies that
\begin{eqnarray*}
&& g_2(x,U)=(U_\epsilon+U_+)^p-U_\epsilon^p
= U_\epsilon^p \big( (1+\tau)^p-1\big)\\
&&\qquad\ge U_\epsilon^p \left( \frac{\tau^p}{2}-L \right)=
\frac{U_+^p}{2}-L U_\epsilon^p.
\end{eqnarray*}
Integrating this formula and using the Young inequality, we obtain
\begin{equation}\label{9cbnmrasuifg}
\int_{\R^n} g_2(x,U(x,0))\,U(x,0)\,dx\ge
\frac{\|U_+(\cdot,0)\|^{p+1}_{L^{p+1}(\R^n)}}{4}-C_\epsilon,
\end{equation}
for some~$C_\epsilon>0$.
On the other hand, by~\eqref{9dkjcc9gg-2},
we have that
\begin{eqnarray*}
\left|\int_{\R^n} g_1(x,U(x,0))\,U(x,0)\,dx\right|
&\le & C\epsilon \int_{\R^n} |h(x)| \,(1+U_+^{p+1}(x,0))\,dx\\&\le &
C+C\epsilon\|U_+(\cdot,0)\|^{p+1}_{L^{p+1}(\R^n)}.
\end{eqnarray*}
By combining this and~\eqref{9cbnmrasuifg},
we get
\begin{eqnarray*}
&&\int_{\R^n} g(x,U(x,0))\,U(x,0)\,dx \\&=&
\int_{\R^n} g_1(x,U(x,0))\,U(x,0)\,dx+
\int_{\R^n} g_2(x,U(x,0))\,U(x,0)\,dx\\&\ge &
\frac{\|U_+(\cdot,0)\|^{p+1}_{L^{p+1}(\R^n)}}{8}-C_\epsilon,
\end{eqnarray*}
if~$\epsilon$ is small enough, up to renaming constants.
Recalling that~$p+1=2^*_s$, the formula above gives the proof of~\eqref{LA2-00}.
\end{proof}

Finally, we recall~\eqref{def I} and~\eqref{DER} and we show the following:

\begin{corollary}\label{cor:bound}
Let~$\epsilon$, $\kappa>0$. There
exists~$M>0$, possibly depending on~$\kappa$, $\epsilon$,
$n$ and~$s$, such that the following statement holds true.

For any~$U\in H^s_a(\R^{n+1}_+)$ 
such that
$$ |{\mathcal{I}}_\epsilon (U)| +
\sup_{{V\in H^s_a(\R^{n+1}_+)}\atop{ [V]_a =1 }}
\big|\langle {\mathcal{I}}_\epsilon'(U),V\rangle\big|\le\kappa$$
one has that
$$ [U]_a \le M.$$
\end{corollary}

\begin{proof} If~$[U]_a=0$ we are done, so we suppose that~$[U]_a\ne0$
and we obtain that
$$ \left|
\langle {\mathcal{I}}_\epsilon'(U),\frac{U}{[U]_a}\rangle\right|
\le\kappa.$$
This and~\eqref{DER} give that
\begin{equation}\label{Rgj67}
\left| [U]_a^2 -\int_{\R^n} g(x,U(x,0)) U(x,0)\,dx \right|
\le \kappa\,[U]_a.\end{equation}
Therefore
\begin{eqnarray*}
\kappa +\frac{\kappa\,[U]_a}{2} &\ge&
{\mathcal{I}}_\epsilon (U)-\frac{1}{2}\left(
[U]_a^2 -\int_{\R^n} g(x,U(x,0)) U(x,0)\,dx\right)\\
&=& -\int_{\R^n} G(x,U(x,0)) \,dx+\frac{1}{2}
\int_{\R^n} g(x,U(x,0)) U(x,0)\,dx \\
&=& \frac{1}{p+1}
\int_{\R^n} g(x,U(x,0)) U(x,0)\,dx
-\int_{\R^n} G(x,U(x,0)) \,dx\\
&&\qquad +
\left(\frac{1}{2}-\frac{1}{p+1}\right)
\int_{\R^n} g(x,U(x,0)) U(x,0)\,dx.
\end{eqnarray*}
Consequently, by fixing~$\delta\in(0,1)$, to be taken conveniently small in the sequel,
and using~\eqref{LA1-00}
and~\eqref{LA2-00},
\begin{eqnarray*}
\kappa +\frac{\kappa\,[U]_a}{2} &\ge &
-C (\epsilon+\delta) \| U_+(\cdot,0)\|_{L^{2^*_s}(\R^n)}^{2^*_s}
- C_{\delta,\epsilon}\\&&\qquad 
+ \left(\frac{1}{2}-\frac{1}{p+1}\right)
\left(\frac{\|U_+(\cdot,0)\|^{2^*_s}_{L^{2^*_s}(\R^n)}}{8}-C_\epsilon\right),
\end{eqnarray*}
for suitable~$C$, $C_{\delta,\epsilon}$ and~$C_\epsilon>0$.
By taking~$\delta$ and~$\epsilon$ appropriately small,
we thus obtain that
$$ \kappa +\frac{\kappa\,[U]_a}{2}\ge
\left(\frac{1}{2}-\frac{1}{p+1}\right)
\frac{\|U_+(\cdot,0)\|^{2^*_s}_{L^{2^*_s}(\R^n)}}{16}-C_{\epsilon},$$
up to renaming the latter constant (this fixes $\delta$ once and for all).
That is
\begin{equation} \label{stima su U}
\|U_+(\cdot,0)\|^{2^*_s}_{L^{2^*_s}(\R^n)} \le 
M_1 ( [U]_a +1),\end{equation}
for a suitable~$M_1$, possibly depending
on~$\kappa$, $\epsilon$,
$n$ and~$s$.

Now we recall~\eqref{Rgj67} and~\eqref{LA1-00}
(used here with~$\delta:=1$),
and we see that
\begin{eqnarray*}
&&\kappa+\frac{\kappa\,[U]_a}{p+1} \\
&\ge&
{\mathcal{I}}_\epsilon (U)-\frac{1}{p+1}\left(
[U]_a^2 -\int_{\R^n} g(x,U(x,0)) U(x,0)\,dx\right)\\
&=& \left( \frac{1}{2}-\frac{1}{p+1}\right)\,[U]_a^2
+ \frac{1}{p+1} \int_{\R^n} g(x,U(x,0)) U(x,0)\,dx\\
&&\qquad -\int_{\R^n} G(x,U(x,0)) \,dx\\
&\ge& \left( \frac{1}{2}-\frac{1}{p+1}\right)\,[U]_a^2
-
C  \| U_+(\cdot,0)\|_{L^{2^*_s}(\R^n)}^{2^*_s}
- C_{\epsilon},
\end{eqnarray*}
for suitable~$C$, $C_\epsilon>0$. As a consequence,
$$ [U]_a^2
\le M_2\big( [U]_a + \| U_+(\cdot,0)\|_{L^{2^*_s}(\R^n)}^{2^*_s}
+1\big),$$
for a suitable~$M_2$, possibly depending
on~$\kappa$, $\epsilon$,
$n$ and~$s$. Hence, from~\eqref{stima su U},
$$ [U]_a^2
\le M_3\big( [U]_a + 
+1\big),$$
for some~$M_3$, possibly depending
on~$\kappa$, $\epsilon$,
$n$ and~$s$.
This implies the desired result.
\end{proof}

\section{Some convergence results in view of Theorem~\ref{TH:MP}}\label{sec:prelim2}

In this section we collect two convergence results 
that we will need in the sequel. 

The first one shows that weak convergence to~0 in~$\dot H^s_a(\R^{n+1}_+)$ 
implies a suitable integral convergence.

\begin{lemma}\label{lemma:conv}
Let~$\alpha$, $\beta>0$ with~$\alpha+\beta\le 2^*_s$.
Let~$V_k\in\dot H^s_a(\R^{n+1}_+)$ be a sequence
such that~$V_k$ converges to~$0$ weakly in~$\dot H^s_a(\R^{n+1}_+)$.
Let~$U_o\in \dot H^s_a(\R^{n+1}_+)$, with~$U_o(\cdot,0)\in
L^\infty(\R^{n})$, and~$\psi\in L^{\rm{a}}(\R^n)\cap L^{\rm{b}}(\R^n)$, where 
$$ \rm{a}:=\begin{cases}
\frac{2^*_s}{2^*_s-\alpha-\beta} & {\mbox{ if }}\alpha +\beta <2^*_s,\\
+\infty &{\mbox{ if }}\alpha +\beta=2^*_s,
\end{cases}$$
and 
$$ \rm{b}:=\begin{cases}
\frac{2^*_s+\alpha}{2^*_s-\alpha-\beta} & {\mbox{ if }}\alpha +\beta <2^*_s,\\
+\infty &{\mbox{ if }}\alpha +\beta=2^*_s,
\end{cases}$$
Then,
up to a subsequence,
$$ \lim_{k\to+\infty} \int_{\R^n} \big|\psi(x)\big|\,\big|(V_k)_+ (x,0)\big|^\alpha
\,\big|U_o (x,0)\big|^\beta\,dx =0.$$
\end{lemma}

\begin{proof} Since weakly convergent sequences are bounded,
we have that~$[V_k]_a\le C_o$, for every~$k\in\N$
and a suitable~$C_o>0$. Accordingly, by~\eqref{equivNorms},
we obtain that~$[v_k]_{\dot{H}^s(\R^n)}\le C_o$, where~$v_k(x):=V_k(x,0)$.
As a consequence, by Theorem~7.1 of~\cite{DPV}, we know that, up to a subsequence,
$v_k$ converges to some~$v$ in~$L^\gamma_{\rm loc}(\R^n)$ for any~$\gamma\in
[1,2^*_s)$, and a.e.: we claim that
\begin{equation}\label{zer}
v=0.
\end{equation}
To prove this, let~$\eta\in C^\infty_0(\R^n)$
and~$\psi$ be the solution of
\begin{equation}\label{90567}
{\mbox{$(-\Delta)^s \psi=\eta$ in~$\R^n$.}}\end{equation}
Also, let~$\Psi$ be the extension of~$\psi$ according to~\eqref{poisson}.
In particular, ${\rm div}(y^a \nabla\Psi)=0$ in~$\R^{n+1}_+$, therefore
$$ \int_{\R^{n+1}_+} {\rm div}(y^a V_k\nabla\Psi)\,dX
=\int_{\R^{n+1}_+} y^a \langle\nabla V_k,\nabla\Psi\rangle\,dX.$$
The latter term is infinitesimal as~$k\to+\infty$, thanks to
the weak convergence of~$V_k$ in~$\dot H^s_a(\R^{n+1}_+)$.
Thus, using the Divergence Theorem in the left hand side
of the identity above, we obtain
$$ \lim_{k\to+\infty} \int_{\R^n} v_k(x) \partial_y \Phi(x)\,dx=0.$$
That is, recalling~\eqref{90567} and the convergence of~$v_k$,
$$ \int_{\R^n} v(x) \eta(x)\,dx=
\lim_{k\to+\infty} \int_{\R^n} v_k(x) \eta(x)\,dx=0.$$
Since~$\eta$ is arbitrary, we have established~\eqref{zer}.

Now we set~$u_o(x):=U_o(x,0)$ and we observe that~$u_o\in L^{2^*_s}(\R^n)$,
thanks to Proposition~\ref{traceIneq}.
Therefore, we can fix~$\epsilon>0$ and find~$R_\epsilon>0$ such that
$$ \int_{\R^n\setminus B_{R_\epsilon}} |u_o(x)|^{2^*_s}\,dx\le \epsilon.$$
In virtue of~\eqref{TraceIneq}, $\| v_k\|_{L^{2^*_s(\R^n)}}\le
S^{-1/2} C_o$.
Consequently, using the H\"older inequality with exponents~$\rm{a}$, 
$2^*_s/\alpha$ and~$2^*_s/\beta$, we deduce that
\begin{equation}\label{679ujj9}
\begin{split}
& \int_{\R^n\setminus B_{R_\epsilon}} \big|\psi(x)\big|
\,\big|(V_k)_+ (x,0)\big|^\alpha
\,\big|U_o (x,0)\big|^\beta\,dx
\\ &\qquad\le \|\psi\|_{L^{\rm{a}}(\R^n)}\,
\left[
\int_{\R^n\setminus B_{R_\epsilon}} 
\big|V_k (x,0)\big|^{2^*_s} \,dx
\right]^{\frac{\alpha}{2^*_s}}\,
\left[
\int_{\R^n\setminus B_{R_\epsilon}} 
\big|U_o (x,0)\big|^{2^*_s}
\,dx
\right]^{\frac{\beta}{2^*_s}}\\
&\qquad\le
\|\psi\|_{L^{\rm{a}}(\R^n)}\, 
\big(S^{-1/2} C_o\big)^{\frac{\alpha}{2^*_s}}\,
\epsilon^{\frac{\beta}{2^*_s}}.
\end{split}
\end{equation}
Now we fix~$\gamma:=\frac{\alpha+2^*_s}{2}$.
Notice that~$\gamma\in(1,2^*_s)$, thus, using the
convergence of~$v_k$ and~\eqref{zer}, we see that
$$ \lim_{k\to+\infty} \|v_k\|_{L^\gamma(B_{R_\epsilon})}=0.$$
In addition,
$$ \int_{\R^n} |u_o(x)|^{2^*_s+\alpha}\,dx
\le \|u_o\|_{L^\infty(\R^n)}^{\alpha}
\int_{\R^n} |u_o(x)|^{2^*_s}
\,dx\le C_*,$$
for some~$C_*>0$.
Therefore we use the H\"older inequality 
with exponents $\rm{b}$, $\frac{\gamma}{\alpha}
=\frac{\alpha+2^*_s}{2\alpha}$ and~$\frac{2^*_s+\alpha}{\beta}$, and we
obtain
\begin{eqnarray*}
&& \lim_{k\to+\infty}
\int_{B_{R_\epsilon}} \big|\psi(x)\big|
\,\big|(V_k)_+ (x,0)\big|^\alpha
\,\big|U_o (x,0)\big|^\beta\,dx \\
&\le& \|\psi\|_{L^{\rm{b}}(\R^n)}\,
\lim_{k\to+\infty}
\left[ \int_{B_{R_\epsilon}} 
\,\big|v_k(x)\big|^\gamma\right]^{\frac{\alpha}{\gamma}}\,
\left[ \int_{\R^n}
\big|u_o (x)\big|^{2^*_s+\alpha}\,dx
\right]^{\frac{\beta}{2^*_s+\alpha}}\\
&=&0.\end{eqnarray*}
{F}rom this and~\eqref{679ujj9}, we see that
$$
\lim_{k\to+\infty}
\int_{\R^n} \big|\psi(x)\big|
\,\big|(V_k)_+ (x,0)\big|^\alpha
\,\big|U_o (x,0)\big|^\beta\,dx\le
\|\psi\|_{L^{\rm{a}}(\R^n)}\,
\big(S^{-1/2} C_o\big)^{\frac{\alpha}{2^*_s}}\,
\epsilon^{\frac{\beta}{2^*_s}}.$$
The desired result then follows by taking~$\epsilon$
as small as we wish.
\end{proof}

As a corollary we have 
\begin{corollary}
Let $V_k$, $U_o$ and $\psi$ as in Lemma \ref{lemma:conv}. Then 
\begin{equation}\label{second}
\| (U_o + (V_k)_+)(\cdot,0)\|_{L^{p+1}(\R^n)}^{p+1}
= \| U_o \|_{L^{p+1}(\R^n)}^{p+1}
+ \|(V_k)_+(\cdot,0)\|_{L^{p+1}(\R^n)}^{p+1} +o_k(1), 
\end{equation}
and 
\begin{equation}\label{usare-1}
\left| \int_{\mathbb{R}^n}\psi(x)\left(  (U_o+(V_{k})_+)^{q+1}(x,0)
-U_o^{q+1}(x,0)\right)\,dx\right| \le C+o_k(1),
\end{equation}
for some $C>0$.
\end{corollary}

\begin{proof}
Formula \ref{second} plainly follows from Lemma~\ref{lemma:conv}, 
by taking~$\psi:=1$ (notice that $p+1=2^*_s$).

To prove \eqref{usare-1}, we use Lemma~\ref{lemma:conv} to see that
\begin{equation}\label{first}\begin{split}
& \int_{\mathbb{R}^n}\psi(x)
\left(  (U_o+(V_{k})_+)^{q+1}(x,0)-U_o^{q+1}(x,0)\right)\,dx\\
 =\,& \int_{\mathbb{R}^n}\psi(x)(U_o+(V_{k})_+)^{q+1}(x,0)\,dx
-\int_{\mathbb{R}^n}\psi(x)U_o^{q+1}(x,0)\,dx\\
=\,&\int_{\mathbb{R}^n}\psi(x)U_o^{q+1}(x,0)\,dx+
\int_{\mathbb{R}^n}\psi(x)(V_{k})_+^{q+1}(x,0)\,dx +o_k(1)
-\int_{\mathbb{R}^n}\psi(x)U_o^{q+1}(x,0)\,dx\\
=\,&\int_{\mathbb{R}^n}\psi(x)(V_{k})_+^{q+1}(x,0)\,dx +o_k(1).
\end{split}\end{equation}
By H\"older inequality with exponents $\frac{2^*_s}{2^*_s-q-1}$ and $\frac{2^*_s}{q+1}$ 
and by Proposition \ref{traceIneq} we have that 
\begin{eqnarray*}
&& \int_{\mathbb{R}^n}\psi(x)(V_{k})_+^{q+1}(x,0)\,dx \\
&&\qquad \le \left(\int_{\mathbb{R}^n}|\psi(x)|^{\frac{2^*_s}{2^*_s-q-1}}(x,0)\,dx\right)^{\frac{2^*_s-q-1}{2^*_s}}
\left(\int_{\mathbb{R}^n}(V_{k})_+^{2^*_s}(x,0)\,dx\right)^{\frac{q+1}{2^*_s}}\\
&&\qquad \le \|\psi\|_{  L^{ \frac{2^*_s}{2^*_s-q-1}}(\R^n)}
\left(  \int_{\mathbb{R}^n}(V_k)_+^{2^*_s}(x,0)\,dx\right)^{ \frac{q+1}{2^*_s} }\\
&&\qquad \le S^{-\frac{q+1}{2}} \|\psi\|_{  L^{\frac{2^*_s}{2^*_s-q-1} }(\R^n) }[V_k]_a.
\end{eqnarray*}
Now notice that in this case $\alpha+\beta=q+1<2^*_s$, and so $\psi\in L^{\rm{a}}(\R^n)$, 
with $\rm{a}=\frac{2^*_s}{2^*_s-q-1}$, by hypothesis. 
Moreover, since $V_k$ is a weakly convergent sequence in $\dot{H}^s_a(\R^{n+1}_+)$, 
then $V_k$ is uniformly bounded in $\dot{H}^s_a(\R^{n+1}_+)$. Hence
$$ \int_{\mathbb{R}^n}\psi(x)(V_{k})_+^{q+1}(x,0)\,dx\le C,$$ 
for a suitable $C>0$. Plugging this information into \eqref{first}, we obtain that 
$$ \left| \int_{\mathbb{R}^n}\psi(x)
\left(  (U_o+(V_{k})_+)^{q+1}(x,0)-U_o^{q+1}(x,0)\right)\,dx\right| \le C+o_k(1),$$
as desired. 
\end{proof}

Now we show that, under an assumption on the positivity 
of the limit function, weak convergence in~$\dot{H}^s_a(\R^{n+1})$ 
implies weak convergence of the positive part.

\begin{lemma}\label{POS}
Assume that~$W_m$ is a sequence of functions in~$\dot H^s_a(\R^{n+1}_+)$
that converges weakly in~$\dot H^s(\R^{n+1}_+)$ to~$
W\in \dot H^s_a(\R^{n+1}_+)$.

Suppose also that for any bounded set~$K\subset\R^{n+1}_+$, we have that
$$ \inf_K W>0.$$
Then, up to a subsequence,~$(W_m)_+\in H^s_a(\R^{n+1}_+)$
and it also
converges weakly in~$\dot H^s_a(\R^{n+1}_+)$ to~$W$.
\end{lemma}

\begin{proof} Notice that~$|\nabla (W_m)_+|=
|\nabla W_m|\chi_{\{W_m>0\}}\le |\nabla W_m|$ a.e., which shows
that~$(W_m)_+\in H^s_a(\R^{n+1}_+)$.

We also recall that, since weakly convergent sequences are bounded,
\begin{equation}\label{C0}
\sup_{m\in\N} \int_{\R^{n+1}_+} y^a |\nabla W_m|^2\,dX\le C_o,\end{equation}
for some~$C_o>0$.

Now we claim that
\begin{equation}\label{DF780}\begin{split}
&{\mbox{for any~$\Phi\in C^\infty_0(\overline{\R^{n+1}_+})$, }}\\&
\lim_{m\to+\infty} \int_{\R^{n+1}_+} y^a \langle\nabla (W_m)_+,\nabla\Phi\rangle\,dX
= \int_{\R^{n+1}_+} y^a \langle\nabla W, \nabla\Phi\rangle\,dX.
\end{split}\end{equation}
For this, we let~$K$ be the support of~$\Phi$.
Up to a subsequence, we know that~$W_m$
converges a.e. to~$W$. Therefore, by Egorov Theorem, fixed~$\epsilon>0$,
there exists~$K_\epsilon\subseteq K$ such that~$W_m$ converges to~$W$
uniformly in~$K_\epsilon$ and~$|K\setminus K_\epsilon|\le\epsilon$.
Then, for any~$X\in K_\epsilon$,
$$ W_m(X) \ge W(X) -\| W_m-W\|_{L^\infty(K_\epsilon)}
\ge \inf_K W -\| W_m-W\|_{L^\infty(K_\epsilon)}
\ge\frac{1}{2}\inf_K W >0,$$
as long as~$m$ is large enough, say $m\ge m_\star(K,\epsilon)$.

Accordingly,~$\nabla (W_m)_+=\nabla W_m$ a.e. in $K_\epsilon$
if~$m\ge m_\star(K,\epsilon)$ and therefore
\begin{equation}\label{9sdvfg898}
\lim_{m\to+\infty} \int_{K_\epsilon} y^a \langle\nabla (W_m)_+, \nabla\Phi\rangle\,dX
= \int_{K_\epsilon} y^a \langle\nabla W, \nabla\Phi\rangle\,dX.\end{equation}
Moreover, for any~$\eta>0$, the absolute continuity of
the integral gives that
$$ \int_{K\setminus K_\epsilon} y^a |\nabla W|^2\,dX
+\int_{K\setminus K_\epsilon} y^a |\nabla \Phi|^2\,dX\le \eta,$$
provided that~$\epsilon$ is small enough, say~$\epsilon\in(0,\epsilon_\star(\eta))$,
for a suitable~$\epsilon_\star(\eta)$.
As a consequence, recalling~\eqref{C0},
\begin{eqnarray*}
&& \left|\int_{K\setminus K_\epsilon} y^a \langle\nabla (W_m)_+,\nabla\Phi\rangle\,dX\right|+
\left|\int_{K\setminus K_\epsilon} y^a \langle\nabla W,\nabla\Phi\rangle\,dX\right|
\\ &\le&
\sqrt{\int_{K\setminus K_\epsilon} y^a |\nabla (W_m)_+|^2\,dX}
\cdot\sqrt{\int_{K\setminus K_\epsilon} y^a |\nabla \Phi|^2\,dX}
\\&&\qquad +
\sqrt{\int_{K\setminus K_\epsilon} y^a |\nabla W|^2\,dX}
\cdot\sqrt{\int_{K\setminus K_\epsilon} y^a |\nabla \Phi|^2\,dX}
\\ &\le& \sqrt{C_o \eta} +\eta.
\end{eqnarray*}
Using this and~\eqref{9sdvfg898}, we obtain that
\begin{eqnarray*}
&& \left| 
\lim_{m\to+\infty} \int_{\R^{n+1}_+} y^a \langle\nabla (W_m)_+, \nabla\Phi\rangle\,dX
-\int_{\R^{n+1}_+} y^a\langle \nabla W, \nabla\Phi\rangle\,dX
\right| \\
&=& \left|
\lim_{m\to+\infty} \int_{K} y^a \langle\nabla (W_m)_+,\nabla\Phi\rangle\,dX
-\int_{K} y^a \langle\nabla W, \nabla\Phi\rangle\,dX
\right|\\
&\le& \left| \lim_{m\to+\infty} \int_{K\setminus K_\epsilon} y^a \langle\nabla (W_m)_+, \nabla\Phi\rangle\,dX
-\int_{K\setminus K_\epsilon} y^a \langle\nabla W, \nabla\Phi\rangle\,dX
\right|
\\ &\le& \sqrt{C_o \eta} +\eta.
\end{eqnarray*}
By taking~$\eta$ as small as we like, we complete the proof of~\eqref{DF780}.

Now we finish the proof of Lemma~\ref{POS}
by a density argument. Let~$\Phi\in \dot H^s_a(\R^{n+1}_+)$ and~$\epsilon>0$.
We take~$\Phi_\epsilon\in C^\infty_0(\R^{n+1}_+)$ such that
$$ \int_{\R^{n+1}_+} y^a |\nabla (\Phi-\Phi_\epsilon)|^2\,dX\le\epsilon.$$
The existence of such~$\Phi_\epsilon$ is guaranteed by~\eqref{D123}.
Then, recalling~\eqref{C0} and~\eqref{DF780} for the function~$\Phi_\epsilon$,
we obtain that
\begin{eqnarray*}
&& \left|
\lim_{m\to+\infty} \int_{\R^{n+1}_+} y^a \langle\nabla (W_m)_+, \nabla\Phi\rangle\,dX
-\int_{\R^{n+1}_+} y^a \langle\nabla W, \nabla\Phi\rangle\,dX
\right| \\
&\le&
\left|
\lim_{m\to+\infty} \int_{\R^{n+1}_+} y^a \langle\nabla (W_m)_+, \nabla\Phi_\epsilon\rangle\,dX
-\int_{\R^{n+1}_+} y^a \langle\nabla W, \nabla\Phi_\epsilon\rangle\,dX
\right|
\\&&\qquad + \left(\sqrt{ C_o }+ \sqrt{\int_{\R^{n+1}_+} y^a|\nabla W|^2\,dX}\right)\,
\sqrt{
\int_{\R^{n+1}_+} y^a |\nabla (\Phi-\Phi_\epsilon)|^2\,dX }\\
&\le& 0+\left(\sqrt{ C_o }+ \sqrt{\int_{\R^{n+1}_+} y^a|\nabla W|^2\,dX}\right)\,\sqrt{\epsilon}.
\end{eqnarray*}
Accordingly, by taking~$\epsilon$ as small as we please, we obtain that
$$\lim_{m\to+\infty} \int_{\R^{n+1}_+} y^a \langle\nabla (W_m)_+, \nabla\Phi\rangle\,dX
=\int_{\R^{n+1}_+} y^a \langle\nabla W, \nabla\Phi\rangle\,dX,$$
for any~$\Phi\in \dot H^s_a(\R^{n+1}_+)$, thus completing the
proof of Lemma~\ref{POS}.
\end{proof}

\section{Palais-Smale condition for~$\mathcal{I}_\epsilon$}\label{sec:PS MP}

Once we have found a minimum of~$\mathcal{I}_\varepsilon$, 
we apply a contradiction procedure to prove the existence of a second critical point. 

Roughly speaking, the idea is the following: let us suppose that $U=0$ is the only critical point; 
thus, we prove some compactness and geometric properties of the functional 
(based on the fact that the critical point is unique), and these facts allow us 
to apply the Mountain Pass Theorem, that provides a second critical point. Hence, 
we reach a contradiction, so $U=0$ cannot be the only critical point of $\mathcal{I}_\varepsilon$. 

As we did in Proposition~\ref{PScond}
for the minimal solution,
also to find the second solution we need to prove that a Palais-Smale condition 
holds true below a certain threshold, as stated in the following result:

\begin{prop}\label{PScond2}
There exists~$C>0$, depending on~$h$, $q$, $n$
and~$s$, such that the following statement
holds true.
Let $\{U_k\}_{k\in\mathbb{N}}\subset \dot{H}^s_a(\mathbb{R}^{n+1}_+)$ be a sequence satisfying
\begin{enumerate}
\item[(i)] $\displaystyle\lim_{k\to+\infty}\mathcal{I}_\epsilon(U_k)= c_\epsilon$, with 
\begin{equation}\label{ceps2}
c_\epsilon+C \epsilon^{\frac{1}{2\gamma}}<
\dfrac{s}{n}S^{\frac{n}{2s}},\end{equation}
where $\gamma=1+\frac{2}{n-2s}$ and ~$S$ is the Sobolev constant appearing in Proposition~\ref{traceIneq},
\item[(ii)] $\displaystyle\lim_{k\to+\infty}\mathcal{I}'_\epsilon(U_k)= 0.$ 
\end{enumerate}
Assume also that~$U=0$ is the only critical point of~$\mathcal{I}_\varepsilon$.

Then~$\{U_k\}_{k\in\mathbb{N}}$ contains a subsequence strongly convergent in $\dot{H}^s_a(\mathbb{R}^{n+1}_+)$.
\end{prop}

\begin{rem} \label{rem:3.3}
The limit in (ii) is intended in the following way 
\begin{eqnarray*}
&& \lim_{k\to+\infty}\|\mathcal{I}'_\epsilon(U_k)\|_
{\mathcal{L}(\dot{H}^s_a(\mathbb{R}^{n+1}_+),\dot{H}^s_a(\mathbb{R}^{n+1}_+))} 
\\ &&\qquad = \lim_{k\to+\infty}\sup_{V\in\dot{H}^s_a(\mathbb{R}^{n+1}_+), 
[V]_{a}=1}  
\left|\langle \mathcal{I}'_\epsilon(U_k), V\rangle\right|
=0,
\end{eqnarray*}
where $\mathcal{L}(\dot{H}^s_a(\mathbb{R}^{n+1}_+),\dot{H}^s_a(\mathbb{R}^{n+1}_+))$ 
consists of all the linear functionals from $\dot{H}^s_a(\mathbb{R}^{n+1}_+)$ 
in $\dot{H}^s_a(\mathbb{R}^{n+1}_+)$.
\end{rem}

We observe that a sequence that satisfies the assumptions of Proposition~\ref{PScond2} is weakly convergent. 
The precise statement goes as follows:

\begin{lemma}\label{lemma:weak}
Let $\{U_k\}_{k\in\mathbb{N}}\subset \dot{H}^s_a(\mathbb{R}^{n+1}_+)$ be a sequence satisfying
the hypotheses of Proposition \ref{PScond2}. 
Assume also that~$U=0$ is the only critical point of~$\mathcal{I}_\varepsilon$.

Then, up to a subsequence, $U_k$ weakly converges to~0 in~$\dot{H}^s_a(\R^{n+1}_+)$ as~$k\to+\infty$.
\end{lemma}

\begin{proof}
Notice that assumptions (i) and (ii) imply that there exists~$\kappa>0$ 
such that 
$$ |\mathcal{I}_\epsilon(U_k)|+\sup_{{V\in H^s_a(\R^{n+1}_+)}\atop{ [V]_a =1 }}
\big|\langle {\mathcal{I}}_\epsilon'(U_k),V\rangle\big|\le\kappa.$$
Hence, by Corollary~\ref{cor:bound} we have that there exists a positive constant~$M$ (independent of~$k$) such that 
\begin{equation}\label{uniformBoundUk}
[U_k]_a\leq M.
\end{equation}
Therefore, there exists a subsequence 
(that we still denote by $U_k$) 
converging weakly to some function $U_\infty\in \dot{H}^s_a(\mathbb{R}^{n+1}_+)$, that is,
\begin{equation}\label{kqwrwet}
U_k\rightharpoonup U_\infty\quad\hbox{ in }\dot{H}^s_a(\mathbb{R}^{n+1}_+),
\end{equation}
as~$k\to+\infty$. We now claim that
\begin{equation}\label{U fa zero}
U_\infty=0.
\end{equation} 
For this, we first observe that, thanks to~\eqref{uniformBoundUk} 
and Theorem~7.1 in~\cite{DPV}, we have that 
\begin{equation}\label{asfewgjgfjr}
U_k(\cdot,0)\rightarrow U_\infty(\cdot,0)\qquad\hbox{ in }L^\alpha_{\rm{loc}}(\R^n),\qquad 1\leq \alpha<2^*_s,
\end{equation}
and so
\begin{equation}\label{punto}
U_k(\cdot,0)\rightarrow U_\infty(\cdot,0) \quad {\mbox{ a. e. }}\R^n.
\end{equation}
Let now $\Psi\in C_0^\infty(\mathbb{R}^{n+1})$, $\psi:=\Psi(\cdot,0)$ 
and~$K:=supp(\psi)$. According to~\eqref{DER}, 
\begin{equation}\begin{split}\label{adweterjutrj}
\langle\mathcal{I}_\epsilon'(U_k),\Psi\rangle =\,& \int_{\R^{n+1}_+}y^a\langle\nabla U_k,\nabla\Psi\rangle\,dX -
\int_{\R^n}g(x,U_k(x,0))\psi(x)\,dx \\
=\,& \int_{\R^{n+1}_+}y^a\langle\nabla U_k,\nabla\Psi\rangle\,dX -
\int_{K}g(x,U_k(x,0))\psi(x)\,dx.
\end{split}\end{equation}
Thanks to~\eqref{kqwrwet}, we have that 
\begin{equation}\label{kqwrwet-1}
\int_{\R^{n+1}_+}y^a\langle\nabla U_k,\nabla\Psi\rangle\,dX \to 
\int_{\R^{n+1}_+}y^a\langle\nabla U_\infty,\nabla\Psi\rangle\,dX 
\end{equation}
as~$k\to+\infty$. Moreover, \eqref{punto} implies that
$$ g(x,U_k(\cdot,0))\to g(x,U_\infty(\cdot,0)) \quad {\mbox{ a. e. }}\R^n,$$
as~$k\to+\infty$. Also, notice that 
$$ (1+r)^p-1\le C(1+r^p), $$
for any~$r\ge0$ and for some positive constant~$C>0$. 
Hence, recalling that~$U_\epsilon>0$, 
thanks to Proposition~\ref{prop:pos},
we can use this with~$r:=t/U_\epsilon$ 
and we obtain that
$$ (U_\epsilon+t)^p-U_\epsilon^p=U_\epsilon^p\left[\left(1+\frac{t}{U_\epsilon}\right)^p-1\right] \le C\,U_\epsilon^p\left(1+\frac{t^p}{U_\epsilon^p}\right) =C\left(U_\epsilon^p+t^p\right).$$
This, formulas~\eqref{fdcvbpfp11}
and~\eqref{def g small} give that 
\begin{equation}\label{alkdfsjgerphitrji}
|g(x,t)|\le C\left(U_\epsilon^p+t^p\right)+\epsilon |h|.\end{equation}
Hence, for any~$k\in\N$, 
\begin{equation}\label{wqrfewhghbenb}
|g(x,U_k(x,0))|\,|\psi| \le C|\psi| \left(U_\epsilon^p(x,0)+U_k^p(x,0)\right)+\epsilon |h||\psi|.
\end{equation}
This means that the sequence~$g(\cdot,U_k(\cdot,0))$ is bounded 
by a sequence that is strongly convergent in~$L^1_{\rm{loc}}(\R^n)$. 
Moreover, by Theorem 4.9 in \cite{brezis} 
we have that there exists a function $f\in L^1_{\rm{loc}}(\R^n)$ such that, 
up to a subsequence, 
\begin{equation}\label{wqrfewhghbenb-1}
C|\psi| \left(U_\epsilon^p(\cdot,0)+U_k^p(\cdot,0)\right)+\epsilon|\psi||h|\le |f|.
\end{equation}
Formulas \eqref{wqrfewhghbenb} and \eqref{wqrfewhghbenb-1} together with \eqref{punto}
imply that we can use the Dominated Convergence Theorem (see e.g. Theorem 4.2 in~\cite{brezis}) 
and we obtain that 
$$ \int_{K}g(x,U_k(x,0))\psi(x)\,dx\to \int_{K}g(x,U_\infty(x,0))\psi(x)\,dx,$$
as~$k\to+\infty$. 
Using this and~\eqref{kqwrwet-1} into~\eqref{adweterjutrj} we have
$$ 
\langle\mathcal{I}_\epsilon'(U_k),\Psi\rangle\to 
\int_{\R^{n+1}_+}y^a\langle\nabla U_\infty,\nabla\Psi\rangle\,dX -
\int_{\R^n}g(x,U_\infty(x,0))\psi(x)\,dx =
\langle\mathcal{I}_\epsilon'(U_\infty),\Psi\rangle,$$
as~$k\to+\infty$. On the other hand, assumption~(ii) implies that 
$$ \langle\mathcal{I}_\epsilon'(U_k),\Psi\rangle\to 0$$
as~$k\to+\infty$. The last two formulas imply that
\begin{equation}\label{qwrewputhjtynmf}
\langle\mathcal{I}_\epsilon'(U_\infty),\Psi\rangle =0, \quad {\mbox{ for any~$\Psi\in C^\infty_0(\R^{n+1}_+)$}}.\end{equation} 
Let now~$\Psi\in\dot{H}^s_a(\R^{n+1})$, with~$\psi:=\Psi(\cdot,0)$. 
Then by~\eqref{D123} there exists 
a sequence of functions~$\Psi_m\in C^\infty_0(\R^{n+1}_+)$, with~$\psi_m:=\Psi_m(\cdot,0)$, such 
that
\begin{equation}\label{wefreypqqpqwpqwpqw}
{\mbox{$\Psi_m\to\Psi$ in~$\dot{H}^s_a(\R^{n+1}_+)$ as~$m\to+\infty$.}}
\end{equation} 
By Proposition~\ref{traceIneq} this implies also that
\begin{equation}\label{wefreypqqpqwpqwpqw-1}
{\mbox{$\psi_m\to\psi$ in~$L^{2^*_s}(\R^{n})$ as~$m\to+\infty$.}}
\end{equation} 
Therefore, from~\eqref{qwrewputhjtynmf} we deduce that for any~$m\in\N$
\begin{equation}\label{distri}
0=\langle\mathcal{I}_\epsilon'(U_\infty),\Psi_m\rangle =
\int_{\R^{n+1}_+}y^a\langle\nabla U_\infty,\nabla\Psi_m\rangle\,dX -
\int_{\R^n}g(x,U_\infty(x,0))\psi_m(x)\,dx.
\end{equation}
Now, \eqref{wefreypqqpqwpqwpqw} 
implies that 
\begin{equation*}\begin{split}
&\left| \int_{\R^{n+1}_+}y^a\langle\nabla U_\infty,\nabla\Psi_m\rangle\,dX -\int_{\R^{n+1}_+}y^a\langle\nabla U_\infty,\nabla\Psi\rangle\,dX\right|
\\\le\, & \int_{\R^{n+1}_+}y^a|\nabla U_\infty|\,|\nabla(\Psi_m-\Psi)|\,dX\\
\le\, & \sqrt{\int_{\R^{n+1}_+}y^a|\nabla U_\infty|^2\,dX} \cdot 
\sqrt{\int_{\R^{n+1}_+}y^a|\nabla(\Psi_m-\Psi)|^2\,dX}\to 0,
\end{split}\end{equation*}
as~$m\to+\infty$. Moreover, by H\"older inequality with exponents~$2^*_s$ 
and~$\frac{2^*_s}{2^*_s-1}$ we get 
\begin{equation*}\begin{split}
&\left| \int_{\R^n}g(x,U_\infty(x,0))\psi_m(x)\,dx- \int_{\R^n}g(x,U_\infty(x,0))\psi(x)\,dx\right| \\
\le \,& \int_{\R^n}|g(x,U_\infty(x,0))|\,|\psi_m(x)-\psi(x)|\,dx\\
\le \,& C\int_{\R^n}(U_\epsilon^p+U_\infty^p)\,|\psi_m(x)-\psi(x)|\,dx+\epsilon\int_{\R^n}|h|\,|\psi_m(x)-\psi(x)|\,dx\\
\le \,& C\left(\int_{\R^n}(U_\epsilon^p+U_\infty^p)^{\frac{2^*_s}{2^*_s-1}}\,dx\right)^{\frac{2^*_s-1}{2^*_s}}\, 
\left(\int_{\R^n}|\psi_m(x)-\psi(x)|^{2^*_s}\,dx\right)^{\frac{1}{2^*_s}}\\
& + \epsilon\left(\int_{\R^n}|h|^{\frac{2^*_s}{2^*_s-1}}\,dx\right)^{\frac{2^*_s-1}{2^*_s}}\, 
\left(\int_{\R^n}|\psi_m(x)-\psi(x)|^{2^*_s}\,dx\right)^{\frac{1}{2^*_s}},
\end{split}\end{equation*}
where we have used \eqref{alkdfsjgerphitrji}. Furthermore, noticing that $\frac{2^*_sp}{2^*_s-1}=2^*_s$, we have that 
$$ \left(U_\epsilon^p+U_\infty^p\right)^{\frac{2^*_s}{2^*_s-1}}\le 
\left(U_\epsilon+U_\infty\right)^{\frac{2^*_sp}{2^*_s-1}}\leq\left(U_\epsilon+U_\infty\right)^{2^*_s},$$ 
up to renaming constants. 
Thus, since~$U_\epsilon$, $U_\infty\in\dot{H}^s_a(\R^{n+1}_+)$, and $h\in L^r(\R^n)$ for every $1\leq r\leq +\infty$, by
Proposition~\ref{traceIneq} we deduce that 
\begin{equation*}\begin{split}
&\left| \int_{\R^n}g(x,U_\infty(x,0))\psi_m(x)\,dx- \int_{\R^n}g(x,U_\infty(x,0))\psi(x)\,dx\right| \\
\le \,& C\, 
\left(\int_{\R^n}|\psi_m(x)-\psi(x)|^{2^*_s}\,dx\right)^{\frac{1}{2^*_s}}\to 0
\end{split}\end{equation*}
as~$m\to+\infty$, thanks to~\eqref{wefreypqqpqwpqwpqw-1}. 
All in all and going back to~\eqref{distri} we obtain that 
$$ 0=\lim_{m\to+\infty}\langle\mathcal{I}_\epsilon'(U_\infty),\Psi_m\rangle=
\langle\mathcal{I}_\epsilon'(U_\infty),\Psi\rangle,$$
and this shows that~\eqref{qwrewputhjtynmf} holds true for any~$\Psi\in\dot{H}^s_a(\R^{n+1})$. Namely,~$U_\infty$ is a critical point 
for~$\mathcal{I}_\epsilon$. Since~$U=0$ is the only critical point of~$\mathcal{I}_\epsilon$, we obtain the claim in~\eqref{U fa zero}. 
This concludes the proof of Lemma~\ref{lemma:weak}.
\end{proof}

As we did in the first part to obtain the existence of the minimum, 
(see in particular Lemma~\ref{tightness}), 
to prove Proposition~\ref{PScond2} we first need to show 
that the sequence is tight, according to Definition \ref{defTight}. 
Then we can prove the following:

\begin{lemma}\label{lemma:tight}
Let $\{U_k\}_{k\in\mathbb{N}}\subset \dot{H}^s_a(\mathbb{R}^{n+1}_+)$ be a sequence satisfying
the hypotheses of Proposition \ref{PScond2}. 
Assume also that~$U=0$ is the only critical point of~$\mathcal{I}_\varepsilon$.

Then, for all $\eta>0$ there exists $\rho>0$ such that for every $k\in\mathbb{N}$ there holds
$$\int_{\mathbb{R}^{n+1}_+\setminus B_\rho^+}{y^a|\nabla U_k|^2\,dX}
+\int_{\mathbb{R}^n\setminus\{B_\rho\cap\{y=0\}\}}{|U_k(x,0)|^{2^*_s}\,dx}
<\eta.$$
In particular, the sequence~$\{U_k\}_{k\in\mathbb{N}}$ is tight.
\end{lemma}

\begin{proof} 
From Lemma~\ref{lemma:weak} we have that
\begin{equation}\begin{split}\label{weak convergence}
& U_k\rightharpoonup 0 \quad \hbox{ in }\dot{H}^s_a(\mathbb{R}^{n+1}_+) \quad {\mbox{ as }}k\to+\infty \\
{\mbox{and }}& U_k(\cdot,0)\rightarrow 0\quad  \hbox{ a.e. in }\mathbb{R}^{n}\quad {\mbox{ as }}k\to+\infty.
\end{split}\end{equation}
Now we proceed by contradiction. 
That is, we suppose that there exists~$\eta_0>0$ such that for every~$\rho>0$ there exists~$k\in\mathbb{N}$ such that 
\begin{equation}\label{contradV2}
\int_{\R^{n+1}_+\setminus B_\rho^+} y^a |\nabla U_k|^2\,dX +
\int_{ \R^n \setminus\left(B_\rho\cap\{y=0\}\right) } (U_k)_+^{2^*_s}(x,0)\,dx \ge \eta_0.
\end{equation}
Proceeding as in \eqref{forse0}, one can prove that actually $k\rightarrow +\infty$ as $\rho\rightarrow +\infty$.

Let $U_\varepsilon$ be the local minimum of the functional $\mathcal{F}_\varepsilon$ found in Theorem~\ref{MINIMUM}.
Since~$U_\epsilon\in\dot{H}^s_a(\R^{n+1}_+)$, from
Propositions~\ref{WeightedSob}
and~\ref{traceIneq} we have that for any~$\epsilon>0$ there exists~$r:=r_\epsilon>0$ such that 
\begin{equation}\label{lasfrepungfmh}\begin{split}
&\int_{\R^{n+1}_+\setminus B_r^+}y^a|\nabla U_\epsilon|^2\,dX + 
\int_{\R^{n+1}_+\setminus B_r^+}y^a|U_\epsilon|^{2\gamma}\,dX \\&\qquad+ 
\int_{\R^{n}\setminus \left(B_r^+\cap \{y=0\}\right)}|U_\epsilon(x,0)|^{2^*_s}\,dx <\epsilon,
\end{split}\end{equation}
where~$\gamma:=1+\frac{2}{n-2s}$. Moreover, by~\eqref{uniformBoundUk} and 
again by Propositions~\ref{WeightedSob}
and~\ref{traceIneq} we deduce that 
\begin{equation}\label{epsBoundV2-3}
\begin{split}
\int_{\R^{n+1}_+}y^a&|\nabla U_k|^2\,dX+\int_{\R^{n+1}_+}{y^a|U_k|^{2\gamma}\,dX}+\int_{\R^n}{|U_k(x,0)|^{2^*_s}\,dx}\\
&+\int_{\R^{n+1}_+}{y^a|\nabla (U_k+U_\varepsilon)|^2\,dX}
+\int_{\R^{n+1}_+}{y^a(|U_k|+U_\varepsilon)^{2\gamma}\,dX}\\
&+\int_{\R^n}{(|U_k(x,0)|+U_\varepsilon(x,0))^{2^*_s}\,dx}\leq \tilde{M},
\end{split}\end{equation}
for some~$\tilde{M}>0$.

Now let~$j_\epsilon\in\N$ be integer part of~$\frac{\tilde{M}}{\epsilon}$, 
and set, for any~$l\in\{0,1,\ldots,j_\epsilon\}$
$$ I_l:=\{(x,y)\in\R^{n+1}_+ : r+l\le |(x,y)|\le r+l+1\}.$$
Notice that~$j_\epsilon\to +\infty$ as~$\epsilon\to 0$. 
Therefore, by~\eqref{epsBoundV2-3} we have that 
\begin{eqnarray*}
(j_\epsilon+1)\epsilon &\ge & \frac{\tilde{M}}{\epsilon}\epsilon\\
&\ge & \sum_{l=0}^{j_\epsilon} \Big( 
\int_{I_l}y^a|\nabla U_k|^2\,dX+\int_{I_l}{y^a|U_k|^{2\gamma}\,dX}
+\int_{I_l\cap\{y=0\}}{|U_k(x,0)|^{2^*_s}\,dx}\\
&&\qquad\quad +\int_{I_l}{y^a|\nabla (U_k+U_\varepsilon)|^2\,dX}
+\int_{I_l}{y^a(|U_k|+U_\varepsilon)^{2\gamma}\,dX}\\
&&\qquad\quad +\int_{I_l\cap\{y=0\}}{(|U_k(x,0)|+U_\varepsilon(x,0))^{2^*_s}\,dx} \Big).
\end{eqnarray*}
This implies that there exists~$\bar{l}\in\{0,1,\ldots,j_\epsilon\}$
such that, up to a subsequence, 
\begin{equation}\label{epsBoundV2}
\begin{split}
\int_{I_{\bar{l}}}y^a&|\nabla U_k|^2\,dX+\int_{I_{\bar{l}}}{y^a|U_k|^{2\gamma}\,dX}+\int_{I_{\bar{l}}\cap\{y=0\}}{|U_k(x,0)|^{2^*_s}\,dx}\\
&+\int_{I_{\bar{l}}}{y^a|\nabla (U_k+U_\varepsilon)|^2\,dX}
+\int_{I_{\bar{l}}}{y^a(|U_k|+U_\varepsilon)^{2\gamma}\,dX}\\
&+\int_{I_{\bar{l}}\cap\{y=0\}}{(|U_k(x,0)|+U_\varepsilon(x,0))^{2^*_s}\,dx}
\leq \varepsilon.
\end{split}\end{equation}

\noindent Let now $\chi\in C^\infty_0(\R^{n+1}_+,[0,1])$ be a cut-off function such that  
\begin{equation}\label{chi}
\chi(x,y)=\begin{cases}
1,\qquad |(x,y)|\leq r+\bar{l},\\
0,\qquad |(x,y)|\geq r+\bar{l}+1,
\end{cases}
\quad {\mbox{ and }} \quad |\nabla \chi|\leq 2.
\end{equation}
Define, for any~$k\in\N$,
\begin{equation}\label{def Wk}
W_{1,k}:=\chi U_k \quad {\mbox{ and }} \quad W_{2,k}:=(1-\chi)U_k.
\end{equation}
Hence~$W_{1,k}+W_{2,k}=U_k$ for any~$k\in\N$.
Moreover, 
\begin{equation}\label{conv Wk}
W_{1,k},\,W_{2,k}\rightharpoonup 0 \quad {\mbox{ in }} \dot{H}^s_a(\R^{n+1}_+),
\end{equation}
as~$k\to+\infty$. Indeed, for any~$\Psi\in\dot{H}^s_a(\R^{n+1}_+)$ 
with~$[\Psi]_a=1$ and~$\delta>0$, we have that 
$$  \left|\int_{\mathbb{R}^{n+1}_+}y^a\langle\nabla U_{k},\nabla\Psi\rangle \,dX\right|\le\frac{\delta}{2},$$
for any~$k$ sufficiently large, say~$k\ge \bar{k}(\delta)$, thanks to~\eqref{weak convergence}. Moreover, the compactness
result in Lemma~\ref{lemma:compact}
implies that
$$ \int_{I_{\bar{l}} }y^a|U_k|^2\,dX\le \frac{\delta^2}{16},$$
for~$k$ large enough (say~$k\ge\bar{k}(\delta)$, 
up to renaming~$\bar{k}(\delta)$). Therefore, recalling~\eqref{def Wk} and~\eqref{chi}
and using H\"older inequality, we obtain that 
\begin{eqnarray*}
&&\left|\int_{\mathbb{R}^{n+1}_+}y^a\langle\nabla W_{1,k},\nabla\Psi\rangle \,dX \right|\\&=& 
\left|\int_{\mathbb{R}^{n+1}_+}y^a\langle\nabla (\chi\,U_{k}),\nabla\Psi\rangle \,dX\right|\\
&\le  & \left|\int_{\mathbb{R}^{n+1}_+}y^a\chi \langle\nabla U_{k},\nabla\Psi\rangle \,dX\right| + 
\left|\int_{\mathbb{R}^{n+1}_+}y^a U_k \langle\nabla\chi,\nabla\Psi\rangle \,dX\right|\\
&\le & \left|\int_{\mathbb{R}^{n+1}_+}y^a\langle\nabla U_{k},\nabla\Psi\rangle \,dX\right|+ \int_{\mathbb{R}^{n+1}_+}y^a |U_k|\, |\nabla\chi|\,|\nabla\Psi|\,dX\\
&\le & \frac{\delta}{2} + 2\sqrt{\int_{\mathbb{R}^{n+1}_+}y^a |U_k|^2\,dX} 
\cdot\sqrt{\int_{\mathbb{R}^{n+1}_+}y^a |\nabla\Psi|^2\,dX}\\
&\le &\frac{\delta}{2}+2\frac{\delta}{4}\\
&=&\delta,
\end{eqnarray*}
which proves~\eqref{conv Wk} for~$W_{1,k}$. The proof for~$W_{2,k}$ 
is similar, and so we omit it.  

Furthermore, from~\eqref{conv Wk} and Theorem~7.1 in~\cite{DPV} we have that 
\begin{equation}\label{convWi}\begin{split}
& W_{i,k}(\cdot,0)\rightarrow 0\quad {\mbox{ a.e. }}\R^n,\\
&{\mbox{and }}(U_\varepsilon+W_{i,k})(\cdot,0)\rightarrow U_\varepsilon(\cdot,0)\ \hbox{ in }L^\alpha_{\rm{loc}}(\R^n),\;\; \forall\;1\leq \alpha<2^*_s,\; i=1,2,
\end{split}\end{equation}
as~$k\to+\infty$. Notice also that there exists a positive constant~$C$ 
(independent of~$k$) such that 
\begin{equation}\label{lim124}
[U_\epsilon +W_{i,k}]_a\le C,
\end{equation}
for~$i\in\{1,2\}$. Let us show~\eqref{lim124} only for~$W_{1,k}$, 
being the proof for~$W_{2,k}$ similar. From~\eqref{def Wk} we obtain that 
\begin{eqnarray*}
&& [W_{i,k}]_a^2=\int_{\R^{n+1}_+}y^a|\nabla W_{1,k}|^2\,dX 
= \int_{\R^{n+1}_+}y^a|\nabla(\chi U_k)|^2\,dX\\
&&\qquad \le 2\int_{\R^{n+1}_+}y^a\chi^2 |\nabla U_k|^2\,dX +
2\int_{\R^{n+1}_+}y^a |U_k|^2|\nabla\chi|^2\,dX\\
&&\qquad \le 2M +8\int_{I_{\bar{l}}}y^a |U_k|^2\,dX\le 2M+8C,
\end{eqnarray*}
for some~$C>0$ independent of~$k$, 
thanks to~\eqref{uniformBoundUk}, \eqref{chi} 
and Lemma~\ref{lemma:compact}. This, together with the fact 
that~$U_\epsilon\in\dot{H}^s_a(\R^{n+1}_+)$, gives~\eqref{lim124}.

Therefore, using hypothesis~(ii), 
\begin{equation}\label{conv1}
\lim_{k\to+\infty}\langle\mathcal{I}_\varepsilon'(U_k),
U_\varepsilon+W_{i,k}\rangle =0,\qquad i=1,2.
\end{equation}
On the other hand, by~\eqref{DER},
\begin{equation}\begin{split}\label{laskgprekjn b}
&\left| 
\langle \mathcal{I}_\varepsilon'(U_k)-\mathcal{I}_\varepsilon'(W_{1,k}),
U_\varepsilon+W_{1,k}\rangle \right| \\
\le\,& \left| \int_{\mathbb{R}^{n+1}_+} {y^a\langle \nabla (1-\chi)U_{k},\nabla(U_\varepsilon+W_{1,k})\rangle\,dX} \right|\\
& +\left| \varepsilon \int_{\mathbb{R}^n}{h(x)\left( (U_\varepsilon + (W_{1,k})_+ )^q-(U_\varepsilon +(U_k)_+)^q \right)(x,0)(U_\varepsilon+W_{1,k})(x,0)\,dx} \right| \\
&+\left| \int_{\mathbb{R}^n} \left( (U_\varepsilon+(W_{1,k})_+)^p
-(U_\varepsilon+(U_k)_+)^p \right)(x,0)(U_\varepsilon+W_{1,k})(x,0)\,dx \right|\\
=:\,& I_1+I_2+I_3.
\end{split}\end{equation}
To estimate~$I_1$, notice that~$I_1\le I_{1,1}+I_{1,2}$, where 
\begin{eqnarray*}
I_{1,1}&:=&\left|\int_{ \mathbb{R}^{n+1}_+} {y^a\langle \nabla (1-\chi)U_{k},\nabla U_\varepsilon\rangle\,dX} \right|\\
{\mbox{and }}\quad I_{1,2}&:=&\left|\int_{\mathbb{R}^{n+1}_+} {y^a\langle \nabla (1-\chi)U_{k},\nabla(\chi U_{k})\rangle\,dX} \right|.
\end{eqnarray*}
We split further~$I_{1,1}$ as 
\begin{eqnarray*}
I_{1,1}\le \left|\int_{\R^{n+1}_+\setminus B^+_{r+\bar{l}} } {y^a(1-\chi)\langle \nabla U_{k},\nabla U_\varepsilon\rangle\,dX} \right| + 
\left|\int_{ I_{\bar{l}}} {y^a U_k\langle \nabla (1-\chi),\nabla U_\varepsilon\rangle\,dX} \right|\\
\end{eqnarray*}
Since~$B^+_r\subset B^+_{r+\bar{l}}$, by H\"older inequality, \eqref{lasfrepungfmh} and~\eqref{uniformBoundUk} we have that
\begin{eqnarray*}
&&\left|\int_{\R^{n+1}_+\setminus B^+_{r+\bar{l}}} {y^a(1-\chi)\langle \nabla U_{k},\nabla U_\varepsilon\rangle\,dX} \right|\\
&&\qquad \le  \sqrt{\int_{\R^{n+1}_+} {y^a|\nabla U_{k}|^2\,dX}}
\cdot\sqrt{\int_{\R^{n+1}_+\setminus B^+_r} {y^a|\nabla U_\epsilon|^2\,dX}}
\le M\epsilon^{1/2}.
\end{eqnarray*}
Moreover, by~\eqref{chi} and applying twice the H\"older inequality 
(first with exponent~$1/2$ and then
with exponents~$\gamma$ and~$\frac{\gamma}{\gamma-1}$) we obtain that
\begin{eqnarray*}
&& \left|\int_{ I_{\bar{l}}} {y^a U_k\langle \nabla (1-\chi),\nabla U_\varepsilon\rangle\,dX} \right|\\
&&\qquad \le 2\left(\int_{I_{\bar{l}}}y^a|U_k|^2\,dX\right)^{1/2}
\left(\int_{I_{\bar{l}}}y^a|\nabla U_\epsilon|^2\,dX\right)^{1/2}\\
&&\qquad \le 2\left(\int_{I_{\bar{l}}}y^a|U_k|^2\,dX\right)^{1/2}
\left(\int_{\R^{n+1}_+}y^a|\nabla U_\epsilon|^2\,dX\right)^{1/2}\\
&&\qquad \le C\left(\int_{I_{\bar{l}}} y^a\,dX\right)^{\frac{\gamma-1}{2\gamma}} 
\left(\int_{I_{\bar{l}}}y^a|U_k|^{2\gamma}\,dX\right)^{\frac{1}{2\gamma}}
\le C\epsilon^{1/2\gamma},
\end{eqnarray*}
up to renaming constants, where~\eqref{epsBoundV2} was used in the last line. 
Hence, 
$$ I_{1,1}\le C\epsilon^{1/2\gamma},$$
for a suitable constant~$C>0$. Let us estimate~$I_{1,2}$: 
\begin{eqnarray*}
I_{1,2}&\le &\left|\int_{I_{\bar{l}}} {y^a|U_k|^2\langle \nabla (1-\chi),\nabla\chi\rangle\,dX} \right|
+ \left|\int_{I_{\bar{l}}} {y^a U_k\chi\langle\nabla (1-\chi),\nabla U_{k}\rangle\,dX} \right|\\ 
&&\qquad +\left|\int_{I_{\bar{l}}} {y^a\chi(1-\chi)|\nabla U_{k}|^2\,dX} \right|
+\left|\int_{I_{\bar{l}}} {y^a(1-\chi)U_k \langle\nabla U_{k},\nabla\chi\rangle\,dX} \right|.
\end{eqnarray*}
Thus, in the same way as before, and using~\eqref{epsBoundV2} once more, 
we obtain that~$I_{1,2}\le C\epsilon^{1/2\gamma}$ for some~$C>0$. Therefore
\begin{equation} \label{est1}
I_1 \le C\epsilon^{1/2\gamma},
\end{equation}
for some positive constant~$C$. 

We estimate now~$I_2$. For this, we first observe that formulas~\eqref{fdcvbpfp11} and~\eqref{def Wk} give that
$$ \left|(U_\epsilon+(W_{1,k})_+)^q-(U_\epsilon+(U_k)_+)^q\right| \le 
L|(W_{1,k})_+-(U_k)_+|^q = L (U_k)_+^q|1-\chi|^q,$$
for a suitable constant~$L>0$. 
Consequently, applying H\"older inequality with 
exponents~$\frac{2^*_s}{2^*_s-1-q}$, $\frac{2^*_s}{q}$ and~$2^*_s$ 
we obtain that 
\begin{eqnarray*}
I_2 &\le & \epsilon \int_{\R^n}|h|\,\left|(U_\epsilon+(W_{1,k})_+)^q-(U_\epsilon+(U_k)_+)^q\right|\,|U_\epsilon+W_{1,k}|\,dx\\
&\le & \epsilon L\int_{\R^n}|h| (U_k)_+^q |U_\epsilon+W_{1,k}|\,dx\\
&\le & \epsilon L \|h\|_{L^{\frac{2^*_s}{2^*_s-1-q}}(\R^n)} 
\left(\int_{\R^n}(U_k)_+^{2^*_s}\,dx\right)^{\frac{q}{2^*_s}}
\left( \int_{\R^n}|U_\epsilon+W_{1,k}|^{2^*_s}\,dx\right)^{\frac{1}{2^*_s}}
\\
&\le & C\epsilon, 
\end{eqnarray*}
for some~$C>0$, thanks to~\eqref{h0}, \eqref{uniformBoundUk} and Proposition~\ref{traceIneq}.

To estimate $I_3$, let us define the auxiliary function
$$ f(t):=(U_\varepsilon+t\chi (U_k)_+ +(1-t)(U_k)_+)^p,\qquad t\in[0,1].$$
Thus, recalling \eqref{def Wk}, we have that
\begin{equation*}\begin{split}
& \left|(U_\varepsilon+(W_{1,k})_+)^p-(U_\varepsilon+(U_k)_+)^p\right|
=\left|(U_\varepsilon+\chi(U_k)_+)^p-(U_\varepsilon+(U_k)_+)^p\right|
\\&\qquad =|f(1)-f(0)|=\left|\int_0^1f'(t)\,dt\right|\\
&\qquad \le p(1-\chi)(U_k)_+\int_0^1|U_\varepsilon+t\chi (U_k)_+ +(1-t)(U_k)_+|^{p-1}\,dt\\
&\qquad \leq p(1-\chi)(U_k)_+(U_\varepsilon+(U_k)_+)^{p-1}
\leq C(1-\chi)(U_k)_+ U_\varepsilon^{p-1}+C(1-\chi)(U_k)_+^p,
\end{split}\end{equation*}
for a suitable positive constant $C$. Therefore,
\begin{equation*}\begin{split}
I_3&\leq C\Big(\int_{\mathbb{R}^n}(1-\chi(x,0))(U_k)_+(x,0)U_\varepsilon^p(x,0)\,dx \\
&\qquad\quad  +\int_{\mathbb{R}^n}(1-\chi(x,0))(U_k)_+^p(x,0) U_\varepsilon(x,0)\,dx \\
&\qquad \quad +
\int_{I_{\bar{l}}} \chi(x,0)(1-\chi(x,0))U_\varepsilon^{p-1}(x,0)(U_k)_+^2(x,0)\,dx 
\\&\qquad\quad +\int_{I_{\bar{l}}}\chi(x,0)(1-\chi(x,0))(U_k)_+^{p+1}(x,0)\,dx\Big)\\
&=: I_{3,1}+I_{3,2}+I_{3,3}+I_{3,4}.
\end{split}\end{equation*}
Concerning $I_{3,1}$ and $I_{3,2}$ we are in the position to apply Lemma \ref{lemma:conv} 
with $V_k:=U_k$, $U_o:=U_\epsilon$ and $\psi:=1-\chi(\cdot,0)$ 
(notice that $\alpha:=1$ and $\beta:=p$ and $\alpha:=p$ and $\beta:=1$, respectively). 
So we obtain that both $I_{3,1}=o_k(1)$ and $I_{3,2}=o_k(1)$. 

Moreover, using H\"older inequality with exponent $\frac{p+1}{p-1}$ and $\frac{p+1}{2}$, 
Proposition \ref{traceIneq} and \eqref{epsBoundV2}, we have 
$$I_{3,3}\leq \left(\int_{I_{\bar{l}}}U_\epsilon^{p+1}(x,0)\,dx\right)^{\frac{p-1}{p+1}}
\left(\int_{I_{\bar{l}}}(U_k)_+^{p+1}(x,0)\,dx\right)^{\frac{2}{p+1}} \le 
C\varepsilon^{\frac{2}{p+1}},$$
for a suitable $C>0$. Finally, making use of \eqref{epsBoundV2} 
once again we obtain that $I_{3,4}\le C\epsilon$, for some $C>0$. 
Consequently, putting all these informations together we get 
$$I_3\leq C\varepsilon^{\frac{2}{p+1}}+o_k(1).$$
All in all, from \eqref{laskgprekjn b} we obtain that 
\begin{equation}\label{boundW1}
|\langle \mathcal{I}_\varepsilon'(U_k)-
\mathcal{I}_\varepsilon'(W_{1,k}),U_\varepsilon+W_{1,k}\rangle|\leq C\varepsilon^{\frac{1}{2\gamma}}
+o_k(1).
\end{equation}
Likewise, it can be checked that
\begin{equation}\label{boundW2}
|\langle \mathcal{I}_\varepsilon'(U_k)-\mathcal{I}_\varepsilon'(W_{2,k}),
U_\varepsilon+W_{2,k}\rangle|\leq C\varepsilon^{\frac{1}{2\gamma}}+o_k(1).
\end{equation}
Therefore, using this and \eqref{conv1}, 
\begin{equation}\label{boundWi}
|\langle \mathcal{I}_\varepsilon'(W_{i,k}),U_\varepsilon+W_{i,k}\rangle|
\leq C\varepsilon^{\frac{1}{2\gamma}} +o_k(1), \qquad i=1,2.
\end{equation}

From now on we organize the proof in three steps as follows: 
in the forthcoming Step 1 and 2 we show lower bounds for $\mathcal{I}_\epsilon(W_{1,k})$ 
and $\mathcal{I}_\epsilon(W_{2,k})$, respectively. Then, in Step 3 we use these estimates 
to obtain a lower bound for $\mathcal{I}_\epsilon(U_k)$ that will give a contradiction 
with the assumptions on $\mathcal{I}_\epsilon$, and so the desired claim in Lemma \ref{lemma:tight}
follows. 
\smallskip 

\noindent {\em Step 1: Lower bound for $\mathcal{I}_\epsilon(W_{1,k})$.} 
From \eqref{def I} and \eqref{DER} we have 
\begin{equation}\begin{split}\label{limW1}
&\mathcal{I}_\varepsilon(W_{1,k})-
\frac{1}{2}\langle \mathcal{I}_\varepsilon'(W_{1,k}),U_\varepsilon+W_{1,k}\rangle\\
 =\,& -\frac12\int_{\R^{n+1}_+}y^a\langle\nabla U_\epsilon,\nabla W_{1,k}\rangle\,dX \\
&-\frac{\varepsilon}{q+1}\int_{\R^n}h(x)\left( 
(U_\epsilon+(W_{1,k})_+)^{q+1}(x,0) -U_\varepsilon^{q+1}(x,0)\right)\,dx \\& + 
\epsilon\int_{\R^n}h(x)U^q_\epsilon(x,0)(W_{1,k})_+(x,0)\,dx \\
& -\frac{1}{p+1}\int_{\R^n}\left( 
(U_\epsilon+(W_{1,k})_+)^{p+1}(x,0) -U_\varepsilon^{p+1}(x,0)\right)\,dx \\& + 
\int_{\R^n}U^p_\epsilon(x,0)(W_{1,k})_+(x,0)\,dx \\
& +\frac{\epsilon}{2}\int_{\R^n}h(x)
\left( (U_\epsilon+(W_{1,k})_+)^{q}(x,0)-U_\epsilon^q(x,0)\right)
\left(U_\epsilon +W_{1,k}\right)(x,0) \,dx \\
& +\frac{1}{2}\int_{\R^n}\left( (U_\epsilon+(W_{1,k})_+)^{p}(x,0)-U_\epsilon(x,0)\right)
\left(U_\epsilon +W_{1,k}\right)(x,0) \,dx.
\end{split}\end{equation}
Thanks to \eqref{conv Wk}, we have that 
$$ \lim_{k\to+\infty}\int_{\R^{n+1}_+}y^a\langle\nabla U_\epsilon,\nabla W_{1,k}\rangle\,dX=0.$$
Moreover, from Lemma \ref{lemma:conv} applied here with $V_k:=W_{1,k}$, $U_o:=U_\epsilon$, 
$\psi:=h$, $\alpha:=1$ and $\beta:=q$ we have that 
\begin{equation}\label{ccc96}
\lim_{k\to+\infty}\int_{\R^n}h(x)U^q_\epsilon(x,0)(W_{1,k})_+(x,0)\,dx =0.\end{equation}
Analogously, by taking $V_k:=W_{1,k}$, $U_o:=U_\epsilon$, 
$\psi:=1$, $\alpha:=1$ and $\beta:=p$ in Lemma \ref{lemma:conv} 
(notice that in this case $\alpha+\beta=p+1=2^*_s$ and $\psi\in L^{\infty}(\R^n)$) 
we obtain that
\begin{equation}\label{ccc95}
\lim_{k\to+\infty}\int_{\R^n}U^p_\epsilon(x,0)(W_{1,k})_+(x,0)\,dx =0.\end{equation}
Taking the limit as $k\to+\infty$ in \eqref{limW1} and using the last three formulas, we obtain that 
\begin{equation}\begin{split}\label{limW1-1}
&\lim_{k\to+\infty}\mathcal{I}_\varepsilon(W_{1,k})-
\frac{1}{2}\langle \mathcal{I}_\varepsilon'(W_{1,k}),U_\varepsilon+W_{1,k}\rangle\\
 =\,& \lim_{k\to+\infty}\Big(-\frac{\varepsilon}{q+1}\int_{\R^n}h(x)\left( 
(U_\epsilon+(W_{1,k})_+)^{q+1}(x,0) -U_\varepsilon^{q+1}(x,0)\right)\,dx \\
& \qquad \quad -\frac{1}{p+1}\int_{\R^n}\left( 
(U_\epsilon+(W_{1,k})_+)^{p+1}(x,0) -U_\varepsilon^{p+1}(x,0)\right)\,dx \\
& \qquad \quad +\frac{\epsilon}{2}\int_{\R^n}h(x)\left( 
(U_\epsilon+(W_{1,k})_+)^{q}(x,0)-U_\epsilon^q(x,0)\right)
\left(U_\epsilon +W_{1,k}\right)(x,0) \,dx \\
& \qquad \quad +\frac{1}{2}\int_{\R^n}\left( (U_\epsilon+(W_{1,k})_+)^{p}(x,0)-U_\epsilon^p(x,0)\right)
\left(U_\epsilon +W_{1,k}\right)(x,0) \,dx\Big).
\end{split}\end{equation}
Now we observe that if $x\in\R^n$ is such that $W_{1,k}(x,0)\le 0$, then 
$$ (U_\epsilon+(W_{1,k})_+)^{q}(x,0)-U_\epsilon^q(x,0)=U_\epsilon^{q}(x,0)-U_\epsilon^q(x,0)=0,$$ 
and so 
\begin{eqnarray*}
&&\int_{\R^n}h(x)\left( 
(U_\epsilon+(W_{1,k})_+)^{q}(x,0)-U_\epsilon^q(x,0)\right)
\left(U_\epsilon +W_{1,k}\right)(x,0) \,dx \\
&=& \int_{\R^n}h(x)\left( 
(U_\epsilon+(W_{1,k})_+)^{q}(x,0)-U_\epsilon^q(x,0)\right)
\left(U_\epsilon +(W_{1,k})_+\right)(x,0) \,dx\\
&=& \int_{\R^n}h(x)
(U_\epsilon+(W_{1,k})_+)^{q+1}(x,0)\,dx - \int_{\R^n} h(x)U_\epsilon^{q+1}(x,0)\,dx \\&&\qquad -
\int_{\R^n} h(x)U_\epsilon^q(x,0) (W_{1,k})_+(x,0) \,dx\\
&=& \int_{\R^n}h(x)\left((U_\epsilon+(W_{1,k})_+)^{q+1}(x,0) - U_\epsilon^{q+1}(x,0)\right)\,dx \\&&\qquad-
\int_{\R^n} h(x)U_\epsilon^q(x,0) (W_{1,k})_+(x,0) \,dx.
\end{eqnarray*}
Analogously 
\begin{eqnarray*}
&& \int_{\R^n}\left( (U_\epsilon+(W_{1,k})_+)^{p}(x,0)-U_\epsilon^p(x,0)\right)
\left(U_\epsilon +W_{1,k}\right)(x,0) \,dx\\
&=& \int_{\R^n}\left((U_\epsilon+(W_{1,k})_+)^{p+1}(x,0) - U_\epsilon^{p+1}(x,0)\right)\,dx -
\int_{\R^n} U_\epsilon^p(x,0) (W_{1,k})_+(x,0) \,dx.
\end{eqnarray*}
Therefore, using once more \eqref{ccc96} and \eqref{ccc95}, from \eqref{limW1-1} we obtain that 
\begin{equation}\begin{split}\label{limW1-2}
&\lim_{k\to+\infty}\mathcal{I}_\varepsilon(W_{1,k})-
\frac{1}{2}\langle \mathcal{I}_\varepsilon'(W_{1,k}),U_\varepsilon+W_{1,k}\rangle\\
 =\,& \lim_{k\to+\infty}\Big(-\frac{\varepsilon}{q+1}\int_{\R^n}h(x)\left( 
(U_\epsilon+(W_{1,k})_+)^{q+1}(x,0) -U_\varepsilon^{q+1}(x,0)\right)\,dx \\
& \qquad \quad -\frac{1}{p+1}\int_{\R^n}\left( 
(U_\epsilon+(W_{1,k})_+)^{p+1}(x,0) -U_\varepsilon^{p+1}(x,0)\right)\,dx \\
& \qquad \quad +\frac{\epsilon}{2} \int_{\R^n}h(x)\left((U_\epsilon+(W_{1,k})_+)^{q+1}(x,0) 
- U_\epsilon^{q+1}(x,0)\right)\,dx \\& \qquad \quad-
\frac{\epsilon}{2}\int_{\R^n} h(x)U_\epsilon^q(x,0) (W_{1,k})_+(x,0) \,dx \\
& \qquad \quad +\frac{1}{2}\int_{\R^n}\left((U_\epsilon+(W_{1,k})_+)^{p+1}(x,0) - U_\epsilon^{p+1}(x,0)\right)\,dx \\& \qquad \quad-\frac{1}{2}
\int_{\R^n} U_\epsilon^p(x,0) (W_{1,k})_+(x,0) \,dx\Big)\\
=\,& \lim_{k\to+\infty}\Big(-\frac{\varepsilon}{q+1}\int_{\R^n}h(x)\left( 
(U_\epsilon+(W_{1,k})_+)^{q+1}(x,0) -U_\varepsilon^{q+1}(x,0)\right)\,dx \\
& \qquad \quad -\frac{1}{p+1}\int_{\R^n}\left( 
(U_\epsilon+(W_{1,k})_+)^{p+1}(x,0) -U_\varepsilon^{p+1}(x,0)\right)\,dx \\
& \qquad \quad +\frac{\epsilon}{2} \int_{\R^n}h(x)\left((U_\epsilon+(W_{1,k})_+)^{q+1}(x,0) 
- U_\epsilon^{q+1}(x,0)\right)\,dx \\
& \qquad \quad +\frac{1}{2}\int_{\R^n}\left((U_\epsilon+(W_{1,k})_+)^{p+1}(x,0) - 
U_\epsilon^{p+1}(x,0)\right)\,dx \\
=\,& \lim_{k\to+\infty}\Big(-\epsilon\left(\frac{1}{q+1}-\frac12\right)\int_{\R^n}h(x)\left( 
(U_\epsilon+(W_{1,k})_+)^{q+1}(x,0) -U_\varepsilon^{q+1}(x,0)\right)\,dx \\
& \qquad \quad +\left(\frac12-\frac{1}{p+1}\right)\int_{\R^n}\left( 
(U_\epsilon+(W_{1,k})_+)^{p+1}(x,0) -U_\varepsilon^{p+1}(x,0)\right)\,dx\Big).
\end{split}\end{equation}
Now we claim that 
\begin{equation}\label{ewhupytlkvdxsmnvew}
\lim_{k\to+\infty}\int_{\R^n}h(x)\left( 
(U_\epsilon+(W_{1,k})_+)^{q+1}(x,0) -U_\varepsilon^{q+1}(x,0)\right)\,dx=0. 
\end{equation}
For this, notice that if $x\in\R^n\setminus B_{r+\bar{l}+1}$, then $W_{1,k}(x,0)=0$, thanks to 
\eqref{def Wk} and \eqref{chi}. Therefore, for any $x\in\R^n\setminus B_{r+\bar{l}+1}$ we have that 
$$ (U_\epsilon+(W_{1,k})_+)^{q+1}(x,0) -U_\varepsilon^{q+1}(x,0) = 
U_\epsilon^{q+1}(x,0) -U_\varepsilon^{q+1}(x,0)=0.$$ 
Thus 
\begin{equation}\begin{split}\label{qwtrewgrhgvr}
&\int_{\R^n}h(x)\left( 
(U_\epsilon+(W_{1,k})_+)^{q+1}(x,0) -U_\varepsilon^{q+1}(x,0)\right)\,dx \\=\,& 
\int_{B_{r+\bar{l}+1}}h(x)\left( 
(U_\epsilon+(W_{1,k})_+)^{q+1}(x,0) -U_\varepsilon^{q+1}(x,0)\right)\,dX.
\end{split}\end{equation}
Thanks to \eqref{convWi}, we have that $W_{1,k}(\cdot, 0)$ converges to zero a.e. in $\R^n$, 
and so $(W_{1,k})_+(\cdot, 0)$ converges to zero a.e. in $\R^n$, as $k\to+\infty$. 
Therefore
$$ U_\epsilon+(W_{1,k})_+^{q+1}(x,0) \to U_\varepsilon^{q+1}(x,0)\quad 
{\mbox{ for a.e. }}x\in\R^n,$$
as $k\to+\infty$. 
Moreover the strong convergence of $W_{1,k}(\cdot,0)$ in $L^{q+1}_{\rm{loc}}(\R^n)$ 
(due again to \eqref{convWi}) and Theorem 4.9 in \cite{brezis} imply that 
there exists a function $F\in L^{q+1}_{\rm{loc}}(\R^n)$ such that 
$|W_{1,k}(x,0)|\le |F(x)|$ for a.e. $x\in\R^n$. This and the boundedness of $U_\epsilon$ 
(see Corollary \ref{coro:bound}) give that 
$$ h(U_\epsilon+(W_{1,k})_+)^{q+1}\le |h|(|U_\epsilon|+|W_{1,k}|)^{q+1}\le C|h|(1+|F|^{q+1})
\in L^1(B_{r+\bar{l}+1}),$$
for a suitable $C>0$. Thus, the Dominated Convergence Theorem applies, and 
together with \eqref{qwtrewgrhgvr} give the convergence in \eqref{ewhupytlkvdxsmnvew}.

Consequently, from \eqref{limW1-2} and \eqref{ewhupytlkvdxsmnvew} we obtain that 
\begin{equation}\begin{split}\label{limPositive}
&\lim_{k\rightarrow\infty}
\left(\mathcal{I}_\varepsilon(W_{1,k})-\frac{1}{2}\langle 
\mathcal{I}_\varepsilon'(W_{1,k}),U_\varepsilon+W_{1,k}\rangle\right)\\
=\,& \left(\frac12-\frac{1}{p+1}\right) \lim_{k\rightarrow\infty}
\int_{\mathbb{R}^n} \left( (U_\varepsilon+(W_{1,k})_+)^{p+1}(x,0)-U_\varepsilon^{p+1}(x,0)\right)
\,dx \geq 0
\end{split}\end{equation}
(recall that $p+1=2^*_s>2$). 
In particular, by \eqref{boundWi} and \eqref{limPositive}, there holds
\begin{equation}\begin{split}\label{infBoundW1}
\mathcal{I}_\varepsilon(W_{1,k})&=  \mathcal{I}_\varepsilon(W_{1,k})
-\frac{1}{2}\langle \mathcal{I}_\varepsilon'(W_{1,k}),U_\varepsilon+W_{1,k}\rangle
+\frac{1}{2}\langle \mathcal{I}_\varepsilon'(W_{1,k}),U_\varepsilon+W_{1,k}\rangle \\
&\geq \mathcal{I}_\varepsilon(W_{1,k})
-\frac{1}{2}\langle \mathcal{I}_\varepsilon'(W_{1,k}),U_\varepsilon+W_{1,k}\rangle
-C\varepsilon^{\frac{1}{2\gamma}} +o_k(1)\\
&\ge -C\varepsilon^{\frac{1}{2\gamma}} +o_k(1),
\end{split}\end{equation}
where $C$ is a positive constant that may change from line to line. 

Formula \eqref{infBoundW1} provides the desired estimate from below 
for $\mathcal{I}_\epsilon(W_{1,k})$. Next step 
is to obtain an estimate from below for $\mathcal{I}_\epsilon(W_{2,k})$ as well.

\smallskip

\noindent {\em Step 2: Lower bound for $\mathcal{I}_\epsilon(W_{2,k})$.} 
We first observe that formula~\eqref{fdcvbpfp11} implies that there exists a constant $L>0$ such that
$$ |(U_\epsilon+(W_{2,k})_+)^q(x,0) -U_\epsilon^q(x,0)|\le L(W_{2,k})_+^q(x,0).$$
Hence 
\begin{equation*}\begin{split}
& \left|\varepsilon
\int_{\mathbb{R}^n} 
h(x)\left((U_\varepsilon+(W_{2,k})_+)^{q}(x,0)-U_\varepsilon^q(x,0)\right)
(U_\varepsilon+W_{2,k})(x,0)\,dx\right|\\
&\qquad \le \varepsilon
\int_{\mathbb{R}^n} 
|h(x)|\left|(U_\varepsilon+(W_{2,k})_+)^{q}(x,0)-U_\varepsilon^q(x,0)\right|
|(U_\varepsilon+W_{2,k})(x,0)|\,dx\\
&\qquad \le \varepsilon L
\int_{\mathbb{R}^n} 
|h(x)| (W_{2,k})_+^{q}(x,0)|(U_\varepsilon+W_{2,k})(x,0)|\,dx\\
&\qquad \le \varepsilon L\left(
\int_{\mathbb{R}^n} |h(x)| (W_{2,k})_+^{q}(x,0)U_\varepsilon(x,0)\,dx +
\int_{\mathbb{R}^n} |h(x)| (W_{2,k})_+^{q+1}(x,0)\,dx\right).
\end{split}\end{equation*}
Thanks to Lemma \ref{lemma:conv} (applied here with $V_k:=W_{2,k}$, $U_o:=U_\epsilon$, 
$\psi:=h$, $\alpha:=q$ and $\beta:=1$) we have that 
$$ \lim_{k\to+\infty}\int_{\mathbb{R}^n} |h(x)| (W_{2,k})_+^{q}(x,0)U_\varepsilon(x,0)\,dx=0.$$ 
Moreover, by H\"older inequality with exponents $\frac{2^*_s}{2^*_s-1-q}$ and $\frac{2^*_s}{q+1}$ 
we obtain that 
\begin{eqnarray*}
&& \int_{\mathbb{R}^n} |h(x)| (W_{2,k})_+^{q+1}(x,0)\,dx\\&&\qquad \le 
\left(\int_{\mathbb{R}^n} |h(x)|^{\frac{2^*_s}{2^*_s-1-q}}\,dx\right)^{\frac{2^*_s-1-q}{2^*_s}} 
\left(\int_{\mathbb{R}^n}(W_{2,k})_+^{2^*_s}(x,0)\,dx\right)^{\frac{q+1}{2^*_s}}
\\&&\qquad \le C \left(\int_{\mathbb{R}^n}(U_{k})_+^{2^*_s}(x,0)\,dx\right)^{\frac{q+1}{2^*_s}} 
\le C [U_k]_a^{q+1}\le C,
\end{eqnarray*}
for some constant $C>0$, 
where we have also used \eqref{h0}, Proposition \ref{traceIneq} and Corollary \ref{coro:bound}. 
The last three formulas imply that 
\begin{equation*}
\left|\varepsilon
\int_{\mathbb{R}^n} 
h(x)\left((U_\varepsilon+(W_{2,k})_+)^{q}(x,0)-U_\varepsilon^q(x,0)\right)
(U_\varepsilon+W_{2,k})(x,0)\,dx\right|\le C\epsilon +o_k(1),\end{equation*}
for a suitable $C>0$.
This, together with \eqref{DER}, \eqref{conv Wk} and \eqref{boundWi} (with $i=2$) gives 
\begin{equation}\begin{split}\label{upBoundW2}
& \int_{\mathbb{R}^{n+1}_+}{y^a|\nabla W_{2,k}|^2\,dX}\\
=\,& \langle \mathcal{I}_\epsilon'(W_{2,k}, U_\epsilon+W_{2,k}\rangle 
- \int_{\mathbb{R}^{n+1}_+}{y^a\langle\nabla W_{2,k}, \nabla U_\epsilon\rangle,dX}\\
&\qquad + \epsilon\int_{\R^n}h(x)\left((U_\varepsilon+(W_{2,k})_+)^{q}(x,0)-U_\varepsilon^q(x,0)\right)
(U_\varepsilon+W_{2,k})(x,0)\,dx \\
&\qquad 
+ \int_{\R^n}\left((U_\varepsilon+(W_{2,k})_+)^{p}(x,0)-U_\varepsilon^p(x,0)\right)
(U_\varepsilon+W_{2,k})(x,0)\,dx \\
 \leq\,& C\varepsilon^{\frac{1}{2\gamma}} +o_k(1) 
+\int_{\R^n}\left((U_\varepsilon+(W_{2,k})_+)^{p}(x,0)-U_\varepsilon^p(x,0)\right)
(U_\varepsilon+W_{2,k})(x,0)\,dx.
\end{split}\end{equation}
Now notice that if $x\in\R^n$ is such that $W_{2,k}\le 0$ then 
$$ (U_\varepsilon+(W_{2,k})_+)^{p}(x,0)-U_\varepsilon^p(x,0) =0,$$
and so 
\begin{eqnarray*}
&& \int_{\R^n}\left((U_\varepsilon+(W_{2,k})_+)^{p}(x,0)-U_\varepsilon^p(x,0)\right)
(U_\varepsilon+W_{2,k})(x,0)\,dx \\
&=& \int_{\R^n}\left((U_\varepsilon+(W_{2,k})_+)^{p}(x,0)-U_\varepsilon^p(x,0)\right)
(U_\varepsilon+(W_{2,k})_+)(x,0)\,dx\\
&=& \int_{\R^n}\left((U_\varepsilon+(W_{2,k})_+)^{p+1}(x,0)
-U_\varepsilon^{p+1}(x,0)\right)\,dx -\int_{\R^n}
U_\varepsilon^{p}(x,0)(W_{2,k})_+(x,0)\,dx.
\end{eqnarray*}
According to Lemma \ref{lemma:conv} (applied here with $V_k:=W_{2,k}$, $U_o:=U_\epsilon$, 
$\psi:=1$, $\alpha:=1$ and $\beta:=p$) we have that 
$$ \lim_{k\to+\infty}\int_{\R^n}
U_\varepsilon^{p}(x,0)(W_{2,k})_+(x,0)\,dx=0.$$ 
Therefore, \eqref{upBoundW2} becomes
\begin{equation}\begin{split}\label{upBoundW2-1}
& \int_{\mathbb{R}^{n+1}_+}{y^a|\nabla W_{2,k}|^2\,dX}\\
\le\, & \int_{\R^n}\left((U_\varepsilon+(W_{2,k})_+)^{p+1}(x,0)
-U_\varepsilon^{p+1}(x,0)\right)\,dx + C\varepsilon^{\frac{1}{2\gamma}} +o_k(1).
\end{split}\end{equation}
Furthermore, it is not difficult to see that that there exist two constants $0<c_1<c_2$ such that 
$$ c_1\leq \frac{(1+t)^{p+1}-1-t^{p+1}}{t^p+t}\leq c_2,\qquad t>0 $$
Thus, setting $t:=\frac{(W_{2,k})_+}{U_\epsilon}$, one has
\begin{eqnarray*}
&& (U_\varepsilon+(W_{2,k})_+)^{p+1}-U_\varepsilon^{p+1}= 
U_\epsilon^{p+1} \left[ \left(  1+\frac{(W_{2,k})_+}{U_\epsilon} \right)^{p+1}-1\right] \\
&&\qquad \le U_\epsilon^{p+1} \left[ c_2 \left(   \frac{(W_{2,k})_+^p}{U_\epsilon^p}
+\frac{(W_{2,k})_+}{U_\epsilon} \right) + \frac{(W_{2,k})_+^{p+1}}{U_\epsilon^{p+1}} \right] \\
&&\qquad =c_2 U_\epsilon (W_{2,k})_+^p +c_2 U_\epsilon^p(W_{2,k})_+ +(W_{2,k})_+^{p+1}.
\end{eqnarray*}
Therefore 
\begin{equation*}\begin{split}
& \int_{\mathbb{R}^n}((U_\varepsilon+(W_{2,k})_+)^{p+1}(x,0)-
U_\varepsilon^{p+1}(x,0))\,dx \\
\leq\,&  c_2\int_{\mathbb{R}^n}{U_\varepsilon^p(x,0) (W_{2,k})_+(x,0)\,dx}
+c_2 \int_{\mathbb{R}^n}{(W_{2,k})_+^p(x,0)U_\varepsilon(x,0) \,dx}
\\&\qquad+\int_{\mathbb{R}^n}{(W_{2,k})_+^{p+1}(x,0) \,dx}.
\end{split}\end{equation*}
Applying Lemma \ref{lemma:conv} once more, we obtain that 
\begin{eqnarray*}
&&\lim_{k\to+\infty}\int_{\mathbb{R}^n}{U_\varepsilon^p(x,0) (W_{2,k})_+(x,0)\,dx}=0\\
{\mbox{and }} && \lim_{k\to+\infty}\int_{\mathbb{R}^n}{(W_{2,k})_+^p(x,0)U_\varepsilon(x,0) \,dx}=0.
\end{eqnarray*}
Hence, going back to \eqref{upBoundW2-1}, we get 
\begin{equation}\label{upW2}
\int_{\mathbb{R}^{n+1}_+}{y^a|\nabla W_{2,k}|^2\,dX}\leq 
\int_{\mathbb{R}^n}{(W_{2,k})_+^{p+1}(x,0)\,dx} +C\varepsilon^{\frac{1}{2\gamma}}+o_k(1).
\end{equation}

Now we observe that, thanks to \eqref{def Wk}, $W_{2,k}=U_k$ outside $B_{r+\bar{l}+1}$. 
So, using \eqref{contradV2} with $\rho:=r+\bar{l}+1$, we have 
\begin{eqnarray*}
&& \int_{\R^{n+1}_+\setminus B_{r+\bar{l}+1}^+} y^a |\nabla W_{2,k}|^2\,dX +
\int_{ \R^n \setminus\left(B_{r+\bar{l}+1}\cap\{y=0\}\right) } (W_{2,k})_+^{2^*_s}(x,0)\,dx \\
&&\qquad =
\int_{\R^{n+1}_+\setminus B_{r+\bar{l}+1}^+} y^a |\nabla U_k|^2\,dX +
\int_{ \R^n \setminus\left(B_{r+\bar{l}+1}\cap\{y=0\}\right) } (U_k)_+^{2^*_s}(x,0)\,dx \ge \eta_0,
\end{eqnarray*}
for some $k$ that depends on $\rho$. 
This implies that either 
$$\int_{ \R^{n+1}_+\setminus B_{r+\bar{l}+1}^+ }{y^a|\nabla W_{2,k}|^2\,dX}\geq\frac{\eta_0}{2}$$
or
$$\int_{\mathbb{R}^n\setminus\{B_{r+\bar{l}+1}\cap \{y=0\}\}}{(W_{2,k})_+^{p+1}(x,0)\,dx}\geq \frac{\eta_0}{2}.
$$
In the first case we have 
$$\int_{\mathbb{R}^{n+1}_+}{y^a|\nabla W_{2,k}|^2\,dX}\geq 
\int_{\R^{n+1}_+\setminus B_{r+\bar{l}+1}^+}{y^a|\nabla W_{2,k}|^2\,dX}
\ge \frac{\eta_0}{2}.$$
From this and \eqref{upW2} it follows 
\begin{equation}\label{usare}
\int_{\mathbb{R}^n}{(W_{2,k})_+^{p+1}(x,0) \,dx}> \frac{\eta_0}{4}.
\end{equation}
In the second case, this inequality holds trivially. 

Accordingly, we can define~$\psi_k:=\alpha_k W_{2,k}$, where
\begin{equation}\label{def alfa}
\alpha^{p-1}_k:=\frac{[W_{2,k}]_a^2}
{\|(W_{2,k})_+(\cdot,0)\|_{L^{p+1}(\mathbb{R}^n)}^{p+1}}.\end{equation}
We claim that 
\begin{equation}\label{wqftrejhlipppuy}
[W_{2,k}]_a^2\le \|(W_{2,k})_+(\cdot,0)\|_{L^{p+1}(\R^n)}^{p+1}+C\epsilon^{\frac{2}{p+1}}+o_k(1),
\end{equation}
for a suitable positive constant $C$. For this, notice 
that \eqref{DER}, \eqref{conv Wk} and \eqref{boundWi} give 
\begin{equation}\begin{split}\label{sqwlgergbedsdc}
[W_{2,k}]_a^2 =\,& \langle \mathcal{I}_\epsilon'(W_{2,k}),U_\epsilon+W_{2,k}\rangle  - 
\int_{\R^{n+1}_+}y^a\langle \nabla U_\epsilon,\nabla W_{2,k}\rangle\,dX \\
&\quad +\varepsilon\int_{\mathbb{R}^n}h(x)\left(  (U_\varepsilon+(W_{2,k})_+)^{q}(x,0)-U_\epsilon^q(x,0)
\right) (U_\epsilon+W_{2,k})(x,0)\,dx
\\&\quad +\int_{\mathbb{R}^n}\left( (U_\varepsilon+(W_{2,k})_+)^{p}(x,0)-U_\epsilon^p(x,0)\right)
(U_\epsilon+W_{2,k})(x,0)\,dx\\
\le\, & C\epsilon^{\frac{1}{2\gamma}} +o_k(1) \\
&\quad +\varepsilon\int_{\mathbb{R}^n}h(x)\left(  (U_\varepsilon+(W_{2,k})_+)^{q}(x,0)-U_\epsilon^q(x,0)
\right) (U_\epsilon+W_{2,k})(x,0)\,dx
\\&\quad +\int_{\mathbb{R}^n}\left( (U_\varepsilon+(W_{2,k})_+)^{p}(x,0)-U_\epsilon^p(x,0)\right)
(U_\epsilon+W_{2,k})(x,0)\,dx.
\end{split}\end{equation}
We can rewrite 
\begin{eqnarray*}
&&\int_{\mathbb{R}^n}h(x)\left(  (U_\varepsilon+(W_{2,k})_+)^{q}(x,0)-U_\epsilon^q(x,0)
\right) (U_\epsilon+W_{2,k})(x,0)\,dx\\
&=& \int_{\mathbb{R}^n}h(x)\left(  (U_\varepsilon+(W_{2,k})_+)^{q+1}(x,0)-U_\epsilon^{q+1}(x,0)\right)\,dx
\\&&\qquad-\int_{\R^n}h(x) U_\epsilon^{q}(x,0)(W_{2,k})_+(x,0)\,dx\\
&=& \int_{\mathbb{R}^n}h(x)\left(  (U_\varepsilon+(W_{2,k})_+)^{q+1}(x,0)-U_\epsilon^{q+1}(x,0)\right)\,dx
+o_k(1),
\end{eqnarray*}
where we have applied once again Lemma \ref{lemma:conv}. Analogously, 
\begin{eqnarray*}
&&\int_{\mathbb{R}^n}\left( (U_\varepsilon+(W_{2,k})_+)^{p}(x,0)-U_\epsilon^p(x,0)\right)
(U_\epsilon+W_{2,k})(x,0)\,dx\\
&=& \int_{\mathbb{R}^n}\left((U_\varepsilon+(W_{2,k})_+)^{p+1}(x,0)-U_\epsilon^{p+1}(x,0)\right)\,dx
+o_k(1).
\end{eqnarray*}
Plugging these informations into \eqref{sqwlgergbedsdc}, we obtain that
\begin{equation}\begin{split}\label{sqwlgergbedsdc-1}
[W_{2,k}]_a^2 \le& 
\epsilon 
\int_{\mathbb{R}^n}h(x)\left(  (U_\varepsilon+(W_{2,k})_+)^{q+1}(x,0)-U_\epsilon^{q+1}(x,0)\right)\,dx\\
&\quad +\int_{\mathbb{R}^n}\left((U_\varepsilon+(W_{2,k})_+)^{p+1}(x,0)-U_\epsilon^{p+1}(x,0)\right)\,dx\\
&\quad +C\epsilon^{\frac{1}{2\gamma}} +o_k(1).
\end{split}\end{equation}

\noindent So using~\eqref{usare-1} into \eqref{sqwlgergbedsdc-1} we obtain 
\begin{equation}\begin{split}\label{sqwlgergbedsdc-2}
[W_{2,k}]_a^2 \le& 
\int_{\mathbb{R}^n}\left((U_\varepsilon+(W_{2,k})_+)^{p+1}(x,0)-U_\epsilon^{p+1}(x,0)\right)\,dx\\
&\quad +C\epsilon^{\frac{1}{2\gamma}} +o_k(1),
\end{split}\end{equation}
up to renaming constants. 
Now we use~\eqref{second} with $V_k:=W_{2,k}$ and $U_o:=U_\epsilon$ and we get 
\begin{eqnarray*}
[W_{2,k}]_a^2 &\le & \int_{\mathbb{R}^n}(U_\varepsilon+(W_{2,k})_+)^{p+1}(x,0)\,dx 
-\int_{\R^n}U_\epsilon^{p+1}(x,0)\,dx \\
&&\qquad +C\epsilon^{\frac{1}{2\gamma}} +o_k(1)\\
&=& \| (U_\eps + (W_{2,k})_+)(\cdot,0)\|_{L^{p+1}(\R^n)}^{p+1}
- \| U_\eps \|_{L^{p+1}(\R^n)}^{p+1}+C\epsilon^{\frac{1}{2\gamma}} +o_k(1)\\
&=&\|(W_{2,k})_+(\cdot,0)\|_{L^{p+1}(\R^n)}^{p+1} +C\epsilon^{\frac{1}{2\gamma}}+o_k(1),
\end{eqnarray*}
and this shows~\eqref{wqftrejhlipppuy}. 

From \eqref{def alfa}, \eqref{wqftrejhlipppuy} and \eqref{usare} we have that 
\begin{equation}\label{star}
\alpha_k^{p-1}= \frac{[W_{2,k}]_a^2}
{\|(W_{2,k})_+(\cdot,0)\|_{L^{p+1}(\mathbb{R}^n)}^{p+1}}
\le 1+C\epsilon^{\frac{1}{2\gamma}} +o_k(1).
\end{equation}
Notice also that, with the choice of $\alpha_k$ in \eqref{def alfa}, it holds
$$ [\psi_k]_a^2=\alpha_k^2 [W_{2,k}]_a^2 = \alpha_k^{p+1} 
\|(W_{2,k})_+(\cdot,0)\|_{L^{p+1}(\R^n)}^{p+1}= \|(\psi_{k})_+(\cdot,0)\|_{L^{p+1}(\R^n)}^{p+1}. $$
Hence, by \eqref{equivNorms} and Proposition~\ref{traceIneq}, we have that 
\begin{eqnarray*}
&& S\le\frac{[\psi_k(\cdot,0)]_{\dot{H}^s(\R^n)}^2}{\|(\psi_k)_+(\cdot,0)\|_{L^{p+1}(\R^n)}^2} =  \frac{[\psi_k]_{a}^2}{\|(\psi_k)_+(\cdot,0)\|_{L^{p+1}(\R^n)}^2}\\
&&\qquad = \frac{\|(\psi_k)_+(\cdot,0)\|_{L^{p+1}(\R^n)}^{p+1}}{\|(\psi_k)_+(\cdot,0)\|_{L^{p+1}(\R^n)}^2} = \|(\psi_k)_+(\cdot,0)\|_{L^{p+1}(\R^n)}^{p-1}.
\end{eqnarray*}
Accordingly, 
$$ \|(W_{2,k})_+(\cdot,0)\|_{L^{p+1}(\R^n)}^{p+1}=\frac{\|(\psi_k)_+(\cdot,0)\|_{L^{p+1}(\R^n)}^{p+1}}{\alpha_k^{p+1}}\ge S^{n/2s}\frac{1}{\alpha_k^{p+1}},$$
where we have used the fact that~$p-1=2^*_s-2=\frac{4s}{n-2s}$. 
This, together with~\eqref{star}, gives that
\begin{equation}\begin{split}\label{SobW2}
S^{n/2s}\le\, & (1+C\epsilon^{\frac{2}{p+1}} +o_k(1))^{\frac{p+1}{p-1}}
\|(W_{2,k})_+(\cdot,0)\|_{L^{p+1}(\R^n)}^{p+1}\\
\le\, & \|(W_{2,k})_+(\cdot,0)\|_{L^{p+1}(\R^n)}^{p+1} 
+C\epsilon^{\frac{1}{2\gamma}} +o_k(1).
\end{split}\end{equation}
Moreover, by \eqref{conv Wk} and Lemma \ref{lemma:conv} we have that 
\begin{eqnarray*}
&& \mathcal{I}_\varepsilon(W_{2,k})-
\frac{1}{2}\langle \mathcal{I}'_\varepsilon(W_{2,k}),U_\varepsilon+W_{2,k}\rangle \\
&=& -\int_{\R^{n+1}_+}y^a\langle\nabla W_{2,k},\nabla U_\epsilon\rangle\,dX \\
&&\quad -\frac{\epsilon}{q+1}\int_{\R^n}h(x)\left( (U_\epsilon+(W_{2,k})_+)^{q+1}(x,0)
-U_\epsilon^{q+1}(x,0)\right)\,dx \\
&&\quad +\epsilon\int_{\R^n}h(x)U_\epsilon^q(x,0)(W_{2,k})_+(x,0)\,dx\\
&&\quad -\frac{1}{p+1}\int_{\R^n}\left( (U_\epsilon+(W_{2,k})_+)^{p+1}(x,0)
-U_\epsilon^{p+1}(x,0)\right)\,dx \\
&&\quad +\int_{\R^n}U_\epsilon^p(x,0)(W_{2,k})_+(x,0)\,dx\\
&&\quad +\frac{\epsilon}{2}\int_{\R^n}h(x)\left( (U_\epsilon+(W_{2,k})_+)^{q}(x,0)
-U_\epsilon^{q}(x,0)\right)(U_\epsilon+W_{2,k})(x,0))\,dx \\
&&\quad +\frac{1}{2}\int_{\R^n}\left( (U_\epsilon+(W_{2,k})_+)^{p}(x,0)
-U_\epsilon^{p}(x,0)\right)(U_\epsilon+W_{2,k})(x,0))\,dx \\
&=& -\epsilon\left(\frac{1}{q+1}-\frac12\right)
\int_{\R^n}h(x)\left( (U_\epsilon+(W_{2,k})_+)^{q+1}(x,0)
-U_\epsilon^{q+1}(x,0)\right)\,dx \\
&&\quad +\left(\frac12-\frac{1}{p+1}\right)
\int_{\R^n}\left( (U_\epsilon+(W_{2,k})_+)^{p+1}(x,0)
-U_\epsilon^{p+1}(x,0)\right)\,dx + o_k(1).
\end{eqnarray*}
We observe that~$\frac12-\frac{1}{p+1}=\frac{s}{n}$. 
Thus, using~\eqref{usare-1} we have that 
\begin{eqnarray*}
&& \mathcal{I}_\varepsilon(W_{2,k})-
\frac{1}{2}\langle \mathcal{I}'_\varepsilon(W_{2,k}),U_\varepsilon+W_{2,k}\rangle \\
&\ge & \frac{s}{n}\int_{\R^n}\left( (U_\epsilon+(W_{2,k})_+)^{p+1}(x,0)
-U_\epsilon^{p+1}(x,0)\right)\,dx -C\epsilon +o_k(1),
\end{eqnarray*}
for some~$C>0$. Therefore, using~\eqref{second} with 
$V_k:=W_{2,k}$ and $U_o:=U_\epsilon$, we have that 
$$ \mathcal{I}_\varepsilon(W_{2,k})-\frac{1}{2}\langle \mathcal{I}'_\varepsilon(W_{2,k}),U_\varepsilon+W_{2,k}\rangle\geq
\frac{s}{n}\|(W_{2,k})_+\|_{L^{p+1}(\R^n)}^{p+1}-C\epsilon +o_k(1).$$
Furthermore, by~\eqref{SobW2},
$$ \mathcal{I}_\varepsilon(W_{2,k})-\frac{1}{2}\langle \mathcal{I}'_\varepsilon(W_{2,k}),
U_\varepsilon+W_{2,k}\rangle\geq\frac{s}{n}S^{n/2s}-C\varepsilon^{\frac{1}{2\gamma}}+o_k(1).$$
This and \eqref{boundWi} give the desired estimate for $\mathcal{I}_\epsilon(W_{2,k})$, namely
\begin{equation}\label{infBoundW2}
\mathcal{I}_\varepsilon(W_{2,k})\geq\frac{s}{n}S^{n/2s}-C\varepsilon^{\frac{1}{2\gamma}}+o_k(1).
\end{equation}

\smallskip

\noindent {\em Step 3: Lower bound for $\mathcal{I}_\epsilon(U_k)$.} 
Now, keeping in mind the estimates obtained in~\eqref{infBoundW1} 
and~\eqref{infBoundW2} for~$\mathcal{I}_\epsilon(W_{1,k})$ 
and~$\mathcal{I}_\epsilon(W_{2,k})$ respectively, 
we will produce an estimate for~$\mathcal{I}_\epsilon(U_k)$. 
Indeed, notice first that~$U_k=\chi U_k +(1-\chi)U_k=W_{1,k}+W_{2,k}$, 
thanks to~\eqref{def Wk}. Hence, recalling~\eqref{def I}, we have 
\begin{eqnarray*}
&&\mathcal{I}_\varepsilon(U_k)\\ &=&\mathcal{I}_\varepsilon(W_{1,k})+\mathcal{I}_\varepsilon(W_{2,k})+
\int_{\mathbb{R}^{n+1}_+}y^a\langle\nabla W_{1,k},\nabla W_{2,k}\rangle\,dX\\
&&\quad -\frac{\epsilon}{q+1}\int_{\R^n} h(x)\left((U_\epsilon+(U_k)_+)^{q+1}(x,0)-U_\epsilon^{q+1}(x,0)\right)\,dx \\&&\quad+\epsilon\int_{\R^n}h(x)U_\epsilon^q(x,0)(U_k)_+(x,0)\,dx\\
&&\quad -\frac{1}{p+1}\int_{\R^n} \left((U_\epsilon+(U_k)_+)^{p+1}(x,0)-U_\epsilon^{p+1}(x,0)\right)\,dx \\&&\quad+\int_{\R^n}U_\epsilon^p(x,0)(U_k)_+(x,0)\,dx\\
&&\quad +\frac{\epsilon}{q+1}\int_{\R^n} h(x)\left((U_\epsilon+(W_{1,k})_+)^{q+1}(x,0)-U_\epsilon^{q+1}(x,0)\right)\,dx \\&&\quad-\epsilon\int_{\R^n}h(x)U_\epsilon^q(x,0)(W_{1,k})_+(x,0)\,dx\\
&&\quad +\frac{1}{p+1}\int_{\R^n} \left((U_\epsilon+(W_{1,k})_+)^{p+1}(x,0)-U_\epsilon^{p+1}(x,0)\right)\,dx \\&&\quad-\int_{\R^n}U_\epsilon^p(x,0)(W_{1,k})_+(x,0)\,dx\\
&&\quad +\frac{\epsilon}{q+1}\int_{\R^n} h(x)\left((U_\epsilon+(W_{2,k})_+)^{q+1}(x,0)-U_\epsilon^{q+1}(x,0)\right)\,dx\\&&\quad -\epsilon\int_{\R^n}h(x)U_\epsilon^q(x,0)(W_{2,k})_+(x,0)\,dx\\
&&\quad +\frac{1}{p+1}\int_{\R^n} \left((U_\epsilon+(W_{2,k})_+)^{p+1}(x,0)-U_\epsilon^{p+1}(x,0)\right)\,dx \\&&\quad-\int_{\R^n}U_\epsilon^p(x,0)(W_{2,k})_+(x,0)\,dx.
\end{eqnarray*} 
Thanks to Lemma~\ref{lemma:conv} we have that 
\begin{eqnarray*}
&&\lim_{k\to+\infty} \int_{\R^n}h(x)U_\epsilon^q(x,0)(U_k)_+(x,0)\,dx=0,\\
&&
\lim_{k\to+\infty}\int_{\R^n}U_\epsilon^p(x,0)(U_k)_+(x,0)\,dx,\\
&&\lim_{k\to+\infty} \int_{\R^n}h(x)U_\epsilon^q(x,0)(W_{1,k})_+(x,0)\,dx=0, 
\\&& \lim_{k\to+\infty}\int_{\R^n}U_\epsilon^p(x,0)(W_{1,k})_+(x,0)\,dx=0,\\
&& \lim_{k\to+\infty}\int_{\R^n}h(x)U_\epsilon^q(x,0)(W_{2,k})_+(x,0)\,dx
\\ {\mbox{ and }} && \lim_{k\to+\infty}\int_{\R^n}U_\epsilon^p(x,0)(W_{2,k})_+(x,0)\,dx=0.
\end{eqnarray*}
Therefore, 
\begin{eqnarray*}
\mathcal{I}_\varepsilon(U_k)&=&\mathcal{I}_\varepsilon(W_{1,k})+\mathcal{I}_\varepsilon(W_{2,k})+
\int_{\mathbb{R}^{n+1}_+}y^a\langle\nabla W_{1,k},\nabla W_{2,k}\rangle\,dX\\
&&\quad -\frac{\epsilon}{q+1}\int_{\R^n} h(x)\left((U_\epsilon+(U_k)_+)^{q+1}(x,0)-U_\epsilon^{q+1}(x,0)\right)\,dx \\
&&\quad -\frac{1}{p+1}\int_{\R^n} \left((U_\epsilon+(U_k)_+)^{p+1}(x,0)-U_\epsilon^{p+1}(x,0)\right)\,dx \\
&&\quad +\frac{\epsilon}{q+1}\int_{\R^n} h(x)\left((U_\epsilon+(W_{1,k})_+)^{q+1}(x,0)-U_\epsilon^{q+1}(x,0)\right)\,dx \\
&&\quad +\frac{1}{p+1}\int_{\R^n} \left((U_\epsilon+(W_{1,k})_+)^{p+1}(x,0)-U_\epsilon^{p+1}(x,0)\right)\,dx \\
&&\quad +\frac{\epsilon}{q+1}\int_{\R^n} h(x)\left((U_\epsilon+(W_{2,k})_+)^{q+1}(x,0)-U_\epsilon^{q+1}(x,0)\right)\,dx \\
&&\quad +\frac{1}{p+1}\int_{\R^n} \left((U_\epsilon+(W_{2,k})_+)^{p+1}(x,0)-U_\epsilon^{p+1}(x,0)\right)\,dx +o_k(1).
\end{eqnarray*} 
Since the terms with~$\epsilon$ in front are bounded 
(see~\eqref{usare-1} and notice that it holds also for~$U_k$ and~$W_{1,k}$), 
we have that 
\begin{equation}\begin{split}\label{ok-1}
\mathcal{I}_\varepsilon(U_k)\geq\,&\mathcal{I}_\varepsilon(W_{1,k})+\mathcal{I}_\varepsilon(W_{2,k})+
\int_{\mathbb{R}^{n+1}_+}y^a\langle\nabla W_{1,k},\nabla W_{2,k}\rangle\,dX\\
&\quad -\frac{1}{p+1}\int_{\R^n} \left((U_\epsilon+(U_k)_+)^{p+1}(x,0)-U_\epsilon^{p+1}(x,0)\right)\,dx \\
&\quad +\frac{1}{p+1}\int_{\R^n} \left((U_\epsilon+(W_{1,k})_+)^{p+1}(x,0)-U_\epsilon^{p+1}(x,0)\right)\,dx \\
&\quad +\frac{1}{p+1}\int_{\R^n} \left((U_\epsilon+(W_{2,k})_+)^{p+1}(x,0)-U_\epsilon^{p+1}(x,0)\right)\,dx \\&\quad-C\epsilon+ o_k(1).
\end{split}\end{equation} 
Now notice that 
\begin{equation}\begin{split}\label{ok1}
& \int_{\mathbb{R}^{n+1}_+}y^a\langle\nabla W_{1,k},\nabla W_{2,k}\rangle\,dX\\
&\qquad = \frac12\int_{\mathbb{R}^{n+1}_+}y^a\langle\nabla(U_k- W_{1,k}),\nabla W_{1,k}\rangle\,dX 
\\&\qquad\qquad+\frac12\int_{\mathbb{R}^{n+1}_+}y^a\langle\nabla(U_k- W_{2,k}),\nabla W_{2,k}\rangle\,dX.
\end{split}\end{equation}
Moreover, from~\eqref{DER} we have that for any $i\in\{1,2\}$
\begin{equation}\begin{split}\label{ok2}
&\langle \mathcal{I}'_\varepsilon(U_k)-\mathcal{I}_\epsilon'(W_{i,k}),
U_\varepsilon+W_{i,k}\rangle\\
=\,& \int_{\mathbb{R}^{n+1}_+}y^a\langle\nabla(U_k- W_{i,k}),\nabla (W_{i,k}+U_\epsilon)\rangle\,dX \\
&\qquad -\epsilon\int_{\R^n} h(x)\left((U_\epsilon+(U_k)_+)^q-
U_\epsilon^q\right)(U_\epsilon+W_{i,k})\,dx 
\\&\qquad-\int_{\R^n}\left((U_\epsilon+(U_k)_+)^p-
U_\epsilon^p\right)(U_\epsilon+W_{i,k})\,dx \\
&\qquad +\epsilon\int_{\R^n} h(x)\left((U_\epsilon+(W_{i,k})_+)^q-
U_\epsilon^q\right)(U_\epsilon+W_{i,k})\,dx \\&\qquad
+\int_{\R^n}\left((U_\epsilon+(W_{i,k})_+)^p-
U_\epsilon^p\right)(U_\epsilon+W_{i,k})\,dx.
\end{split}\end{equation}
We claim that 
\begin{equation}\label{speriamo}\begin{split}
&\int_{\R^n} h(x)\left((U_\epsilon+(U_k)_+)^q(x,0)-
U_\epsilon^q(x,0)\right)(U_\epsilon+W_{i,k})(x,0)\,dx\le C+o_k(1)\\
{\mbox{and }} &\int_{\R^n} h(x)\left((U_\epsilon+(W_{i,k})_+)^q(x,0)-
U_\epsilon^q(x,0)\right)(U_\epsilon+W_{i,k})(x,0)\,dx \le C+o_k(1),
\end{split}\end{equation}
for some~$C>0$. 
Let us prove the first estimate in~\eqref{speriamo}. 
For this, we notice that if~$x\in\R^n$ is such that~$U_k(x,0)\ge0$
then also~$W_{i,k}(x,0)\ge0$, thanks to the definition of~$W_{i,k}$ 
given in~\eqref{def Wk}. Hence 
\begin{eqnarray*}
&&\int_{\R^n} h(x)\left((U_\epsilon+(U_k)_+)^q(x,0)-
U_\epsilon^q(x,0)\right)(U_\epsilon+W_{i,k})(x,0)\,dx\\
&&\qquad =\int_{\R^n} h(x)\left((U_\epsilon+(U_k)_+)^q(x,0)-
U_\epsilon^q(x,0)\right)(U_\epsilon+(W_{i,k})_+)(x,0)\,dx\\
&&\qquad \le \int_{\R^n} h(x)\left((U_\epsilon+(U_k)_+)^q(x,0)-
U_\epsilon^q(x,0)\right)(U_\epsilon+(U_{k})_+)(x,0)\,dx\\
&&\qquad =\int_{\R^n} h(x)\left((U_\epsilon+(U_k)_+)^{q+1}(x,0)-
U_\epsilon^{q+1}(x,0)\right)\,dx\\&&\qquad\qquad-\int_{\R^n}h(x)U_\epsilon^{q}(x,0)
(U_{k})_+)(x,0)\,dx\\
&&\qquad \le C+o_k(1),
\end{eqnarray*}
for a suitable~$C>0$, thanks to~\eqref{usare-1} 
(that holds true also for~$U_k$) and Lemma~\ref{lemma:conv}. 
Analogously one can prove also the second estimate in~\eqref{speriamo}, 
and this finishes the proof of~\eqref{speriamo}.

Hence, from~\eqref{ok-1}, \eqref{ok1}, \eqref{ok2} and \eqref{speriamo} we get 
\begin{eqnarray*}
&&\mathcal{I}_\varepsilon(U_k)\\ &\ge &\mathcal{I}_\varepsilon(W_{1,k})+\mathcal{I}_\varepsilon(W_{2,k})+
\frac12 \langle \mathcal{I}_\epsilon'(U_k)-\mathcal{I}_\epsilon'(W_{1,k}), 
U_\epsilon+W_{1,k}\rangle \\&&\quad+ 
\frac12 \langle \mathcal{I}_\epsilon'(U_k)-\mathcal{I}_\epsilon'(W_{2,k}), 
U_\epsilon+W_{2,k}\rangle\\
&&\quad -\frac{1}{p+1}\int_{\R^n} \left((U_\epsilon+(U_k)_+)^{p+1}(x,0)-U_\epsilon^{p+1}(x,0)\right)\,dx \\
&&\quad +\frac{1}{p+1}\int_{\R^n} \left((U_\epsilon+(W_{1,k})_+)^{p+1}(x,0)-U_\epsilon^{p+1}(x,0)\right)\,dx \\
&&\quad +\frac{1}{p+1}\int_{\R^n} \left((U_\epsilon+(W_{2,k})_+)^{p+1}(x,0)-U_\epsilon^{p+1}(x,0)\right)\,dx \\
&&\quad +\frac{1}{2}\int_{\R^n}\left( (U_\epsilon+(U_k)_+)^p(x,0)
-U_\epsilon^p(x,0)\right)(U_\epsilon+W_{1,k})(x,0)\,dx \\
&&\quad -\frac{1}{2}\int_{\R^n}\left( (U_\epsilon+(W_{1,k})_+)^p(x,0)
-U_\epsilon^p(x,0)\right)(U_\epsilon+W_{1,k})(x,0)\,dx \\
&&\quad +\frac{1}{2}\int_{\R^n}\left( (U_\epsilon+(U_k)_+)^p(x,0)
-U_\epsilon^p(x,0)\right)(U_\epsilon+W_{2,k})(x,0)\,dx \\
&&\quad -\frac{1}{2}\int_{\R^n}\left( (U_\epsilon+(W_{2,k})_+)^p(x,0)
-U_\epsilon^p(x,0)\right)(U_\epsilon+W_{1,k})(x,0)\,dx \\&&\quad-C\epsilon +o_k(1) .
\end{eqnarray*}
Moreover, the estimates in~\eqref{boundW1} and~\eqref{boundW2} give 
\begin{eqnarray*}
\mathcal{I}_\varepsilon(U_k)&\ge &
\mathcal{I}_\varepsilon(W_{1,k})+\mathcal{I}_\varepsilon(W_{2,k})\\
&&\quad -\frac{1}{p+1}\int_{\R^n} \left((U_\epsilon+(U_k)_+)^{p+1}(x,0)-U_\epsilon^{p+1}(x,0)\right)\,dx \\
&&\quad +\frac{1}{p+1}\int_{\R^n} \left((U_\epsilon+(W_{1,k})_+)^{p+1}(x,0)-U_\epsilon^{p+1}(x,0)\right)\,dx \\
&&\quad +\frac{1}{p+1}\int_{\R^n} \left((U_\epsilon+(W_{2,k})_+)^{p+1}(x,0)-U_\epsilon^{p+1}(x,0)\right)\,dx \\
&&\quad +\frac{1}{2}\int_{\R^n}\left( (U_\epsilon+(U_k)_+)^p(x,0)
-U_\epsilon^p(x,0)\right)(U_\epsilon+W_{1,k})(x,0)\,dx \\
&&\quad -\frac{1}{2}\int_{\R^n}\left( (U_\epsilon+(W_{1,k})_+)^p(x,0)
-U_\epsilon^p(x,0)\right)(U_\epsilon+W_{1,k})(x,0)\,dx \\
&&\quad +\frac{1}{2}\int_{\R^n}\left( (U_\epsilon+(U_k)_+)^p(x,0)
-U_\epsilon^p(x,0)\right)(U_\epsilon+W_{2,k})(x,0)\,dx \\
&&\quad -\frac{1}{2}\int_{\R^n}\left( (U_\epsilon+(W_{2,k})_+)^p(x,0)
-U_\epsilon^p(x,0)\right)(U_\epsilon+W_{1,k})(x,0)\,dx \\&&\quad-C\epsilon^{\frac{1}{2\gamma}} +o_k(1).
\end{eqnarray*}
Now we use Lemma~\ref{lemma:conv} once more to see that, 
for~$i\in\{1,2\}$, 
\begin{eqnarray*}
&& \int_{\R^n}\left( (U_\epsilon+(W_{i,k})_+)^p(x,0)
-U_\epsilon^p(x,0)\right) (U_\epsilon+W_{i,k})(x,0)\,dx\\
&&\qquad 
= \int_{\R^n}  ( U_\epsilon+(W_{i,k})_+ )^{p+1}(x,0)\,dx
-\int_{\R^n} U_\epsilon^{p+1}(x,0)\,dx +o_k(1).
\end{eqnarray*}
Hence, using this and collecting some terms, we have 
\begin{equation}\begin{split}\label{ok forse}
&\mathcal{I}_\varepsilon(U_k)\\
\ge \,&
\mathcal{I}_\varepsilon(W_{1,k})+\mathcal{I}_\varepsilon(W_{2,k})\\
&\quad -\frac{1}{p+1}\int_{\R^n} \left((U_\epsilon+(U_k)_+)^{p+1}(x,0)-U_\epsilon^{p+1}(x,0)\right)\,dx \\
&\quad +\frac{1}{2}\int_{\R^n}\left( (U_\epsilon+(U_{k})_+)^p(x,0)
-U_\epsilon^p(x,0)\right)(U_\epsilon+W_{1,k})(x,0)\,dx \\
&\quad +\frac{1}{2}\int_{\R^n}\left( (U_\epsilon+(U_{k})_+)^p(x,0)
-U_\epsilon^p(x,0)\right)(U_\epsilon+W_{2,k})(x,0)\,dx \\
&\quad -\left(\frac12-\frac{1}{p+1}\right)\int_{\R^n} \left((U_\epsilon+(W_{1,k})_+)^{p+1}(x,0)-U_\epsilon^{p+1}(x,0)\right)\,dx \\
&\quad -\left(\frac12-\frac{1}{p+1}\right)
\int_{\R^n} \left((U_\epsilon+(W_{2,k})_+)^{p+1}(x,0)-U_\epsilon^{p+1}(x,0)\right)\,dx
\\&\quad-C\epsilon^{\frac{1}{2\gamma}} +o_k(1).
\end{split}\end{equation}
Now we claim that 
\begin{equation}\label{speriamo-1}
\lim_{k\to+\infty}\int_{\R^n}\left( (U_\epsilon+(U_{k})_+)^p(x,0)
-U_\epsilon^p(x,0)\right)U_\epsilon(x,0)\,dx=0.
\end{equation}
Indeed, we first observe that for any~$a\ge b\ge0$
$$ a^p-b^p =p\int_b^at^{p-1}\,dt\le pa^{p-1}(a-b).$$ 
Hence, taking~$a:=U_\epsilon+(U_{k})_+$ and~$b:=U_\epsilon$, we have that
$$ |(U_\epsilon+(U_k)_+)^p-U_\epsilon^p|\le p(U_\epsilon+(U_k)_+)^{p-1}(U_k)_+.$$
Accordingly, 
\begin{eqnarray*}
&&\int_{\R^n}\left( (U_\epsilon+(U_{k})_+)^p(x,0)
-U_\epsilon^p(x,0)\right)U_\epsilon(x,0)\,dx\\
&&\qquad \le p \int_{\R^n}(U_\epsilon+(U_{k})_+)^{p-1}(x,0)
(U_k)_+(x,0)U_\epsilon(x,0)\,dx.\\
\end{eqnarray*}
We now use H\"older inequality with exponents~$\frac{2^*_s}{p-1}=\frac{n}{2s}$ 
and~$\frac{n}{n-2s}$ and obtain 
\begin{eqnarray*}
&&\int_{\R^n}\left( (U_\epsilon+(U_{k})_+)^p(x,0)
-U_\epsilon^p(x,0)\right)U_\epsilon(x,0)\,dx\\
&&\qquad \le p \left(\int_{\R^n}(U_\epsilon+(U_{k})_+)^{2^*_s}(x,0)\,dx
\right)^{\frac{2s}{n}} \left(\int_{\R^n}
(U_k)_+^{\frac{n}{n-2s}}(x,0)U_\epsilon^{\frac{n}{n-2s}}(x,0)\,dx
\right)^{\frac{n-2s}{n}}\\
&&\qquad \le C [U_\epsilon+(U_{k})_+]_a^{p-1} \left(\int_{\R^n}
(U_k)_+^{\frac{n}{n-2s}}(x,0)U_\epsilon^{\frac{n}{n-2s}}(x,0)\,dx
\right)^{\frac{n-2s}{n}}\\
&&\qquad \le C\left(\int_{\R^n}
(U_k)_+^{\frac{n}{n-2s}}(x,0)U_\epsilon^{\frac{n}{n-2s}}(x,0)\,dx
\right)^{\frac{n-2s}{n}},
\end{eqnarray*}
for some positive~$C$ that may change from line to line, thanks to Proposition~\ref{traceIneq} and~\eqref{uniformBoundUk}. 
Now the desired claim in~\eqref{speriamo-1} simply follows 
by using Lemma~\ref{lemma:conv} with~$V_k:=U_k$, $U_o:=U_\epsilon$, 
$\psi:=1$, $\alpha:=\frac{n}{n-2s}$ and~$\beta:=\frac{n}{n-2s}$ 
(notice that~$\alpha+\beta=2^*_s$). 

From~\eqref{speriamo-1} we deduce that 
\begin{eqnarray*}
&&\int_{\R^n}\left( (U_\epsilon+(U_{k})_+)^p(x,0)
-U_\epsilon^p(x,0)\right)(U_\epsilon+W_{1,k})(x,0)\,dx 
\\&&\qquad  +\int_{\R^n}\left( (U_\epsilon+(U_{k})_+)^p(x,0)
-U_\epsilon^p(x,0)\right)(U_\epsilon+W_{2,k})(x,0)\,dx\\
&&\qquad = \int_{\R^n}\left( (U_\epsilon+(U_{k})_+)^p(x,0)
-U_\epsilon^p(x,0)\right)(U_\epsilon+W_{1,k}+W_{2,k})(x,0)\,dx+o_k(1)\\
&&\qquad = \int_{\R^n}\left( (U_\epsilon+(U_{k})_+)^p(x,0)
-U_\epsilon^p(x,0)\right)(U_\epsilon + U_k)(x,0)\,dx+o_k(1)\\
&&\qquad = \int_{\R^n}\left( (U_\epsilon+(U_{k})_+)^{p+1}(x,0)
-U_\epsilon^{p+1}(x,0)\right)(x,0)\,dx+o_k(1),
\end{eqnarray*}
where Lemma~\ref{lemma:conv} was used once again in the last line. 
Plugging this information into~\eqref{ok forse} we obtain 
\begin{eqnarray*}
&&\mathcal{I}_\varepsilon(U_k)\\&\ge &
\mathcal{I}_\varepsilon(W_{1,k})+\mathcal{I}_\varepsilon(W_{2,k})\\
&&\quad +\left(\frac12-\frac{1}{p+1}\right)
\int_{\R^n} \left((U_\epsilon+(U_k)_+)^{p+1}(x,0)-U_\epsilon^{p+1}(x,0)\right)\,dx \\
&&\quad -\left(\frac12-\frac{1}{p+1}\right)\int_{\R^n} \left((U_\epsilon+(W_{1,k})_+)^{p+1}(x,0)-U_\epsilon^{p+1}(x,0)\right)\,dx \\
&&\quad -\left(\frac12-\frac{1}{p+1}\right)\int_{\R^n} \left((U_\epsilon+(W_{2,k})_+)^{p+1}(x,0)-U_\epsilon^{p+1}(x,0)\right)\,dx-C\epsilon^{\frac{1}{2\gamma}} +o_k(1).
\end{eqnarray*}
Now we use~\eqref{second} with $U_o:=U_\epsilon$ and $V_k:=U_k$, $V_k:=W_{1,k}$ and $V_k:=W_{2,k}$ 
respectively, and so 
\begin{eqnarray*}
\mathcal{I}_\varepsilon(U_k)&\ge &
\mathcal{I}_\varepsilon(W_{1,k})+\mathcal{I}_\varepsilon(W_{2,k})
+\left(\frac12-\frac{1}{p+1}\right)
\int_{\R^n}(U_k)_+^{p+1}(x,0)\,dx\\
&&\qquad -\left(\frac12-\frac{1}{p+1}\right)\int_{\R^n} (W_{1,k})_+^{p+1}(x,0)\,dx  \\
&&\qquad -\left(\frac12-\frac{1}{p+1}\right)\int_{\R^n} (W_{2,k})_+^{p+1}(x,0)\,dx-C\epsilon^{\frac{1}{2\gamma}} +o_k(1).
\end{eqnarray*}
Notice now that for any~$x\in\R^n$
\begin{eqnarray*}
&&(U_k)_+^{p+1}(x,0) - (W_{1,k})_+^{p+1}(x,0)-(W_{2,k})_+^{p+1}(x,0)\\
&=& (U_k)_+^{p+1}(x,0) - \chi^{p+1}(x,0)(U_{k})_+^{p+1}(x,0)-(1-\chi(x,0))^{p+1}(U_{k})_+^{p+1}(x,0)\\
&=&(U_k)_+^{p+1}(x,0)\left(1-\chi^{p+1}(x,0)-(1-\chi)^{p+1}(x,0)\right)\ge 0. 
\end{eqnarray*}
This and the fact that~$p+1>2$ give 
$$
\mathcal{I}_\varepsilon(U_k)\ge 
\mathcal{I}_\varepsilon(W_{1,k})+\mathcal{I}_\varepsilon(W_{2,k})
-C\epsilon^{\frac{1}{2\gamma}} +o_k(1).
$$
Finally, this, together with~\eqref{infBoundW1} and~\eqref{infBoundW2}, 
implies that 
\begin{equation*}
\mathcal{I}_\varepsilon(U_k)\geq \frac{s}{n}S^{n/2s}-C\varepsilon^{\frac{1}{2\gamma}}+o_k(1),
\end{equation*}
up to renaming constants. Therefore, taking the limit as~$k\to+\infty$
we have 
$$ c_\epsilon =\lim_{k\to+\infty}\mathcal{I}_\epsilon(U_k)\ge 
 \frac{s}{n}S^{n/2s}-C\varepsilon^{\frac{1}{2\gamma}}. $$
This gives a contradiction with~\eqref{ceps2}
and finishes the proof of Lemma~\ref{lemma:tight}. 
\end{proof}

We are now in the position to show that the functional~$\mathcal{I}_\epsilon$ 
introduced in~\eqref{def I} satisfies a Palais-Smale condition. 

\begin{proof}[Proof of Proposition \ref{PScond2}]
Thanks to Lemma~\ref{lemma:weak} we know that the sequence~$U_k$ 
weakly converges to~0 in~$\dot{H}^s_a(\R^{n+1}_+)$ as~$k\to+\infty$. 

For any~$k\in\N$, we set~$V_k:=U_\epsilon +U_k$, where~$U_\epsilon$ 
is the local minimum of~$\mathcal{F}_\epsilon$ found Theorem~\ref{MINIMUM}.
Since~$U_\epsilon$ is a critical point of~$\mathcal{F}_\epsilon$, 
from~\eqref{pqoeopwoegi} we deduce that 
$$ \int_{\R^{n+1}_+}y^a\langle\nabla U_\epsilon,\nabla U_k\rangle\,dX =
\epsilon\int_{\R^n}h(x)U_\epsilon^q(x,0)U_k(x,0)\,dx +
\int_{\R^n}U_\epsilon^p(x,0)U_k(x,0)\,dx.$$
Therefore, recalling~\eqref{f ext} and~\eqref{def I} we have
\begin{equation}\begin{split}\label{aslpooooooo}
\mathcal{F}_\epsilon(V_k) =\,& \frac12\int_{\R^{n+1}_+}
y^a|\nabla(U_\epsilon+U_k)|^2\,dX \\
&\quad -\frac{\epsilon}{q+1}\int_{\R^n}h(x)(U_\epsilon+U_k)_+^{q+1}(x,0)\,dx
-\frac{1}{p+1}\int_{\R^n}(U_\epsilon+U_k)_+^{p+1}(x,0)\,dx\\
=\,& \mathcal{I}_\epsilon(U_k)+\mathcal{F}_\epsilon(U_\epsilon)
+\int_{\R^{n+1}_+}y^a\langle\nabla U_\epsilon,\nabla U_k\rangle\,dX\\
&\quad +\frac{\epsilon}{q+1}\int_{\R^n}h(x)\left(  (U_\epsilon+(U_k)_+)^{q+1}(x,0) -(U_\epsilon+U_k)_+^{q+1}(x,0)\right)\,dx\\
&\quad +\frac{1}{p+1}\int_{\R^n}\left(  (U_\epsilon+(U_k)_+)^{p+1}(x,0) -(U_\epsilon+U_k)_+^{p+1}(x,0)\right)\,dx\\
&\quad -\epsilon\int_{\R^n}h(x)U_\epsilon^q(x,0)(U_k)_+(x,0)\,dx 
-\int_{\R^n}U_\epsilon^p(x,0)(U_k)_+(x,0)\\
=\,& \mathcal{I}_\epsilon(U_k)+\mathcal{F}_\epsilon(U_\epsilon)\\
&\quad +\frac{\epsilon}{q+1}\int_{\R^n}h(x)\Big(  (U_\epsilon+(U_k)_+)^{q+1}(x,0) -(U_\epsilon+U_k)_+^{q+1}(x,0)\\
&\qquad \qquad\quad  +(q+1)U_\epsilon^q(x,0)(U_k-(U_k)_+)(x,0)\Big)\,dx\\
&\quad +\frac{1}{p+1}\int_{\R^n}\Big(  (U_\epsilon+(U_k)_+)^{p+1}(x,0) -(U_\epsilon+U_k)_+^{p+1}(x,0)\\
&\qquad \qquad\quad +(p+1)U_\epsilon^p(x,0)(U_k-(U_k)_+)(x,0)\Big)\,dx.
\end{split}\end{equation}
We now claim that 
\begin{equation}\label{aslgttrjuytk}
(U_\epsilon+(U_k)_+)^{r+1}(x,0) -(U_\epsilon+U_k)_+^{r+1}(x,0) +(r+1)U_\epsilon^r(x,0)(U_k-(U_k)_+)(x,0)\le 0,
\end{equation}
for any~$x\in\R^n$ and~$r\in\{p,q\}$. Indeed, the claim is trivially 
true if~$U_k(x,0)\ge0$. Hence we suppose that~$U_k(x,0)<0$, and so~\eqref{aslgttrjuytk} becomes
\begin{equation}\label{aslgttrjuytk-1}
U_\epsilon^{r+1}(x,0) -(U_\epsilon+U_k)_+^{r+1}(x,0) +(r+1)U_\epsilon^r(x,0)(U_k-(U_k)_+)(x,0)\le 0.\end{equation}
Given~$a>0$, the function~$f(t):=(a+t)^{r+1}_+$, for~$t\in\R$, is convex, 
and therefore it satisfies for any~$b<0$
$$ f(b)\ge f(0) +f'(0)b,$$
that is
$$ (a+b)^{r+1}_+\ge a^{r+1}+(r+1)a^r b.$$ 
Thus, taking~$a:=U_\epsilon(x,0)$ and~$b:=U_k(x,0)$ we have 
$$ (U_\epsilon+U_k)_+^{r+1}(x,0)\ge U_\epsilon^{r+1}(x,0)+(r+1)U_\epsilon^r(x,0) U_k(x,0),$$
which shows~\eqref{aslgttrjuytk-1}, and in turn~\eqref{aslgttrjuytk}. 

Accordingly, using~\eqref{aslgttrjuytk} into~\eqref{aslpooooooo} we get 
\begin{equation}\label{minore}
\mathcal{F}_\epsilon(V_k)\le \mathcal{I}_\epsilon(U_k)+\mathcal{F}_\epsilon(U_\epsilon).
\end{equation}
This and assumption~(i) in Proposition~\ref{PScond2} imply that 
\begin{equation}\label{questa}
|\mathcal{F}_\epsilon(V_k)|\le C,
\end{equation}
for a suitable~$C>0$ independent of~$k$. 

Now we recall that~$U_\epsilon$ is a critical point of~$\mathcal{F}_\epsilon$. 
Hence, from~\eqref{pqoeopwoegi} we deduce that for 
any~$\Psi\in\dot{H}^s_a(\R^{n+1}_+)$ with~$\psi:=\Psi(\cdot,0)$
\begin{equation}\label{dopo-1}
\int_{\R^{n+1}_+}y^a\langle\nabla U_\epsilon,\nabla\Psi\rangle\,dX 
= \epsilon\int_{\R^n}h(x)U_\epsilon^q(x,0)\psi(x)\,dx 
+\int_{\R^n}U_\epsilon^p(x,0)\psi(x)\,dx.
\end{equation}
Moreover, 
from~\eqref{pqoeopwoegi} and~\eqref{DER} we have that 
\begin{eqnarray*}
&&\langle\mathcal{F}_\epsilon'(V_k),\Psi\rangle \\
&=& \int_{\R^{n+1}_+}y^a\langle\nabla V_k,\nabla\Psi\rangle\,dX
-\epsilon\int_{\R^n}h(x)(V_k)_+^q(x,0)\psi(x)\,dx 
-\int_{\R^n} (V_k)_+^p(x,0)\psi(x)\,dx\\
&=&\langle\mathcal{I}_\epsilon'(U_k),\Psi\rangle +\int_{\R^{n+1}_+}y^a\langle\nabla U_\epsilon,\nabla\Psi\rangle\,dX \\
&&\quad +\epsilon\int_{\R^n}h(x)\left( (U_\epsilon+(U_k)_+)^q(x,0)
-U_\epsilon^q(x,0)\right)\psi(x)\,dx\\
&&\quad +\int_{\R^n}\left( (U_\epsilon+(U_k)_+)^p(x,0)
-U_\epsilon^p(x,0)\right)\psi(x)\,dx\\
&&\quad -\epsilon\int_{\R^n}h(x)(V_k)_+^q(x,0)\psi(x)\,dx 
-\int_{\R^n} (V_k)_+^p(x,0)\psi(x)\,dx.
\end{eqnarray*}
Using~\eqref{dopo-1} in the formula above 
and recalling that~$V_k=U_\epsilon+U_k$, we obtain  
\begin{equation}\begin{split}\label{lasgtprjutr}
&\langle\mathcal{F}_\epsilon'(V_k),\Psi\rangle\,=\,
\langle\mathcal{I}_\epsilon'(U_k),\Psi\rangle \\
&\qquad +\epsilon\int_{\R^n}h(x)\left( (U_\epsilon+(U_k)_+)^q(x,0)
-(U_\epsilon+U_k)_+^q(x,0)\right)\psi(x)\,dx\\
&\qquad +\int_{\R^n}\left( (U_\epsilon+(U_k)_+)^p(x,0)
-(U_\epsilon+U_k)_+^p(x,0)\right)\psi(x)\,dx.
\end{split}\end{equation}
We claim that 
\begin{equation}\begin{split}\label{viene?}
& \lim_{k\to+\infty}\int_{\R^n}h(x)\left( (U_\epsilon+(U_k)_+)^q(x,0)
-(U_\epsilon+U_k)_+^q(x,0)\right)\psi(x)\,dx=0\\
{\mbox{and }} & \lim_{k\to+\infty} \int_{\R^n}\left( (U_\epsilon+(U_k)_+)^p(x,0)
-(U_\epsilon+U_k)_+^p(x,0)\right)\psi(x)\,dx=0.
\end{split}\end{equation}
Notice that if~$U_k(x,0)\ge0$ then 
\begin{eqnarray*}
&& (U_\epsilon+(U_k)_+)^q(x,0) -(U_\epsilon+U_k)_+^q(x,0)\\&&\qquad= (U_\epsilon+U_k)^q(x,0)
-(U_\epsilon+U_k)^q(x,0)=0\\
{\mbox{and }} && (U_\epsilon+(U_k)_+)^p(x,0) -(U_\epsilon+U_k)_+^p(x,0)
\\&&\qquad= (U_\epsilon+U_k)^p(x,0)
-(U_\epsilon+U_k)^p(x,0)=0.
\end{eqnarray*}
Therefore the claim becomes
\begin{equation}\begin{split}\label{viene?-1}
& \lim_{k\to+\infty}\int_{\R^n\cap\{U_k(\cdot,0)<0\}}h(x)\left( U_\epsilon^q(x,0)
-(U_\epsilon+U_k)_+^q(x,0)\right)\psi(x)\,dx=0\\
{\mbox{and }} & \lim_{k\to+\infty} \int_{\R^n\cap\{U_k(\cdot,0)<0\}}\left( U_\epsilon^p(x,0)
-(U_\epsilon+U_k)_+^p(x,0)\right)\psi(x)\,dx=0.
\end{split}\end{equation}
Now, we recall that Lemma \ref{lemma:weak} here and the compact embedding in Theorem 7.1 in \cite{DPV} 
imply that $U_k(\cdot,0)\to0$ a.e. in $\R^n$ as $k\to+\infty$. 
Moreover, we notice that, by the H\"older inequality with exponents $\frac{2^*_s}{2^*_s-1-q}$, 
$\frac{2^*_s}{q}$ and $2^*_s$,  
\begin{eqnarray*}
&&\left| \int_{\R^n\cap\{U_k(\cdot,0)<0\}}h(x) U_\epsilon^q(x,0)\psi(x)\,dx\right| 
\\&&\qquad \le 
\left( \int_{\R^n\cap\{U_k(\cdot,0)<0\} } |h(x)|^{\frac{2^*_s}{2^*_s-1-q}}
\,dx\right)^{\frac{2^*_s-1-q}{2^*_s}}
\left( \int_{\R^n\cap\{U_k(\cdot,0)<0\} } U_\epsilon^{2^*_s}(x,0)\,dx\right)^{\frac{q}{2^*_s}}
\\&&\qquad\cdot
\left( \int_{\R^n\cap\{U_k(\cdot,0)<0\} }|\psi(x)|^{2^*_s} \,dx\right)^{\frac{1}{2^*_s}}\\
&&\qquad \le \|h\|_{L^{\frac{2^*_s}{2^*_s-1-q}}(\R^n)}S^{-q/2}[U_\epsilon]_a^q S^{-1/2}[\Psi]_a\le C,
\end{eqnarray*}
for some $C>0$, thanks to \eqref{h0} and Proposition \ref{traceIneq}. Consequently
$$ h\left( U_\epsilon^q(\cdot,0)
-(U_\epsilon+U_k)_+^q(\cdot,0)\right)\psi\le |h| U_\epsilon^q(\cdot,0)
|\psi|\in L^1(\R^n\cap\{U_k(\cdot,0)<0\}).$$
Hence, by the Dominated Convergence Theorem we get the first limit in \eqref{viene?-1}.

To prove the second limit in \eqref{viene?-1}, we use the H\"older inequality 
with exponents $\frac{2^*_s}{p}=\frac{2n}{n+2s}$ and $2^*_s=\frac{2n}{n-2s}$ 
to see that
\begin{eqnarray*}
&&\left| \int_{\R^n\cap\{U_k(\cdot,0)<0\}}U_\epsilon^p(x,0)\psi(x)\,dx\right| \\
&\le& 
\left( \int_{\R^n\cap\{U_k(\cdot,0)<0\} } U_\epsilon^{2^*_s}(x,0)\,dx\right)^{\frac{n+2s}{2n}}
\left( \int_{\R^n\cap\{U_k(\cdot,0)<0\} }|\psi(x)|^{2^*_s} \,dx\right)^{\frac{1}{2^*_s}}\\
&\le & S^{-p/2}[U_\epsilon]_a^p S^{-1/2}[\Psi]_a\le C,
\end{eqnarray*}
for a suitable $C>0$, where \eqref{TraceIneq} was also used. 
Therefore, 
$$ \left( U_\epsilon^p(\cdot,0)
-(U_\epsilon+U_k)_+^p(\cdot,0)\right)\psi\le U_\epsilon^p(\cdot,0)|\psi|
\in L^1(\R^n\cap\{U_k(\cdot,0)<0\}).$$
So the second limit in \eqref{viene?-1} follows from the Dominated Convergence Theorem.
This shows~\eqref{viene?-1} and so the proof of~\eqref{viene?} is finished. 

As a consequence of~\eqref{lasgtprjutr}, \eqref{viene?} and 
assumption~(ii) in Proposition~\ref{PScond2} we have 
that
\begin{equation}\label{dlghjtrjh}
{\mbox{$\mathcal{F}_\epsilon'(V_k)\to0$ as~$k\to+\infty$}}\end{equation} 
in the sense of Remark~\ref{rem:3.3}.
This, together with~\eqref{questa} and Lemma~\ref{lemma bound}, implies that
the sequence~$V_k$ is uniformly bounded in~$\dot{H}^s_a(\R^{n+1}_+)$, 
namely there exists a constant~$M>0$ such that
\begin{equation}\label{boundM}
{\mbox{$[V_k]_a\le M$ for all~$k\in\N$.}}\end{equation} 
Hence,~$V_k$ is weakly convergent in~$\dot{H}^s_a(\R^{n+1}_+)$ 
to some function~$V_0$. Since~$V_k=U_\epsilon+U_k$ and~$U_k$ 
weakly converges to~0 in~$\dot{H}^s_a(\R^{n+1}_+)$ as~$k\to+\infty$ 
(see Lemma~\ref{lemma:weak}), it turns out that~$V_0=U_\epsilon$. 
Also, we recall that~$U_\epsilon$ is positive, thanks to Proposition~\ref{prop:pos}.
Therefore, we are in the position to apply Lemma~\ref{POS} 
with~$W_m:=V_k$ and~$W:=U_\epsilon$, and we obtain that 
\begin{equation}\label{qwwrtruoiungf}
{\mbox{$(V_k)_+$ weakly converges to $U_\epsilon$ in~$\dot{H}^s_a(\R^{n+1}_+)$ as~$k\to+\infty$.}}
\end{equation}
We also show that
\begin{equation}\label{qwwrtruoiungf-1}
{\mbox{the sequence $\{V_k\}_k$ is tight, according to Definition~\ref{defTight}.}}
\end{equation}
For this, we fix~$\eta>0$. Thanks to Lemma~\ref{lemma:tight}, we have that 
there exists~$\rho_1>0$ such that 
$$ \int_{\R^{n+1}_+\setminus B_{\rho_1}^+}y^a|\nabla U_k|^2\,dX<\frac{\eta}{4},$$
for any~$k\in\N$. Moreover, since~$U_\epsilon\in\dot{H}^s_a(\R^{n+1}_+)$, 
there exists~$\rho_2>0$ such that 
$$ \int_{\R^{n+1}_+\setminus B_{\rho_2}^+}y^a|\nabla U_\epsilon|^2\,dX<\frac{\eta}{4}.$$
We take~$\rho:=\max\{\rho_1,\rho_2\}$, and so the two formulas above
give that 
\begin{eqnarray*}
&&\int_{\R^{n+1}_+\setminus B_\rho^+}y^a|\nabla V_k|^2\,dX \\
& = & \int_{\R^{n+1}_+\setminus B_\rho^+}y^a|\nabla U_k|^2\,dX 
+\int_{\R^{n+1}_+\setminus B_\rho^+}y^a|\nabla U_\epsilon|^2\,dX
\\&&\qquad+2\int_{\R^{n+1}_+\setminus B_\rho^+}y^a\langle\nabla U_k,\nabla U_\epsilon\rangle\,dX\\
&\le &\int_{\R^{n+1}_+\setminus B_\rho^+}y^a|\nabla U_k|^2\,dX 
+\int_{\R^{n+1}_+\setminus B_\rho^+}y^a|\nabla U_\epsilon|^2\,dX\\
&&\qquad 
+2\sqrt{\int_{\R^{n+1}_+\setminus B_\rho^+}y^a|\nabla U_k|^2\,dX}\cdot
\sqrt{\int_{\R^{n+1}_+\setminus B_\rho^+}y^a|\nabla U_\epsilon|^2\,dX}\\
&\le & \frac{\eta}{4}+\frac{\eta}{4}+2\frac{\sqrt{\eta}}{2}\frac{\sqrt{\eta}}{2} =\eta.
\end{eqnarray*}
This shows~\eqref{qwwrtruoiungf-1}. 

Also, Theorem~1.1.4 in~\cite{evans} gives the existence of two measures 
on~$\R^n$ and~$\R^{n+1}_+$, $\nu$ and~$\mu$ respectively, such 
that~$(V_k)_+^{2^*_s}(\cdot,0)$ converges to~$\nu$ 
and~$y^a|\nabla (V_k)_+|^2$ converges to~$\mu$ as~$k\to+\infty$, 
according to Definition~1.1.2 in~\cite{evans} (see also Definition \ref{convMeasures}). 
This, \eqref{qwwrtruoiungf} and~\eqref{qwwrtruoiungf-1} imply that 
the hypotheses of Proposition~\ref{CCP} are satisfied, 
and so there exist an at most countable set~$J$ and 
three families~$\{x_j\}_{j\in J}\in\R^n$, $\{\nu_j\}_{j\in J}$ 
and~$\{\mu_j\}_{j\in J}$, with~$\nu_j,\mu_j\ge0$ such that 
\begin{equation}\label{lklhjgghf}
{\mbox{$(V_k)_+^{2^*_s}$ converges to~$\nu=U_\epsilon^{2^*_s}+\sum_{j\in J}\nu_j\delta_{x_j}$ as~$k\to+\infty$,}}
\end{equation}
\begin{equation}\label{lklhjgghf-2}
{\mbox{$y^a|\nabla(V_k)_+|^2$ converges to~$\mu\ge y^a|\nabla U_\epsilon|^2+\sum_{j\in J}\mu_j\delta_{(x_j,0)}$ as~$k\to+\infty$}}
\end{equation}
and 
\begin{equation}\label{lklhjgghf-3}
\mu_j\ge S \nu_j^{2/2^*_s}\quad {\mbox{ for all }} j\in J.
\end{equation}

We claim now that $\nu_j=\mu_j=0$ for every $j\in J$. 
To prove this, we argue by contradiction and we suppose that there exists $j\in J$ 
such that $\mu_j=0$. We denote $X_j:=(x_j,0)$, we fix $\delta>0$ and we take a cut-off function 
$\phi_\delta\in C^\infty(\R^{n+1}_+,[0,1])$ such that 
$$
\phi_\delta(X)=\begin{cases}
1,\qquad {\mbox{ if }} X\in B^+_{\delta/2}(X_j),\\
0,\qquad {\mbox{ if }} X\in (B^+_{\delta}(X_j))^c,
\end{cases}
\quad {\mbox{ and }} \quad |\nabla \phi_\delta|\leq \frac{C}{\delta},$$
for some $C>0$. 

Now, it is not difficult to show that the sequence $\phi_\delta(V_k)_+$ is uniformly 
bounded in $\dot{H}^s_a(\R^{n+1}_+)$. Therefore, from \eqref{dlghjtrjh} and \eqref{pqoeopwoegi} 
we have that 
\begin{equation}\begin{split}\label{sixth}
0=\,& \lim_{k\to+\infty} \langle\mathcal{F}_\epsilon'(V_k),\phi_\delta(V_k)_+\rangle\\
=\,& \lim_{k\to+\infty}\Big( \int_{\R^{n+1}_+}y^a\langle\nabla V_k,\nabla(\phi_\delta (V_k)_+)\rangle\,dX\\
&\qquad \qquad -\epsilon\int_{\R^n}h(x)(V_k)_+^{q+1}(x,0)\phi_\delta(x,0)\,dx 
-\int_{\R^n}(V_k)_+^{p+1}(x,0)\phi_\delta(x,0)\,dx\Big)\\
=\,& \lim_{k\to+\infty}\Big( \int_{\R^{n+1}_+}y^a|\nabla (V_k)_+|^2\phi_\delta \,dX
+\int_{\R^{n+1}_+}y^a\langle\nabla (V_k)_+,\nabla\phi_\delta\rangle (V_k)_+\,dX\\
&\qquad \qquad -\epsilon\int_{\R^n}h(x)(V_k)_+^{q+1}(x,0)\phi_\delta(x,0)\,dx 
-\int_{\R^n}(V_k)_+^{p+1}(x,0)\phi_\delta(x,0)\,dx\Big).
\end{split}\end{equation}
We recall that $p+1=2^*_s$ and we use \eqref{lklhjgghf} and \eqref{lklhjgghf-2} to see that 
\begin{eqnarray}\label{qqqqqqqqqqqpppp} 
\lim_{k\to+\infty}\int_{\R^n}(V_k)_+^{p+1}(x,0)\phi_\delta(x,0)\,dx&=& \int_{\R^n}\phi_\delta(x,0)\,d\nu\\
{\mbox{and }}  \lim_{k\to+\infty}\int_{\R^{n+1}_+}y^a|\nabla (V_k)_+|^2\phi_\delta \,dX &=& 
\int_{\R^{n+1}_+}\phi_\delta\,d\mu\label{qqqqqqqqqqqpppp-1}.
\end{eqnarray}
Also, the weak convergence in \eqref{qwwrtruoiungf}, \eqref{equivNorms} 
and Theorem 7.1 in \cite{DPV} imply that 
$(V_k)_+(\cdot,0)$ strongly converges to $U_\epsilon(\cdot,0)$ in $L^r_{\rm{loc}}(\R^n)$ 
as $k\to+\infty$, for any $r\in[1,2^*_s)$. Accordingly, 
\begin{eqnarray*}
&& \left|\int_{\R^n}h(x)(V_k)_+^{q+1}(x,0)\phi_\delta(x,0)\,dx- 
\int_{\R^n}h(x)U_\epsilon^{q+1}(x,0)\phi_\delta(x,0)\,dx\right| \\
&&\qquad \le \|h\|_{L^\infty(\R^n)} 
\left| \int_{B_\delta^+(X_j)\cap\{y=0\} } 
\left((V_k)_+^{q+1}(x,0)-U_\epsilon^{q+1}(x,0)\right)\,dx\right|\to 0,
\end{eqnarray*}
as $k\to+\infty$, since $1<q+1<2^*_s$. This implies that 
$$ \lim_{k\to+\infty}\int_{\R^n}h(x)(V_k)_+^{q+1}(x,0)\phi_\delta(x,0)\,dx
= \int_{\R^n}h(x)U_\epsilon^{q+1}(x,0)\phi_\delta(x,0)\,dx.$$ 
Taking the limit as $\delta\to0$ we have 
\begin{equation}\label{forth}\begin{split}
&\lim_{\delta\to0} \lim_{k\to+\infty}\int_{\R^n}h(x)(V_k)_+^{q+1}(x,0)\phi_\delta(x,0)\,dx
\\&\qquad= \lim_{\delta\to0}\int_{B_\delta^+(X_j)\cap\{y=0\} }h(x)U_\epsilon^{q+1}(x,0)\phi_\delta(x,0)\,dx=0.
\end{split}\end{equation}
Finally, we claim that 
\begin{equation}\label{fifth}
\lim_{\delta\to0}\lim_{k\to+\infty} 
\int_{\R^{n+1}_+}y^a\langle\nabla (V_k)_+,\nabla\phi_\delta\rangle (V_k)_+\,dX=0.
\end{equation}
For this, we apply the H\"older inequality and we use \eqref{boundM} to obtain that 
\begin{equation}\begin{split}\label{slasgkrehrtohurt}
&\left|\int_{\R^{n+1}_+}y^a\langle\nabla (V_k)_+,\nabla\phi_\delta\rangle (V_k)_+\,dX \right|
\\&\qquad
=\left| \int_{B_\delta^+(X_j)}y^a\langle\nabla (V_k)_+,\nabla\phi_\delta\rangle (V_k)_+\,dX\right|\\
&\qquad 
\le \left( \int_{B_\delta^+(X_j)}y^a|\nabla (V_k)_+|^2\,dX\right)^{1/2}
\left( \int_{B_\delta^+(X_j)}y^a(V_k)_+^2|\nabla\phi_\delta|^2\,dX\right)^{1/2}\\&\qquad\le 
M\left( \int_{B_\delta^+(X_j)}y^a(V_k)_+^2|\nabla\phi_\delta|^2\,dX\right)^{1/2}.
\end{split}\end{equation}
Again by \eqref{boundM} and Lemma \ref{lemma:compact},
we deduce that 
\begin{eqnarray*}&& \left| \int_{B_\delta^+(X_j)}y^a(V_k)_+^2|\nabla\phi_\delta|^2\,dX 
- \int_{B_\delta^+(X_j)}y^aU_\epsilon^2|\nabla\phi_\delta|^2\,dX\right| \\
&&\quad\le \frac{C^2}{\delta^2} \left| \int_{B_\delta^+(X_j)}y^a(V_k)_+^2\,dX - 
\int_{B_\delta^+(X_j)}y^aU_\epsilon^2\,dX\right| \to 0,\end{eqnarray*}
as $k\to+\infty$. Hence 
\begin{equation}\label{sldujhktpriujr}
\lim_{k\to+\infty} \int_{B_\delta^+(X_j)}y^a(V_k)_+^2|\nabla\phi_\delta|^2\,dX= 
\int_{B_\delta^+(X_j)}y^aU_\epsilon^2|\nabla\phi_\delta|^2\,dX.
\end{equation}
Now by the
H\"older inequality with exponents $\gamma$ and $\frac{\gamma}{\gamma-1}$ we have that 
\begin{eqnarray*}
&& \int_{B_\delta^+(X_j)}y^aU_\epsilon^2|\nabla\phi_\delta|^2\,dX \\&\le & 
\left( \int_{B_\delta^+(X_j)}y^aU_\epsilon^{2\gamma}\,dX\right)^{\frac1{\gamma}} 
\left(\int_{B_\delta^+(X_j)}y^a|\nabla\phi_\delta|^{\frac{2\gamma}{\gamma-1}}\,dX
\right)^{\frac{\gamma-1}{\gamma}}\\
&\le & \frac{C^2}{\delta^2}
\left( \int_{B_\delta^+(X_j)}y^aU_\epsilon^{2\gamma}\,dX\right)^{\frac1{\gamma}}
\left(\int_{B_\delta^+(X_j)}y^a\,dX\right)^{\frac{\gamma-1}{\gamma}}\\
&\le & C\delta^{\frac{(n+a+1)(\gamma-1)}{\gamma}-2}
\left( \int_{B_\delta^+(X_j)}y^aU_\epsilon^{2\gamma}\,dX\right)^{\frac1{\gamma}},
\end{eqnarray*}
up to renaming constants. Since $\frac{(n+a+1)(\gamma-1)}{\gamma}-2=0$, this implies that 
$$ \int_{B_\delta^+(X_j)}y^aU_\epsilon^2|\nabla\phi_\delta|^2\,dX\le C 
\left( \int_{B_\delta^+(X_j)}y^aU_\epsilon^{2\gamma}\,dX\right)^{\frac1{\gamma}},$$
for a suitable positive constant $C$. Hence, 
$$ \lim_{\delta\to0} \int_{B_\delta^+(X_j)}y^aU_\epsilon^2|\nabla\phi_\delta|^2\,dX =0.$$
This,
together with \eqref{slasgkrehrtohurt} and \eqref{sldujhktpriujr}, proves \eqref{fifth}. 

From \eqref{sixth}, \eqref{qqqqqqqqqqqpppp}, \eqref{qqqqqqqqqqqpppp-1}, \eqref{forth} and \eqref{fifth} 
we obtain that 
\begin{eqnarray*}
&& 0= \lim_{\delta\to0}\lim_{k\to+\infty} \langle\mathcal{F}_\epsilon'(V_k),\phi_\delta(V_k)_+\rangle\\
\\&&\qquad= \lim_{\delta\to0}
\Big( \int_{\R^{n+1}_+}\phi_\delta \,d\mu-\int_{\R^n}\phi_\delta(x,0)\,d\nu\Big)
\ge  \mu_j-\nu_j.
\end{eqnarray*}
Therefore, this and \eqref{lklhjgghf-3} give that $\nu_j\ge \mu_j\ge S\nu_j^{2/2^*_s}$. 
Hence, either $\nu_j=\mu_j=0$ or $\nu_j^{1-2/2^*_s}\ge S$. 
Since we are in the case $\mu_j\neq0$, the first possibility cannot occur. As a consequence,
\begin{equation}\label{maggiore}
\nu_j\ge S^{n/2s}.
\end{equation}

Now, taking the limit as $k\to+\infty$ in \eqref{minore} and 
recalling assumption (i) of Proposition \ref{PScond2},  \eqref{dlghjtrjh} 
and \eqref{boundM}, we have that 
\begin{equation}\begin{split}\label{vabbe0}
&c_\epsilon +\mathcal{F}_\epsilon(U_\epsilon)\\
 \ge\, & \lim_{k\to+\infty}\left(\mathcal{F}_\epsilon(V_k)
-\frac12\langle \mathcal{F}_\epsilon'(V_k),V_k\rangle \right)\\
=\,& \lim_{k\to+\infty} \left[\left(\frac12-\frac{1}{p+1}\right)\int_{\R^n}(V_k)_+^{p+1}(x,0)\, dx 
-\epsilon\left(\frac{1}{q+1}-\frac12\right)\int_{\R^n}h(x)(V_k)_+^{q+1}(x,0)\,dx\right].
\end{split}\end{equation}
We claim that 
\begin{equation}\label{vabbe}
\lim_{k\to+\infty}\int_{\R^n}(V_k)_+^{p+1}(x,0)\, dx \ge S^{n/2s}+ 
\int_{\R^n}U_\epsilon^{p+1}(x,0)\, dx.
\end{equation}
For this, we take a sequence $\{\varphi_m\}_{m\in\N}\in C^\infty_0(\R^n,[0,1])$ 
such that $\displaystyle\lim_{m\to+\infty}\varphi_m(x)=1$ for any $x\in\R^n$. 
By \eqref{lklhjgghf} we have that 
$$ \lim_{k\to+\infty}\int_{\R^n}(V_k)_+^{p+1}(x,0)\, dx \ge 
\lim_{k\to+\infty}\int_{\R^n}(V_k)_+^{p+1}(x,0)\varphi_m(x)\, dx 
= \int_{\R^n}\varphi_m(x)\, d\nu.$$ 
Moreover, thanks to Fatou's lemma and \eqref{maggiore}, 
$$ \lim_{m\to+\infty}\int_{\R^n}\varphi_m(x)\,d\nu \ge \int_{\R^n}\,d\mu \ge S^{n/2s} + 
\int_{\R^n}U_\epsilon^{p+1}(x,0)\,dx.$$ 
The last two formulas imply that  
\begin{eqnarray*}
&& \lim_{k\to+\infty}\int_{\R^n}(V_k)_+^{p+1}(x,0)\, dx= 
\lim_{m\to+\infty}\lim_{k\to+\infty}\int_{\R^n}(V_k)_+^{p+1}(x,0)\, dx 
\\ &&\qquad\ge \lim_{m\to+\infty} \int_{\R^n}\varphi_m(x)\,d\nu 
\ge S^{n/2s} + 
\int_{\R^n}U_\epsilon^{p+1}(x,0)\,dx,
\end{eqnarray*}
which gives the desired result in \eqref{vabbe}. 
We now show that 
\begin{equation}\label{vabbe-1}
\lim_{k\to+\infty}\int_{\R^n}h(x)(V_k)_+^{q+1}(x,0)\,dx = 
\int_{\R^n}h(x)U_\epsilon^{q+1}(x,0)\,dx.
\end{equation}
Indeed, thanks to \eqref{boundM} we know that $[(V_k)_+]_a\le M$. 
Therefore, Proposition \ref{traceIneq} and Theorem 7.1 in \cite{DPV} imply that 
\begin{eqnarray*}
&& \|(V_k)_+(\cdot,0)-U_\epsilon(\cdot,0)\|_{L^{2^*_s}(\R^n)}\le 2M\\
{\mbox{and }} && (V_k)_+(\cdot,0)\to U_\epsilon(\cdot,0) \ {\mbox{ in }} L^{q+1}_{\rm{loc}}( \R^n )
\ {\mbox{ as }}k\to+\infty. 
\end{eqnarray*}
Thus, we fix $R>0$ and we use \eqref{h1}, \eqref{h0} and the
H\"older inequality to obtain that  
\begin{eqnarray*}
&&\left| \int_{\R^n}h(x)\left((V_k)_+(x,0)-U_\epsilon(x,0)\right)^{q+1}\,dx\right| \\
&&\qquad \le  
\int_{B_R}|h(x)|\,\left|(V_k)_+(x,0)-U_\epsilon(x,0)\right|^{q+1}\,dx \\&&\qquad\qquad+ 
\int_{\R^n\setminus B_R}|h(x)|\,\left|(V_k)_+(x,0)-U_\epsilon(x,0)\right|^{q+1}\,dx\\
&&\qquad \le  \|h\|_{L^\infty(\R^n)} \|(V_k)_+(\cdot,0)-U_\epsilon(\cdot,0)\|_{L^{q+1}(B_R)} \\&&\qquad\qquad+ 
\|h\|_{L^{\frac{2^*_s}{2^*_s-q-1}}(\R^n\setminus B_R)} \|(V_k)_+(\cdot,0)-U_\epsilon(\cdot,0)\|_{L^{2^*_s}(\R^n)}\\
&&\qquad \le C\|(V_k)_+(\cdot,0)-U_\epsilon(\cdot,0)\|_{L^{q+1}(B_R)}\,dx + 
(2M)^{q+1}\|h\|_{L^{\frac{2^*_s}{2^*_s-q-1}}(\R^n\setminus B_R)}. 
\end{eqnarray*}
Hence, letting first $k\to+\infty$ and then $R\to+\infty$, we obtain \eqref{vabbe-1}. 

Also, we observe that $\frac12-\frac{1}{p+1}=\frac{s}{n}$. 
Using this and plugging \eqref{vabbe} and \eqref{vabbe-1} into \eqref{vabbe0} we obtain that 
\begin{eqnarray*}
&&c_\epsilon + \mathcal{F}_\epsilon(U_\epsilon) \\&\ge & \frac{s}{n}S^{n/2s} + 
\left(\frac12-\frac{1}{p+1}\right)\int_{\R^n}U_\epsilon^{p+1}(x,0)\, dx 
\\&&\qquad-\epsilon\left(\frac{1}{q+1}-\frac12\right)\int_{\R^n} h(x)
U_\epsilon^{q+1}(x,0)\,dx\\
&=& \frac{s}{n}S^{n/2s}+\mathcal{F}_\epsilon(U_\epsilon).\end{eqnarray*}
Hence 
$$ c_\epsilon\ge \frac{s}{n}S^{n/2s}, $$
and this is a contradiction with \eqref{ceps2}.

As a consequence, necessarily $\mu_j=\nu_j=0$ for any $j\in J$. Hence, by \eqref{lklhjgghf}
\begin{equation}\label{chiama}
\lim_{k\to+\infty}\int_{\R^n}
(V_k)_+^{2^*_s}(x,0)\varphi(x)\,dx = 
\int_{\R^n} U_\epsilon^{2^*_s}(x,0)\varphi(x)\,dx,\end{equation}
for any $\varphi\in C_0(\R^n)$. 
Furthermore, by Lemma \ref{lemma:tight} and the fact that $U_\epsilon(\cdot,0)\in L^{2^*_s}(\R^n)$
(thanks to Proposition \ref{traceIneq}), we have that for any $\eta>0$ there exists $\rho>0$ 
such that 
$$ \int_{\R^n\setminus B_\rho} (V_k)_+^{2^*_s}(x,0)\,dx <\eta.$$ 
Thus we are in the position to apply Lemma \ref{PSL-2}
with $v_k:=(V_k)_+(\cdot,0)$ 
and $v:=U_\epsilon(\cdot,0)$, and we obtain that 
$(V_k)_+(\cdot,0)\to U_\epsilon(\cdot,0)$ in $L^{2^*_s}(\R^n)$ as $k\to+\infty$. 
Then, by Lemma \ref{PSL-1}
(again applied with $v_k:=(V_k)_+(\cdot,0)$ 
and $v:=U_\epsilon(\cdot,0)$) we have that 
\begin{eqnarray*}
&& \lim_{k\to+\infty}\int_{\R^n}|(V_k)_+^q(x,0)-U_\epsilon^q(x,0)|^{\frac{2^*_s}{q}}\,dx =0\\
{\mbox{and }} && \lim_{k\to+\infty}\int_{\R^n}|(V_k)_+^p(x,0)-U_\epsilon^p(x,0)|^{\frac{2n}{n+2s}}\,dx =0.
\end{eqnarray*}
Therefore, we can fix $\delta\in(0,1)$ (that we will take arbitrarily small in the sequel), 
and say that 
\begin{equation}\label{pegyhrehehrht}
\begin{split}
&\int_{\R^n}|(V_k)_+^q(x,0)-(V_m)_+^q(x,0)|^{\frac{2^*_s}{q}}\,dx \\&\qquad+ 
\int_{\R^n}|(V_k)_+^p(x,0)-(V_m)_+^p(x,0)|^{\frac{2n}{n+2s}}\,dx\le \delta
\end{split}\end{equation}
for $k$ and $m$ sufficiently large (say bigger that some $k_\star(\delta)$). 

We now take $\Psi\in\dot{H}^s_a(\R^{n+1}_+)$ with $\psi:=\Psi(\cdot,0)$ and such that 
\begin{equation}\label{hlkhfjfdsafgh}
[\Psi]_a=1. 
\end{equation}
From \eqref{dlghjtrjh} we have that, for large $k$ 
(say $k\ge k_\star(\delta)$, up to renaming $k_\star(\delta)$), we deduce that 
$$ \left|\langle \mathcal{F}_\epsilon'(V_k),\Psi\rangle\right|\le\delta. $$
As a consequence of this and \eqref{pqoeopwoegi}, 
\begin{eqnarray*}
&& \left| \int_{\R^{n+1}_+}y^a\langle\nabla V_k,\nabla \Psi\rangle\,dX \right.\\
&&\qquad\left.
-\epsilon \int_{\R^n}h(x)(V_k)_+^q(x,0)\psi(x)\,dx -\int_{\R^n}(V_k)_+^p(x,0)\psi(x)\,dx\right|\le 
\delta.
\end{eqnarray*}
In particular, for $k,m\ge k_\star(\delta)$, 
\begin{eqnarray*}
&&\Big| \int_{\R^{n+1}_+}y^a\langle\nabla(V_k-V_m),\nabla \Psi\rangle\,dX 
\\&&\qquad -\epsilon \int_{\R^n}h(x)\left((V_k)_+^q(x,0)-(V_m)_+^q(x,0)\right)\psi(x)\,dx 
\\&&\qquad -\int_{\R^n}\left((V_k)_+^p(x,0)-(V_m)_+^p(x,0)\right)\psi(x)\,dx\Big|\le 
2\delta.
\end{eqnarray*}
Now we use the H\"older inequality with exponents $\frac{2n}{n+2s-q(n-2s)}$, 
$\frac{2^*_s}{q}=\frac{2n}{q(n-2s)}$ and $2^*_s=\frac{2n}{n-2s}$, and with exponents 
$\frac{2^*_s}{p}=\frac{2n}{n+2s}$ and $2^*_s$, and we obtain that 
\begin{eqnarray*}
&&\left| \int_{\R^{n+1}_+}y^a\langle\nabla(V_k-V_m),\nabla \Psi\rangle\,dX \right| \\
&\le & \epsilon\left|\int_{\R^n}h(x)\left((V_k)_+^q(x,0)-(V_m)_+^q(x,0)\right)\psi(x)\,dx\right| 
\\&&\qquad
+ \left|\int_{\R^n}\left((V_k)_+^p(x,0)-(V_m)_+^p(x,0)\right)\psi(x)\,dx\right| +2\delta \\
&\le & \epsilon \left[\int_{\R^n}|h(x)|^{\frac{2n}{n+2s-q(n-2s)}}\,dx
\right]^{\frac{n+2s-q(n-2s)}{2n}}\\ &&\qquad\cdot
\left[ \int_{\R^n}\left|(V_k)_+^q(x,0)-(V_m)_+^q(x,0)\right|^{\frac{2^*_s}{q}}\,dx\right]^{\frac{q(n-2s)}{2n}}
\left[\int_{\R^n}|\psi(x)|^{2^*_s}\,dx\right]^{\frac{1}{2^*_s}} \\
&&\qquad + \left[ \int_{\R^n}\left|(V_k)_+^p(x,0)-(V_m)_+^p(x,0)\right|^{\frac{2n}{n+2s}}\,dx\right]^{\frac{n+2s}{2n}}
\left[ \int_{\R^n}|\psi(x)|^{2^*_s}\,dx\right]^{\frac{1}{2^*_s}} +2\delta.
\end{eqnarray*}
Hence, from \eqref{h0} and \eqref{pegyhrehehrht} we have that 
$$ \left| \int_{\R^{n+1}_+}y^a\langle\nabla(V_k-V_m),\nabla \Psi\rangle\,dX \right|\le 
C\delta^{\frac{q(n-2s)}{2n}}\|\psi\|_{L^{2^*_s}(\R^n)} +C\delta^{\frac{n+2s}{2n}}
\|\psi\|_{L^{2^*_s}(\R^n)} +2\delta,$$
for a suitable positive constant $C$. Now notice that \eqref{TraceIneq} and \eqref{hlkhfjfdsafgh} 
imply that $\|\psi\|_{L^{2^*_s}(\R^n)}\le S^{-1/2}[\Psi]_a=S^{-1/2}$, and so 
\begin{equation}\label{menomale}
\left| \int_{\R^{n+1}_+}y^a\langle\nabla(V_k-V_m),\nabla \Psi\rangle\,dX \right|\le 
C\delta^a,\end{equation}
for some $C,a>0$, as long as $k,m\ge k_\star(\delta)$. 
Also, 
$$ \nabla(V_k-V_m) = \nabla V_k -\nabla V_m = \nabla(U_k+U_\epsilon)-\nabla(U_m+U_\epsilon) 
=\nabla U_k-\nabla U_m.$$ 
Hence, plugging this into \eqref{menomale}, we have 
$$ \left| \int_{\R^{n+1}_+}y^a\langle\nabla(U_k-U_m),\nabla \Psi\rangle\,dX \right|\le 
C\delta^a.$$ 
Since this inequality is valid for any $\Psi$ satisfying \eqref{hlkhfjfdsafgh}, 
we deduce that 
$$ [U_k-U_m]_a\le C\delta^a, $$
namely $U_k$ is a Cauchy sequence in $\dot{H}^s_a(\R^{n+1}_+)$. 
Then, the desired result plainly follows. 
\end{proof}

\section{Bound on the minimax value}\label{sec:BMMV}

The goal of this section is to show that
the minimax value (computed along a suitable path)
lies below the critical threshold 
given by
Proposition~\ref{PScond2}.
The chosen path will be a suitably cut-off
rescaling of the fractional Sobolev minimizers
introduced in~\eqref{tale}.

To start with,
we set
\begin{equation}\label{67ydv8888kk}
\tilde G(x,U):=\int_0^{U} \tilde g(x,t)\,dt
\end{equation}
and
$$ \tilde g(x,t):= \begin{cases}
(U_\varepsilon+t)^q-U_\varepsilon^q,\;\hbox{ if }t\geq 0,
\\ 0\;\hbox{ if }t< 0.
\end{cases}$$
We observe that
\begin{equation}\label{G0poo}
\tilde G(x,0)=0.\end{equation}
Also, we see that~$\tilde g(x,t)\ge0$ for any~$t\in\R$,
and so
\begin{equation}\label{G0po}
{\mbox{$\tilde G(x,U)\ge0$ for any~$U\ge0$.}}\end{equation}
Moreover, recalling~\eqref{h2}, we write the ball~$B$
in~\eqref{h2} as~$B_{\mu_0}(\xi)$
for some~$\xi\in\R^n$ and~$\mu_0>0$.
We fix a cut-off function~$\bar\phi\in C^\infty_0(B_{\mu_0}(\xi),\,[0,1])$
with
\begin{equation}\label{PP0dsfhuw5r6t78y009}
{\mbox{$\bar\phi(x)=1$ for any~$x\in B_{\mu_0/2}(\xi)$.}}\end{equation}
The quantities~$\xi\in\R^n$ and~$\mu_0>0$, as well as
the cut-off~$\bar\phi$, are fixed from now on.
Also, if~$z$ is as in~\eqref{tale}, given~$\mu>0$, we let
\begin{equation}\label{tale2} z_{\mu,\xi}(x):=\mu^{-\frac{n-2s}{2}}
z\left( \frac{x-\xi}{\mu}\right).\end{equation}
Let also~$\bar Z_{\mu,\xi}$ be the extension of~$\bar\phi z_{\mu,\xi}$,
according to~\eqref{poisson}.

{F}rom \eqref{tale}, we know that
\begin{equation}\label{SO89sd}
S= \frac{[z]_{\dot H^s(\R^n)}^2}{\|z\|^2_{L^{2^*_s}(\R^n)}}\end{equation}
and~$(-\Delta)^s z=z^p$. Thus,
by testing this equation against~$z$ itself, we obtain that
$$ [z]_{\dot H^s(\R^n)}^2 = \|z\|^{2^*_s}_{L^{2^*_s}(\R^n)},$$
which, together with~\eqref{SO89sd}, gives that
$$ \|z\|_{L^{2^*_s}(\R^n)}=S^{\frac{n-2s}{4s}}$$
and so
$$ [z]_{\dot H^s(\R^n)}^2 = S^{\frac{n}{2s}}.$$ 
Moreover, by scaling, we have that
\begin{equation}\label{8dvsuyuscal}
\|z_{\mu,\xi}\|_{L^{2^*_s}(\R^n)}=\|z\|_{L^{2^*_s}(\R^n)}=
S^{\frac{n-2s}{4s}}\end{equation}
and
\begin{equation*}
[z_{\mu,\xi}]_{\dot H^s(\R^n)}^2 =
[z]_{\dot H^s(\R^n)}^2= S^{\frac{n}{2s}}.\end{equation*}
{F}rom the equivalence of norms in~\eqref{equivNorms}
and Proposition~21 in~\cite{SV}, we have that
\begin{equation}\label{8dvsuyuscal-2}
[\bar Z_{\mu,\xi}]_a^2
=[\bar\phi z_{\mu,\xi}]_{\dot H^s(\R^n)}^2\le S^{\frac{n}{2s}} + C\mu^{n-2s},\end{equation}
for some~$C>0$.

This setting is fixed from now on, together with
the minimum~$u_\epsilon(x)=U_\epsilon(x,0)$ given in Theorem~\ref{MINIMUM}.
Now we show that the effect of the cut-off on the Lebesgue norm
of the rescaled Sobolev minimizers is negligible when~$\mu$
is small. The quantitative statement is the following:

\begin{lemma}\label{LE:NU90}
We have that
$$ \int_{\R^n} |\bar\phi^{2^*_s}-1|\, z_{\mu,\xi}^{2^*_s} \,dx
\le C\mu^{n},$$
for some~$C>0$.
\end{lemma}

\begin{proof}
We observe that
\begin{equation*}
\begin{split}
& \int_{\R^n\setminus B_{\frac{\mu_0}2}(\xi)} z_{\mu,\xi}^{2^*_s} (x)\,dx
= \mu^{-n} 
\int_{\R^n\setminus B_{\frac{\mu_0}2}(\xi)} 
z^{2^*_s}\left( \frac{x-\xi}{\mu}\right)\,dx \\
&\qquad =
\int_{\R^n\setminus B_{\frac{\mu_0}{2\mu}} } 
z^{2^*_s}(y)\,dy
\le C\,\int_{\R^n\setminus B_{\frac{\mu_0}{2\mu}} } 
|y|^{-2n}\,dy \le C\mu^{n},
\end{split}
\end{equation*}
for some~$C>0$ (that may vary from line to line and may also
depend on~$\mu_0$). As a consequence, recalling~\eqref{PP0dsfhuw5r6t78y009},
we have that
\begin{equation*}
\int_{\R^n} |\bar\phi^{2^*_s}-1|\, z_{\mu,\xi}^{2^*_s} \,dx
=\int_{\R^n\setminus B_{\frac{\mu_0}2}(\xi)}
|\bar\phi^{2^*_s}-1|\, z_{\mu,\xi}^{2^*_s} \,dx 
\le C\mu^{n}. \qedhere
\end{equation*}
\end{proof}

The next result states that we can always ``concentrate the mass
near the positivity set of $h$'', in order
to detect a positive integral out of it.

\begin{lemma}\label{SP}
We have that
\begin{equation}\label{TO:P:SP}
\int_{\R^n} h(x)\,
\tilde G\big( x,t\bar\phi(x)\,z_{\mu,\xi}(x)\big)\,dx
\ge 0,\end{equation}
for any~$\mu>0$ and any~$t\ge0$.
\end{lemma}

\begin{proof} 
We have that $\bar\phi(x)=0$ if~$x\in \R^n\setminus
B_{\mu_0}(\xi)$. Thus, using~\eqref{G0poo},
we have that
$$ \tilde G\big(x,t\bar\phi(x)\,z_{\mu,\xi}(x) \big)=0$$
for any~$x\in  \R^n\setminus
B_{\mu_0}(\xi)$.
Therefore
$$ \int_{\R^n} h(x)\,
\tilde G\big( x,t\bar\phi(x)\,z_{\mu,\xi}(x)\big)\,dx
=\int_{B_{\mu_0}(\xi)} h(x)\,\tilde G
\big( x,t\bar\phi(x)\,z_{\mu,\xi}(x)\big)\,dx.$$
Then, the desired result follows from~\eqref{h2}
and~\eqref{G0po}.
\end{proof}

Now we check
that the geometry
of the mountain pass is satisfied by the functional~$\mathcal{I}_\varepsilon$.
Indeed, we first observe that Proposition~\ref{prop:zero}
gives that~$0$ is a local minimum for the functional~$
\mathcal{I}_\varepsilon$. The next result shows that
the path induced by the function~$\bar Z_{\mu,\xi}$
attains negative values, in a somehow uniform way
(the uniform estimates in $\mu$
in Lemma \ref{yu799}
will be needed in
the subsequent Corollary \ref{CC-below-PS} and,
from these facts, we will be able to deduce the
mountain pass geometry, check that the minimax
values stays below the critical threshold and complete the proof
of Theorem~\ref{TH:MP} in the forthcoming
Section~\ref{sec:proof}). To this goal,
it is useful
to introduce the auxiliary functional
\begin{equation}\label{def I2}
\mathcal{I}^\star_\varepsilon(U):=
\frac{1}{2}\int_{\mathbb{R}^{n+1}_+}{y^a|\nabla U|^2\,dX}-
\int_{\mathbb{R}^n}{G^\star(x,U(x,0))\,dx},\end{equation}
where
$$ G^\star(x,U):=\int_0^{U} g^\star(x,t)\,dt$$
and
$$ g^\star(x,t):= \begin{cases} 
(U_\varepsilon+t)^p-U_\varepsilon^p,\;\hbox{ if }t\geq 0,
\\ 0\;\hbox{ if }t< 0.
\end{cases}$$
By~\eqref{def I} and~\eqref{67ydv8888kk}, we see that~$G=G^\star+\epsilon h\tilde G$.
Thus,
as a consequence of Lemma~\ref{SP}, we have that
\begin{equation}\label{789hgbnjjjjjKK}
\mathcal{I}_\varepsilon(t\bar Z_{\mu,\xi})
=\mathcal{I}_\varepsilon^\star(t\bar Z_{\mu,\xi})
-\epsilon\int_{\R^n} h(x)\,\tilde G(x,t\bar Z_{\mu,\xi}(x))\,dx
\leq
\mathcal{I}_\varepsilon^\star(t\bar Z_{\mu,\xi}).\end{equation}
Then we have:

\begin{lemma}\label{yu799}
There exists~$\mu_1\in(0,\mu_0)$ such that
$$ \lim_{t\to+\infty} 
\sup_{\mu\in(0,\mu_1)}
\mathcal{I}_\varepsilon^\star(t\bar Z_{\mu,\xi}) =-\infty.$$
In particular, there exists~$T_1>0$ such that
\begin{equation}\label{789hgshh67g000io}
\sup_{\mu\in(0,\mu_1)}
\mathcal{I}_\varepsilon^\star(t\bar Z_{\mu,\xi})\le 0
\end{equation}
for any~$t\ge T_1$.
\end{lemma}

\begin{proof} We observe that, if~$U\ge0$,
\begin{equation}\label{9sg78ujohFgj}
G^\star(x,U)=\int_0^{U} (U_\varepsilon+t)^p-U_\varepsilon^p\,dt
=\frac{(U_\varepsilon+U)^{p+1} -U_\varepsilon^{p+1}}{p+1}
-U_\varepsilon^p\,U.
\end{equation}
Moreover,
$$ U_\epsilon^{p}(x,0)\,
\bar Z_{\mu,\xi}(x,0)\le u_\epsilon^{p+1}(x)+
z_{\mu,\xi}^{p+1}(x).$$
Using this and~\eqref{9sg78ujohFgj},
we obtain that
\begin{eqnarray*}
&&G^\star(x,t\bar Z_{\mu,\xi}(x,0))\\
&=& \frac{1}{p+1}\left((U_\epsilon(x,0) +t\bar Z_{\mu,\xi}(x,0))^{p+1}-
U_\epsilon^{p+1}(x,0)\right)
-tU_\epsilon^p (x,0)\,\bar Z_{\mu,\xi}(x,0)
\\ &\ge&
\frac{1}{p+1}\Big(
(t\bar \phi(x) z_{\mu,\xi}(x))^{p+1}-
u_\epsilon^{p+1}(x)\Big)
-t u_\epsilon^{p+1}(x) -tz^{p+1}_{\mu,\xi}(x)
.\end{eqnarray*}
Thus, integrating over~$\R^n$
and recalling~\eqref{h0}, \eqref{8dvsuyuscal} and the fact that~$2^*_s=p+1$, we get
\begin{equation}\label{78990iuii-0}
\begin{split}
& \int_{\R^n}
G^\star(x,t\bar Z_{\mu,\xi})\,dx\\ &\qquad\ge \frac{t^{p+1}}{p+1}\int_{\R^n}
\bar\phi^{p+1}(x) z_{\mu,\xi}^{p+1}(x)\,dx
-C-Ct,\end{split}\end{equation}
for some~$C>0$ (up to renaming constants).

Now we deduce from Lemma~\ref{LE:NU90}
that there exists~$\mu_1\in(0,1)$ such that
if~$\mu\in(0,\mu_1)$ then
\begin{equation*}
\int_{\R^n}
\bar\phi^{p+1}(x) z_{\mu,\xi}^{p+1}(x)\,dx \ge \frac{ S^{\frac{n}{2s}} }{2}.
\end{equation*}
Now, by inserting this
into~\eqref{78990iuii-0}, we obtain that, if~$\mu\in(0,\mu_1)$,
then
$$ \int_{\R^n}
G^\star(x,t\bar Z_{\mu,\xi})\,dx\ge \frac{ S^{\frac{n}{2s}}\,t^{p+1} }{2(p+1)}
-C-Ct.$$
This and~\eqref{def I2}
give that
$$ \mathcal{I}_\varepsilon^\star(t\bar Z_{\mu,\xi}) \le
\frac{t^2 [\bar Z_{\mu,\xi}]^2_a}{2}
+C(1+t)-\frac{ S^{\frac{n}{2s}}\,t^{p+1} }{2(p+1)}.$$
Hence, recalling~\eqref{8dvsuyuscal-2},
$$ \mathcal{I}_\varepsilon^\star(t\bar Z_{\mu,\xi}) \le
C(1+t+t^2)-\frac{ S^{\frac{n}{2s}}\,t^{p+1} }{2(p+1)},$$
up to renaming constants, for any~$\mu\in(0,\mu_1)$.
Since~$p+1>2$,
the desired claim easily follows.
\end{proof}

Now we introduce a series of purely elementary, but useful,
estimates.

\begin{lemma}\label{ABC-1}
For any~$a$, $b\ge0$ and any~$p>1$, we have that
\begin{equation}\label{9sghjtrf}
(a+b)^p\ge a^p +b^p.\end{equation}
Also, if~$a$, $b>0$, we have that
\begin{equation}\label{9sghjtrf-2}
(a+b)^p> a^p +b^p.\end{equation}
\end{lemma}

\begin{proof} If either~$a=0$ or~$b=0$, then~\eqref{9sghjtrf}
is obvious. So we can suppose that~$a\ne0$ and~$b\ne0$.
We let~$f(b):=(a+b)^p - a^p-b^p$.
Notice that~$f'(b)=p\big( (a+b)^{p-1} - b^{p-1}\big)>0$, since~$a>0$.
Hence
$$ (a+b)^p-a^p-b^p = f(b)> f(0)=0,$$
since~$b>0$, as desired.
\end{proof}

The result in Lemma~\ref{ABC-1} can be made more precise when~$p\ge2$,
as follows:

\begin{lemma}\label{ABC-2}
Let~$p\ge2$. Then, there exists~$c_p>0$ such that,
for any~$a$ and~$b\ge0$,
$$ (a+b)^p\ge a^p +b^p + c_p\,a^{p-1} b.$$
\end{lemma}

\begin{proof} If~$a=0$, then we are done, so we suppose~$a\ne0$
and we set~$t_o:=b/a$.
For any~$t>0$, we let
$$ f(t):=\frac{ (1+t)^p-1-t^p }{ t }.$$
{F}rom~\eqref{9sghjtrf-2} (used here with~$a:=1$ and~$b:=t$),
we know that~$f(t)>0$ for any~$t>0$. Moreover
$$ \lim_{t\to 0} f(t)=
\lim_{t\to 0} \frac{ 1 + pt +o(t)-1-t^p }{ t }=p,$$
hence $f$ can be continuously extended over~$[0,+\infty)$
by setting~$f(0):=p$. Furthermore,
\begin{eqnarray*}
&& \lim_{t\to +\infty} f(t)=
\lim_{t\to +\infty} t^{p-1} 
\left( \left(\frac{1}{t}+1\right)^p-\frac{1}{t^p}-1\right)
\\&&\qquad=\lim_{t\to +\infty}
t^{p-1}
\left( 1+\frac{p}{t}+o\left(\frac{1}{t}\right)-\frac{1}{t^p}-1\right) 
\\ &&\qquad=\lim_{t\to +\infty} t^{p-2}\left( p
+\frac{o\left(\frac{1}{t}\right)}{ \frac{1}{t} }
\right)-\frac{1}{t} 
= \left\{ \begin{matrix}
p & {\mbox{ if }} p=2,\\
+\infty & {\mbox{ if }} p>2.
\end{matrix}\right.
\end{eqnarray*}
In any case,
$$ \lim_{t\to +\infty} f(t) \ge f(0)=p,$$
hence
$$ c_p:=\inf_{[0,+\infty)} f =\min_{[0,+\infty)} f >0.$$
As a consequence,
\begin{eqnarray*}
&& (a+b)^p- a^p -b^p - c_p\,a^{p-1} b\\
&=& a^p \big( (1+t_o)^p -1- t_o^p -c_p\, t_o\big)\\
&=& a^p t_o \big( f(t_o)-c_p\big)\\
&\ge& 0,\end{eqnarray*}
as desired.
\end{proof}

It is worth to stress that the result in Lemma~\ref{ABC-2}
does not hold when~$p\in(1,2)$, differently than what
is often stated in the literature:
as a counterexample, one can take~$b=1$ and observe that
\begin{eqnarray*}
&& \lim_{a\to0} \frac{(a+b)^p- a^p -b^p}{a^{p-1} b}
=\lim_{a\to0} \frac{(a+1)^p- a^p -1}{a^{p-1}}\\
&& \qquad
=\lim_{a\to0} \frac{1+pa+o(a)- a^p -1}{a^{p-1}}
=\lim_{a\to0} pa^{2-p} + a^{1-p} o(a)- a=0
\end{eqnarray*}
when~$p\in(1,2)$. In spite of this additional difficulty,
when~$p\in(1,2)$ one can obtain a variant of Lemma~\ref{ABC-2}
under an additional assumption on the size of~$b$.
The precise statement goes as follows:

\begin{lemma}\label{ABC-3}
Let~$p\in(1,2)$ and~$\kappa>0$.
Then, there exists~$c_{p,\kappa}>0$ such that,
for any~$a>0$, $b\ge0$, with~$\frac{b}{a}\in[0,\kappa]$, we have
$$ (a+b)^p\ge a^p +b^p + c_{p,\kappa}\,a^{p-1} b.$$
\end{lemma}

\begin{proof}
The proof is a variation of
the one of Lemma~\ref{ABC-2}. Full details
are provided for the facility of the reader.
We set~$t_o:=\frac{b}{a}\in [0,\kappa]$.
For any~$t>0$, we let
$$ f(t):=\frac{ (1+t)^p-1-t^p }{ t }.$$
{F}rom~\eqref{9sghjtrf-2} (used here with~$a:=1$ and~$b:=t$),
we know that~$f(t)>0$ for any~$t>0$. Moreover,
$f$ can be continuously extended over~$[0,+\infty)$
by setting~$f(0):=p$. Therefore
$$ c_{p,\kappa}:=\min_{[0,\kappa]} f>0.$$
As a consequence,
\begin{eqnarray*}
&& (a+b)^p- a^p -b^p - c_{p,\kappa}\,a^{p-1} b\\
&=& a^p \big( (1+t_o)^p -1- t_o^p -c_{p,\kappa}\, t_o\big)\\
&=& a^p t_o \big( f(t_o)-c_{p,\kappa}\big)\\
&\ge& 0,\end{eqnarray*}
as desired.
\end{proof}

Now we consider the functional introduced in~\eqref{def I2},
deal with the path induced by the function~$z$
in~\eqref{tale} (suitably scaled and cut-off)
and show that the associated
mountain pass level for~$\mathcal{I}^\star_\varepsilon$
lies below the critical
threshold~$\frac{s}{n} S^{\frac{n}{2s}}$
(see Proposition~\ref{PScond2}).
The precise result goes as follows:

\begin{lemma}\label{7cu889k9}
There exists~$\mu_\star\in(0,\mu_0)$ such that if~$\mu\in(0,\mu_\star)$
then we have
\begin{equation}\label{8tyunbvc666678}
\sup_{t\ge0} \mathcal{I}^\star_\varepsilon(t\bar Z_{\mu,\xi})<
\frac{s}{n} S^{\frac{n}{2s}}.\end{equation}
\end{lemma}

\begin{proof} We will take~$\mu_\star\le \mu_1$,
where~$\mu_1>0$ was introduced in
Lemma~\ref{yu799}.
We also take~$T_1$ as in Lemma~\ref{yu799}. Then, by~\eqref{789hgshh67g000io},
\begin{equation}\label{8tyunbvc666678-pre}
\sup_{t\ge T_1} 
\sup_{\mu\in(0,\mu_\star)}
\mathcal{I}^\star_\varepsilon(t\bar Z_{\mu,\xi})
\le \sup_{t\ge T_1}
\sup_{\mu\in(0,\mu_1)}
\mathcal{I}^\star_\varepsilon(t\bar Z_{\mu,\xi}) 
\le 0
<\frac{s}{n} S^{\frac{n}{2s}}.\end{equation}
Consequently, 
we have that
the claim in~\eqref{8tyunbvc666678}
holds true if we prove that,
for any~$\mu\in(0,\mu_\star)$,
\begin{equation}\label{9ikUUYTGp}
\sup_{t\in[0,T_1]} \mathcal{I}^\star_\varepsilon(t\bar Z_{\mu,\xi})<
\frac{s}{n} S^{\frac{n}{2s}}.\end{equation}
To this goal, we set
\begin{equation}\label{CJ}
m:=\left\{
\begin{matrix}
2 & {\mbox{ if }} n>4s,\\
2^*_s-1 & {\mbox{ if }} n\in(2s,4s],
\end{matrix}
\right.
\end{equation}
and
\begin{equation}\label{CJ-2}
\Omega:=\left\{
\begin{matrix}
B_{{2}{\sqrt{\mu}}}(\xi)\setminus B_{{\sqrt{\mu}}}(\xi)
& {\mbox{ if }} n>4s,\\
\R^n & {\mbox{ if }} n\in(2s,4s].
\end{matrix}
\right.
\end{equation}
For further reference, we point out that, if~$n\in(2s,4s]$, then~$m-2=\frac{6s-n}{n-2s}>0$,
and so
\begin{equation}\label{CJ-200}
{\mbox{$m-2\ge0$ for every $n>2s$.}}
\end{equation}
We claim that, for any~$t\in [0,T_1]$,
any~$\mu\in(0,\mu_\star)$
and any~$x\in \Omega$, we have
\begin{equation}\label{FUND-ABC}
G^\star\big(x,t\bar\phi(x) z_{\mu,\xi}(x)\big)\ge
\frac{t^{2^*_s}\bar\phi^{2^*_s} (x)\,
z_{\mu,\xi}^{2^*_s}(x)}{2^*_s}+\frac{c\,
u_\varepsilon^{2^*_s-m}(x)
\,t^m\bar\phi^m (x)\,z_{\mu,\xi}^m(x)}{m},\end{equation}
for some~$c>0$.

To prove it, we distinguish two cases, according to
whether~$n>4s$ or~$n\in(2s,4s]$.
If~$n>4s$, we take~$a:=u_\varepsilon(x)$
and~$b\ge0$, with~$b\le
t\bar\phi(x) z_{\mu,\xi}(x)$, and~$x\in\Omega=B_{{2}{\sqrt{\mu}}}
(\xi)\setminus B_{{\sqrt{\mu}}}(\xi)$.
Notice that, in this case,
\begin{equation}\label{FJghy67890hgfddfg}
a\ge
\inf_{ B_{2\sqrt{\mu}}(\xi)\setminus B_{{\sqrt{\mu}}} (\xi) }
u_\varepsilon\ge
\inf_{ B_{2}(\xi)}
u_\varepsilon \ge a_0,\end{equation}
for some~$a_0>0$.
Moreover, from~\eqref{tale},
$$ b\le t z_{\mu,\xi}(x) =
t\mu^{-\frac{n-2s}{2}} z\left( \frac{x-\xi}{\mu}\right)
= \frac{c_\star t\mu^{-\frac{n-2s}{2}} }{ \left( 1+ 
\left|\frac{x-\xi}{\mu}\right|^2
\right)^{\frac{n-2s}{2}}}
= \frac{c_\star t\mu^{\frac{n-2s}{2}} }{ \left( \mu^2+ 
|x-\xi|^2\right)^{\frac{n-2s}{2}}}
.$$
Since~$x\in B_{2\sqrt{\mu}}(\xi)\setminus B_{{\sqrt{\mu}}} (\xi)$,
we obtain that~$|x-\xi|\ge {\sqrt{\mu}}$ and so
$$ b\le 
\frac{c_\star t\mu^{\frac{n-2s}{2}} }{ \left( \mu^2+
\mu\right)^{\frac{n-2s}{2}}} \le
\frac{c_\star t\mu^{\frac{n-2s}{2}} }{ 
\mu^{\frac{n-2s}{2}}} \le c_\star T_1.$$
{F}rom this and~\eqref{FJghy67890hgfddfg} we obtain that~$b/a\le\kappa$,
for some~$\kappa>0$, hence we can apply
Lemma~\ref{ABC-3} and obtain that
\begin{eqnarray*}
G^\star\big(x,t\bar\phi(x) z_{\mu,\xi}(x)\big)
&=&\int_0^{t\bar\phi(x) z_{\mu,\xi}(x)} \big[(u_\epsilon(x)+b)^p-u_\epsilon^p(x)\big]\,db
\\&=&
\int_0^{t\bar\phi(x) z_{\mu,\xi}(x)} \big[ (a+b)^p-a^p\big]\,db
\\ &\ge& \int_0^{t\bar\phi(x) z_{\mu,\xi}(x)} \big[
b^p +c_{p,\kappa} a^{p-1}b\big]\,db
\\ &=& \frac{
\big( t\bar\phi(x) z_{\mu,\xi}(x)\big)^{p+1} }{p+1}
+ c_{p,\kappa} u_\epsilon^{p-1}(x)\, \frac{\big(
t\bar\phi(x) z_{\mu,\xi}(x) \big)^2}{2}.
\end{eqnarray*}
This and~\eqref{CJ}
complete the proof of~\eqref{FUND-ABC}
when~$n>4s$ (recall that~$p+1=2^*_s$).

Now we prove~\eqref{FUND-ABC} when~$n\in(2s,4s]$.
In this case, we 
observe that
$$ p=\frac{n+2s}{n-2s}\ge 2.$$
So we choose~$a\ge0$,
with~$a\le t\bar\phi(x) z_{\mu,\xi}(x)$,
and~$b:=u_\varepsilon(x)$, and we can
use Lemma~\ref{ABC-2} to obtain that
\begin{eqnarray*}
G^\star\big(x,t\bar\phi(x) z_{\mu,\xi}(x)\big)
&=&\int_0^{t\bar\phi(x) z_{\mu,\xi}(x)} \big[
(u_\epsilon(x)+a)^p-u_\epsilon^p(x)\big]\,da
\\&=&
\int_0^{t\bar\phi(x) z_{\mu,\xi}(x)} \big[
(a+b)^p-b^p\big]\,da
\\ &\ge& \int_0^{t\bar\phi(x) z_{\mu,\xi}(x)} \big[
a^p +c_{p} a^{p-1}b\big]\,da
\\ &=&
\frac{ \big( t\bar\phi(x) z_{\mu,\xi}(x)\big)^{p+1} }{p+1}
+c_p \frac{ \big( t\bar\phi(x) z_{\mu,\xi}(x)\big)^{p} }{p}\,
u_\epsilon(x).\end{eqnarray*}
This and~\eqref{CJ}
imply~\eqref{FUND-ABC}
when~$n\in(2s,4s]$.
With this, we have completed the proof of~\eqref{FUND-ABC}.

Now we claim that,
for any~$t\in [0,T_1]$,
any~$\mu\in(0,\mu_\star)$
and any~$x\in \R^n$,
\begin{equation}\label{FUND-ABC-2}
G^\star\big(x,t\bar\phi(x) z_{\mu,\xi}(x)\big)\ge
\frac{t^{2^*_s}\bar\phi^{2^*_s} (x)\,
z_{\mu,\xi}^{2^*_s}(x)}{2^*_s}
.\end{equation}
We remark that~\eqref{FUND-ABC} is
a stronger inequality than~\eqref{FUND-ABC-2}, but~\eqref{FUND-ABC}
only holds in~$\Omega$, while~\eqref{FUND-ABC-2} holds in the whole
of~$\R^n$ (this is an advantadge in the case~$n>4s$, according to~\eqref{CJ-2}).
To prove~\eqref{FUND-ABC-2},
we use Lemma~\ref{ABC-1}, with~$a:=u_\epsilon(x)$
and~$b\ge0$, to see that
\begin{eqnarray*}
G^\star\big(x,t\bar\phi(x) z_{\mu,\xi}(x)\big)
&=&\int_0^{t\bar\phi(x) z_{\mu,\xi}(x)} \big[(u_\epsilon(x)+b)^p-u_\epsilon^p(x)\big]\,db
\\&=&
\int_0^{t\bar\phi(x) z_{\mu,\xi}(x)}\big[ (a+b)^p-a^p\big]\,db
\\ &\ge& \int_0^{t\bar\phi(x)z_{\mu,\xi}(x)} b^p\,db
\\ &=& \frac{ \big( t\bar\phi(x)\,z_{\mu,\xi}(x)\big)^{p+1} }{p+1},
\end{eqnarray*}
and this establishes~\eqref{FUND-ABC-2}.

By combining~\eqref{FUND-ABC} and~\eqref{FUND-ABC-2},
we obtain that
\begin{equation}\label{FUNF-78h}
\begin{split}
& \int_{\R^n} G^\star\big(x,t\bar\phi(x) z_{\mu,\xi}(x)\big)\,dx\\
&\qquad=
\int_{\R^n\setminus\Omega} G^\star\big(x,t\bar\phi(x) z_{\mu,\xi}(x)\big)\,dx
+
\int_{\Omega} G^\star\big(x,t\bar\phi(x) z_{\mu,\xi}(x)\big)\,dx\\
&\qquad\ge
\int_{\R^n\setminus\Omega}
\frac{t^{2^*_s}\bar\phi^{2^*_s} (x)\,
z_{\mu,\xi}^{2^*_s}(x)}{2^*_s}\,dx
\\ &\qquad\qquad+\int_{\Omega}
\frac{t^{2^*_s}\bar\phi^{2^*_s} (x)\,
z_{\mu,\xi}^{2^*_s}(x)}{2^*_s}+\frac{c\,
u_\varepsilon^{2^*_s-m}(x)
\,t^m\bar\phi^m (x)\,z_{\mu,\xi}^m(x)}{m}\,dx
\\ &\qquad\ge
\frac{t^{2^*_s}}{2^*_s}
\int_{\R^n} \bar\phi^{2^*_s} (x)\,z_{\mu,\xi}^{2^*_s}(x)\,dx
+ \frac{c\,t^m}{m}
\int_{\Omega} u_\varepsilon^{2^*_s-m}(x)
\,\bar\phi^m (x)\,z_{\mu,\xi}^m(x) \,dx.
\end{split}
\end{equation}
Now we claim that
\begin{equation}\label{89-0fgUIjjk906}
\int_{\Omega} u_\varepsilon^{2^*_s-m}(x)
\,\bar\phi^m (x)\,z_{\mu,\xi}^m(x) \,dx \ge c' \mu^\beta,
\end{equation}
for some~$c'>0$, where
\begin{equation}\label{CJ-beta}
\beta:=\left\{
\begin{matrix}
\displaystyle\frac{n}{2} & {\mbox{ if }} n>4s,\\
\, \\
\displaystyle\frac{n-2s}{2} & {\mbox{ if }} n\in(2s,4s],
\end{matrix}
\right.
\end{equation}
To prove this, when~$n>4s$
we remark that, for small~$\mu$,
we have~$B_{{2}{\sqrt{\mu}}}(\xi)\subseteq B_{\mu_0/2}(\xi)$,
and~$\bar\phi=1$ in this set, due to~\eqref{PP0dsfhuw5r6t78y009}.
So, we use~\eqref{CJ} and~\eqref{CJ-2} and we have that
\begin{eqnarray*}
&& \int_{\Omega} u_\varepsilon^{2^*_s-m}(x)
\,\bar\phi^m (x)\,z_{\mu,\xi}^m(x) \,dx
=
\int_{B_{{2}{\sqrt{\mu}}}(\xi)\setminus B_{{\sqrt{\mu}}}(\xi)}
u_\varepsilon^{2^*_s-2}(x)\,z_{\mu,\xi}^2(x) \,dx
\\ &&\qquad\ge
\inf_{B_2(\xi)} u_\varepsilon^{2^*_s-2}
\,\int_{B_{{2}{\sqrt{\mu}}}(\xi)\setminus B_{{\sqrt{\mu}}}(\xi)}
z_{\mu,\xi}^2(x) \,dx
\\ &&\qquad=
\inf_{B_2(\xi)} u_\varepsilon^{2^*_s-2}
\,\mu^{-(n-2s)}\,\int_{B_{{2}{\sqrt{\mu}}}(\xi)\setminus B_{{\sqrt{\mu}}}(\xi)}
z^2\left(\frac{x-\xi}{\mu}\right) \,dx
\\&&\qquad=
\inf_{B_2(\xi)} u_\varepsilon^{2^*_s-2}
\,\mu^{2s}\,\int_{B_{\frac{2}{\sqrt{\mu}}}
\setminus B_{\frac{1}{\sqrt{\mu}}}}
z^2(y)\,dy.
\end{eqnarray*}
Thus, recalling~\eqref{tale} and taking~$\mu$ suitably small, we have that
\begin{eqnarray*}
&& \int_{\Omega} u_\varepsilon^{2^*_s-m}(x)
\,\bar\phi^m (x)\,z_{\mu,\xi}^m(x) \,dx
\ge c_1 \mu^{2s}\int_{1/\sqrt{\mu}}^{2/\sqrt{\mu}}
\frac{\rho^{n-1}d\rho}{(1+\rho^2)^{n-2s}}
\\ &&\qquad
\ge c_1 \mu^{2s}\int_{1/\sqrt{\mu}}^{2/\sqrt{\mu}}
\frac{\rho^{n-1}d\rho}{(2\rho^2)^{n-2s}}
= c_2 \mu^{\frac{n}{2}},
\end{eqnarray*}
for some~$c_1$, $c_2>0$.
This proves~\eqref{89-0fgUIjjk906} when~$n>4s$.

Now we prove~\eqref{89-0fgUIjjk906} when~$n\in(2s,4s]$.
For this, we exploit~\eqref{CJ} and~\eqref{CJ-2}
and we observe that
\begin{eqnarray*}
&& \int_{\Omega} u_\varepsilon^{2^*_s-m}(x)\,\bar{\phi}^m(x)\, z_{\mu,\xi}^m(x)\,dx
=
\int_{\R^n} u_\varepsilon (x)\,\bar{\phi}^{2^*_s-1}(x)\, z_{\mu,\xi}^{2^*_s-1}(x)\,dx
\\ &&\qquad
\ge \mu^{-\frac{n+2s}{2}}
\int_{B_{2\sqrt{\mu}}(\xi)} u_\varepsilon(x)\, 
z^p \left(\frac{x-\xi}{\mu}\right)\,dx
\\ &&\qquad\ge
\mu^{-\frac{n+2s}{2}}
\,\inf_{B_1(\xi)}u_\varepsilon
\,\int_{B_{2\sqrt{\mu}}(\xi)} 
z^p\left(\frac{x-\xi}{\mu}\right)\,dx
\\ &&\qquad=
\mu^{\frac{n-2s}{2}}
\,\inf_{B_1(\xi)}u_\varepsilon\,
\int_{B_1} 
z^p(y)\,dy
\\ &&\qquad\ge c' \mu^{\frac{n-2s}{2}},
\end{eqnarray*}
for some~$c'>0$, which establishes~\eqref{89-0fgUIjjk906} when~$n\in(2s,4s]$.
The proof of~\eqref{89-0fgUIjjk906}
is thus complete.

Now,
by inserting \eqref{89-0fgUIjjk906}
into~\eqref{FUNF-78h}, we obtain that
\begin{equation}\label{FUNF-78hIII}
\int_{\R^n} G^\star\big(x,t\bar\phi(x) z_{\mu,\xi}(x)\big)\,dx\ge
\frac{t^{2^*_s}}{2^*_s}
\int_{\R^n} \bar\phi^{2^*_s} (x)\,z_{\mu,\xi}^{2^*_s}(x)\,dx
+ \frac{c\,t^m\,\mu^\beta}{m},\end{equation}
for some~$c>0$, up to renaming constants.

By Lemma~\ref{LE:NU90}
and~\eqref{FUNF-78hIII},
we conclude that
\begin{equation*}
\int_{\R^n} G^\star\big(x,t\bar\phi(x) z_{\mu,\xi}(x)\big)\,dx\ge
\frac{t^{2^*_s}}{2^*_s}
\int_{\R^n} \,z_{\mu,\xi}^{2^*_s}(x)\,dx
+ \frac{c\,\mu^\beta\,t^m}{m}-\frac{C\mu^{n}\,t^{2^*_s}}{2^*_s}.\end{equation*}
This and~\eqref{8dvsuyuscal} give that
\begin{equation*}
\int_{\R^n} G^\star(x,t\bar\phi z_{\mu,\xi})\,dx \ge
\frac{t^{2^*_s} \,S^{\frac{n}{2s}}}{2^*_s}+
\frac{c\,\mu^\beta\,t^m}{m}
-\frac{C \mu^n \,t^{2^*_s}}{2^*_s},\end{equation*}
for some~$c$, $C>0$.
As a consequence, recalling~\eqref{8dvsuyuscal-2},
we obtain that
\begin{eqnarray*}
\mathcal{I}^\star_\varepsilon(t\bar Z_{\mu,\xi}) &\le&
\frac{t^2\,[\bar Z_{\mu,\xi}]_a^2}{2}
-
\frac{t^{2^*_s} \,S^{\frac{n}{2s}}}{2^*_s}-
\frac{c\,\mu^\beta\,t^m}{m}
+\frac{C \mu^n \,t^{2^*_s}}{2^*_s}
\\ &\le&
\frac{t^{2} \,S^{\frac{n}{2s}}}{2}
-
\frac{t^{2^*_s} \,S^{\frac{n}{2s}}}{2^*_s}-
\frac{c\,\mu^\beta\,t^m}{m}
+\frac{C \mu^n \,t^{2^*_s}}{2^*_s}
+\frac{C \mu^{n-2s} \,t^{2}}{2},
\end{eqnarray*}
and so, up to renaming constants,
\begin{equation}\label{8dfuihgs019}
\mathcal{I}^\star_\varepsilon(t\bar Z_{\mu,\xi})
\le S^{\frac{n}{2s}} \Psi(t),\end{equation}
with
$$ \Psi(t):=
\frac{t^{2}}{2}
-\frac{t^{2^*_s} }{2^*_s}-
\frac{c\,\mu^\beta\,t^m}{m}
+\frac{C \mu^n \,t^{2^*_s}}{2^*_s}
+\frac{C \mu^{n-2s} \,t^{2}}{2},$$
for some~$C$, $c>0$.

Now we claim that
\begin{equation}\label{PSIMA}
\sup_{t\ge0}\Psi(t)<\frac{s}{n}
,\end{equation}
provided that~$\mu>0$ is suitably small.
To check this, we notice that~$\Psi(0)=0$
and
$$ \lim_{t\to+\infty}\Psi(t)=-\infty,$$
since~$2^*_s>\max\{2,m\}$ (recall~\eqref{CJ}).
As a consequence, $\Psi$ attains its maximum at some point~$T\in[0,+\infty)$.
If~$T=0$, then~$\Psi(T)=0$ and~\eqref{PSIMA}
is obvious, so we can assume that~$T\in(0,+\infty)$.
Accordingly, we have that~$\Psi'(T)=0$.
Therefore
$$ 0=\frac{\Psi'(T)}{T}=
1-T^{2^*_s-2}-
c\,\mu^\beta\,T^{m-2}
+C \mu^n \,T^{2^*_s-2}
+C \mu^{n-2s}.$$
So we set
$$ \Phi_\mu(t):=
1-t^{2^*_s-2}-
c\,\mu^\beta\,t^{m-2}
+C \mu^n \,t^{2^*_s-2}
+C \mu^{n-2s}$$
and we have that~$T=T(\mu)$ is a solution
of~$\Phi_\mu(T)=0$.
We remark that
$$ \Phi_\mu'(t)=
-(2^*_s-2)(1-C \mu^n)t^{2^*_s-3}-
c\,\mu^\beta\,(m-2) t^{m-3} <0,$$
since~$m-2\ge0$ and~$(2^*_s-2)(1-C \mu^n)\ge0$
for small~$\mu$ (recall~\eqref{CJ-200}). This says that~$\Phi_\mu$ is strictly decreasing,
hence~$T=T(\mu)$ is the unique solution of~$\Phi_\mu(T(\mu))=0$.
It is now convenient to write~$\tau(\mu):=T(\mu^{\frac{1}{\beta}})$
and~$\eta:=\mu^\beta$,
so that our equation becomes
\begin{eqnarray*}
&& 0 =\Phi_\mu(T(\mu)) = \Phi_\mu (\tau(\mu^\beta))=\Phi_\mu(\tau(\eta))
\\ &&\qquad=
1-(1-C \eta^{\frac{n}{\beta}})(\tau(\eta))^{2^*_s-2}-
c\,\eta\,(\tau(\eta))^{m-2}
+C \eta^{\frac{n-2s}{\beta}}.
\end{eqnarray*}
Accordingly, if we differentiate in~$\eta$, we have that
\begin{equation}\label{0syufertyuiuytrgb}\begin{split}
& 0=\frac{\partial}{\partial \eta}\left(
1-(1-C \eta^{\frac{n}{\beta}})(\tau(\eta))^{2^*_s-2}-
c\,\eta\,(\tau(\eta))^{m-2}
+C \eta^{\frac{n-2s}{\beta}}
\right) \\
&\qquad=
-(2^*_s-2)(1-C \eta^{\frac{n}{\beta}})(\tau(\eta))^{2^*_s-3}\tau'(\eta)
+C\,\frac{n}{\beta} \eta^{\frac{n}{\beta}-1} 
(\tau(\eta))^{2^*_s-2}
\\ &\qquad\qquad-
c\,(\tau(\eta))^{m-2}
-
c(m-2)\,\eta\,(\tau(\eta))^{m-3}\tau'(\eta)
+{\frac{C(n-2s)}{\beta}} \eta^{\frac{n-2s}{\beta}-1}.
\end{split}\end{equation}
Now we claim that
\begin{equation}\label{BETA78HGV}
\frac{n-2s}{\beta}-1 >0.
\end{equation}
Indeed, using~\eqref{CJ-beta}, we see that
\begin{eqnarray*}
\frac{n-2s}{\beta}-1 &=& 
\left\{
\begin{matrix}
\displaystyle\frac{2(n-2s)}{n}-1 & {\mbox{ if }} n>4s,\\
\, \\
2-1 & {\mbox{ if }} n\in(2s,4s],
\end{matrix}\right.\\
&=&
\left\{
\begin{matrix}
\displaystyle\frac{n-4s}{n} & {\mbox{ if }} n>4s,\\
\, \\
1 & {\mbox{ if }} n\in(2s,4s],\end{matrix}\right.
\end{eqnarray*}
which proves~\eqref{BETA78HGV}.

Now we observe that when~$\mu=0$, we have that~$T=1$ is a solution
of~$\Phi_0(t)=0$, i.e.~$T(0)=1$ and so~$\tau(0)=1$.
Hence, we evaluate~\eqref{0syufertyuiuytrgb} at~$\eta=0$
and we conclude that
$$ 0=
-(2^*_s-2)\tau'(0) - c.$$
We remark that~\eqref{BETA78HGV}
was used here.
Then, we obtain
$$ \tau'(0)=-\frac{c}{2^*_s-2},$$
which gives that
$$ \tau(\eta) =1-\frac{c\eta}{2^*_s-2} + o(\eta)$$
and so
$$ T(\mu)=\tau(\mu^\beta)=1-\frac{c\mu^\beta}{2^*_s-2} + o(\mu^\beta)=
1-c_o\mu^\beta+ o(\mu^\beta),$$
for some~$c_o>0$.
Therefore
\begin{eqnarray*}
&& \sup_{t\ge0}\Psi(t) =\Psi(T(\mu))\\
&&\qquad =
(1+C \mu^{n-2s})\frac{(T(\mu))^{2}}{2}
-(1-C \mu^n)\frac{(T(\mu))^{2^*_s} }{2^*_s}-
\frac{c\,\mu^\beta\,(T(\mu))^m}{m}\\
&&\qquad=
(1+C \mu^{n-2s})\frac{(1-c_o\mu^\beta+ o(\mu^\beta))^{2}}{2}
-(1-C \mu^n)\frac{(1-c_o\mu^\beta+ o(\mu^\beta))^{2^*_s} }{2^*_s}
\\ &&\qquad\qquad-
\frac{c\,\mu^\beta\,(1-c_o\mu^\beta+ o(\mu^\beta))^m}{m}\\
&&\qquad=
(1+C \mu^{n-2s})\frac{1-2c_o\mu^\beta}{2}
-(1-C \mu^n)\frac{1-2^*_sc_o\mu^\beta}{2^*_s}-
\frac{c\,\mu^\beta}{m}
+ o(\mu^\beta)\\
&&\qquad=
\frac{1-2c_o\mu^\beta}{2}
-\frac{1-2^*_sc_o\mu^\beta}{2^*_s}-
\frac{c\,\mu^\beta}{m}
+ o(\mu^\beta)
\\ &&\qquad=
\frac{1}{2}-\frac{1}{2^*_s}
-\frac{c\,\mu^\beta}{m}
+ o(\mu^\beta)
\\ &&\qquad<
\frac{1}{2}-\frac{1}{2^*_s}
\\ &&\qquad=\frac{s}{n}
,\end{eqnarray*}
and this proves~\eqref{PSIMA}.

Using~\eqref{8dfuihgs019}
and~\eqref{PSIMA}, we obtain that
$$ \sup_{t\in[0,T_1]}
\mathcal{I}^\star_\varepsilon(t\bar Z_{\mu,\xi})
\le S^{\frac{n}{2s}} \sup_{t\ge0} \Psi(t)<
S^{\frac{n}{2s}}\cdot\frac{s}{n},$$
which proves~\eqref{9ikUUYTGp}
and so it completes the proof of Lemma~\ref{7cu889k9}.
\end{proof}

By combining \eqref{789hgbnjjjjjKK}
with Lemmata~\ref{yu799} and~\ref{7cu889k9}, we obtain:

\begin{corollary}\label{CC-below-PS}
There exists~$\mu_\star>0$ such that if~$\mu\in(0,\mu_\star)$
we have that
\begin{eqnarray*}
&& \lim_{t\to+\infty}
\mathcal{I}_\varepsilon(t\bar Z_{\mu,\xi}) =-\infty
\\{\mbox{and }}&&
\sup_{t\ge0}
\mathcal{I}_\varepsilon(t\bar Z_{\mu,\xi})<
\frac{s}{n} S^{\frac{n}{2s}}.\end{eqnarray*}
\end{corollary}

The result in Corollary~\ref{CC-below-PS}
says that the path induced by the function~$\bar Z_{\mu,\xi}$
is a mountain pass path which
lies below the critical threshold given in Proposition~\ref{PScond2}
(so, from now on, the value of~$\mu\in(0,\mu_\star)$
will be fixed so that
Corollary~\ref{CC-below-PS} holds true).

\section{Proof of Theorem \ref{TH:MP}}\label{sec:proof}

In this section we establish Theorem \ref{TH:MP}. 
For this, we argue by contradiction and we suppose that $U=0$ is the only critical point 
of $\mathcal{I}_\epsilon$. As a consequence, 
the functional~$\mathcal{I}_\epsilon$
verifies 
the Palais-Smale condition below the critical level given in~\eqref{ceps2}, according to Proposition \ref{PScond2}.

In addition,~$\mathcal{I}_\epsilon$ fulfills the mountain pass geometry,
and the minimax level~$c_\epsilon$ stays strictly
below
the level~$\frac{s}{n}S^{\frac{n}{2s}}$, as shown in
Proposition~\ref{prop:zero}
and
Corollary~\ref{CC-below-PS}.

Hence, for small~$\epsilon$, we have that~$
c_\epsilon+C \epsilon^{\frac{1}{2\gamma}}$ remains strictly below~$
\frac{s}{n}S^{\frac{n}{2s}}$, thus satisfying~\eqref{ceps2}.

Then, exploiting Proposition~\ref{PScond2}
and the Mountain Pass Theorem in~\cite{AmRab, GG},
we obtain the existence of another critical point,
in contradiction with
the assumption.
This ends the proof of
Theorem \ref{TH:MP}.

\end{document}